\documentclass[]{amsart}

\usepackage{amssymb}
\usepackage{amsmath}
\usepackage{amsthm}
\usepackage{graphics}
\usepackage{booktabs}

\newcommand{\diag}{\mathop{\mathrm {diag}}\nolimits}

\newcommand{\Ad}{\mathop{\mathrm {Ad}}\nolimits}
\newcommand{\adj}{\mathop{\mathrm {ad}}\nolimits}
\newcommand{\Hom}{\mathop{\mathrm {Hom}}\nolimits}

\newcommand{\sgn}{\mathop{\mathrm {sgn}}\nolimits}

\newcommand{\Sym}{\mathop{\mathrm {Sym}}\nolimits}
\newcommand{\id}{\mathop{\mathrm {id}}\nolimits}

\newcommand{\sI}{\sqrt{-1}}

\newcommand{\cC}{\mathcal{C}}

\newcommand{\cE}{\mathcal{E}}

\newcommand{\cI}{\mathcal{I}}
\newcommand{\cJ}{\mathcal{J}}

\newcommand{\cR}{\mathcal{R}}

\newcommand{\cT}{\mathcal{T}}

\newcommand{\rB}{\mathrm{B}}

\newcommand{\rI}{\mathrm{I}}

\newcommand{\rM}{\mathrm{M}}

\newcommand{\rP}{\mathrm{P}}

\newcommand{\rf}{\mathrm{f}}

\newcommand{\rk}{\mathrm{k}}

\newcommand{\rrq}{\mathrm{q}}

\newcommand{\bA}{\mathbb{A}}
\newcommand{\bC}{\mathbb{C}}

\newcommand{\bH}{\mathbb{H}}

\newcommand{\bQ}{\mathbb{Q}}
\newcommand{\bR}{\mathbb{R}}
\newcommand{\bZ}{\mathbb{Z}}

\newcommand{\mk}{\mathbf{k}}

\newcommand{\mn}{\mathbf{n}}
\newcommand{\mmp}{\mathbf{p}}

\newcommand{\gH}{\mathfrak{H}}

\newcommand{\gS}{\mathfrak{S}}
\newcommand{\ga}{\mathfrak{a}}

\newcommand{\g }{\mathfrak{g}}

\newcommand{\gk}{\mathfrak{k}}
\newcommand{\gl}{\mathfrak{l}}

\newcommand{\gn}{\mathfrak{n}}

\newcommand{\gp}{\mathfrak{p}}
\newcommand{\gs}{\mathfrak{s}}

\newtheorem{thm}{Theorem}[section]
\newtheorem{lem}[thm]{Lemma}
\newtheorem{prop}[thm]{Proposition}

\newtheorem{rem}[thm]{Remark}

\numberwithin{equation}{section}
\theoremstyle{remark}

\newcommand{\GR}{\Gamma_{\bR}}
\newcommand{\GC}{\Gamma_{\bC}}

\newcommand{\pa}{{\partial}}
\newcommand{\vp}{{\varphi}}

\newcommand{\s}{{\sigma}}

\newcommand{\varep}{{\varepsilon}}

\begin{document}
\allowdisplaybreaks


\title[Whittaker functions on $\mathrm{GL}(4,\bR )$]
{Whittaker functions on $\mathrm{GL}(4,\bR )$ and archimedean Bump-Friedberg integrals}

\author{Miki Hirano, Taku Ishii, and Tadashi Miyazaki}

\address{(M.H.) Department of Mathematics, Faculty of Science, Ehime University,
2-5, Bunkyo-cho, Matsuyama, Ehime, 790-8577, Japan}
\email{hirano.miki.mf@ehime-u.ac.jp}

\address{(T.I.) Faculty of Science and Technology,
Seikei University, 3-3-1 Kichijoji-Kitamachi,
Musashino,
Tokyo, 180-8633, Japan}
\email{ishii@st.seikei.ac.jp}

\address{(T.M.) Department of Mathematics, College of Liberal Arts and Sciences, Kitasato University, 1-15-1 Kitasato, Minamiku, Sagamihara, Kanagawa, 252-0373,
Japan}
\email{miyaza@kitasato-u.ac.jp}


\footnote{This work was supported by JSPS KAKENHI Grant Numbers JP18K03252, JP19K03450, JP19K03452.}

\maketitle 

\begin{abstract}
We give explicit formulas of Whittaker functions on ${\rm GL}(4,\bR)$ for all irreducible generic representations.
As an application, we determine test vectors which attain the associated $L$-factors for Bump-Friedberg integrals on ${\rm GL}(4,\bR)$. 
\end{abstract}

\section*{Introduction}

The study of Whittaker functions of irreducible generic representations of ${\rm GL}(n)$ are not only meaningful from a representation theoretical point of view but also important in advancing the theory of automorphic forms via their Fourier expansions. 
In our previous paper \cite{HIM}, the authors gave explicit formulas of archimedean Whittaker functions on ${\rm GL}(3) $ 
and applied them to archimedean theory of automorphic $L$-functions on $ {\rm GL}(3) \times {\rm GL}(2) $.
The present paper describes explicit formulas of Whittaker functions on ${\rm GL}(4,\bR) $ and their application 
to the standard and the exterior square $L$-functions on $ {\rm GL}(4) $. 
Our main results include the following.
\begin{itemize}
\item We give explicit formulas of moderate growth Whittaker functions on ${\rm GL}(4,\bR)$ 
for all irreducible generic representations at their minimal ${\rm O}(4) $-types (\S 5).
\item We determine the test vectors (see below) for Bump-Friedberg integrals on ${\rm GL}(4,\bR)$ (\S 6).
\end{itemize}

\medskip 

Let $ \Pi = \otimes'_v \Pi_v $ be a cuspidal automorphic representation of ${\rm GL}(n,{\bA}) $,
where $ \bA $ is the adele ring of $ \bQ $. 
As is well known (\cite{Shalika_001}), Fourier expansion of a cusp form in $\Pi $ 
can be written in terms of the global Whittaker function.
Then $ \Pi $ has global Whittaker model and thus each local component $ \Pi_v $ also has local Whittaker model,
that is, $ \Pi_v $ is generic.
According to the result of Vogan \cite{Vogan_001}, an irreducible generic representation 
$ \Pi_{\infty} $ of ${\rm GL}(n,\bR)$ is isomorphic to an irreducible generalized principal series representation induced from 
a parabolic subgroup corresponding to a partition $ (2,\ldots, 2, 1, \ldots, 1) $ of $n$.

Let $ (\tau, V_{\tau}) $ be a multiplicity one ${\rm O}(n)$-type of an irreducible generalized principal series representation $\Pi_{\infty} $ of ${\rm GL}(n,\bR)$. 
The local multiplicity one theorem asserts that there exists a unique (up to constant multiple) 
${\rm O}(n)$-embedding $ \varphi $ from $ V_{\tau} $ to the space 
 $ {\rm Wh}(\Pi_{\infty}, \psi)^{\rm mg} $ of moderate growth Whittaker functions with a character of the group of upper triangular matrices with diagonal entries equal to $1$. 
For this embedding $\varphi$ and each $v\in V_{\tau}$, the Whittaker function $ \varphi(v) $ can be regarded as a smooth function on $ (\bR_+)^{n-1} $. 
See \S \ref{subsec:Fn_def_whittaker} for the precise.
There are two ways to arrive at an explicit formula of Whittaker function: 
\begin{itemize}
\item  
Analysis of a system of partial differential equations satisfied by Whittaker functions.
\item 
Manipulation of Jacquet integrals of Whittaker functions. 
\end{itemize}
In either case, we may face the following fundamental problems. 
\begin{itemize}
\item[(a)] How to describe representation theory of $ {\rm O}(n) $?
\item[(b)] How to describe special functions of $(n-1)$ variables? 
\end{itemize}

\medskip 

When $ \Pi_{\infty} $ is the class one principal series representation with the minimal ${\rm O}(n)$-type $(\tau,V_{\tau})$, the first problem (a) can be ignored, since $(\tau,V_{\tau})$ is trivial and (scalar-valued) Whittaker function can be obtained as an eigen-function of Capelli elements. 
With regard to the second problem (b), 
Bump (\cite{Bump}, analysis of partial differential equations for $n=3$) and 
Stade (\cite{Stade_001}, \cite{Stade_002}, evaluation of Jacquet integrals for general $n$)
obtained Mellin-Barnes integral representations for Whittaker functions of the form
\begin{align*}
  \varphi(v) (y_1,\dotsc, y_{n-1})  
& = \frac{1}{(4\pi \sqrt{-1})^{n-1}} \int_{s_{n-1}} \!\! \cdots \int_{s_1} V_n(v; s_1,\dotsc,s_{n-1}) 
   \, y_1^{-s_1} \cdots y_{n-1}^{-s_{n-1}} \, ds_1\cdots ds_{n-1}.
\end{align*}
Bump expressed $ V_3(v;s_1,s_2) $ as a ratio of gamma functions. Stade's formula, which is a recursive integral formula between $V_{n}$ and $ V_{n-2} $,  
implies that the Mellin-Barnes kernel $ V_{n}$ can not be written as a ratio of gamma functions when $n \ge 4$.
Moreover, based on the result of \cite{Stade_002},  
Stade and the second author \cite{Ishii_Stade_001} found a recursive relation between $ V_{n} $ and $V_{n-1}$, 
and gave a fundamental solution to the system of partial differential equations 
satisfied by the class one principal series Whittaker functions.

Oda and the second author \cite{Ishii_Oda_001} extended the class one result of \cite{Ishii_Stade_001} 
to cases of principal series representations induced from the parabolic subgroup of type $(1,\dotsc,1)$. 
In addition to Capelli equations, so-called ``Dirac-Schmid equations'' are needed to characterize Whittaker functions. 
To derive Dirac-Schmid equations which are consequence of $ (\mathfrak{gl}(n,\bR), {\rm O}(n)) $-structures 
of $\Pi_{\infty} $ around the minimal ${\rm O}(n)$-type $ (\tau, V_{\tau}) $, 
we need matrix elements of $ (\tau, V_{\tau}) $ and
some Clebsch-Gordan coefficients of the tensor product. 
Since the minimal ${\rm O}(n)$-type of the principal series representation 
is $k$-th exterior product $ \wedge^k \bC^n $ of the standard $n$-dimensional representation for some $ 0 \le k \le n $, 
it is not hard to describe a system of partial differential equations for Whittaker functions 
by using a {\it basis} of $ \wedge^k \bC^n $. 
The fundamental solution of the system and Mellin-Barnes integral representations of 
moderate growth Whittaker functions are given in \cite{Ishii_Oda_001}.

Note that, in the recent paper \cite{Ishii_Miyazaki_001}, the second and the third authors noticed that the formulas in 
\cite{Ishii_Stade_001} and \cite{Ishii_Oda_001} can be derived from Jacquet integrals for Godement sections
of the principal series representations.

\medskip 

Leaving the principal series representations, we are faced with the problem (a).
Our idea is a continuation of the line taken in \cite{HIM}. 
When $ n= 4 $, the remaining generic representations are the generalized principal series induced 
from parabolic subgroups of types $ (2,1,1) $ and $(2,2)$. 
Though we may use the Gelfand-Tsetlin basis of $ {\rm O}(4) $, the matrix elements are very complicated
and Clebsch-Gordan coefficients are less understood (cf. \cite{Ishii_Miyazaki_001}). 
To overcome the difficulties, the most significant idea in this paper is not to use {\it basis}, but to use  {\it generators}.
This idea has already been used in the case $n=3$ in \cite{HIM}, but 
our construction of generators requires more effort for an irreducible representation of ${\rm O}(4)$. 
Despite the many relations among generators, 
we can enjoy great benefits from the simplified expressions of $ {\rm O}(4) $-actions via generators.
We note that our generator becomes a basis when $ \wedge^k \bC^4 $, and hence, there was no
 need for this idea in the principal series situations. 

After establishing representation theory of ${\rm O}(4) $ in \S \ref{sec:rep_K}, we determine 
$ (\mathfrak{gl}(4,\bR), {\rm O}(4))$-structures of (generalized) principal series representations 
around their minimal ${\rm O}(4)$-types in \S \ref{sec:gps}.
As a consequence, we can get a system of partial differential equations for Whittaker functions in Proposition \ref{prop:PDE2}.

\medskip

Since we concentrate on moderate growth Whittaker functions, 
our approach for the system of partial differential equations in this paper is different from \cite{Ishii_Oda_001}. 
As in \cite{HIM}, through the reduction of the system, with the aid of the general theory of Whittaker functions, 
we want to show the following including the case of principal series.
\begin{itemize}
\item For some specific vector $ v_0 \in V_{\tau} $, Whittaker function $\varphi(v_0) $ satisfies essentially the same 
system of partial differential equations as the class one principal series Whittaker functions.
\item 
We can determine $ \varphi(v) $ for all $ v \in V_{\tau} $ from the function $ \varphi(v_0) $,
and our system in Proposition \ref{prop:PDE2} characterizes Whittaker functions $ \{ \varphi(v) \mid v \in V_{\tau} \} $.
\end{itemize}
However, our approach via differential equations faces a new problem that did not arise in the case $n=3$.
In a few exceptional cases (cf. Proposition \ref{prop:PDE_reduction_classone}), our Dirac-Schmid equations can not distinguish between two representations
that are not isomorphic, that is, 
Whittaker functions satisfy the same system of partial differential equations for those representations.
This phenomenon did not occur in our previous work \cite{HIM} and
we overcome this new obstruction by using Jacquet integral (\S \ref{subsec:Jacquet}).

\medskip

We remark that our argument includes new proof for the principal series Whittaker functions  
and does not rely on the results of \cite{Stade_001} and \cite{Ishii_Oda_001}.
In view of explicit formulas of the class one principal series Whittaker functions in 
\cite{Stade_001} and \cite{Ishii_Stade_001}, we have two possibilities to express Mellin-Barnes kernel
(cf. Proposition \ref{prop:classone_MB}).
In this paper, we adopt the same expression as in \cite{Stade_001} which is a natural extension of the formulas given in \cite{HIM}.
Moreover, our integral representation here works well with application to the Bump-Friedberg integral.
See Theorems \ref{thm:EF1}, \ref{thm:EF2} and \ref{thm:EF3} for our first main results.

\medskip 

The last section of this paper is devoted to an application to archimedean zeta integrals. 
Thanks to Fourier expansions via Whittaker functions, various zeta integrals for automorphic $L$-functions on ${\rm GL}(n)$ unfold to Whittaker models, and the associated local zeta integrals become integral transforms of Whittaker functions. 
Thus, explicit formulas of Whittaker functions are essential in precisely evaluating the integral transformations. 
Also, we believe explicit formulas of Whittaker functions will serve as an important step toward number theoretic applications not only computations of archimedean zeta integrals (cf. \cite{Buttcane}, \cite{GSW}).

If the local zeta integral is equal to the expected local $L$-factor, 
such identity will be a strong tool for arithmetic properties of automorphic $L$-functions (cf. \cite{HMN}). 
But such coincidence can not be expected in general.  
For example, archimedean zeta integral for Rankin-Selberg $L$-function on $ {\rm GL}(n) \times {\rm GL}(m) $ 
is not expected to attain exactly the $L$-factor when $ |n-m| \ge 2 $. 
On the other hand, we can expect that such a coincidence occurs when $|n-m|=1 $ as in our previous paper \cite{HIM}. 
We call the local datum such as Whittaker functions attaining the expected $L$-factors {\it test vectors} for the local zeta integrals. 
It is widely open regarding archimedean test vectors even their existence (cf. \cite{HJ}).

Our target in this paper is the Bump-Friedberg integral. 
When $ n =2m$, the archimedean Bump-Friedberg integral is
\begin{align*}
 Z(s_1,s_2,W,\Phi) & = \int_{N_{m} \backslash {\rm GL}(m,\bR)} \int_{N_{m} \backslash {\rm GL}(m,\bR)}
 W( \tilde{\iota} (g_1,g_2) ) \Phi( (0,\dotsc,0,1) g_2) 
\\
& \times |\det g_1|^{s_1-1/2}  |\det g_2|^{-s_1+s_2+1/2} \,dg_1 dg_2,
\end{align*}
where $ N_m$ is the maximal unipotent subgroup of ${\rm GL}(m,\bR) $ consisting of upper triangular matrices, 
$ W $ is a Whittaker functions for $ \Pi_{\infty} $, 
$ \Phi $ is a Schwartz-Bruhat function on $\bR^{m} $,
and $  \tilde{\iota}  $ is an embedding $ {\rm GL}(m,\bR) \times {\rm GL}(m,\bR) \to {\rm GL}(n,\bR)$.
See \cite{Ishii_001} for the precise. 

By unramified computation \cite[Theorem 3]{Bump_Friedberg_001}, 
the archimedean zeta integral above is expected to be related to 
the the product $ L(s_1, \Pi_{\infty})  L(s_2,\Pi_{\infty},\wedge^2) $ of the standard and the exterior square $L$-factors 
for $\Pi_{\infty} $.
When $ \Pi_{\infty} $ is the class one principal series representation, Stade \cite[Theorem 3.3]{Stade_002} proved 
the identity
$$
 Z(s_1, s_2, W, \Phi) = L(s_1, \Pi_{\infty}) L(s_2, \Pi_{\infty}, \wedge^2)
$$
for the class one principal series Whittaker function $W$ and the Gaussian $\Phi$. 
Then we can expect the existence of test vectors $ (W, \Phi) $ for other generic representations. 
As an extension of \cite{Stade_002}, the second author \cite{Ishii_001} gave test vectors explicitly 
for the principal series representation of ${\rm GL}(n,\bR)$ by using the explicit formulas in \cite{Ishii_Oda_001}. 

\medskip 

For the generalized principal series representations, as in \cite{HIM} and \cite{Ishii_001}, 
finding test vectors is a much harder task than the principal series representations (cf. \cite{Lin}).
Similar to the study of explicit formulas, the calculus of archimedean zeta integrals requires the computation of 
two objects:  
integrations over compact groups, and Mellin transforms of Whittaker functions.

When $ n= 4$, our compact group is ${\rm O}(2) \times {\rm O}(2)$.
Unfortunately, as in \cite{HIM}, the integration over ${\rm O}(2) \times {\rm O}(2) $ vanishes 
for the minimal ${\rm O}(4)$-type Whittaker function $W$ in many cases. 
Then we move ${\rm O}(4)$-type of Whittaker functions by applying differential operators, 
so that the integration over ${\rm O}(2) \times {\rm O}(2)$ does not vanish. 
Further, for that Whittaker function $W$, we need to understand right $ {\rm O}(2) \times {\rm O}(2) $ 
translation of $W$ precisely.
Here we can again benefit from the use of generators of $ V_{\tau} $.
  
After the integration over $ {\rm O}(2) \times {\rm O}(2) $, 
the zeta integral $ Z(s_1,s_2,W,\Phi) $ is reduced to a linear combination of special values $ V_4(v; s_1,s_2,s_1+s_2) $
of Mellin transform of Whittaker functions. 
Based on our explicit formulas in \S \ref{sec:EF}, with the aid of Barnes' first and second lemmas, 
we can proceed with our computation.
But in general, the integral $ Z(s_1,s_2, W,\Phi) $ is expected to become 
$  P(s_1,s_2) L(s_1, \Pi_{\infty}) L(s_2, \Pi_{\infty}, \wedge^2) $ with some polynomial $ P(s_1,s_2) $. 
The test vector problem is nothing but to find $ (W, \Phi )$ so that $P(s_1,s_2) = 1 $, and  
as far as we know, there is no guiding principle for this problem. 
After trial and error, we reach test vectors for Bump-Friedberg integrals.   
See Theorems \ref{thm:BF1}, \ref{thm:BF2} and \ref{thm:BF3} for our second main results.  

\medskip 

Since various archimedean zeta integrals are evaluated 
for the class one principal series representations (cf. \cite{Stade_002}), there will be plenty of room for study.
We hope our example opens some windows to test vector problems for ${\rm GL}_n $-integrals. 
Especially, the use of generators, which is our important idea in this paper, 
will contribute to the further development of the archimedean theory of automorphic $L$-functions.

\section{Basic objects}
\label{sec:Fn_basic_objects}




\subsection{Groups and algebras}
\label{subsec:group_algebra}
Throughout this paper we write  $ 1_n $ the unit matrix of degree $n$.  
The zero matrix of size $m\times n$ is denoted by $ O_{m,n} $.
Let $G=\mathrm{GL}(4,\bR )$ be the general linear group 
of degree $4$ over $\bR$. 
Let $N$ be the group of 
upper triangular matrices in $G$ 
with diagonal entries equal to $1$. 
Let $A$ be the group 
of diagonal matrices in $G$ with positive diagonal entries. 
Let $K=\mathrm{O}(4)$ be the orthogonal group of degree $4$. 
Then $K$ is a maximal compact subgroup of $G$, 
and we have an Iwasawa decomposition $G=NAK$ of $G$. 
It is convenient to fix 
the coordinates on $N$ and $A$ as follows:
\begin{align*}
&x=\begin{pmatrix}
1&x_{1,2}&x_{1,3}&x_{1,4}\\
0&1&x_{2,3}&x_{2,4}\\
0&0&1&x_{3,4}\\
0&0&0&1
\end{pmatrix} \in N,&
&y=\diag (y_1y_2y_3y_4, y_2y_3y_4, y_3y_4, y_4)\in A,
\end{align*}
where $x_{i,j}\in \bR \ (1\le i<j\le 4)$ and 
$y_k\in \bR_+\ (1\leq k\leq 4)$. 
Here $\bR_+$ means the set of positive real numbers.

Let 
$G_1=\mathrm{GL}(1,\bR )=\bR^\times$ and 
$G_2=\mathrm{GL}(2,\bR )$. 
We fix an Iwasawa decomposition 
$G_2=N_2A_2K_2$ of $G_2$ with 
\begin{align*}
&N_2=\left\{
 \begin{pmatrix} 1&x_{1,2}\\ 0&1 \end{pmatrix} 
\biggl|  \, x_{1,2}\in \bR \right\},&
&A_2=\{\diag (a_1, a_2)\mid a_1,a_2\in \bR_+\},&
&K_2=\mathrm{O}(2).
\end{align*}
We define an embedding $\iota \colon G_2\times G_2\to G$ by 
\begin{align*}
&\iota (g_1,g_2)=
\begin{pmatrix}
g_1&O_{2,2}\\
O_{2,2}&g_2
\end{pmatrix}&
&(g_1,g_2\in G_2). 
\end{align*}
For $\theta ,\theta_1,\theta_2\in \bR$, 
we set 
\begin{align*}
&\rk_\theta^{(2)} =
\begin{pmatrix}
\cos \theta &\sin \theta \\
-\sin \theta &\cos \theta
\end{pmatrix} \in K_2,&
&\rk_{\theta_1,\theta_2}^{(2,2)}
=\iota (\rk_{\theta_1}^{(2)},\rk_{\theta_2}^{(2)} )
\in K.
\end{align*}

Let $\g$, $\gn$, $\ga$ and $\gk$ be 
the associated Lie algebras of $G$, $N$, $A$ and $K$, respectively. 
In this article, we identify the complexification $\g_{\bC}$ of $\g$ 
with $\g \gl (4,\bC)$. 
Let $\gp$ be the orthogonal complement of $\gk$ in $\g$ 
with respect to the Killing form, that is, 
$\gp =\{ X\in \g \mid {}^tX=X\}$. 
The complexification of $\gp$ is denoted by $\gp_{\bC}$. 
The adjoint action of $G$ on $\g_{\bC}$ and 
its differential are denoted by Ad and ad, respectively.

Let $L$ be a Lie subgroup of $G$, and $\gl$ 
the associated Lie algebra of $L$. 
Let $\gl_{\bC}$ be the complexification of $\gl$.
The universal enveloping algebra of $\gl_\bC$ and its center are 
denoted by $U(\gl_{\bC} )$ and $Z(\gl_{\bC})$, respectively. 
We identify $\gl$ with the space of left $L$-invariant 
vector fields on $L$ in the usual way, that is, 
\begin{align*}
(R(X)f)(g)& =\frac{d}{dt}\biggl|_{t=0} f(g\exp (tX)) & (g\in L)
\end{align*}
for $X\in \gl$ and a differentiable function $f$ on $L$. 
Then $U(\gl_{\bC} )$ is identified with the algebra of 
left $L$-invariant differential operators on $L$. 
Let $C^\infty (L)$ be 
the space of smooth functions on $L$. 
We equip the space $C^\infty (L)$ with the topology 
of uniform convergence on compact sets of a function 
and its derivatives. 

For $1\leq i,j\leq 4$, let
$E_{i,j}$ be the matrix in $\g$ 
with $1$ at the $(i,j)$-th entry and $0$ at other entries. 
We set 
\begin{align*}
&E^{{\gk}}_{i,j}=E_{i,j}-E_{j,i},&
&E^{{\gp}}_{i,j}=E_{i,j}+E_{j,i}&
&(1\leq i,j\leq 4).
\end{align*}
Then $\{E_{i,j}\mid 1\leq i,j\leq 4\}$, 
$\{E_{i,j}^{\gk}\mid 1\leq i<j\leq 4\}$ and 
$\{E_{i,j}^{\gp}\mid 1\leq i\leq j\leq 4\}$ are bases of 
$\g$, $\gk$ and $\gp$, respectively. 

We define a matrix $\cE=(\cE_{i,j})_{1\leq i,j\leq 4}$ of size $4$ 
with entries in $U(\g_\bC )$ by 
\begin{align*}
\cE_{i,j} 
& = \begin{cases} 
E_{i,i}-\tfrac{5-2i}{2}&\text{if $i=j $}, \\
E_{i,j}&\text{if $i\neq j$}.
\end{cases}
\end{align*}
We define the Capelli elements 
$\cC_1$, $\cC_2$, $\cC_3$, $\cC_4$ by 
the identity 
\begin{align*}
\mathrm{Det}(t1_4+\cE )
=t^4+\cC_1t^3+\cC_2t^2+\cC_3t+\cC_4
\end{align*}
in a variable $t$. Here $\mathrm{Det}$ means the vertical determinant 
defined by 
\begin{align*}
&\mathrm{Det}(X)=\sum_{w\in \gS_4}\sgn (w)
X_{1,w(1)}X_{2,w(2)}X_{3,w(3)}X_{4,w(4)},&
&X=(X_{i,j})_{1\leq i,j\leq 4}
\end{align*}
with the symmetric group $\gS_4$ of degree $4$. 
It is known that the Capelli elements $\cC_1$, $\cC_2$, $\cC_3$, $\cC_4$ 
generate $Z(\g_\bC )$ as a $\bC$-algebra 
(\textit{cf}. \cite[\S 11]{MR1116239}).

\subsection{Whittaker functions}
\label{subsec:Fn_def_whittaker}

We define the standard character $\psi_{\bR} \colon \bR \to \bC^\times$ by 
\begin{align*}
\psi_{\bR} (t) &=\exp (2\pi \sI t) 
& (t\in \bR),
\end{align*}
and a character $\psi_{(c_1,c_2,c_3)}$ of $N$ by 
\begin{align*}
&\psi_{(c_1,c_2,c_3)}(x)
=\psi_{\bR} (c_1x_{1,2}+c_2x_{2,3}+c_3x_{3,4})&
&(x=(x_{i,j})\in N) 
\end{align*}
for $(c_1,c_2,c_3)\in \bR^3$. 
Then unitary characters of $N$ are exhausted 
by the characters of this form. 
We say that $\psi_{(c_1,c_2,c_3)}$ is non-degenerate 
if $(c_1,c_2,c_3)\in (\bR^\times )^3$. 
For $c\in \bR$, 
the character $\psi_{(c ,c ,c )}$ is simply denoted by $\psi_c$.

We regard $C^\infty (G)$ as a $G$-module 
via the right translation. 
For a non-degenerate character $\psi$ of $N$, 
let $C^\infty (N\backslash G;\psi )$ be 
the subspace of $C^\infty (G)$ consisting of 
all functions $f$ satisfying 
\begin{align*}
f(xg)&=\psi (x)f(g)
&(x\in {N},\ g\in {G}).
\end{align*}
For an admissible representation $(\Pi ,H_{\Pi})$ of $G$, let 
\begin{align*}
&{\cI}_{\Pi,\psi }=\Hom_{(\g_\bC ,K)}
(H_{\Pi ,K} ,C^\infty (N\backslash G;\psi )_K).
\end{align*}
Here $H_{\Pi ,K}$ and $C^\infty (N\backslash G;\psi )_K$ are 
the subspaces of $H_{\Pi}$ and $C^\infty (N\backslash G;\psi )$ 
consisting of all $K$-finite vectors, respectively. 
We define the subspace ${\cI}_{\Pi,\psi }^{\mathrm{mg}}$ 
of ${\cI}_{\Pi,\psi }$ consisting of 
all homomorphisms $\Phi$ such that  
$\Phi (f)$ $(f\in H_{\Pi ,K})$ are 
moderate growth functions. 
We define the space ${\mathrm{Wh}}(\Pi ,\psi )$ of 
Whittaker functions for $(\Pi ,\psi )$ by 
\begin{align*}
&{\mathrm{Wh}}(\Pi ,\psi )=
\bC \textrm{-span}\{\Phi (f) \mid f\in H_{\Pi ,{K}},\ 
\Phi \in {\cI}_{\Pi,\psi }\}, 
\end{align*}
and define the subspace 
${\mathrm{Wh}}(\Pi ,\psi )^{\mathrm{mg}}$ of 
${\mathrm{Wh}}(\Pi ,\psi )$ by 
\begin{align*}
&{\mathrm{Wh}}(\Pi ,\psi )^{\mathrm{mg}}=
\bC \textrm{-span}\{\Phi (f) \mid f\in H_{\Pi ,{K}},\ 
\Phi \in {\cI}_{\Pi,\psi }^{\mathrm{mg}}\}.
\end{align*}

Let $\varphi \colon V_{\tau}\to {\mathrm{Wh}}(\Pi ,\psi )$ 
be a $K$-embedding with a $K$-type $(\tau ,V_\tau )$ of $\Pi$. 
By definition, we have 
\begin{align}
\label{eqn:Fn_whitt_ngk}
&\varphi (v)(xgk)=\psi (x)\varphi (\tau (k)v)(g)&
&(v\in V_\tau,\ x\in {N},\ g\in {G},\ k\in {K}).
\end{align}
Because of the Iwasawa decomposition ${G}={N}{A}{K}$, 
$\varphi $ is characterized by 
its restriction $v\mapsto \varphi (v)|_{{A}}$ to ${A}$. 
We call $v\mapsto \varphi (v)|_{{A}}$ 
the radial part of $\varphi$.

Assume that $\Pi$ is irreducible. 
Then the multiplicity one theorem 
(\textit{cf.} \cite{Shalika_001}, \cite{Wallach_001}) 
tells that the intertwining space 
${\cI}_{\Pi,\psi }^{\mathrm{mg}}$ 
is at most one dimensional. 
By the result of Matumoto 
\cite[Corollary 2.2.2, Theorem 6.2.1]{Matumoto_001} 
for $\mathrm{SL}(n,\bR )$, 
we know that ${\cI}_{\Pi,\psi }\neq 0$ if and only if 
$\Pi$ is large in the sense of Vogan \cite{Vogan_001}.

For $(c_1,c_2,c_3)\in (\bR^\times)^3$, 
there is a $G$-isomorphism 
\begin{align*}
\Xi_{(c_1,c_2,c_3)}\colon C^\infty (N\backslash G;\psi_1)
\to C^\infty (N\backslash G;\psi_{(c_1,c_2,c_3)})
\end{align*}
defined by $\Xi_{(c_1,c_2,c_3)}(f)(g)
=f ( \diag (c_1c_2c_3, c_2c_3, c_3, 1) g)$ ($g\in G$). 
Hence, it suffices to consider the case of $\psi_1$. 
In this paper, we give explicit formulas of the radial part 
of a $K$-embedding $\varphi \colon V_{\tau}\to 
{\mathrm{Wh}}(\Pi ,\psi_1)^{\mathrm{mg}}$ 
for an irreducible admissible 
large representation $\Pi$ of $G$ and the minimal $K$-type 
$(\tau ,V_\tau )$ of $\Pi$.

\subsection{Generalized principal series representations}
\label{subsec:Rn_def_gps}

We shall specify certain representations of 
$G_1$ and $G_2$ as follows. 
\begin{itemize}
\item 
For $\nu \in \bC$ and $\delta \in \{0,1\}$, 
we define a character $\chi_{(\nu ,\delta )}$ of $G_1$ by 
\begin{align*}
&\chi_{(\nu ,\delta )}(t)
=\sgn (t)^{\delta }|t|^{\nu }
\hspace{10mm}(t\in G_1).
\end{align*}
\item 
For $\nu \in \bC$ and $\kappa \in \bZ_{\geq 2}$, 
let $(D_{(\nu , \kappa )},\gH_{(\nu , \kappa )})$ be 
an irreducible Hilbert representation of $G_2$ 
such that 
$D_{(\nu ,\kappa )}(t1_2)=t^{2\nu }\ (t\in \bR_+)$ and 
$D_{(\nu ,\kappa )}\simeq D^+_{\kappa }\oplus D^-_{\kappa }$ 
as $(\gs \gl (2,\bR ),\mathrm{SO}(2))$-modules, where 
$D^\pm_{\kappa }$ is the discrete series representations of $\mathrm{SL}(2,\bR )$ with 
the minimal $\mathrm{SO}(2)$-type 
\[
\mathrm{SO}(2)\ni \rk_\theta^{(2)}
\mapsto e^{\pm \sI \kappa \theta }\in \bC^\times .
\]
Such representation $D_{(\nu , \kappa )}$ is unique 
up to infinitesimal equivalence. 
\end{itemize}

Let us define generalized principal series representations of 
$G=\mathrm{GL}(4,\bR)$. 
Let $\mn \in \{(1,1,1,1),(2,1,1),(2,2)\}$, and we associate 
the block upper triangular subgroup 
$P_{\mn}=N_{\mn}M_{\mn}$ of $G$, where 
\begin{align*}
M_{(1,1,1,1)} & =\{\diag (m_1,m_2,m_3,m_4)\mid 
m_1,m_2,m_3,m_4\in G_1\}
\simeq G_1\times G_1\times G_1\times G_1,\\
M_{(2,1,1)}& =\{
\iota (m_1,\diag (m_2,m_3))\mid 
m_1\in G_2,\  m_2,m_3\in G_1\}
\simeq G_2\times G_1\times G_1,\\
M_{(2,2)}& =\{
\iota (m_1,m_2) \mid 
m_1,m_2\in G_2\}
\simeq G_2\times G_2,
\end{align*}
and
\begin{gather*}
N_{(1,1,1,1)} =N, \quad 
 N_{(2,1,1)} =\{(x_{i,j})\in N\mid x_{1,2}=0\}, \quad 
N_{(2,2)} =\{(x_{i,j})\in N\mid x_{1,2}=x_{3,4}=0\}.
\end{gather*}
Let $(\sigma ,U_\sigma )$ be a Hilbert representation of 
$M_\mn $ of the following form: 
\begin{itemize}
\item When $\mn =(1,1,1,1)$, let 
$\sigma =\chi_{(\nu_1,\delta_1)}\boxtimes \chi_{(\nu_2,\delta_2)}
\boxtimes \chi_{(\nu_3,\delta_3)}\boxtimes \chi_{(\nu_4,\delta_4)}$ 
with $\nu_1,\nu_2,\nu_3,\nu_4\in \bC$, 
$\delta_1,\delta_2,\delta_3,\delta_4\in \{0,1\}$, that is, 
$U_\sigma =\bC$ and 
\begin{align*}
&\sigma (m)=\chi_{(\nu_1,\delta_1)}(m_1)\chi_{(\nu_2,\delta_2)}(m_2)
\chi_{(\nu_3,\delta_3)}(m_3)\chi_{(\nu_4,\delta_4)}(m_4)
\end{align*}
for $m=\diag (m_1,m_2,m_3,m_4)\in M_{(1,1,1,1)}$.  \smallskip 

\item When $\mn =(2,1,1)$, let 
$\sigma =D_{(\nu_1,\kappa_1)}\boxtimes \chi_{(\nu_2,\delta_2)}
\boxtimes \chi_{(\nu_3,\delta_3)}$ 
with $\nu_1,\nu_2,\nu_3\in \bC$, 
$\kappa_1\in \bZ_{\geq 2}$, 
$\delta_2,\delta_3\in \{0,1\}$, that is, 
$U_\sigma =\gH_{(\nu_1, \kappa_1)}$ and 
\begin{align*}
&\sigma (m)=\chi_{(\nu_2,\delta_2)}(m_2)
\chi_{(\nu_3,\delta_3)}(m_3)D_{(\nu_1,\kappa_1)}(m_1)
\end{align*}
for $m=\iota (m_1,\diag (m_2,m_3))\in M_{(2,1,1)}$. \smallskip 

\item When $\mn =(2,2)$, let 
$\sigma =D_{(\nu_1,\kappa_1)}\boxtimes D_{(\nu_2,\kappa_2)}$ 
with $\nu_1,\nu_2\in \bC$, 
$\kappa_1,\kappa_2\in \bZ_{\geq 2}$, that is, 
$U_\sigma =\gH_{(\nu_1, \kappa_1)}\boxtimes_{\bC} 
\gH_{(\nu_2, \kappa_2)}$ and 
\begin{align*}
&\sigma (m)=D_{(\nu_1,\kappa_1)}(m_1)\boxtimes 
D_{(\nu_2,\kappa_2)}(m_2)
\end{align*}
for $m=\iota (m_1,m_2)\in M_{(2,2)}$.
\end{itemize}
Moreover, we extend the representation $\sigma $ 
to $P_\mn =N_\mn M_\mn $ by 
\begin{align*}
&\sigma (xm)=\sigma (m)&(x\in N_{\mn},\ m\in M_\mn). 
\end{align*}
We define the function $\rho_{\mn} $ on $P_{\mn}$ by 
\begin{align*}
&\rho_{\mn} (p)=|\det (\Ad_{\gn_{\mn}} (p))|^{\frac{1}{2}}&&(p\in P_{\mn}), 
\end{align*}
where $\Ad_{\gn_{\mn}} $ means the adjoint action 
on the Lie algebra $\gn_{\mn}$ of $N_{\mn}$. 

Let $H(\sigma )^0$ be the space 
of $U_\sigma$-valued continuous functions $f$ on $K$ 
satisfying 
\begin{align*}
&f(mk)=\sigma (m)f(k)&
&(m\in {K}\cap M_{\mn},\ k\in {K}),
\end{align*}
on which ${G}$ acts by 
\begin{align*}
&(\Pi_\sigma (g)f)(k)=\rho_{\mn} (\mmp (kg))
\sigma (\mmp (kg))f(\mk (kg))&
&(g\in {G},\ k\in {K},\ f\in H(\sigma )^0).
\end{align*}
Here $kg=\mmp (kg)\mk (kg)$ is the decomposition of $kg$ 
with respect to the decomposition ${G}=P_{\mn}{K}$. 
We define a representation 
$(\Pi_\sigma ,H(\sigma ))$ of $G$ as the completion of 
$(\Pi_\sigma ,H(\sigma )^0)$ 
with respect to the inner product 
\begin{align*}
&\langle f_1,f_2\rangle_{H({\sigma})}
=\int_{{K}}\langle f_1(k),f_2(k)\rangle_{{\sigma}}dk&
&(f_1,f_2\in H(\sigma )^0),
\end{align*}
where $dk$ is the Haar measure on ${K}$, and 
$\langle \cdot ,\cdot \rangle_{{\sigma}}$ is the inner product 
on the Hilbert space $U_\sigma$. 
We call $(\Pi_\sigma ,H(\sigma ))$ a generalized principal 
series representation or a $P_\mn$-principal series representation. 
The subspace of $H(\sigma )$ 
consisting of all ${K}$-finite vectors is denoted by $H(\sigma)_K$.
We write the action of $\g_\bC$ on $H(\sigma)_K$ 
induced from $\Pi_\sigma$ by the same symbol $\Pi_\sigma$. 
If $\mn =(1,1,1,1)$, we call $(\Pi_\sigma ,H(\sigma ))$ 
a principal series representation. 

Because of Vogan's characterization \cite[Theorem 6.2 (f)]{Vogan_001} with 
\cite[Corollary 2.8]{Speh_Vogan_001}, 
any irreducible admissible large representation $\Pi$ of ${G}$ 
is infinitesimally equivalent 
to some $\Pi_\sigma $, which is induced from 
one of the following representations  
\begin{align*}
\sigma &=
\begin{cases}
 \chi_{(\nu_1,\delta_1)}\boxtimes \chi_{(\nu_2,\delta_2)}
\boxtimes \chi_{(\nu_3,\delta_3)}\boxtimes \chi_{(\nu_4,\delta_4)}
& \text{$(\delta_1\geq \delta_2\geq \delta_3\geq \delta_4)$},\\
  D_{(\nu_1,\kappa_1)}\boxtimes \chi_{(\nu_2,\delta_2)}
\boxtimes \chi_{(\nu_3,\delta_3)}
& \text{$(\delta_2\geq \delta_3)$},\\
 D_{(\nu_1,\kappa_1)}\boxtimes D_{(\nu_2,\kappa_2)}
&\text{$(\kappa_1\geq \kappa_2)$}.
\end{cases}
\end{align*}

Any generalized principal series representation of $G$ 
can be regarded as a subrepresentation 
of a principal series representation as 
(\ref{eqn:P211_embed}) and (\ref{eqn:P22_embed}) 
in \S \ref{sec:gps}, 
and the quotient representation 
is not large in the sense of Vogan. 
Hence, by the results of Kostant 
\cite[Theorems 5.5 and 6.6.2]{Kostant_001} and 
Matumoto \cite[Corollary 2.2.2, Theorem 6.1.6]{Matumoto_001} 
with the standard arguments, we have
\begin{align}
 \label{eqn:dimWh_gps}
&\dim_\bC  {\cI}_{\Pi_\sigma ,\psi_1}=4!,&
&\dim_\bC  {\cI}_{\Pi_\sigma ,\psi_1}^{\rm mg}=1.
\end{align}

\subsection{The gamma functions and Mellin-Barnes integrals}
\label{sec:Barnes}

We recall some basic facts of the gamma functions and Mellin-Barnes integrals.

The gamma function $\Gamma(s) $ is holomorphic on $ \bC\backslash \{ 0,-1,-2,\dotsc, \} $ and has a simple pole 
at $ s = m $ for any $ m \in \bZ_{\le 0} $.  
As usual, we set 
\begin{align*}
 \Gamma_\bR (s) & =\pi^{-s/2}\Gamma (s/2), 
&\Gamma_\bC (s)=2(2\pi )^{-s}\Gamma (s).
\end{align*}  
The functional equation $ \Gamma(s+1) = s \Gamma(s) $ implies that 
\begin{align}
\label{eqn:gamma_FE}
  \GR(s+2) &=(2\pi)^{-1}s\GR(s), &
  \GC(s+1) & = (2\pi)^{-1} s\GC(s).
\end{align}
The duplication formula $ \Gamma(s)\Gamma(s+\tfrac12) = 2^{1-2s} \sqrt{\pi }\Gamma(2s) $ implies that 
\begin{align}
\label{eqn:gamma_dup}
 \GR(s) \GR(s+1) = \GC(s).
\end{align}
For $ a \in \bC $ and $ i \in \bZ $, we introduce the Pochhammer symbol $ (a)_i $ by 
\begin{align*}
 (a)_i = \frac{\Gamma(a+i)}{\Gamma(a)}.
\end{align*}

In this paper, we treat Mellin-Barnes integrals. 
Let $ s_1,s_2,s_3 \in \bC $.
For $ a_i, b_j, c_k  \in \bC$ $(1 \le i,j \le 6, 1\le k \le 2) $, 
we assume that 
$ {\rm Re}(s_i+b_j+b_5),  {\rm Re}(s_i+b_j+b_6)  > 0 $ for $ (i,j) = (1,1), (2, 2), (2, 3), (3,4) $.  
Let
\begin{align*}
 V(s_1,s_2,s_3) 
& = \GR(s_1+a_1) \GR(s_1+a_2) \GR(s_2+a_3) \GR(s_2+a_4) \GR(s_3+a_5) \GR(s_3+a_6) 
\\
& \quad \times 
 \frac{1}{4\pi \sqrt{-1}} \int_q
  \frac{ \GR(s_1-q+b_1) \GR(s_2-q+b_2) \GR(s_2-q+b_3) \GR(s_3-q+b_4)  } 
   {\GR(s_1+s_2-q+c_1) \GR(s_2+s_3-q+c_2)} \\
& \quad  \times \GR(q+b_5) \GR(q+b_6) \,dq. 
\end{align*}
Here the path $ \int_{q} $ is a vertical line from 
$ {\rm Re}(q)-\sqrt{-1}\infty $ to $ {\rm Re}(q) + \sqrt{-1} \infty $ 
with the real part
$$ \max \{ -{\rm Re}(b_5), -{\rm Re}(b_6) \} < {\rm Re}(q) 
< \min \{ {\rm Re} (s_1+b_1), {\rm Re} (s_2+b_2), {\rm Re} (s_2+b_3), {\rm Re} (s_3+b_4) \}. $$
By using $ V(s_1,s_2,s_3) $, we consider a function $ f_V(y_1,y_2, y_3) $ on $(\bR_+)^3 $ by
\begin{align*}
 f_V(y_1,y_2,y_3) =\frac{1}{ (4\pi \sqrt{-1})^3} \int_{s_3} \int_{s_2} \int_{s_1} V(s_1,s_2,s_3) \, y_1^{-s_1} y_2^{-s_2} y_3^{-s_3} 
\,ds_1ds_2ds_3.
\end{align*}
Here the paths $ \int_{s_i} $ $(i=1,2,3)$ are vertical lines from 
$ {\rm Re}(s_i)-\sqrt{-1}\infty $ to $ {\rm Re}(s_i) + \sqrt{-1} \infty $ with sufficiently large real parts. 
More precisely, 
$ {\rm Re}(s_i+a_j)> 0 $ for $ (i,j)=(1,1)$, $(1,2)$, $(2,3)$, $(2,4)$, $(3,5)$, $(3,6) $
and 
$ {\rm Re}(s_i+b_j+b_k) > 0 $ for $ (i,j,k) =(1,1,5)$, $(1,1,6)$, $(2,2,5)$, $(2,2,6)$, $(2,3,5)$, $(2,3,6)$, $(3,4,5)$, $(3,4,6). $

We will describe $ \varphi(v)|_A $ as a linear combination of $f_V $. 
Hereafter we sometimes omit to mention the paths of integrations. 

\medskip

The following statement known to Barnes' (first) lemma plays important role in our 
calculation of explicit formulas of Whittaker functions and archimedean zeta integrals.

\begin{lem}[{\cite[Lemmas 8.5]{HIM}}]
\label{lem:Barnes1st}

For $ a_1,a_2, b_1,b_2 \in \bC $ such that  ${\rm Re}(a_i+b_j) > 0 $ $ (1 \le i,j \le 2) $, it holds that 
\begin{align*}
& \frac{1}{4\pi \sqrt{-1}} \int_z 
 \GR(z+a_1)\GR(z+a_2)\GR(-z+b_1)\GR(-z+b_2) \,dz
\\
& = \frac{\GR(a_1+b_1)\GR(a_1+b_2)\GR(a_2+b_1)\GR(a_2+b_2)}{\GR(a_1+a_2+b_1+b_2)}. 
\end{align*}
Here the path of integration $ \int_z $ is the vertical line from
$ {\rm Re}(z) - \sqrt{-1} \infty $ to $ {\rm Re}(z) + \sqrt{-1} \infty $ with the real part
$$ \max \{ -{\rm Re}(a_1), -{\rm Re}(a_2) \} < {\rm Re}(z) < \min \{ {\rm Re}(b_1), {\rm Re}(b_2) \}. $$
\end{lem}

The following lemma called  Barnes' second lemma is also useful.

\begin{lem}[{\cite[Lemmas 8.7]{HIM}}] 
\label{lem:Barnes2nd}
For $ a_1,a_2, b_1,b_2,b_3 \in \bC $ such that  ${\rm Re}(a_i+b_j) > 0 $ $ (1 \le i \le 2, 1 \le j \le 3) $, it holds that 
\begin{align*}
& \frac{1}{4\pi \sqrt{-1}} \int_z 
 \frac{\GR(z+a_1)\GR(z+a_2)\GR(-z+b_1)\GR(-z+b_2)\GR(-z+b_3)}
  { \GR(-z+a_1+a_2+b_1+b_2+b_3)} \,dz
\\
& = \frac{\GR(a_1+b_1)\GR(a_1+b_2)\GR(a_1+b_3) 
   \GR(a_2+b_1)\GR(a_2+b_2)\GR(a_2+b_3)}
  {\GR(a_1+a_2+b_1+b_2)\GR(a_1+a_2+b_1+b_3)\GR(a_1+a_2+b_2+b_3)}. 
\end{align*}
Here the path of integration $ \int_z $ is the vertical line from
$ {\rm Re}(z) - \sqrt{-1} \infty $ to $ {\rm Re}(z) + \sqrt{-1} \infty $ with the real part
$$
\max \{ -{\rm Re}(a_1), -{\rm Re}(a_2) \} < {\rm Re}(z) < \min \{ {\rm Re}(b_1), {\rm Re}(b_2), {\rm Re}(b_3) \}. $$
\end{lem}

\section{Representations of $K=\mathrm{O}(4)$}
\label{sec:rep_K}

\subsection{The parametrization}
\label{subsec:par_rep_K}

In this subsection, we will show that irreducible representations of $K$ 
is parametrized by 
\[
\Lambda_K =
\{\lambda =(\lambda_1,\lambda_2,\lambda_3)\in \bZ^2\times \{0,1\}
\mid \lambda_1\geq \lambda_2\geq 0,\ \lambda_2\lambda_3=0\}. 
\]
For $\gamma =(\gamma_1,\gamma_2,\dotsc,\gamma_n),
\gamma' =(\gamma_1',\gamma_2',\dotsc,\gamma_n')\in \bR^n$, 
we write $\gamma <_{\mathrm{lex}}\gamma'$ if and only if 
there is $1\leq m\leq n$ such that $\gamma_i=\gamma_i'$ ($1\leq i\leq m-1$) 
and $\gamma_m<\gamma_m'$. 
For $\gamma ,\gamma' \in \bR^n$, we write 
$\gamma \leq_{\mathrm{lex}}\gamma' $ 
if and only if either 
$\gamma <_{\mathrm{lex}}\gamma'$ or $\gamma = \gamma'$ holds. 
Then $\leq_{\mathrm{lex}}$ is a total order on $\bR^n$, 
and we call it the lexicographical order. 

The identity component of $K$ is 
$\mathrm{SO}(4)=\{k\in K \mid \det k=1\}$. 
Let us recall the highest weight theory \cite[Theorem 4.28]{Knapp_002} 
for $\mathrm{SO}(4)$. 
Let $(\tau ,V_\tau )$ be a finite-dimensional representation of 
$\mathrm{SO}(4)$. 
For $\gamma =(\gamma_1,\gamma_2)\in \bZ^2$, 
we define a subspace $V_\tau (\gamma )$ of $V_\tau$ by 
\begin{align*}
&V_\tau (\gamma )=\{v\in V_\tau \mid 
\tau (E_{1,2}^{\gk})v=\sI \gamma_1v,\ \ 
\tau (E_{3,4}^{\gk})v=\sI \gamma_2v\}.
\end{align*} 
Here the differential of $\tau$ is denoted again by $\tau$. 
We call $\gamma \in \bZ^2$ a weight of $(\tau,V_\tau )$ if and only if 
the corresponding subspace $V_\tau (\gamma )$ has nonzero elements. 
For a weight $\gamma$ of $(\tau,V_\tau )$, 
nonzero vectors in $V_\tau (\gamma )$ are called weight vectors 
with weight $\gamma$. 
We call $\lambda_\tau \in \bZ^2$ the highest weight of $(\tau,V_\tau )$ 
if and only if $\lambda_\tau $ is the weight of $(\tau,V_\tau )$ satisfying 
$\gamma \leq_{\mathrm{lex}}\lambda_\tau $ 
for any weight $\gamma $ of $(\tau,V_\tau )$. 
It is known that every irreducible 
representation  of $\mathrm{SO}(4)$ 
is finite dimensional, 
and $\tau \mapsto \lambda_\tau$ defines a bijection from 
the set of equivalence classes of irreducible representations 
of $\mathrm{SO}(4)$ to the set 
\[
\Lambda_{\mathrm{SO}(4)}=\{(\lambda_1,\lambda_2)\in \bZ^2\mid 
\lambda_1\geq |\lambda_2|\}.
\]
If $\tau$ is irreducible, then $V_\tau (\lambda_\tau )$ is one dimensional,
and any weight $ \gamma $ of $ (\tau, V_{\tau}) $ satisfies 
\[
 \gamma \in \{ \lambda_{\tau} -(m_1+m_2, m_1-m_2) \mid m_1,m_2 \in \bZ_{\ge 0} \} .
\]
Moreover, it is known that any finite dimensional representations of 
compact groups are completely reducible.

For $(\lambda_1 ,\lambda_2)\in \Lambda_{\mathrm{SO}(4)}$, 
let $(\tau_{\mathrm{SO}(4),(\lambda_1 ,\lambda_2)},
V_{\mathrm{SO}(4),(\lambda_1 ,\lambda_2)})$ be 
an irreducible representation of $\mathrm{SO}(4)$ 
with highest weight $(\lambda_1 ,\lambda_2)$. 
By Weyl's dimension formula \cite[Theorem 4.48]{Knapp_002}, 
for $(\lambda_1 ,\lambda_2)\in \Lambda_{\mathrm{SO}(4)}$, 
we have 
\begin{align}
\label{eqn:dim_SO4_rep}
\dim V_{\mathrm{SO}(4),(\lambda_1 ,\lambda_2)}
=(\lambda_1+\lambda_2+1)(\lambda_1-\lambda_2+1).
\end{align}
Let $k_0=\diag (1,1,1,-1)$. We have $k_0^2=1_4$ and 
$K=\mathrm{SO}(4)\sqcup \mathrm{SO}(4)k_0$. 
Hence, 
a representation of $K$ is characterized by the actions of $\mathrm{SO}(4)$ 
and $k_0$. Since 
\begin{align*}
&\Ad (k_0)E_{1,2}^{\gk}
=E_{1,2}^{\gk},&
&\Ad (k_0)E_{3,4}^{\gk}
=-E_{3,4}^{\gk},
\end{align*}
we know that, 
for $(\lambda_1,\lambda_2)\in \Lambda_{\mathrm{SO}(4)}$, 
the composite of $\tau_{\mathrm{SO}(4),(\lambda_1,\lambda_2)}$ and 
\[
\mathrm{SO}(4)\ni h\mapsto 
k_0hk_0
\in \mathrm{SO}(4)
\]
defines an irreducible representation of $\mathrm{SO}(4)$ 
with highest weight $(\lambda_1,-\lambda_2)$. 
By these facts, we obtain the following lemma.

\begin{lem}
\label{lem:par_Krep}
Let $(\tau,V_\tau)$ be an irreducible representation of $K$. 
Let $(\lambda_1,\lambda_2)$ be the highest weight of 
$\tau |_{\mathrm{SO(4)}}$. Then we have $\lambda_2\geq 0$ and 
\[
V_\tau \simeq 
\begin{cases}
V_{\mathrm{SO}(4),(\lambda_1,0)}&\text{if $\lambda_2=0$},\\
V_{\mathrm{SO}(4),(\lambda_1,\lambda_2)}\oplus 
V_{\mathrm{SO}(4),(\lambda_1,-\lambda_2)}
&\text{if $\lambda_2>0$}
\end{cases}
\]
as $\mathrm{SO}(4)$-modules. 
Let $\lambda_\tau =(\lambda_1,\lambda_2,\lambda_3)\in \Lambda_K$ with 
\[
\lambda_3=\begin{cases}
1&\text{if $\lambda_2=0$ and $\tau (k_0)|_{V_\tau ((\lambda_1,0))}=-1$},\\
0&\text{otherwise}. 
\end{cases}
\]
Then $\tau \mapsto \lambda_\tau$ defines a bijection from 
the set of equivalence classes of irreducible representations 
of $K$ to $\Lambda_K$. 
\end{lem}

\subsection{Constructions of irreducible representations of $K$}
\label{subsec:irred_Krep}

In this subsection, we give concrete realizations of 
irreducible representations of $K$. 

We define a representation 
$(\tau_{\mathrm{st}},V_{\mathrm{st}})$ of $K$ by 
$V_{\mathrm{st}}=\rM_{4,1}(\bC)\simeq \bC^4$ and 
\begin{align*}
&\tau_{\mathrm{st}}(h)v=hv&
&(h\in K,\ v\in V_{\mathrm{st}}).
\end{align*}
Here $hv$ is the ordinal product of matrices $h$ and $v$. 
The differential of $\tau_{\mathrm{st}}$ is denoted again 
by $\tau_{\mathrm{st}}$. 
Then we have $\tau_{\mathrm{st}}(X)v=Xv$ 
$(X \in \gk_{\bC},\ v\in V_{\mathrm{st}})$. 
For $1\leq i\leq 4$, $\xi_i$ denotes the matrix unit 
in $V_{\mathrm{st}}=\rM_{4,1}(\bC )$ 
with $1$ at $(i,1)$-th entry and $0$ at other entries. 
Moreover, we set 
\begin{align*}
&\zeta_{1}=\xi_1-\sI \xi_2,&
&\zeta_{2}=\xi_1+\sI \xi_2,&
&\zeta_{3}=\xi_3-\sI \xi_4,&
&\zeta_{4}=\xi_3+\sI \xi_4.&
\end{align*}
For $1\leq i,j\leq 4$, we define $\xi_{ij},\zeta_{ij}\in 
V_{\mathrm{st}}\wedge_\bC V_{\mathrm{st}}$ by 
$\xi_{ij}=\xi_i\wedge \xi_j$ and $\zeta_{ij}=\zeta_i\wedge \zeta_j$. 
Here we note $\xi_{ij}=-\xi_{ji}$, $\zeta_{ij}=-\zeta_{ji}$ and 
$\xi_{ii}=\zeta_{ii}=0$ for $1\leq i,j\leq 4$.

We define the graded $\bC$-algebra 
${\cR}=\bigoplus_{\lambda_1\geq \lambda_2\geq 0}
{\cR}_{(\lambda_1,\lambda_2)}$ by 
\begin{align*}
{\cR} &=\Sym (V_{\mathrm{st}})\otimes_\bC 
\Sym (V_{\mathrm{st}}\wedge_\bC V_{\mathrm{st}}),\\ 
{\cR}_{\lambda }
& =\Sym^{\lambda_1-\lambda_2} (V_{\mathrm{st}})\otimes_\bC 
\Sym^{\lambda_2} (V_{\mathrm{st}}\wedge_\bC V_{\mathrm{st}}). 
\end{align*}
Here $\Sym (V)=\bigoplus_{m\geq 0}\Sym^m(V)$ is the symmetric algebra on $V$ 
with the usual grading for a $\bC$-vector space $V$. 
We regard ${\cR}$ as a $K$-module via the action $\cT$ 
which is induced from $\tau_{\mathrm{st}}$. 
Then ${\cR}_{(\lambda_1,\lambda_2)}$ is a $K$-submodule of ${\cR}$. 

For $v\in V_{\rm st}$ and $v'\in V_{\mathrm{st}}\wedge_\bC V_{\mathrm{st}}$, 
we denote the elements $v\otimes 1$ and $1\otimes v'$ of $\cR$ 
simply by $v$ and $v'$, respectively. 
Then we note that 
\begin{align*}
&\{\xi_i\mid 1\leq i\leq 4\}\cup \{\xi_{jk}\mid 1\leq j<k\leq 4\},&
&\{\zeta_i\mid 1\leq i\leq 4\}\cup \{\zeta_{jk}\mid 1\leq j<k\leq 4\}
\end{align*}
are two systems of generators of $\cR$ as a $\bC$-algebra. 
By direct computation, for $ 1 \le a,b,i,j \le 4 $ and $ k \in \{1,3\} $, we have 
\begin{align*}
 & \cT(E_{a,b}^{\gk}) \xi_i = \delta_{b,i} \xi_a - \delta_{a,i} \xi_b, & 
 & \cT(E_{a,b}^{\gk}) \xi_{ij} = \delta_{b,i} \xi_{aj}+\delta_{b,j}\xi_{ia}-\delta_{a,i}\xi_{bj}-\delta_{a,j}\xi_{ib}, 
\end{align*}
and
\begin{align*}
 \cT(E_{k,k+1}^{\gk}) \zeta_i 
 & = \begin{cases} (-1)^i \sqrt{-1} \zeta_i & \text{if $ i \in \{k, k+1\}$}, \\
   0 & \text{otherwise},
\end{cases}
\\
 \cT(E_{k,k+1}^{\gk}) \zeta_{ij} 
 & = \begin{cases}
    ((-1)^i + (-1)^j)  \sqrt{-1} \zeta_{ij} & \text{if $ i \in \{k, k+1\}$ and $j \in \{k,k+1\} $}, \\
    (-1)^i  \sqrt{-1} \zeta_{ij} & \text{if $ i \in \{k, k+1\}$ and $j \notin \{k,k+1\} $}, \\
    (-1)^j  \sqrt{-1} \zeta_{ij} & \text{if $ i \notin \{k, k+1\}$ and $j \in \{k,k+1\} $}, \\
    0 & \text{if $ i \notin \{k, k+1\}$ and $j \notin \{k,k+1\} $}.
\end{cases}
\end{align*}

For $1\leq i,j,k\leq 4$, we define elements 
$\widehat{\xi}$, 
$\widehat{\xi}_i$, $\widehat{\xi}_{ijk}$, 
$\widehat{\xi}_{ij}$, $\widehat{\xi}_{1234}$ of $\cR$ by 
\begin{align*}
\widehat{\xi} &=(\xi_1)^2+(\xi_2)^2+(\xi_3)^2+(\xi_4)^2,\\
\widehat{\xi}_i &=\xi_1\xi_{i1}+\xi_2\xi_{i2}+\xi_3\xi_{i3}+\xi_4\xi_{i4},&
\widehat{\xi}_{ijk}
& =\xi_{i}\xi_{jk}-\xi_{j}\xi_{ik}+\xi_{k}\xi_{ij},\\
\widehat{\xi}_{ij}
& =\xi_{i1}\xi_{j1}+\xi_{i2}\xi_{j2}
+\xi_{i3}\xi_{j3}+\xi_{i4}\xi_{j4},&
\widehat{\xi}_{1234}
& =\xi_{12}\xi_{34}-\xi_{13}\xi_{24}+\xi_{14}\xi_{23}.
\end{align*}
Here we note 
$\widehat{\xi}_{ijk}=\widehat{\xi}_{jki}=\widehat{\xi}_{kij}
=-\widehat{\xi}_{ikj}=-\widehat{\xi}_{jik}=-\widehat{\xi}_{kji}$, 
$\widehat{\xi}_{iik}=\widehat{\xi}_{iji}=\widehat{\xi}_{ijj}=0$, 
and $\widehat{\xi}_{ij}=\widehat{\xi}_{ji}$ for $1\leq i,j,k\leq 4$. 
Let $I_{{\cR}}$ be the ideal of ${\cR}$ generated by 
\begin{align*}
&\{\widehat{\xi},\widehat{\xi}_{1234}\}\cup 
\{\widehat{\xi}_i\mid 1\leq i\leq 4\}\cup 
\{\widehat{\xi}_{ij}\mid 1\leq i\leq j\leq 4\}\cup 
\{\widehat{\xi}_{ijk}\mid 1\leq i<j<k\leq 4\}.
\end{align*}
\begin{lem}
\label{lem:Krep_ideal}
The ideal $I_{{\cR}}$ is $K$-invariant. 
Let $(\lambda_1,\lambda_2)\in \bZ^2$ such that 
$\lambda_1\geq \lambda_2\geq 0$. 
If $\lambda_1\leq 1$, we have 
$\cR_{(\lambda_1,\lambda_2)}\cap I_{{\cR}}=\{0\}$. 
If $\lambda_1>1$, the highest weight of 
$\mathrm{SO}(4)$-module 
$\cR_{(\lambda_1,\lambda_2)}\cap I_{{\cR}}$ 
is lower than $(\lambda_1,\lambda_2)$ in the lexicographical order. 
\end{lem}
\begin{proof}
By direct computation, we have 
\begin{align*}
&\cT (E_{a,b}^{\gk})\widehat{\xi}=0,\hspace{50mm}
\cT (E_{a,b}^{\gk})\widehat{\xi}_i
=\delta_{b,i}\widehat{\xi}_a-\delta_{a,i}\widehat{\xi}_b,\\
&\cT (E_{a,b}^{\gk})\widehat{\xi}_{ijk}
=\delta_{b,i}\widehat{\xi}_{ajk}
+\delta_{b,j}\widehat{\xi}_{iak}
+\delta_{b,k}\widehat{\xi}_{ija}
-\delta_{a,i}\widehat{\xi}_{bjk}
-\delta_{a,j}\widehat{\xi}_{ibk}
-\delta_{a,k}\widehat{\xi}_{ijb},\\
&\cT (E_{a,b}^{\gk})\widehat{\xi}_{ij}
=\delta_{b,i}\widehat{\xi}_{aj}
+\delta_{b,j}\widehat{\xi}_{ia}
-\delta_{a,i}\widehat{\xi}_{bj}
-\delta_{a,j}\widehat{\xi}_{ib},\hspace{15mm}
\cT (E_{a,b}^{\gk})\widehat{\xi}_{1234}=0
\end{align*}
and 
\begin{align*}
&\cT (k_0)
\widehat{\xi}=\widehat{\xi},\hspace{1cm}
\cT (k_0)\widehat{\xi}_i
=(-1)^{\delta_{i,4}}\widehat{\xi}_i,\hspace{1cm}
\cT (k_0)
\widehat{\xi}_{ijk}
=(-1)^{\delta_{i,4}+\delta_{j,4}+\delta_{k,4}}\widehat{\xi}_{ijk},\\[2mm]
&\cT (k_0)
\widehat{\xi}_{ij}
=(-1)^{\delta_{i,4}+\delta_{j,4}}\widehat{\xi}_{ij},
\hspace{3cm}
\cT (k_0)
\widehat{\xi}_{1234}
=-\widehat{\xi}_{1234}
\end{align*}
for $1\leq a,b,i,j,k\leq 4$. 
Hence, $I_{{\cR}}$ is $K$-invariant. 

For $i\in \{1,2\}$ and $j\in \{3,4\}$, 
we define $\widehat{\zeta}_{i}^{(k)},\widehat{\zeta}_{j}^{(k)}
\in \cR_{(2,1)}$ ($k=1,2$) by 
\begin{align*}
&\widehat{\zeta}_{i}^{(1)}=\widehat{\xi}_1+(-1)^i\sI \widehat{\xi}_2,&
&\widehat{\zeta}_{j}^{(1)}=\widehat{\xi}_3+(-1)^j\sI \widehat{\xi}_4,\\
&\widehat{\zeta}_{i}^{(2)}=\widehat{\xi}_{134}+(-1)^i\sI \widehat{\xi}_{234},&
&\widehat{\zeta}_{j}^{(2)}=\widehat{\xi}_{123}+(-1)^j\sI \widehat{\xi}_{124}.&
\end{align*}
Then, for $i\in \{1,3\}$, $1\leq j\leq 4$ and $k\in \{1,2\}$, we have 
\begin{align*}
&\cT (E_{i,i+1}^{\gk})\widehat{\zeta}_{j}^{(k)}
=\begin{cases}
(-1)^j\sI \widehat{\zeta}_{j}^{(k)}&\text{if $j\in \{i,i+1\}$},\\
0&\text{otherwise}.
\end{cases}
\end{align*}
For $1\leq i\leq j\leq 4$, we define 
$\widehat{\zeta}_{ij}\in \cR_{(2,2)}$ by 
\begin{align*}
&\widehat{\zeta}_{ij}=
\begin{cases}
\widehat{\xi}_{11}
+(-1)^i\sI \widehat{\xi}_{21}+(-1)^j\sI \widehat{\xi}_{12}
-(-1)^{i+j}\widehat{\xi}_{22}
&\text{if $i,j\in \{1,2\}$},\\
\widehat{\xi}_{33}
+(-1)^i\sI \widehat{\xi}_{43}+(-1)^j\sI \widehat{\xi}_{34}
-(-1)^{i+j}\widehat{\xi}_{44}
&\text{if $i,j\in \{3,4\}$},\\
\widehat{\xi}_{13}
+(-1)^i\sI \widehat{\xi}_{23}+(-1)^j\sI \widehat{\xi}_{14}
-(-1)^{i+j}\widehat{\xi}_{24}
&\text{otherwise}.
\end{cases}
\end{align*}
Then, for $i\in \{1,3\}$ and $1\leq j\leq k\leq 4$, we have 
\begin{align*}
&\cT (E_{i,i+1}^{\gk})\widehat{\zeta}_{jk}
= \begin{cases}
((-1)^j+(-1)^k)\sI \widehat{\zeta}_{jk}&
\text{if $j\in \{i,i+1\}$ and $k\in \{i,i+1\}$},\\[1mm]
(-1)^j\sI \widehat{\zeta}_{jk}&
\text{if $j\in \{i,i+1\}$ and $k\not\in \{i,i+1\}$},\\[1mm]
(-1)^k\sI \widehat{\zeta}_{jk}&
\text{if $j\not\in \{i,i+1\}$ and $k\in \{i,i+1\}$},\\[1mm]
0&\text{if $j\not\in \{i,i+1\}$ and $k\not\in \{i,i+1\}$}.
\end{cases}
\end{align*}
Since the ideal $I_{{\cR}}$ is generated by 
\begin{align*}
&\{\widehat{\xi},\widehat{\xi}_{1234}\}\cup 
\{\widehat{\zeta}_i^{(j)}\mid 1\leq i\leq 4,\ j\in \{1,2\}\}\cup 
\{\widehat{\zeta}_{ij}\mid 1\leq i\leq j\leq 4\},
\end{align*}
we obtain the latter part of the assertion by the above equalities. 
\end{proof}

Let $\widehat{\cT}$ be the action of $K$ on 
$\cR /I_{\cR }$ induced from $\cT$. 
Let $\rrq_\cR \colon \cR \ni r\mapsto r+I_{\cR}\in \cR /I_{\cR }$ be the 
natural surjection. 
For $\lambda =(\lambda_1,\lambda_2,\lambda_3)\in \Lambda_K$,  
we define a representation $(\tau_\lambda,V_\lambda )$ of $K$ by 
\begin{align*}
&\tau_\lambda (k)=(\det k)^{\lambda_3}\widehat{\cT}(k)
\quad (k\in K),&
&V_\lambda =\rrq_{\cR}(\cR_{(\lambda_1,\lambda_2)}). 
\end{align*}
The differential of $\tau_{\lambda}$ is denoted again 
by $\tau_{\lambda}$. 
Later, we will show that $(\tau_\lambda,V_\lambda )$ 
is an irreducible representation of $K$ corresponding to $\lambda $ 
 (Proposition \ref{prop:irred_Krep}).

Let $S_{\lambda}$ be the set of 
$l=(l_1,l_2,l_3,l_4,l_{12},l_{13},l_{14},l_{23},l_{24},l_{34})
\in (\bZ_{\geq 0})^{10}$ satisfying 
\begin{align*}
&l_1+l_2+l_3+l_4=\lambda_1-\lambda_2,&
&l_{12}+l_{13}+l_{14}+l_{23}+l_{24}+l_{34}=\lambda_2.
\end{align*}
For $l=(l_1,l_2,l_3,l_4,l_{12},l_{13},l_{14},l_{23},l_{24},l_{34})
\in S_\lambda$, 
we set 
\begin{align*}
&u_l=\rrq_\cR \!\left(\prod_{1\leq i\leq 4}(\xi_i)^{l_i}
\prod_{1\leq j<k\leq 4}(\xi_{jk})^{l_{jk}}\right),&
&v_l=\rrq_\cR \!\left(\prod_{1\leq i\leq 4}(\zeta_i)^{l_i}
\prod_{1\leq j<k\leq 4}(\zeta_{jk})^{l_{jk}}\right).
\end{align*}
We note that $\{u_l\}_{l\in S_\lambda }$ and $\{v_l\}_{l\in S_\lambda }$ 
form two systems of generators of $V_\lambda $ as 
a $\bC$-vector space. It is convenient to set $u_{l}=v_{l}=0$ 
if $l\not\in (\bZ_{\geq 0})^{10}$. 
We set $ {\bf 0} = (0,0,0,0,0,0,0,0,0,0) $ and 
\begin{align*}
&e_1=(1,0,0,0,0,0,0,0,0,0),&
&e_2=(0,1,0,0,0,0,0,0,0,0),\\
&e_3=(0,0,1,0,0,0,0,0,0,0),&
&e_4=(0,0,0,1,0,0,0,0,0,0),\\
&e_{12}=e_{21}=(0,0,0,0,1,0,0,0,0,0),&
&e_{13}=e_{31}=(0,0,0,0,0,1,0,0,0,0),\\
&e_{14}=e_{41}=(0,0,0,0,0,0,1,0,0,0),&
&e_{23}=e_{32}=(0,0,0,0,0,0,0,1,0,0),\\ 
&e_{24}=e_{42}=(0,0,0,0,0,0,0,0,1,0),&
&e_{34}=e_{43}=(0,0,0,0,0,0,0,0,0,1).
\end{align*}

\begin{lem}
\label{lem:rel_ul}
Let $\lambda =(\lambda_1,\lambda_2,\lambda_3)\in \Lambda_K$. 

\noindent 
(i) When $\lambda_1-\lambda_2\geq 2$, 
for $l\in S_{\lambda -(2,0,0)}$, we have 
\begin{align*}
&u_{l+2e_1}+u_{l+2e_2}+u_{l+2e_3}+u_{l+2e_4}=0.
\end{align*}
(ii) When $\lambda_1>\lambda_2>0$, 
for $l\in S_{\lambda -(2,1,0)}$, 
we have 
\begin{align*}
&\sum_{1\leq j\leq 4,\, j\neq i}
\sgn (j-i)u_{l+e_j+e_{ij}}=0&
&(1\leq i\leq 4), 
\end{align*}
that is, 
\begin{align*}
& u_{l+e_2+e_{12}} + u_{i+e_3+e_{13}} + u_{l+ e_4+e_{14}} = 0 , 
& -u_{l+e_1+e_{12}} + u_{i+e_3+e_{23}} + u_{l+ e_4+e_{24}} = 0, 
\\
& -u_{l+e_1+e_{13}} - u_{i+e_2+e_{23}} + u_{l+ e_4+e_{34}} = 0 , 
& -u_{l+e_1+e_{14}} - u_{i+e_2+e_{24}} - u_{l+ e_3+e_{34}} = 0, 
\end{align*}
and we also have
\begin{align*}
& u_{l+e_i+e_{jk}}-u_{l+e_j+e_{ik}}+u_{l+e_k+e_{ij}}=0 &  
(1\leq i<j<k\leq 4).
\end{align*}
\noindent 
(iii) When $\lambda_2\geq 2$, 
for $l\in S_{\lambda -(2,2,0)}$, we have 
\begin{align*}
&\sum_{1\leq k\leq 4,\, k\not\in \{i,j\}}
\sgn ((k-i)(k-j))u_{l+e_{ik}+e_{jk}}=0&& 
(1\leq i,j\leq 4), 
\end{align*}
that is, 
\begin{align*}
&  u_{l+2e_{12}} +  u_{l+2e_{13}} + u_{l+2e_{14}} = 0, 
&  u_{l+2e_{12}} +  u_{l+2e_{23}} + u_{l+2e_{24}} = 0,
\\
&  u_{l+2e_{13}} +  u_{l+2e_{23}} + u_{l+2e_{34}} = 0, 
&  u_{l+2e_{14}} +  u_{l+2e_{24}} + u_{l+2e_{34}} = 0,
\end{align*}
\begin{align*}
& u_{l+e_{13}+e_{23}} + u_{l+e_{14} + e_{24}} = 0, & 
& u_{l+e_{12}+e_{13}} + u_{l+e_{24} + e_{34}} = 0, & 
& u_{l+e_{13}+e_{14}} + u_{l+e_{23} + e_{24}} = 0, 
\\
& -u_{l+e_{12}+e_{23}} + u_{l+e_{14} + e_{34}} = 0, & 
& u_{l+e_{12}+e_{14}} - u_{l+e_{23} + e_{34}} = 0, & 
& -u_{l+e_{12}+e_{24}} - u_{l+e_{13} + e_{34}} = 0, 
\end{align*}
and $$u_{l+e_{12}+e_{34}}-u_{l+e_{13}+e_{24}}+u_{l+e_{14}+e_{23}}=0. $$
\end{lem}
\begin{proof}
The assertion follows immediately from the definition of $I_{\cR}$. 
\end{proof}

\begin{lem}
\label{lem:Kact_ul_vl}
Let $\lambda =(\lambda_1,\lambda_2,\lambda_3)\in \Lambda_K$, and 
\[
l=(l_1,l_2,l_3,l_4,l_{12},l_{13},l_{14},l_{23},l_{24},l_{34})
\in S_{\lambda}. 
\]
(i) For 
$\varepsilon_1,\varepsilon_2,\varepsilon_3,\varepsilon_4\in \{\pm 1\}$ 
and $1\leq i<j\leq 4$, 
we have 
\begin{align*}
\tau_\lambda  (\diag (\varepsilon_1,\varepsilon_2,\varepsilon_3,\varepsilon_4))u_l
& =\varepsilon_1^{l_1+l_{12}+l_{13}+l_{14}+\lambda_3}
 \varepsilon_2^{l_2+l_{12}+l_{23}+l_{24}+\lambda_3}
 \varepsilon_3^{l_3+l_{13}+l_{23}+l_{34}+\lambda_3}
 \varepsilon_4^{l_4+l_{14}+l_{24}+l_{34}+\lambda_3}u_l,
\\
\tau_\lambda (E_{i,j}^\gk )u_l
& =l_ju_{l-e_j+e_i}-l_iu_{l-e_i+e_j} \\
& 
   +\sum_{1\leq k\leq 4,\, k\not\in \{ i,j\}} 
\sgn ((k-i)(k-j))
(l_{kj}u_{l-e_{kj}+e_{ki}}-l_{ki}u_{l-e_{ki}+e_{kj}}).
\end{align*}
Here we put $l_{ji}=l_{ij}$ $(1\leq i<j\leq 4)$. 

\noindent (ii) For 
$\varepsilon_1,\varepsilon_2,\varepsilon_3,\varepsilon_4\in \{\pm 1\}$ 
and $\theta_1,\theta_2\in \bR$, 
we have 
\begin{align*}
\tau_\lambda (\diag (\varepsilon_1,\varepsilon_2,\varepsilon_3,\varepsilon_4))v_l
&=\varepsilon_1^{l_1+l_2+l_{12}+l_{13}+l_{14}+l_{23}+l_{24}+\lambda_3}
 \varepsilon_2^{l_{12}+\lambda_3}
 \varepsilon_3^{l_3+l_4+l_{13}+l_{14}+l_{23}+l_{24}+l_{34}+\lambda_3}
 \varepsilon_4^{l_{34}+\lambda_3}\\
&\quad \times 
\begin{cases}
v_{l}&
\text{if $\varepsilon_1\varepsilon_2=1$ and $\varepsilon_3\varepsilon_4=1$},\\
v_{(l_2,l_1,l_3,l_4,l_{12},l_{23},l_{24},l_{13},l_{14},l_{34})}&
\text{if $\varepsilon_1\varepsilon_2=-1$ and $\varepsilon_3\varepsilon_4=1$},\\
v_{(l_1,l_2,l_4,l_3,l_{12},l_{14},l_{13},l_{24},l_{23},l_{34})}&
\text{if $\varepsilon_1\varepsilon_2=1$ and $\varepsilon_3\varepsilon_4=-1$},\\
v_{(l_2,l_1,l_4,l_3,l_{12},l_{24},l_{23},l_{14},l_{13},l_{34})}&
\text{if $\varepsilon_1\varepsilon_2=-1$ 
and $\varepsilon_3\varepsilon_4=-1$}, \smallskip
\end{cases}
\\
\tau_{\lambda}(\rk_{\theta_1,\theta_2}^{(2,2)})v_l
& =e^{\sI (l_2+l_{24}+l_{23}-l_1-l_{13}-l_{14})\theta_1
+\sI (l_4+l_{24}+l_{14}-l_3-l_{13}-l_{23})\theta_2}v_l.
\end{align*}
\end{lem}
\begin{proof}
Direct computation. 
\end{proof}

\begin{prop}
\label{prop:irred_Krep}
(i) The correspondence $\lambda \leftrightarrow \tau_\lambda$ gives 
a bijection between $\Lambda_K$ and the set of equivalence classes of 
irreducible representations of $K$. 

\noindent (ii) Let $\lambda =(\lambda_1,\lambda_2,\lambda_3)\in \Lambda_K$. 
If $\lambda_2>0$, let $S^\circ_{\lambda }$ be the subset of $S_{\lambda }$ 
consisting of all 
$l=(l_1,l_2,l_3,l_4,l_{12},l_{13},l_{14},l_{23},l_{24},l_{34})$ satisfying 
\begin{align*}
&(l_3>0,\ l_4=l_{12}=l_{34}=0,\ l_{14}+l_{23}+l_{24}\leq 1)\\
\text{or}\ \ 
&(l_3=l_4=l_{24}=0,\ l_{12}>0,\ l_{14}+l_{23}+l_{34}\leq 1)\\ 
\text{or}\ \ 
&(l_3=l_4=l_{12}=0,\ l_{14}+l_{23}+l_{24}+l_{34}\leq 1).
\end{align*}
If $\lambda_2=0$, let $S^\circ_{\lambda }$ be the subset of $S_{\lambda }$ 
consisting of all 
$l=(l_1,l_2,l_3,l_4,0,0,0,0,0,0)$ satisfying $l_4\leq 1$. 
Then $\{u_l\}_{l\in S^\circ_{\lambda }}$ is a basis of $V_\lambda $. 
\end{prop}
\begin{proof}
Let $\lambda =(\lambda_1,\lambda_2,\lambda_3)\in \Lambda_K$. 
The cardinality $\# S^\circ_{\lambda }$ of $S^\circ_{\lambda }$ 
is given by 
\[
\# S^\circ_{\lambda }= \begin{cases}
2(\lambda_1-\lambda_2+1)(\lambda_1+\lambda_2+1)&\text{if $\lambda_2>0$},\\
(\lambda_1+1)^2&\text{if $\lambda_2=0$}, 
\end{cases}
\]
and it coincides with the dimension of an irreducible 
representation of $K$ corresponding to $\lambda$ by (\ref{eqn:dim_SO4_rep}) 
and Lemma \ref{lem:par_Krep}. 

We note that the highest weight of $\mathrm{SO}(4)$-module 
$\cR_{(\lambda_1,\lambda_2)}$ is $(\lambda_1,\lambda_2)$, 
and the corresponding weight space 
is given by $\bC (\zeta_2)^{\lambda_1-\lambda_2}(\zeta_{24})^{\lambda_2}$. 
By the latter part of Lemma \ref{lem:Krep_ideal}, 
we have $v_{(\lambda_1-\lambda_2)e_{2}+\lambda_2e_{24}}=
\rrq_{\cR}((\zeta_2)^{\lambda_1-\lambda_2}(\zeta_{24})^{\lambda_2})\neq 0$ 
and know that the highest weight of 
$\tau_{\lambda}|_{\mathrm{SO}(4)}$ is 
also $(\lambda_1,\lambda_2)$, 
and the corresponding weight space is given by 
$\bC v_{(\lambda_1-\lambda_2)e_{2}+\lambda_2e_{24}}$. 
Moreover, if $\lambda_2=0$, we have 
$\tau_\lambda (k_0)v_{\lambda_1e_{2}}
=(-1)^{\lambda_3}v_{\lambda_1e_{2}}$. 
Therefore, by Lemma \ref{lem:par_Krep}, 
$\tau_\lambda$ has an irreducible subrepresentation of $K$ 
corresponding to $\lambda$ whose dimension is 
$\# S^\circ_{\lambda }$. 
Hence, in order to complete the proof, 
it suffices to show that 
$\{u_l\}_{l\in S^\circ_{\lambda }}$ generates $V_\lambda $ 
as a $\bC$-vector space. 

Our task is to show that, 
for any $l\in S_\lambda$, 
the vector $u_l$ can be expressed as 
a linear combination of the vectors $u_{l'}$ $(l'\in S_\lambda^\circ )$. 
In the case of $\lambda_2=0$, 
it follows immediately from Lemma \ref{lem:rel_ul} (i). 
Let us consider the case of $\lambda_2>0$. For 
\[
l=(l_1,l_2,l_3,l_4,l_{12},l_{13},l_{14},l_{23},l_{24},l_{34})\in S_\lambda,
\]
we have the following assertions by Lemma \ref{lem:rel_ul}:
\begin{enumerate}
\item[(1)]By Lemma \ref{lem:rel_ul} (ii), 
the vector $u_l$ can be expressed as 
a linear combination of the vectors $u_{l'}$ with 
$l'=(l_1',l_2',l_3',0,
l_{12}',l_{13}',l_{14}',l_{23}',l_{24}',l_{34}')\in S_\lambda $. 

\item[(2)]By Lemma \ref{lem:rel_ul} (ii), if 
$l_3>0$ and $l_{12}+l_{34}>0$, then 
the vector $u_l$ can be expressed as 
a linear combination of the vectors $u_{l'}$ with 
\[
l'=(l_1',l_2',l_3-1,l_4,
l_{12}',l_{13}',l_{14}',l_{23}',l_{24}',l_{34}')\in S_\lambda 
\] 
satisfying  
$l_{12}'+l_{34}'=l_{12}+l_{34}-1$ and 
$l_{13}'+l_{14}'+l_{23}'+l_{24}'=l_{13}+l_{14}+l_{23}+l_{24}+1$. 

\item[(3)]By Lemma \ref{lem:rel_ul} (iii), if 
$l_{14}+l_{23}+l_{24}+l_{34}>1$, then 
the vector $u_l$ can be expressed as 
a linear combination of the vectors $u_{l'}$ with 
\[
l'=(l_1,l_2,l_3,l_4,
l_{12}',l_{13}',l_{14}',l_{23}',l_{24}',l_{34}')\in S_\lambda 
\] 
satisfying  
$l_{14}'+l_{23}'+l_{24}'+l_{34}'<l_{14}+l_{23}+l_{24}+l_{34}$. 
\end{enumerate}
By these assertions, we know that, for any $l\in S_\lambda$,   
the vector $u_l$ can be expressed as 
a linear combination of the vectors $u_{l'}$ with 
\[
l'=(l_1',l_2',l_3',0,
l_{12}',l_{13}',l_{14}',l_{23}',l_{24}',l_{34}')\in S_\lambda 
\] 
satisfying $l_3'(l_{12}'+l_{34}')=0$ and 
$l_{14}'+l_{23}'+l_{24}'+l_{34}'\leq 1$. 
Hence, the proof is completed 
by the relation $u_{l+e_{12}+e_{24}}=-u_{l+e_{13}+e_{34}}$ 
($l\in S_{\lambda -(2,2,0)}$)
in Lemma \ref{lem:rel_ul} (iii). 
\end{proof}

\subsection{Some lemmas for tensor products}
\label{subsec:lem_tensor}

We regard $\gp_\bC$ as a $K$-module via the adjoint action $\Ad$. 
For later use, we prepare the following lemmas.

\begin{lem}
\label{lem:tensor1}
Let $\lambda =(\lambda_1,\lambda_2,\lambda_3),
\lambda'=(\lambda_1',\lambda_2',\lambda_3')\in \Lambda_K$ and 
$\mu_1,\mu_2\in \bZ$ 
such that $\mu_1\geq \mu_2\geq 0$ and 
$\lambda -(\mu_1,\mu_2,0)\in \Lambda_K$. 
Let $\rB_{\lambda }^{(\mu_1,\mu_2)}\colon 
\cR_{(\mu_1,\mu_2)}\otimes_{\bC}V_{\lambda -(\mu_1,\mu_2,0)}\to 
V_{\lambda }$ be a $\bC$-linear map defined by  
\begin{align*}
&\rB_{\lambda }^{(\mu_1,\mu_2)}(v\otimes \rrq_{\cR} (v'))
=\rrq_{\cR} (vv')&
&(v\in \cR_{(\mu_1,\mu_2)},\ v'\in 
\cR_{(\lambda_1-\mu_1,\lambda_2-\mu_2)}).
\end{align*}
Then $\rB_{\lambda }^{(\mu_1,\mu_2)}$ is a surjective 
$K$-homomorphism, and we have 
\[
\Hom_K(\cR_{(\mu_1,\mu_2)}\otimes_{\bC}V_{\lambda -(\mu_1,\mu_2,0)}, 
V_{\lambda'})
= \begin{cases}
\bC \,\rB_{\lambda }^{(\mu_1,\mu_2)}&\text{if $\lambda'=\lambda $},\\
\{0\}&\text{if $\lambda <_{\mathrm{lex}}\lambda'$}.
\end{cases}
\]
\end{lem}
\begin{proof}
By definition, we know that 
$\rB_{\lambda }^{(\mu_1,\mu_2)}$ is a surjective $K$-homomorphism. 
The subspace of 
$\cR_{(\mu_1,\mu_2)}\otimes_{\bC}V_{\lambda -(\mu_1,\mu_2,0)}$ 
consisting of all vectors $v$ satisfying 
\begin{align*}
&(\cT \otimes \tau_{\lambda -(\mu_1,\mu_2,0)})
\bigl(E_{2i-1,2i}^{\gk}\bigr)v=\sI \lambda_i'v&
& (i \in \{1,2\})
\end{align*} 
is equal to 
\begin{align*}
\begin{cases}
\bC \,(\xi_2)^{\mu_1-\mu_2} (\xi_{24})^{\mu_2}\otimes 
v_{(\lambda_1+\lambda_2-\mu_1-\mu_2)e_2+(\lambda_2-\mu_2)e_{24}}&
\text{if}\ (\lambda_1',\lambda_2')=(\lambda_1,\lambda_2),\\
\{0\}&\text{if}\ (\lambda_1,\lambda_2)<_{\mathrm{lex}}(\lambda_1',\lambda_2').
\end{cases}
\end{align*}
Moreover, if $\lambda_2=\mu_2=0$, we have 
\begin{align*}
(\cT \otimes \tau_{\lambda -(\mu_1,0,0)})(k_0)
 (\xi_2)^{\mu_1}\otimes v_{(\lambda_1-\mu_1)e_2}
=(-1)^{\lambda_3}
 (\xi_2)^{\mu_1}\otimes v_{(\lambda_1-\mu_1)e_2}. 
\end{align*}
Therefore, we obtain the assertion. 
\end{proof}

\begin{lem}
\label{lem:tensor2}
We define $\bC$-linear maps 
$\rI^{\gp}_{(1,\delta )}\colon \cR_{(1,\delta )}\to 
\gp_{\bC}\otimes_{\bC}\cR_{(1,\delta )}$ $(\delta \in \{0,1\})$ by 
\begin{align*}
&\mathrm{I}_{(1,0)}^{\gp}(\xi_i)=
\sum_{k=1}^4
E_{i,k}^{\gp}\otimes \xi_k,&
&\mathrm{I}_{(1,1)}^{\gp}(\xi_{ij})=\sum_{k=1}^4
(E_{i,k}^{\gp}\otimes \xi_{kj}+E_{j,k}^{\gp}\otimes \xi_{ik})&
\end{align*}
for $1\leq i,j\leq 4$. 
Then $\mathrm{I}_{(1,0)}^{\gp}$ and $\mathrm{I}_{(1,1)}^{\gp}$ are 
$K$-homomorphisms. 
\end{lem}
\begin{proof}
We know that 
$\mathrm{I}_{(1,0)}^{\gp}$ and $\mathrm{I}_{(1,1)}^{\gp}$ are 
$\mathrm{SO}(4)$-homomorphisms 
by \cite[Proposition 1.3]{Ishii_Oda_001}. 
By direct computation, we can confirm that 
\begin{align*}
(\adj \otimes \cT )(k_0)\mathrm{I}_{(1,0)}^{\gp}(\xi_i)
&=\mathrm{I}_{(1,0)}^{\gp}(\cT (k_0)\xi_i), & 
(\adj \otimes \cT )(k_0)\mathrm{I}_{(1,1)}^{\gp}(\xi_{ij})
&=\mathrm{I}_{(1,1)}^{\gp}(\cT (k_0)\xi_{ij})&
\end{align*}
hold for $1\leq i,j\leq 4$. 
Therefore, we obtain the assertion. 
\end{proof}

\section{The minimal $K$-types of generalized principal series representations}
\label{sec:gps}

\subsection{The realization of $D_{(\nu , \kappa )}$}

In this subsection, we introduce a realization of 
$(D_{(\nu , \kappa )},\gH_{(\nu , \kappa )})$, which is 
a subrepresentation of some principal series representation of $G_2$. 
See \cite[Chapter 3]{HIM} for details. 

Let $P_{(1,1)}=N_2M_{(1,1)}$ be 
the upper triangular subgroup of $G_2$ with 
\[
M_{(1,1)}=\{m=\diag (m_1,m_2)\mid m_1,m_2\in \bR^\times \}.
\]
Let $\nu_1,\nu_2\in \bC$ and $\delta_1,\delta_2\in \{0,1\}$. 
We set $\sigma =\chi_{(\nu_1,\delta_1)}\boxtimes \chi_{(\nu_2,\delta_2)}$. 
We regard $\sigma $ as character of $P_{(1,1)}$ and 
define a character $\rho_{(1,1)}$ of $P_{(1,1)}$ by 
\begin{align*}
&\sigma (xm)=\chi_{(\nu_1,\delta_1)}(m_1)\chi_{(\nu_2,\delta_2)}(m_2),&
&\rho_{(1,1)} (xm)=\left|\frac{m_1}{m_2}\right|^{\frac{1}{2}}&
\end{align*}
for $x\in N_2$ and $m=\diag (m_1,m_2)\in M_{(1,1)}$.
Let $H(\sigma )^0$ be the space 
of continuous functions $f$ on $K_2$ 
satisfying 
\begin{align*}
&f(\diag (\varepsilon_1,\varepsilon_2)k)
=\varepsilon_1^{\delta_1}\varepsilon_2^{\delta_2}f(k)&
&(\varepsilon_1,\varepsilon_2\in \{\pm 1\},\ k\in K_2),
\end{align*}
on which $G_2$ acts by 
\begin{align*}
&(\Pi_\sigma (g)f)(k)=\rho_{(1,1)} (\mmp (kg))
\sigma (\mmp (kg))f(\mk (kg))&
&(g\in G_2,\ k\in K_2,\ f\in H(\sigma )^0).
\end{align*}
Here $kg=\mmp (kg)\mk (kg)$ is the decomposition of $kg$ 
with respect to the decomposition $G_2=P_{(1,1)}K_2$. 
We define a representation 
$(\Pi_\sigma ,H(\sigma ))$ of $G_2$ as the completion of 
$(\Pi_\sigma ,H(\sigma )^0)$ 
with respect to the $L^2$-inner product on $K_2$. 
We call $(\Pi_\sigma ,H(\sigma ))$ 
a principal series representation of $G_2$. 
Let $H(\sigma )_{K_2}$ be the subspace of $H(\sigma )$ 
consisting of all $K_2$-finite vectors, 
and take a basis 
$\{{\rf}_{(\sigma,q)}\}_{q \in \delta_1-\delta_2+2\bZ}$ 
of $H(\sigma )_{K_2}$ by 
\begin{align*}
&{\rf}_{(\sigma,q)}(\diag (\varepsilon_1,\varepsilon_2)\rk_\theta^{(2)})
=\varepsilon_1^{\delta_1}\varepsilon_2^{\delta_2}e^{\sI q\theta}&
&(\varepsilon_1,\varepsilon_2\in \{\pm 1\},\ \theta \in \bR).
\end{align*}
Then, for $\theta \in \bR$ and $\varepsilon_1,\varepsilon_2\in \{\pm 1\}$, 
we have 
\begin{align}
\label{eqn:GL2ps_Kact1}
&\Pi_{\sigma}(\rk_\theta^{(2)} ){\rf}_{(\sigma ,q)}
=e^{\sI q\theta }{\rf}_{(\sigma ,q)},&
&\Pi_\sigma 
(\diag (\varepsilon_1,\varepsilon_2)){\rf}_{(\sigma ,q)}
=
\varepsilon_1^{\delta_1} \varepsilon_2^{\delta_2} 
{\rf}_{(\sigma ,\varepsilon_1\varepsilon_2q)}.
\end{align}

Let $\nu \in \bC$ and $\kappa \in \bZ_{\geq 1}$. 
We set $\widehat{\sigma}=\chi_{(\nu +(\kappa -1)/2,\delta )}
\boxtimes \chi_{(\nu -(\kappa -1)/2,0)}$ 
with $\delta \in \{0,1\}$ such that 
$\delta \equiv \kappa \bmod 2$. 
Then the subrepresentation of $\Pi_{\widehat{\sigma}}$ 
on the closure of 
\begin{align}
\label{eqn:R2_ds_in_ps}
&\bigoplus_{q\in \kappa +2\bZ_{\geq 0}}
\{\bC\,{\rf}_{(\widehat{\sigma},q)}
+\bC\,{\rf}_{(\widehat{\sigma},-q)}\}
\end{align}
satisfies the definition 
of $(D_{(\nu ,\kappa )},\gH_{(\nu ,\kappa )})$ 
in \S \ref{subsec:Rn_def_gps}. 
Hereafter, we regard $(D_{(\nu ,\kappa )},\gH_{(\nu ,\kappa )})$ as 
this subrepresentation of $\Pi_{\widehat{\sigma}}$. 
We note that the $K_2$-finite part $\gH_{(\nu ,\kappa ),K_2}$ 
of $\gH_{(\nu ,\kappa )}$ coincides with 
the space (\ref{eqn:R2_ds_in_ps}).

\subsection{$P_{(1,1,1,1)}$-principal series representations}
\label{subsec:P_1111_ps}

Let $\sigma = \chi_{(\nu_1,\delta_1)}
\boxtimes \chi_{(\nu_2,\delta_2)} \boxtimes \chi_{(\nu_3,\delta_3)}
\boxtimes \chi_{(\nu_4,\delta_4)} $ 
with $\nu_1,\nu_2,\nu_3,\nu_4 \in \bC$, 
$\delta_1,\delta_2,\delta_3,\delta_4 \in \{0,1\}$ 
such that $\delta_1 \geq \delta_2 \geq \delta_3\geq \delta_4$. 
The group $K\cap M_{(1,1,1,1)}$ consists of the elements
\begin{align*}
\diag (\varepsilon_1,\varepsilon_2,\varepsilon_3,\varepsilon_4)\qquad 
(\varepsilon_1,\varepsilon_2,\varepsilon_3,\varepsilon_4
\in \{\pm 1\}).
\end{align*}
Because of Lemma \ref{lem:Kact_ul_vl} (i), for $ \lambda \in \Lambda_K $, we have
\begin{align*}
\Hom_{K\cap M_{(1,1,1,1)}}(V_{\lambda },U_{\sigma ,K\cap M_{(1,1,1,1)}})
& = \begin{cases}
\bC\,\eta_{\sigma}
&\text{if  $\lambda =(\delta_1-\delta_4,\delta_2-\delta_3,\delta_3)$},\\
\{0\}&\text{if $\lambda<_{\rm lex} (\delta_1-\delta_4, \delta_2-\delta_3, \delta_3)$},
\end{cases}
\end{align*}
where $\eta_{\sigma}\colon 
V_{(\delta_1-\delta_4,\delta_2-\delta_3,\delta_3)}\to U_{\sigma, K \cap M_{(1,1,1,1)}} $ 
is a $\bC$-linear map defined by 
\begin{align*}
\eta_{\sigma}(u_{l})= 
\begin{cases} 
1 &  \mbox{if $ l = (\delta_1-\delta_2)e_1+(\delta_3-\delta_4)e_4+(\delta_2-\delta_3)e_{12}$}, \\
0 &  \mbox{otherwise}
\end{cases}
\end{align*}
for $l\in S_{(\delta_1-\delta_4,\delta_2-\delta_3,\delta_3)}$. 
By the Frobenius reciprocity law \cite[Theorem 1.14]{Knapp_002}, we have 
\begin{align}
\label{eqn:P1111_minKtype}
&\Hom_{K}(V_{\lambda },H(\sigma )_K) 
 = \begin{cases}
\displaystyle 
\bC\,\hat{\eta}_{\sigma}
&\text{if $\lambda =(\delta_1-\delta_4,\delta_2-\delta_3,\delta_3)$},\\
\{0\}
&\text{if $\lambda <_{\rm lex} (\delta_1-\delta_4, \delta_2-\delta_3, \delta_3)$}
\end{cases} & 
&(\lambda \in \Lambda_K),
\end{align}
where 
$\hat{\eta}_{\sigma}(v)(k)
=\eta_{\sigma}(\tau_{(\delta_1-\delta_4,\delta_2-\delta_3,\delta_3)}(k)v)$ 
for $v\in V_{(\delta_1-\delta_4,\delta_2-\delta_3,\delta_3)}$ and $k\in {K}$. 
We call $\tau_{(\delta_1-\delta_4,\delta_2-\delta_3,\delta_3)}$ 
the minimal $K$-type of $\Pi_\sigma$.

\begin{prop}[{cf. \cite[Lemma 2.5]{Ishii_Oda_001}}]
\label{prop:P1111_DS}
Retain the notation. 

\noindent (i) For $f\in H (\sigma )_K$, we have 
\begin{align*}
&\Pi_\sigma (\cC_1)f=
(\nu_1+\nu_2+\nu_3+\nu_4)f,\\
&\Pi_\sigma (\cC_2)f=
(\nu_1\nu_2+\nu_1\nu_3+\nu_1\nu_4+\nu_2\nu_3+\nu_2\nu_4+\nu_3\nu_4)f,\\
&\Pi_\sigma (\cC_3)f=
(\nu_1\nu_2\nu_3+\nu_1\nu_2\nu_4+\nu_1\nu_3\nu_4+\nu_2\nu_3\nu_4)f,\\
&\Pi_\sigma (\cC_4)f=
\nu_1\nu_2\nu_3\nu_4f.
\end{align*}
\noindent (ii) Assume $ (\delta_1,\delta_2,\delta_3,\delta_4) = (1,0,0,0) $ or  $(1,1,1,0) $.
For $1\leq i\leq 4$, we have 
\begin{align*}
2 \{ (\delta_1-\delta_2)\nu_1+(\delta_3-\delta_4) \nu_4 \} \,\hat{\eta}_{\sigma }(u_{e_i})
=\sum_{k=1}^4
\Pi_\sigma (E_{i,k}^{\gp}) 
\hat{\eta}_{\sigma }(u_{e_k}).
\end{align*}
\noindent (iii) Assume $ (\delta_1,\delta_2,\delta_3,\delta_4) = (1,1,0,0). $
For $1\leq i<j\leq 4$, we have 
\begin{align*}
&2(\nu_1+\nu_2)\,\hat{\eta}_{\sigma }(u_{e_{ij}})
=\Pi_\sigma (E_{i,i}^{\gp}+E_{j,j}^{\gp})\hat{\eta}_{\sigma }(u_{e_{ij}})\\
&\hspace{2mm}
+\sum_{1\leq k\leq 4,\, k\not\in \{i,j\}}
\{\sgn (j-k)\Pi_\sigma (E_{i,k}^{\gp})\hat{\eta}_{\sigma }(u_{e_{kj}})
+\sgn (k-i)\Pi_\sigma (E_{j,k}^{\gp})\hat{\eta}_{\sigma }(u_{e_{ik}})\}.
\end{align*}
\end{prop}

\begin{proof}
The first statement (i) is \cite[Proposition 2.2]{HIM}. 
By definition, for $1\leq i,j\leq 4$ and $f\in H({\sigma})_K$, 
we have 
\begin{align}
\label{eqn:pf_P1111DS_001}
&2(\Pi_{{\sigma}}(E_{i,j})f)(1_4)=
\begin{cases}
2\nu_i+5-2i&\text{if $i=j$}, \\
0&\text{if $i<j$}.
\end{cases}
\end{align}

Assume $ \delta_1-\delta_4 = 1$. 
Let $\lambda =(1,\delta_2-\delta_3,\delta_3)$.
We have 
\[
\Hom_K(\cR_{(1, \delta_2-\delta_3 )}\otimes_{\bC}
V_{\lambda -(1, \delta_2-\delta_3 ,0)}, 
H(\sigma)_K)=\bC\,\hat{\eta}_{\sigma }\circ 
\rB_{\lambda}^{(1,\delta_2-\delta_3 )}
\]
by Lemma \ref{lem:tensor1} and (\ref{eqn:P1111_minKtype}). 
We define a $K$-homomorphism 
$\rP_\sigma \colon \gp_{\bC}\otimes_{\bC}H(\sigma)_K
\to H(\sigma)_K$ by $X\otimes f\mapsto \Pi_{\sigma}(X)f$. 
Then there is a constant $c_{\delta_2-\delta_3}$ such that  
\begin{align}
\label{eqn:pre_dirac0}
c_{\delta_2-\delta_3} \,\hat{\eta}_{\sigma }\circ 
\rB_{\lambda}^{(1,\delta_2-\delta_3 )}=
\rP_\sigma \circ (\id_{\gp_{\bC}}\otimes (\hat{\eta}_{\sigma }
\circ \rB_{\lambda}^{(1,\delta_2-\delta_3)}))\circ 
(\rI^\gp_{(1,\delta_2-\delta_3 )}\otimes 
\id_{V_{\lambda -(1,\delta_2-\delta_3 ,0)}}),
\end{align}
since the right hand side is an element of 
$\Hom_K(\cR_{(1,\delta_2-\delta_3 )}\otimes_{\bC}
V_{\lambda -(1,\delta_2-\delta_3 ,0)}, H(\sigma)_K)$. 

Let us consider the case of $\delta_2-\delta_3=0$. 
Considering the image of $\xi_i\otimes u_{{\bf 0}}$ 
under the both sides of (\ref{eqn:pre_dirac0}), we have 
\begin{align*}
c_0\,\hat{\eta}_{\sigma }(u_{e_i})
=\sum_{k=1}^4
\Pi_\sigma (E_{i,k}^{\gp}) 
\hat{\eta}_{\sigma }(u_{e_k})&
&(1\leq i\leq 4).
\end{align*}
Hence, in order to prove the statement (ii), 
it suffices to show 
$c_{0} = 2\{(\delta_1-\delta_2)\nu_1 + (\delta_3-\delta_4)\nu_4 \} $. 
When $ (\delta_1,\delta_2,\delta_3,\delta_4) = (1,0,0,0) $,
using Lemma \ref{lem:Kact_ul_vl} (i) and 
(\ref{eqn:pf_P1111DS_001}),  
we have 
\begin{align*}
c_0
&= c_{0}\,\hat{\eta}_{\sigma }(u_{e_{1}})(1_4) \\
& = \sum_{k=1}^4(\Pi_{\sigma}(E_{1,k}^{\gp}) 
\hat{\eta}_{\sigma }(u_{e_k}))(1_4)\\
& 
=2 ( \Pi_{{\sigma}} (E_{1,1}) \hat{\eta}_{\sigma}(u_{e_1}))(1_4)
+\sum_{k=2}^4 \{ 2(
\Pi_{{\sigma}}(E_{1,k})
 \hat{\eta}_{\sigma}(u_{e_k})) (1_4)
- \hat{\eta}_{\sigma}(
\tau_{\lambda }(E_{1,k}^{\gk})u_{e_k})(1_4) \} \\
&=2\nu_1.
\end{align*}
Similarly, when $ (\delta_1,\delta_2,\delta_3,\delta_4) = (1,1,1,0) $, we know $ c_0 = 2\nu_4 $
to obtain the statement (ii).

Assume $(\delta_1,\delta_2,\delta_3,\delta_4)=(1,1,0,0)$. 
Considering the image of $\xi_{ij}\otimes u_{{\bf 0}}$ 
under the both sides of (\ref{eqn:pre_dirac0}), we have 
\begin{align*}
c_1\,\hat{\eta}_{\sigma }(u_{e_{ij}})
& =\Pi_\sigma (E_{i,i}^{\gp}+E_{j,j}^{\gp})\hat{\eta}_{\sigma }(u_{e_{ij}})\\
&
+\sum_{1\leq k\leq 4,\, k\not\in \{i,j\}}
\{\sgn (j-k)\Pi_\sigma (E_{i,k}^{\gp})\hat{\eta}_{\sigma }(u_{e_{kj}})
+\sgn (k-i)\Pi_\sigma (E_{j,k}^{\gp})\hat{\eta}_{\sigma }(u_{e_{ik}})\}
\end{align*}
for $1\leq i<j\leq 4$. 
Hence, in order to prove the statement (iii), 
it suffices to show $c_{1}=2(\nu_1+\nu_2)$. 
Using Lemma \ref{lem:Kact_ul_vl} (i) and 
(\ref{eqn:pf_P1111DS_001}), we have 
\begin{align*}
c_1
& =  c_{1}\,\hat{\eta}_{\sigma }(u_{e_{12}})(1_4)\\
&=
(\Pi_\sigma (E_{1,1}^{\gp}+E_{2,2}^{\gp})
\hat{\eta}_{\sigma }(u_{e_{12}}))(1_4) 
-\bigl( \Pi_\sigma (E_{1,3}^{\gp})\hat{\eta}_{\sigma }(u_{e_{23}})
 +\Pi_\sigma (E_{1,4}^{\gp})\hat{\eta}_{\sigma }(u_{e_{24}}) \bigr) (1_4)\\
&\phantom{=}
+\bigl( \Pi_\sigma (E_{2,3}^{\gp})\hat{\eta}_{\sigma }(u_{e_{13}})
+\Pi_\sigma (E_{2,4}^{\gp})\hat{\eta}_{\sigma }(u_{e_{14}})
\bigr) (1_4)\\
&=2 \bigl(\Pi_\sigma (E_{1,1}^{}+E_{2,2}^{}) \hat{\eta}_{\sigma }(u_{e_{12}}) \bigr)(1_4) \\
&\phantom{=}
- 2\bigl(\Pi_{{\sigma}}(E_{1,3})
  \hat{\eta}_{\sigma}(u_{e_{23}}) \bigr)(1_4)
+ \hat{\eta}_{\sigma}( \tau_{\lambda}(E_{1,3}^{\gk})u_{e_{23}})(1_4)
- 2\bigl(\Pi_{{\sigma}}(E_{1,4})
  \hat{\eta}_{\sigma}(u_{e_{24}}) \bigr)(1_4)
+ \hat{\eta}_{\sigma}( \tau_{\lambda}(E_{1,4}^{\gk})u_{e_{24}})(1_4) \\
&\phantom{=}
+ 2\bigl(\Pi_{{\sigma}}(E_{2,3})
  \hat{\eta}_{\sigma}(u_{e_{13}}) \bigr)(1_4)
- \hat{\eta}_{\sigma}( \tau_{\lambda}(E_{2,3}^{\gk})u_{e_{13}})(1_4)
+ 2\bigl(\Pi_{{\sigma}}(E_{2,4})
  \hat{\eta}_{\sigma}(u_{e_{14}}) \bigr)(1_4)
- \hat{\eta}_{\sigma}( \tau_{\lambda}(E_{2,4}^{\gk})u_{e_{14}})(1_4)\\
&=2(\nu_1+\nu_2),
\end{align*}
as desired.
\end{proof}


\subsection{$P_{(2,1,1)}$-principal series representations}
\label{subsec:P_211_gps}

Let $\sigma =D_{(\nu_1,\kappa_1)}\boxtimes \chi_{(\nu_2,\delta_2)}
\boxtimes \chi_{(\nu_3,\delta_3)}$ 
with $\nu_1,\nu_2,\nu_3\in \bC$, 
$\kappa_1\in \bZ_{\geq 2}$, 
$\delta_2,\delta_3\in \{0,1\}$ such that $\delta_2\geq \delta_3$. 
We set 
$\widehat{\sigma}_1=\chi_{(\nu_1+(\kappa_1-1)/2,\delta_1)}
\boxtimes \chi_{(\nu_1-(\kappa_1-1)/2,0)}$ 
with $\delta_1\in \{0,1\}$ such that 
$\delta_1\equiv \kappa_1\bmod 2$. 
The group $K\cap M_{(2,1,1)}$ is generated by the elements 
\begin{align*}
&\rk_{\theta_1,0}^{(2,2)}\qquad (\theta_1\in \bR),&
&\diag (\varepsilon_1,\varepsilon_2,\varepsilon_3,\varepsilon_4)\qquad 
(\varepsilon_1,\varepsilon_2,\varepsilon_3,\varepsilon_4
\in \{\pm 1\}).
\end{align*}
Because of (\ref{eqn:GL2ps_Kact1}), 
these elements act on $U_{\sigma , K\cap M_{(2,1,1)}}
=\gH_{(\nu_1,\kappa_1),K_2}$ by 
\begin{align*} 
&\sigma (\rk_{\theta_1,0}^{(2,2)})
{\rf}_{(\widehat{\sigma}_1,q)}
=e^{\sI q\theta_1 }
{\rf}_{(\widehat{\sigma}_1,q)},
&\sigma (\diag (\varepsilon_1,\varepsilon_2,\varepsilon_3,\varepsilon_4))
{\rf}_{(\widehat{\sigma}_1,q)}
=\varepsilon_1^{\kappa_1}\varepsilon_3^{\delta_2}\varepsilon_4^{\delta_3}
{\rf}_{(\widehat{\sigma}_1,\varepsilon_1\varepsilon_2q)}
\end{align*}
for $q\in \kappa_1+2\bZ$ such that $|q|\geq \kappa_1$.
By these equalities and Lemma \ref{lem:Kact_ul_vl} (ii), 
for $\lambda \in \Lambda_K$, we have 
\begin{align*}
&\Hom_{K\cap M_{(2,1,1)}}(V_{\lambda },U_{\sigma ,K\cap M_{(2,1,1)}})
= \begin{cases}
\displaystyle 
\bC\,\eta_{\sigma}
&\text{ if $\lambda =(\kappa_1,\delta_2-\delta_3,\delta_3)$},\\
\{0\}&\text{ if $ \lambda <_{\mathrm{lex}}
(\kappa_1,\delta_2-\delta_3,\delta_3)$},
\end{cases}
\end{align*}
where $\eta_{\sigma}\colon 
V_{(\kappa_1,\delta_2-\delta_3,\delta_3)}\to U_{\sigma, K\cap M_{(2,1,1)}} $ 
is a $\bC$-linear map defined by 
\begin{align*}
\eta_{\sigma}(v_{l})
= \begin{cases}
{\rf}_{(\widehat{\sigma}_1,\kappa_1)}
&\text{ if  $l=(\kappa_1-\delta_2+\delta_3)e_{2}+(\delta_2-\delta_3)e_{24}
\text{ or } l=(\kappa_1-\delta_2+\delta_3)e_{2}+(\delta_2-\delta_3)e_{23}$},\\
(-1)^{\delta_3}{\rf}_{(\widehat{\sigma}_1,-\kappa_1)}
&\text{ if $ l=(\kappa_1-\delta_2+\delta_3)e_{1}+(\delta_2-\delta_3)e_{14} $
 or $l=(\kappa_1-\delta_2+\delta_3)e_{1}+(\delta_2-\delta_3)e_{13}$},\\
0&\text{ otherwise}
\end{cases}
\end{align*}
for $l\in S_{(\kappa_1,\delta_2-\delta_3,\delta_3)}$. 
By the Frobenius reciprocity law \cite[Theorem 1.14]{Knapp_002}, we have 
\begin{align}
\label{eqn:P211_minKtype}
&\Hom_{K}(V_{\lambda },H(\sigma )_K)
=\begin{cases}
\displaystyle 
\bC\,\hat{\eta}_{\sigma}
&\text{if $ \lambda =(\kappa_1,\delta_2-\delta_3,\delta_3)$},\\
\{0\}
&\text{if $ \lambda <_{\mathrm{lex}}(\kappa_1,\delta_2-\delta_3,\delta_3) $}
\end{cases} &
&(\lambda \in \Lambda_K),
\end{align}
where 
$\hat{\eta}_{\sigma}(v)(k)
=\eta_{\sigma}(\tau_{(\kappa_1,\delta_2-\delta_3,\delta_3)}(k)v)$ 
for $v\in V_{(\kappa_1,\delta_2-\delta_3,\delta_3)}$ and $k\in {K}$. 
We call $\tau_{(\kappa_1,\delta_2-\delta_3,\delta_3)}$ 
the minimal $K$-type of $\Pi_\sigma$. 

We define $\iota_\sigma \colon U_\sigma \to \bC$ 
by $\iota_\sigma  (f)=f(1_2)$, and let $\widehat{\sigma}=\widehat{\sigma}_1\boxtimes \chi_{(\nu_2,\delta_2)}\boxtimes \chi_{(\nu_3,\delta_3)}$. 
Then $\Pi_\sigma $ is embedded into a principal series representation 
$\Pi_{\widehat{\sigma}}$ via 
\begin{equation}
\label{eqn:P211_embed}
\rI_\sigma \colon 
H(\sigma)\ni f\mapsto \iota_\sigma \circ f\in H(\widehat{\sigma}). 
\end{equation}

\begin{prop}
\label{prop:P211_DS}
Retain the notation.

\noindent (i) For $f\in H (\sigma )_K$, we have 
\begin{align*}
&\Pi_\sigma (\cC_1)f=(2\nu_1+\nu_2+\nu_3)f,\\
&\Pi_\sigma (\cC_2)f
=\bigl(\nu_1^2+2\nu_1\nu_2+2\nu_1\nu_3
+\nu_2\nu_3-\tfrac{(\kappa_1-1)^2}{4}\bigr)f,\\
&\Pi_\sigma (\cC_3)f
=\bigl(\nu_1^2\nu_2+\nu_1^2\nu_3
+2\nu_1\nu_2\nu_3-\tfrac{(\kappa_1-1)^2}{4}\nu_2
-\tfrac{(\kappa_1-1)^2}{4}\nu_3\bigr)f,\\
&\Pi_\sigma (\cC_4)f
=\bigl(\nu_1^2-\tfrac{(\kappa_1-1)^2}{4}\bigr)\nu_2\nu_3f.
\end{align*}
\noindent (ii)  
For $l\in S_{(\kappa_1-1,\delta_2-\delta_3,\delta_3)}$ and 
$1\leq i\leq 4$, we have 
\begin{align*}
2\nu_1\,\hat{\eta}_{\sigma }(u_{l+e_i})
=\sum_{k=1}^4
\Pi_\sigma (E_{i,k}^{\gp}) 
\hat{\eta}_{\sigma }(u_{l+e_k}).
\end{align*}
\noindent (iii) Assume $(\delta_2,\delta_3)=(1,0)$. 
For $l\in S_{(\kappa_1-1,0,0)}$ and 
$1\leq i<j\leq 4$, we have 
\begin{align*}
&2(\nu_1+\nu_2)\,\hat{\eta}_{\sigma }(u_{l+e_{ij}})
=\Pi_\sigma (E_{i,i}^{\gp}+E_{j,j}^{\gp})\hat{\eta}_{\sigma }(u_{l+e_{ij}})\\
&\hspace{2mm}
+\sum_{1\leq k\leq 4,\, k\not\in \{i,j\}}
\{\sgn (j-k)\Pi_\sigma (E_{i,k}^{\gp})\hat{\eta}_{\sigma }(u_{l+e_{kj}})
+\sgn (k-i)\Pi_\sigma (E_{j,k}^{\gp})\hat{\eta}_{\sigma }(u_{l+e_{ik}})\}.
\end{align*}
\end{prop}

\begin{proof}
The statement (i) follows immediately from 
Proposition \ref{prop:P1111_DS} (i) and (\ref{eqn:P211_embed}). 
By definition, for $1\leq i,j\leq 4$ and $f\in H(\widehat{\sigma})_K$, 
we have 
\begin{align}
\label{eqn:pf_P211DS_001}
&2(\Pi_{\widehat{\sigma}}(E_{i,j})f)(1_4)= \begin{cases}
2\nu_1+\kappa_1+2&\text{if $i=j=1$},\\
2\nu_2-1&\text{if $i=j=3$},\\
0&\text{if $i<j$}.
\end{cases}
\end{align}

Let $\lambda =(\kappa_1,\delta_2-\delta_3,\delta_3)$ and 
$\delta \in \{0,1\}$ such that $ \delta_2-\delta_3 \ge \delta$. 
We have 
\[
\Hom_K(\cR_{(1,\delta )}\otimes_{\bC}
V_{\lambda -(1,\delta ,0)}, 
H(\sigma)_K)=\bC\,\hat{\eta}_{\sigma }\circ 
\rB_{\lambda}^{(1,\delta )}
\]
by Lemma \ref{lem:tensor1} and (\ref{eqn:P211_minKtype}). 
We define a $K$-homomorphism 
$\rP_\sigma \colon \gp_{\bC}\otimes_{\bC}H(\sigma)_K
\to H(\sigma)_K$ by $X\otimes f\mapsto \Pi_{\sigma}(X)f$. 
Then there is a constant $c_\delta$ such that  
\begin{align}
\label{eqn:pre_dirac2}
c_\delta \,\hat{\eta}_{\sigma }\circ 
\rB_{\lambda}^{(1,\delta )}=
\rP_\sigma \circ (\id_{\gp_{\bC}}\otimes (\hat{\eta}_{\sigma }
\circ \rB_{\lambda}^{(1,\delta )}))\circ 
(\rI^\gp_{(1,\delta )}\otimes 
\id_{V_{\lambda -(1,\delta ,0)}}),
\end{align}
since the right hand side is an element of 
$\Hom_K(\cR_{(1,\delta )}\otimes_{\bC}
V_{\lambda -(1,\delta ,0)}, H(\sigma)_K)$. 

Let us consider the case of $\delta =0$. 
Considering the image of $\xi_i\otimes u_l$ under the both sides of (\ref{eqn:pre_dirac2}), we have 
\begin{align*}
&c_0\,\hat{\eta}_{\sigma }(u_{l+e_i})
=\sum_{k=1}^4
\Pi_\sigma (E_{i,k}^{\gp}) 
\hat{\eta}_{\sigma }(u_{l+e_k})&
&(l\in S_{\lambda -(1,0,0)},\ 1\leq i\leq 4).
\end{align*}
Hence, in order to prove the statement (ii), 
it suffices to show $c_{0}=2\nu_1$. 
Using Lemmas \ref{lem:rel_ul}, \ref{lem:Kact_ul_vl} and 
(\ref{eqn:pf_P211DS_001}), for $l\in S_{\lambda -(1,0,0)}$, 
we have 
\begin{align*}
 c_{0}\iota_\sigma ({\eta}_{\sigma }(u_{l+e_{1}}))
& =\rI_{\sigma}(c_{0}\,\hat{\eta}_{\sigma }(u_{l+e_{1}}))(1_4) \\
& =\mathrm{I}_{\sigma}\!\left(\sum_{k=1}^4\Pi_{\sigma}(E_{1,k}^{\gp}) 
\hat{\eta}_{\sigma }(u_{l+e_k})\right)(1_4)\\
&=2\bigl(
\Pi_{\widehat{\sigma}} (E_{1,1})
\mathrm{I}_{\sigma}(\hat{\eta}_{\sigma}(u_{l+e_1}))
\bigr)(1_4)\\
& \quad 
+\sum_{k=2}^4\bigl\{2\bigl(
\Pi_{\widehat{\sigma}}(E_{1,k})
\mathrm{I}_{\sigma}(\hat{\eta}_{\sigma}(u_{l+e_k}))
\bigr)(1_4)
-\mathrm{I}_{\sigma}(\hat{\eta}_{\sigma}(
\tau_{\lambda }(E_{1,k}^{\gk})u_{l+e_k}))(1_4)\bigr\}\\
&=2\nu_1\iota_\sigma (\eta_{\sigma}(u_{l+e_1})).
\end{align*}
Our task is to show the existence of $l\in S_{\lambda -(1,0,0)}$ 
such that $\iota_\sigma (\eta_{\sigma}(u_{l+e_1}))\neq 0$. 
When $\delta_2=\delta_3$, since 
\begin{align*}
&u_{\kappa_1e_1}+\sI u_{(\kappa_1-1)e_1+e_2}
=\rrq_{\cR}(\zeta_2 (\xi_1)^{\kappa_1-1})
=2^{-\kappa_1+1}
\rrq_{\cR}(\zeta_2(\zeta_2+\zeta_1)^{\kappa_1-1}),
\end{align*}
we have $\iota_\sigma (\eta_{\sigma}(u_{\kappa_1e_1}))
+\sI \iota_\sigma (\eta_{\sigma}(u_{(\kappa_1-1)e_1+e_2}))
=2^{-\kappa_1+1}$. 
When $(\delta_2,\delta_3)=(1,0)$, since  
\begin{align*}
u_{(\kappa_1-1)e_1+e_{13}}&=\rrq_{\cR}( (\xi_1)^{\kappa_1-1} \xi_{13})
=2^{-\kappa_1-1}
\rrq_{\cR} ((\zeta_2+\zeta_1)^{\kappa_1-1}
(\zeta_{23}+\zeta_{24}+\zeta_{13}+\zeta_{14})),
\end{align*}
we have $\iota_\sigma (\eta_{\sigma}(u_{(\kappa_1-1)e_1+e_{13}}))
=2^{-\kappa_1+1}$. 
Therefore, we obtain the statement (ii).

Assume $ (\delta_2,\delta_3) = (1,0) $.
Let us consider the case of $\delta =1$. 
Considering the image of $\xi_{ij}\otimes u_l$ 
under the both sides of (\ref{eqn:pre_dirac2}), we have 
\begin{align*}
&c_1\,\hat{\eta}_{\sigma }(u_{l+e_{ij}})
=\Pi_\sigma (E_{i,i}^{\gp}+E_{j,j}^{\gp})\hat{\eta}_{\sigma }(u_{l+e_{ij}})\\
&\hspace{2mm}
+\sum_{1\leq k\leq 4,\, k\not\in \{i,j\}}
\{\sgn (j-k)\Pi_\sigma (E_{i,k}^{\gp})\hat{\eta}_{\sigma }(u_{l+e_{kj}})
+\sgn (k-i)\Pi_\sigma (E_{j,k}^{\gp})\hat{\eta}_{\sigma }(u_{l+e_{ik}})\}
\end{align*}
for $l\in S_{\lambda -(1,1,0)}$ and 
$1\leq i<j\leq 4$. 
Hence, in order to prove the statement (iii), 
it suffices to show $c_{1}=2(\nu_1+\nu_2)$. 
Using Lemmas \ref{lem:rel_ul}, \ref{lem:Kact_ul_vl} and 
(\ref{eqn:pf_P211DS_001}), we have 
\begin{align*}
& c_{1}\,\iota_\sigma ({\eta}_{\sigma }(u_{(\kappa_1-1)e_1+e_{13}})) \\
&=\mathrm{I}_{\sigma}(c_{1}\,\hat{\eta}_{\sigma }(u_{(\kappa_1-1)e_1+e_{13}}))(1_4)\\
&=\mathrm{I}_{\sigma}\!\bigl(
\Pi_\sigma (E_{1,1}^{\gp}+E_{3,3}^{\gp})
\hat{\eta}_{\sigma }(u_{(\kappa_1-1)e_1+e_{13}})\\
&\phantom{=}
+\Pi_\sigma (E_{1,2}^{\gp})\hat{\eta}_{\sigma }(u_{(\kappa_1-1)e_1+e_{23}})
+\Pi_\sigma (E_{2,3}^{\gp})\hat{\eta}_{\sigma }(u_{(\kappa_1-1)e_1+e_{12}})\\
&\phantom{=}
+\Pi_\sigma (E_{3,4}^{\gp})\hat{\eta}_{\sigma }(u_{(\kappa_1-1)e_1+e_{14}})
-\Pi_\sigma (E_{1,4}^{\gp})\hat{\eta}_{\sigma }(u_{(\kappa_1-1)e_1+e_{34}})
\bigr)\!(1_4)\\
&=2\bigl(\Pi_{\widehat{\sigma}}(E_{1,1}+E_{3,3})
\mathrm{I}_{\sigma}(\hat{\eta}_{\sigma }(u_{(\kappa_1-1)e_1+e_{13}})\bigr)(1_4)\\
&\phantom{=}
+2\bigl(\Pi_{\widehat{\sigma}}(E_{1,2})
\mathrm{I}_{\sigma}(\hat{\eta}_{\sigma}(u_{(\kappa_1-1)e_1+e_{23}}))\bigr)(1_4)
-\mathrm{I}_{\sigma}(\hat{\eta}_{\sigma}(
\tau_{\lambda}(E_{1,2}^{\gk})u_{(\kappa_1-1)e_1+e_{23}}))(1_4)\\
&\phantom{=}
+2\bigl(\Pi_{\widehat{\sigma}}(E_{2,3})
\mathrm{I}_{\sigma}(\hat{\eta}_{\sigma}(u_{(\kappa_1-1)e_1+e_{12}}))\bigr)(1_4)
-\mathrm{I}_{\sigma}(\hat{\eta}_{\sigma}(
\tau_{\lambda}(E_{2,3}^{\gk})u_{(\kappa_1-1)e_1+e_{12}}))(1_4)\\
&\phantom{=}
+2\bigl(\Pi_{\widehat{\sigma}}(E_{3,4})
\mathrm{I}_{\sigma}(\hat{\eta}_{\sigma}(u_{(\kappa_1-1)e_1+e_{14}}))\bigr)(1_4)
-\mathrm{I}_{\sigma}(\hat{\eta}_{\sigma}(
\tau_{\lambda}(E_{3,4}^{\gk})u_{(\kappa_1-1)e_1+e_{14}}))(1_4)\\
&\phantom{=}
-2\bigl(\Pi_{\widehat{\sigma}}(E_{1,4})
\mathrm{I}_{\sigma}(\hat{\eta}_{\sigma}(u_{(\kappa_1-1)e_1+e_{34}}))\bigr)(1_4)
+\mathrm{I}_{\sigma}(\hat{\eta}_{\sigma}(
\tau_{\lambda}(E_{1,4}^{\gk})u_{(\kappa_1-1)e_1+e_{34}}))(1_4)\\
&=2(\nu_1+\nu_2)\iota_\sigma (\eta_{\sigma}(u_{(\kappa_1-1)e_1+e_{13}})).
\end{align*}
Since $\iota_\sigma (\eta_{\sigma}(u_{(\kappa_1-1)e_1+e_{13}}))
=2^{-\kappa_1+1}\neq 0$, we obtain the statement (iii). 
\end{proof}

\subsection{$P_{(2,2)}$-principal series representations}
\label{subsec:P22}

Let $\sigma =D_{(\nu_1,\kappa_1)}\boxtimes D_{(\nu_2,\kappa_2)}$ 
with $\nu_1,\nu_2\in \bC$ and 
$\kappa_1,\kappa_2\in \bZ_{\geq 2}$ such that $\kappa_1\geq \kappa_2$. 
For $i\in \{1,2\}$, we set 
$\widehat{\sigma}_i=\chi_{(\nu_i+(\kappa_i-1)/2,\delta_i)}
\boxtimes \chi_{(\nu_i-(\kappa_i-1)/2,0)}$ 
with $\delta_i\in \{0,1\}$ such that 
$\delta_i\equiv \kappa_i\bmod 2$. 
The group $K\cap M_{(2,2)}$ is generated by the elements 
\begin{align*}
&\rk_{\theta_1,\theta_2}^{(2,2)}\qquad (\theta_1,\theta_2\in \bR),&
&\diag (\varepsilon_1,\varepsilon_2,\varepsilon_3,\varepsilon_4)\qquad 
(\varepsilon_1,\varepsilon_2,\varepsilon_3,\varepsilon_4
\in \{\pm 1\}).
\end{align*}
Because of (\ref{eqn:GL2ps_Kact1}), 
these elements act on $U_{\sigma ,K\cap M_{(2,2)}}
=\gH_{(\nu_1,\kappa_1),K_2}\boxtimes_\bC \gH_{(\nu_2,\kappa_2),K_2}$ by 
\begin{align*} 
&\sigma (\rk_{\theta_1,\theta_2}^{(2,2)})
{\rf}_{(\widehat{\sigma}_1,q_1)}\boxtimes 
{\rf}_{(\widehat{\sigma}_2,q_2)}
=e^{\sI (q_1\theta_1+q_2\theta_2) }
{\rf}_{(\widehat{\sigma}_1,q_1)}\boxtimes 
{\rf}_{(\widehat{\sigma}_2,q_2)},\\
&\sigma (\diag (\varepsilon_1,\varepsilon_2,\varepsilon_3,\varepsilon_4))
{\rf}_{(\widehat{\sigma}_1,q_1)}\boxtimes 
{\rf}_{(\widehat{\sigma}_2,q_2)}
=\varepsilon_1^{\kappa_1}\varepsilon_3^{\kappa_2}
{\rf}_{(\widehat{\sigma}_1,\varepsilon_1\varepsilon_2q_1)}\boxtimes 
{\rf}_{(\widehat{\sigma}_2,\varepsilon_3\varepsilon_4q_2)}
\end{align*}
for $q_1\in \kappa_1+2\bZ$ and $q_2\in \kappa_2+2\bZ$ 
such that $|q_1|\geq \kappa_1$ and $|q_2|\geq \kappa_2$.
By these equalities and Lemma \ref{lem:Kact_ul_vl} (ii), 
for $\lambda \in \Lambda_K$, we have 
\begin{align*}
&\Hom_{K\cap M_{(2,2)}}(V_{\lambda },U_{\sigma ,K\cap M_{(2,2)}})
=\begin{cases}
\displaystyle 
\bC\,\eta_{\sigma}
&\text{if $ \lambda =(\kappa_1,\kappa_2,0)$},\\
\{0\}&\text{if $ \lambda <_{\mathrm{lex}} 
(\kappa_1,\kappa_2,0)$},
\end{cases}
\end{align*}
where $\eta_{\sigma}\colon 
V_{(\kappa_1,\kappa_2,0)}\to U_{\sigma, K\cap M_{(2,2)}} $ 
is a $\bC$-linear map defined by 
\begin{align*}
&\eta_{\sigma}(v_{l})
=\begin{cases}
{\rf}_{(\widehat{\sigma}_1,\kappa_1)}\boxtimes 
{\rf}_{(\widehat{\sigma}_2,\kappa_2)}
&\text{if $ l=(\kappa_1-\kappa_2)e_{2}+\kappa_2e_{24}$},\\
{\rf}_{(\widehat{\sigma}_1,\kappa_1)}\boxtimes 
{\rf}_{(\widehat{\sigma}_2,-\kappa_2)}
&\text{if $ l=(\kappa_1-\kappa_2)e_{2}+\kappa_2e_{23}$},\\
{\rf}_{(\widehat{\sigma}_1,-\kappa_1)}\boxtimes 
{\rf}_{(\widehat{\sigma}_2,\kappa_2)}
&\text{if $ l=(\kappa_1-\kappa_2)e_{1}+\kappa_2e_{14}$},\\
{\rf}_{(\widehat{\sigma}_1,-\kappa_1)}\boxtimes 
{\rf}_{(\widehat{\sigma}_2,-\kappa_2)}
&\text{if $l=(\kappa_1-\kappa_2)e_{1}+\kappa_2e_{13}$},\\
0&\text{otherwise}
\end{cases}
\end{align*}
for $l\in S_{(\kappa_1,\kappa_2,0)}$. 
By the Frobenius reciprocity law \cite[Theorem 1.14]{Knapp_002}, we have 
\begin{align}
\label{eqn:P22_minKtype}
&\Hom_{K}(V_{\lambda },H(\sigma )_K)
= \begin{cases}
\displaystyle 
\bC\,\hat{\eta}_{\sigma}
&\text{if $ \lambda =(\kappa_1,\kappa_2,0)$},\\
\{0\}
&\text{if $ \lambda <_{\mathrm{lex}}(\kappa_1,\kappa_2,0) $}
\end{cases} &
&(\lambda \in \Lambda_K),
\end{align}
where 
$\hat{\eta}_{\sigma}(v)(k)
=\eta_{\sigma}(\tau_{(\kappa_1,\kappa_2,0)}(k)v)$ 
for $v\in V_{(\kappa_1,\kappa_2,0)}$ and $k\in {K}$. 
We call $\tau_{(\kappa_1,\kappa_2,0)}$ 
the minimal $K$-type of $\Pi_\sigma$. 

Let $\widehat{\sigma}=\widehat{\sigma}_1\boxtimes \widehat{\sigma}_2$. 
We define a $\bC$-linear form $\iota_\sigma \colon U_\sigma \to \bC$ 
by 
\begin{align*}
&\iota_\sigma  (f_1\boxtimes f_2)=f_1(1_2)f_2(1_2)&
&(f_1\in \gH_{(\nu_1,\kappa_1)},\ f_2\in  \gH_{(\nu_2,\kappa_2)}).
\end{align*}
Then $\Pi_\sigma $ is embedded into a principal series representation 
$\Pi_{\widehat{\sigma}}$ via 
\begin{equation}
\label{eqn:P22_embed}
\rI_\sigma \colon 
H(\sigma)\ni f\mapsto \iota_\sigma \circ f\in H(\widehat{\sigma}). 
\end{equation}

\begin{prop}
\label{prop:P22_DS}
Retain the notation. 

\noindent (i) For $f\in H (\sigma )_K$, we have 
\begin{align*}
&\Pi_\sigma (\cC_1)f=(2\nu_1+2\nu_2)f,\\
&\Pi_\sigma (\cC_2)f
=\bigl(\nu_1^2+4\nu_1\nu_2+\nu_2^2
-\tfrac{(\kappa_1-1)^2+(\kappa_2-1)^2}{4}\bigr)f,\\
&\Pi_\sigma (\cC_3)f
=\bigl(2\nu_1^2\nu_2+2\nu_1\nu_2^2
-\tfrac{(\kappa_2-1)^2\nu_1+(\kappa_1-1)^2\nu_2}{2}\bigr)f,\\
&\Pi_\sigma (\cC_4)f
=\bigl(\nu_1^2-\tfrac{(\kappa_1-1)^2}{4}\bigr)
\bigl(\nu_2^2-\tfrac{(\kappa_2-1)^2}{4}\bigr)f.
\end{align*}
\noindent (ii) Assume $\kappa_1>\kappa_2$. 
 For $l\in S_{(\kappa_1-1,\kappa_2,0)}$ and 
$1\leq i\leq 4$, we have 
\begin{align*}
2\nu_1\,\hat{\eta}_{\sigma }(u_{l+e_i})
=\sum_{k=1}^4
\Pi_\sigma (E_{i,k}^{\gp}) 
\hat{\eta}_{\sigma }(u_{l+e_k}).
\end{align*}
\noindent (iii) 
For $l\in S_{(\kappa_1-1,\kappa_2-1,0)}$ and 
$1\leq i<j\leq 4$, we have 
\begin{align*}
&2(\nu_1+\nu_2)\,\hat{\eta}_{\sigma }(u_{l+e_{ij}})
=\Pi_\sigma (E_{i,i}^{\gp}+E_{j,j}^{\gp})\hat{\eta}_{\sigma }(u_{l+e_{ij}})\\
&\hspace{2mm}
+\sum_{1\leq k\leq 4,\ k\not\in \{i,j\}}
\{\sgn (j-k)\Pi_\sigma (E_{i,k}^{\gp})\hat{\eta}_{\sigma }(u_{l+e_{kj}})
+\sgn (k-i)\Pi_\sigma (E_{j,k}^{\gp})\hat{\eta}_{\sigma }(u_{l+e_{ik}})\}.
\end{align*}
\end{prop}
\begin{proof}
The proof of this proposition is similar to that of 
Proposition \ref{prop:P211_DS}. 
The statement (i) follows immediately from 
Proposition \ref{prop:P1111_DS} (i) and (\ref{eqn:P22_embed}). 
By definition, for $1\leq i,j\leq 4$ and $f\in H(\widehat{\sigma})_K$, we have 
\begin{align}
\label{eqn:pf_P22DS_001}
&2(\Pi_{\widehat{\sigma}}(E_{i,j})f)(1_4)=\begin{cases}
2\nu_1+\kappa_1+2&\text{if $i=j=1$},\\
2\nu_2-\kappa_2-2&\text{if $i=j=4$},\\
0&\text{if $i<j$}.
\end{cases}
\end{align}

Let $\lambda =(\kappa_1,\kappa_2,0)$ and 
$\delta \in \{0,1\}$ such that 
$\kappa_1-1\geq \kappa_2-\delta $. 
We have 
\[
\Hom_K(\cR_{(1,\delta )}\otimes_{\bC}
V_{\lambda -(1,\delta ,0)}, 
H(\sigma)_K)=\bC\,\hat{\eta}_{\sigma }\circ 
\rB_{\lambda}^{(1,\delta )}
\]
by Lemma \ref{lem:tensor1} and (\ref{eqn:P22_minKtype}). 
We define a $K$-homomorphism 
$\rP_\sigma \colon \gp_{\bC}\otimes_{\bC}H(\sigma)_K
\to H(\sigma)_K$ by $X\otimes f\mapsto \Pi_{\sigma}(X)f$. 
Then there is a constant $c_\delta$ such that  
\begin{align}
\label{eqn:pre_dirac1}
c_\delta \,\hat{\eta}_{\sigma }\circ 
\rB_{\lambda}^{(1,\delta )}=
\rP_\sigma \circ (\id_{\gp_{\bC}}\otimes (\hat{\eta}_{\sigma }
\circ \rB_{\lambda}^{(1,\delta )}))\circ 
(\rI^\gp_{(1,\delta )}\otimes 
\id_{V_{\lambda -(1,\delta ,0)}}),
\end{align}
since the right hand side is an element of 
$\Hom_K(\cR_{(1,\delta )}\otimes_{\bC}
V_{\lambda -(1,\delta ,0)}, H(\sigma)_K)$.

Let us consider the case of $\delta =1$. 
Considering the image of $\xi_{ij}\otimes u_l$ 
under the both sides of (\ref{eqn:pre_dirac1}), we have 
\begin{align*}
&c_1\,\hat{\eta}_{\sigma }(u_{l+e_{ij}})
=\Pi_\sigma (E_{i,i}^{\gp}+E_{j,j}^{\gp})\hat{\eta}_{\sigma }(u_{l+e_{ij}})\\
&\hspace{2mm}
+\sum_{1\leq k\leq 4,\, k\not\in \{i,j\}}
\{\sgn (j-k)\Pi_\sigma (E_{i,k}^{\gp})\hat{\eta}_{\sigma }(u_{l+e_{kj}})
+\sgn (k-i)\Pi_\sigma (E_{j,k}^{\gp})\hat{\eta}_{\sigma }(u_{l+e_{ik}})\}
\end{align*}
for $l\in S_{\lambda -(1,1,0)}$ and 
$1\leq i<j\leq 4$. 
Hence, in order to prove the statement (iii), 
it suffices to show $c_{1}=2(\nu_1+\nu_2)$. 
Let 
\[
l'=(\kappa_1-\kappa_2,0,0,0,0,l_{13}',l_{14}',l_{23}',l_{24}',0)\in 
S_{\lambda -(1,1,0)}.
\]
Using Lemmas \ref{lem:rel_ul}, \ref{lem:Kact_ul_vl} and 
(\ref{eqn:P22_embed}), 
we have 
\begin{align*}
&c_{1}\,\iota_\sigma ({\eta}_{\sigma }(u_{l'+e_{14}}))\\
&=\mathrm{I}_{\sigma}(c_{1}\,\hat{\eta}_{\sigma }(u_{l'+e_{14}}))(1_4)\\
&=\mathrm{I}_{\sigma}\!\bigl(
\Pi_\sigma (E_{1,1}^{\gp}+E_{4,4}^{\gp})\hat{\eta}_{\sigma }(u_{l'+e_{14}})
+\Pi_\sigma (E_{1,2}^{\gp})\hat{\eta}_{\sigma }(u_{l'+e_{24}})
+\Pi_\sigma (E_{3,4}^{\gp})\hat{\eta}_{\sigma }(u_{l'+e_{13}})\\
&\phantom{=}
+\Pi_\sigma (E_{1,3}^{\gp})\hat{\eta}_{\sigma }(u_{l'+e_{34}})
+\Pi_\sigma (E_{2,4}^{\gp})\hat{\eta}_{\sigma }(u_{l'+e_{12}})
\bigr)\!(1_4)\\
&=2\bigl(\Pi_{\widehat{\sigma}} (E_{1,1}+E_{4,4})
\mathrm{I}_{\sigma}(\hat{\eta}_{\sigma}(u_{l'+e_{14}}))
\bigr)(1_4)\\
&\phantom{=}
+2\bigl(
\Pi_{\widehat{\sigma}}(E_{1,2})
\mathrm{I}_{\sigma}(\hat{\eta}_{\sigma}(u_{l'+e_{24}})) \bigr)(1_4)
-\mathrm{I}_{\sigma}(\hat{\eta}_{\sigma}(
\tau_{(\kappa_1,\kappa_2,0)}(E_{1,2}^{\gk})u_{l'+e_{24}}))(1_4)\\
&\phantom{=}
+2\bigl(
\Pi_{\widehat{\sigma}}(E_{3,4})
\mathrm{I}_{\sigma}(\hat{\eta}_{\sigma}(u_{l'+e_{13}}))
\bigr)(1_4)
-\mathrm{I}_{\sigma}(\hat{\eta}_{\sigma}(
\tau_{(\kappa_1,\kappa_2,0)}(E_{3,4}^{\gk})u_{l'+e_{13}}))(1_4)\\
&\phantom{=}
+2\bigl(\Pi_{\widehat{\sigma}}(E_{1,3})
\mathrm{I}_{\sigma}(\hat{\eta}_{\sigma}(u_{l'+e_{34}}))
\bigr)(1_4)
-\mathrm{I}_{\sigma}(\hat{\eta}_{\sigma}(
\tau_{(\kappa_1,\kappa_2,0)}(E_{1,3}^{\gk})u_{l'+e_{34}}))(1_4)\\
&\phantom{=}
+2\bigl(\Pi_{\widehat{\sigma}}(E_{2,4})
\mathrm{I}_{\sigma}(\hat{\eta}_{\sigma}(u_{l'+e_{12}}))
\bigr)(1_4)
-\mathrm{I}_{\sigma}(\hat{\eta}_{\sigma}(
\tau_{(\kappa_1,\kappa_2,0)}(E_{2,4}^{\gk})u_{l'+e_{12}}))(1_4)\\
&=2(\nu_1+\nu_2)\iota_\sigma (\eta_{\sigma}(u_{l'+e_{14}})).
\end{align*}
Our task is to show the existence of $l'$ 
such that $\iota_\sigma (\eta_{\sigma}(u_{l'+e_{14}}))\neq 0$ if $\kappa_2 \ge 2$.
Since 
\begin{align*}
&\sum_{i=1}^2\sum_{j=3}^4  (\sI )^{i+j}u_{(\kappa_1-\kappa_2)e_1+(\kappa_2-1)e_{14}+e_{ij}} \\
& =\rrq_{\cR}(\xi_1^{\kappa_1-\kappa_2}\xi_{14}^{\kappa_2-1}\zeta_{24})\\
&=2^{-\kappa_1-\kappa_2+2}(\sI )^{-\kappa_2+1}
\rrq_{\cR}((\zeta_2+\zeta_1)^{\kappa_1-\kappa_2}
(\zeta_{24}-\zeta_{23}-\zeta_{13}+\zeta_{14})^{\kappa_2-1}
\zeta_{24}),
\end{align*}
we have 
\begin{align}
\label{eqn:pf_P22DS_002}
&\sum_{i=1}^2\sum_{j=3}^4(\sI )^{i+j}
\iota_\sigma (\eta_{\sigma}(
u_{(\kappa_1-\kappa_2)e_1+(\kappa_2-1)e_{14}+e_{ij}}))
=2^{-\kappa_1-\kappa_2+2}(\sI )^{-\kappa_2+1}
\end{align}
and complete the proof of the statement (iii).

Assume $\kappa_1>\kappa_2$. 
Let us consider the case of $\delta =0$. 
Considering the image of $\xi_i\otimes u_l$ 
under the both sides of (\ref{eqn:pre_dirac1}), we have 
\begin{align*}
&c_0\,\hat{\eta}_{\sigma }(u_{l+e_i})
=\sum_{k=1}^4
\Pi_\sigma (E_{i,k}^{\gp}) 
\hat{\eta}_{\sigma }(u_{l+e_k})&
&(l\in S_{\lambda -(1,0,0)},\ 1\leq i\leq 4). 
\end{align*}
Hence, in order to prove the statement (ii), 
it suffices to show $c_{0}=2\nu_1$. 
Using Lemmas \ref{lem:rel_ul}, \ref{lem:Kact_ul_vl} and 
(\ref{eqn:pf_P22DS_001}), for $l\in S_{\lambda -(1,0,0)}$, we have 
\begin{align*}
c_{0}\iota_\sigma ({\eta}_{\sigma }(u_{l+e_{1}})) 
& =\rI_{\sigma}(c_{0}\,\hat{\eta}_{\sigma }(u_{l+e_{1}}))(1_4) \\
& =\mathrm{I}_{\sigma}\!\left(\sum_{k=1}^4\Pi_{\sigma}(E_{1,k}^{\gp}) 
\hat{\eta}_{\sigma }(u_{l+e_k})\right)(1_4)\\
&=2\bigl(
\Pi_{\widehat{\sigma}} (E_{1,1})
\mathrm{I}_{\sigma}(\hat{\eta}_{\sigma}(u_{l+e_1}))
\bigr)(1_4)\\
& \quad 
+\sum_{k=2}^4\bigl\{2\bigl(
\Pi_{\widehat{\sigma}}(E_{1,k})
\mathrm{I}_{\sigma}(\hat{\eta}_{\sigma}(u_{l+e_k}))
\bigr)(1_4)
-\mathrm{I}_{\sigma}(\hat{\eta}_{\sigma}(
\tau_{\lambda }(E_{1,k}^{\gk})u_{l+e_k}))(1_4)\bigr\}\\
&=2\nu_1\iota_\sigma (\eta_{\sigma}(u_{l+e_1})).
\end{align*}
By (\ref{eqn:pf_P22DS_002}), we know that 
there is $l\in S_{\lambda -(1,0,0)}$ 
such that $\iota_\sigma (\eta_{\sigma}(u_{l+e_1}))\neq 0$. 
Therefore, we obtain the statement (ii). 
\end{proof}


\section{Partial differential equations for Whittaker functions}
\label{sec:PDE}

\subsection{System of partial differential equations}
\label{subsec;system_PDE}

In this subsection we give a system of partial differential equations 
satisfied by radial parts of Whittaker functions using 
Propositions \ref{prop:P1111_DS}, \ref{prop:P211_DS} and \ref{prop:P22_DS}.
Hereafter we say that we are in {\bf case 1}, {\bf case 2} and {\bf case 3} if $ \Pi_{\sigma} $ are 
$ P_{(1,1,1,1)} $, $ P_{(2,1,1)} $ and $ P_{(2,2)} $-principal series representations, respectively. 
We divide the cases 1 and 2 into subclasses according as $ \delta_i  $.
\begin{itemize}
\item Case 1:  
$ \sigma = \chi_{(\nu_1,\delta_1)} \boxtimes \chi_{(\nu_2,\delta_2)} \boxtimes 
\chi_{(\nu_3,\delta_3)} \boxtimes \chi_{(\nu_4,\delta_4)}  $ with $ \delta_1 \ge \delta_2 \ge \delta_3 \ge \delta_4 $.
\begin{itemize}
\item Case 1-(i): $ (\delta_1,\delta_2,\delta_3,\delta_4) = (0,0,0,0),(1,1,1,1). $
\item Case 1-(ii):  $ (\delta_1,\delta_2,\delta_3,\delta_4) = (1,0,0,0). $
\item Case 1-(iii): $ (\delta_1,\delta_2,\delta_3,\delta_4) = (1,1,0,0). $
\item Case 1-(iv): $ (\delta_1,\delta_2,\delta_3,\delta_4) = (1,1,1,0). $
\end{itemize}
\item Case 2: 
$ \sigma = D_{(\nu_1,\kappa_1)} \boxtimes \chi_{(\nu_2,\delta_2)} \boxtimes \chi_{(\nu_3,\delta_3)} $ 
with $ \delta_2 \ge \delta_3 $.
\begin{itemize}
\item Case 2-(i): $ (\delta_2,\delta_3) = (0,0), (1,1).  $
\item Case 2-(ii): $ (\delta_2,\delta_3) = (1,0). $
\end{itemize}
\item Case 3:
$ \sigma = D_{(\nu_1,\kappa_1)} \boxtimes D_{(\nu_2,\kappa_2)} $ with $ \kappa_1 \ge \kappa_2 $.
\end{itemize} 
We introduce the following notation to discuss three kinds of the generalized principal series representations simultaneously.
\begin{itemize}
\item 
Case 1: 
$ \kappa_1 := \delta_1-\delta_4 $, $ \kappa_2 :=  \delta_2-\delta_3 $ and 
$ \nu_1' := \begin{cases} \nu_1&  \text{case 1-(i),(ii),(iii)}, \\ \nu_4 & \text{case 1-(iv)}. \end{cases} $ \smallskip 
\item
Case 2: 
$ \delta_1 := \begin{cases} 0 & \text{if $ \kappa_1 $ is even}, \\ 1 & \text{if $ \kappa_1 $ is odd}, \end{cases} $ 
$ \kappa_2:= \delta_2-\delta_3 $ and $ \nu_1' := \nu_1 $. \smallskip 
\item 
Case 3: 
$ \delta_1:= \begin{cases} 0 & \text{if $ \kappa_1 $ is even}, \\ 1 & \text{if $ \kappa_1 $ is odd}, \end{cases} $
$ \delta_2:= \begin{cases} 0 & \text{if $ \kappa_2 $ is even}, \\ 1 & \text{if $ \kappa_2 $ is odd}, \end{cases} $
$ \delta_3:= 0 $ and $ \nu_1' := \nu_1. $
\end{itemize}
For $ 1 \le i \le 4 $, we set 
\begin{align*}
\gamma_i = \begin{cases} 
 \s_i(\nu_1,\nu_2,\nu_3,\nu_4) & \text{case 1}, \\
 \s_i(\nu_1+\tfrac{\kappa_1-1}{2},\nu_1-\tfrac{\kappa_1-1}{2}, \nu_2,\nu_3) &\text{case 2}, \\
 \s_i(\nu_1+\tfrac{\kappa_1-1}{2},\nu_1-\tfrac{\kappa_1-1}{2}, \nu_2+\tfrac{\kappa_2-1}{2}, \nu_2-\tfrac{\kappa_2-1}{2}) 
 & \text{case 3},
\end{cases}
\end{align*}
where 
\[ \s_i(a_1,a_2,a_3,a_4) = \sum_{ 1 \le k_1 <  k_2 < \dots < k_i \le 4} a_{k_1} a_{k_2} \cdots a_{k_i} \]
is the $ i $-th elementary symmetric polynomial.
Then $ \tau_{(\kappa_1,\kappa_2,\delta_3)} $ is the minimal $K$-type of $ \Pi_{\sigma} $,
and the results of Propositions \ref{prop:P1111_DS}, \ref{prop:P211_DS} and \ref{prop:P22_DS} 
are summarized as follows: 

\begin{itemize}
\item For $f \in H(\sigma)_K$ and $ 1 \le i \le 4 $, we have 
\begin{align} \label{Cap_eigen}
 \Pi_{\sigma}( \cC_i) f = \gamma_i f.
\end{align}
\item 
Assume $\kappa_1>\kappa_2$. 
For $l\in S_{(\kappa_1-1,\kappa_2,\delta_3)}$ and 
$1\leq i\leq 4$, we have 
\begin{align} \label{DiracSchmid1}
2\nu_1' \hat{\eta}_{\sigma }(u_{l+e_i})
=\sum_{k=1}^4
\Pi_\sigma (E_{i,k}^{\gp}) 
\hat{\eta}_{\sigma }(u_{l+e_k}).
\end{align}
\item 
Assume $ \kappa_2 \ge 1 $.
For $l\in S_{(\kappa_1-1,\kappa_2-1,0)}$ and 
$1\leq i<j\leq 4$, we have 
\begin{align} \label{DiracSchmid2}
\begin{split}
&2(\nu_1+\nu_2)\,\hat{\eta}_{\sigma }(u_{l+e_{ij}}) 
 =\Pi_\sigma (E_{i,i}^{\gp}+E_{j,j}^{\gp})\hat{\eta}_{\sigma }(u_{l+e_{ij}})\\
&\hspace{2mm}
+\sum_{1\leq k\leq 4, \, k\not\in \{i,j\}}
\{\sgn (j-k)\Pi_\sigma (E_{i,k}^{\gp})\hat{\eta}_{\sigma }(u_{l+e_{kj}})
+\sgn (k-i)\Pi_\sigma (E_{j,k}^{\gp})\hat{\eta}_{\sigma }(u_{l+e_{ik}})\}.
\end{split}
\end{align}
\end{itemize}

\begin{lem}[{\cite[Lemma 2.1]{HIM}}] 
\label{lem:Eij_radial} 
Let $f$ be a function in $ C^{\infty}(N\backslash G; \psi_1) $. 
Then, for $ 1 \le i\le j \le 4 $ and $ y = {\rm diag}(y_1y_2y_3y_4,y_2y_3y_4,y_3y_4,y_4) \in A $, we have
\begin{align*}
 (R(E_{i,j}) f)(y) = \begin{cases} 
   (-\pa_{i-1}+\pa_i) f(y) & \mbox{ if } j=i, \\
   2 \pi \sqrt{-1} y_i f(y) & \mbox{ if } j=i+1, \\
   0 & \mbox{ otherwise}, \end{cases}
\end{align*}
where $ \pa_0 = 0 $ and $ \pa_i = y_i \dfrac{\pa}{\pa y_i} $ $(1 \le i \le 4) $.
\end{lem}

\begin{lem} \label{lem:Capelli}
For $ X, Y \in U(\g_{\bC}) $, we write $ X \equiv Y $ if $ X- Y \in (\bC E_{1,3}+\bC E_{1,4}+\bC E_{2,4}) U(\g_{\bC}) $.
Then we have the following:
\begin{gather}
\label{eqn:redC1}
\begin{split}
 \cC_1 & \equiv E_{1,1} + E_{2,2} + E_{3,3} + E_{4,4}, 
\end{split}
\\
\label{eqn:redC2}
\begin{split}
 \cC_2  
& \equiv  \sum_{1 \le i<j \le 4} (E_{i,i} + i-\tfrac52 )(E_{j,j} + j -  \tfrac52) 
   - (E_{1,2})^2 - (E_{2,3})^2 - (E_{3,4})^2 
    + E_{1,2} E_{1,2}^{\gk} + E_{2,3} E_{2,3}^{\gk} + E_{3,4}E_{3,4}^{\gk}, 
\end{split}
\\
\label{eqn:redC3}
\begin{split}
 \cC_3 
& \equiv (E_{1,1}-\tfrac32)(E_{2,2}-\tfrac12)(E_{3,3}+\tfrac12) + (E_{1,1}-\tfrac32)(E_{2,2}-\tfrac12)(E_{4,4}+\tfrac32) \\
& \quad 
+ (E_{1,1}-\tfrac32)(E_{3,3}+\tfrac12)(E_{4,4}+\tfrac32) +(E_{2,2}-\tfrac12) (E_{3,3}+\tfrac12)(E_{4,4}+\tfrac32) \\
&  \quad 
- (E_{1,2})^2 (E_{3,3}+E_{4,4}+2) + E_{1,2}(E_{3,3}+E_{4,4}+2) E_{1,2}^{\gk}  \\
& \quad - (E_{2,3})^2 (E_{1,1}+E_{4,4})  + E_{2,3} (E_{1,1}+E_{4,4}) E_{2,3}^{\gk} \\
& \quad
- (E_{3,4})^2 (E_{1,1}+E_{2,2}-2)  + E_{3,4}(E_{1,1}+E_{2,2}-2) E_{3,4}^{\gk} 
- E_{1,2}E_{2,3} E_{1,3}^{\gk} - E_{2,3}E_{3,4}E_{2,4}^{\gk},
\end{split}
\end{gather}
\begin{align}
\label{eqn:redC4}
\begin{split}
 \cC_4
& \equiv (E_{1,1}-\tfrac32)(E_{2,2}-\tfrac12)(E_{3,3}+\tfrac12)(E_{4,4}+\tfrac32) \\
& \quad  - (E_{1,2})^2 (E_{3,3}+\tfrac12) (E_{4,4}+\tfrac32) + E_{1,2}(E_{3,3}+\tfrac12)(E_{4,4}+\tfrac32) E_{1,2}^{\gk} \\ 
& \quad
 - (E_{2,3})^2 (E_{1,1}-\tfrac32)(E_{4,4}+\tfrac32)  + E_{2,3}(E_{1,1}-\tfrac32)(E_{4,4}+\tfrac32) E_{2,3}^{\gk} \\
&\quad - (E_{3,4})^2 (E_{1,1}-\tfrac32) (E_{2,2}-\tfrac12) + E_{3,4}(E_{1,1}-\tfrac32)(E_{2,2}-\tfrac12) E_{3,4}^{\gk} \\
& \quad
 - E_{1,2} E_{2,3} (E_{4,4}+\tfrac12) E_{1,3}^{\gk} - E_{2,3} E_{3,4} (E_{1,1}-\tfrac32) E_{2,4}^{\gk} 
  + E_{1,2}E_{2,3} E_{3,4} E_{1,4}^{\gk} \\
& \quad
  + (E_{1,2})^2(E_{3,4})^2 - (E_{1,2})^2 E_{3,4} E_{3,4}^{\gk} 
  -E_{1,2} (E_{3,4})^2 E_{1,2}^{\gk} + E_{1,2} E_{3,4} E_{1,2}^{\gk} E_{3,4}^{\gk}. 
\end{split}
\end{align}
\end{lem}

\begin{proof}
Recall that the Capelli elements $ \cC_p $ $ (1\le p \le 4) $ are given by
\begin{align*} 
   \cC_{p} = \sum_{\scriptstyle  1\le i_1 <i_2 < \cdots < i_p \le 4, \atop \scriptstyle  w \in {\mathfrak S}_p}
  {\rm sgn}(w)  \cE_{i_1, i_{w(1)}} \cE_{i_2, i_{w(2)}} \cdots \cE_{i_p, i_{w(p)}}.  
\end{align*}
In view of 
\begin{align*}
 \cC_1& = \sum_{1 \le i \le 4} (E_{i,i}+i-\tfrac52) =  \sum_{1 \le i \le 4} E_{i,i},
&
 \cC_2 
& = \sum_{1 \le i<j \le 4} \{ (E_{i,i} + i-\tfrac52 )(E_{j,j} + j -  \tfrac52)  - E_{i,j} E_{j,i}\}
\end{align*}
and $ E_{j,i} = E_{i,j} -E_{i,j}^{\gk} $, we get (\ref{eqn:redC1}) and (\ref{eqn:redC2}).

To consider $ \cC_3 $ and $ \cC_4 $, we define subsets 
$ {\mathfrak S}_{3,q} $ $ (1 \le q \le 3) $ of $  {\mathfrak S}_{3} $ and
$ {\mathfrak S}_{4,q} $ $ (1 \le q \le 5) $ of $  {\mathfrak S}_{4} $ by
\begin{align*}
  {\mathfrak S}_{p,q} & = \{ w \in {\mathfrak S}_p \mid \text{$w$ is a cyclic permutations of length $q$} \}
  & ( 1 \le q \le p), \\
   \mathfrak{S}_{4,5} &= \{ (1\ 2)(3\ 4), (1\ 3)(2\ 4), (1\ 4)(2\ 3) \}.
\end{align*}
We set
\begin{align*} 
   \cC_{p,q} = \sum_{\scriptstyle  1\le i_1 <i_2 < \cdots < i_p \le 4, \atop \scriptstyle  w \in {\mathfrak S}_{p,q}}
  {\rm sgn}(w)  \cE_{i_1, i_{w(1)}} \cE_{i_2, i_{w(2)}} \cdots \cE_{i_p, i_{w(p)}}.
\end{align*}
Then we can write $  \cC_3 = \cC_{3,1} + \cC_{3,2} + \cC_{3,3} $ with 
\begin{align*}
 \cC_{3,1} 
& = \sum_{1 \le i<j<k \le 4} (E_{i,i} + i-\tfrac52)(E_{j,j} + j -  \tfrac52)(E_{k,k} + k-\tfrac52),
\\
 \cC_{3,2} 
& = -\sum_{\scriptstyle 1 \le i<j \le 4,  \,  k \notin \{i,j \} } E_{i,j} E_{j,i} (E_{k,k} + k-\tfrac52), 
\\
 \cC_{3,3} 
& = \sum_{1 \le i<j<k \le 4} (E_{i,j} E_{j,k} E_{k,i} + E_{i,k} E_{j,i} E_{k,j} ).
\end{align*} 
Using   
\begin{align*}
 E_{i,j} E_{j,i} (E_{k,k} + k-\tfrac52)
& = E_{i,j} (E_{i,j} - E_{i,j}^{\gk}) (E_{k,k} + k-\tfrac52) 
   = (E_{i,j})^2  (E_{k,k} + k-\tfrac52) - E_{i,j}  (E_{k,k} + k-\tfrac52) E_{i,j}^{\gk}
\end{align*}
for $ 1 \le i< j \le 4 $ and $ k \notin \{i,j\} $, we know 
\begin{align*}
 \cC_{3,2} 
& \equiv 
- \sum_{\scriptstyle 1 \le i \le 3, \,  k \notin \{i,i+1\} } 
   \{ (E_{i,i+1})^2  (E_{k,k} + k-\tfrac52) - E_{i,i+1}  (E_{k,k} + k-\tfrac52) E_{i,i+1}^{\gk} \}.
\end{align*}
Similarly, for $ 1 \le i<j<k \le 4 $, we know
\begin{align*}
 E_{i,j} E_{j,k} E_{k,i} 
& = E_{i,j} E_{j,k} E_{i,k} - E_{i,j} E_{j,k} E_{i,k}^{\gk}
  \equiv - E_{i,j} E_{j,k} E_{i,k}^{\gk}, 
&  E_{i,k} E_{j,i} E_{k,j}  \equiv 0.
\end{align*}
Here we used $ E_{i,j}^{\gk} E_{j,k} = E_{j,k} E_{i,j}^{\gk} + E_{i,k} $.
Then we have   
\begin{align*}
 \cC_{3,3} 
& \equiv -E_{1,2} E_{2,3} E_{1,3}^{\gk} - E_{2,3} E_{3,4} E_{2,4}^{\gk}.
\end{align*} 
Next we treat $ \cC_4$. 
We can write $ \cC_4 = \cC_{4,1} + \cC_{4,2} + \cC_{4,3} + \cC_{4,4} + \cC_{4,5} $ with 
\begin{align*}
 \cC_{4,1}
& = (E_{1,1}-\tfrac32)(E_{2,2}-\tfrac12)(E_{3,3}+\tfrac12)(E_{4,4}+\tfrac32), 
\\
 \cC_{4,2}
& = -\sum_{ 1 \le i<j \le 4 } E_{i,j} E_{j,i} \sum_{k,l \notin \{ i,j\}, \,k<l } (E_{k,k}+k-\tfrac52)(E_{l,l}+l-\tfrac52),
\\
 \cC_{4,3}
& = \sum_{\scriptstyle 1 \le i<j<k \le 4, \,  \scriptstyle l \notin \{ i,j,k \}} 
     (E_{i,j}E_{j,k}E_{k,i} + E_{i,k}E_{j,i}E_{k,j})(E_{l,l}+l-\tfrac52), 
\\
 \cC_{4,4}
& = - E_{1,2} E_{2,3} E_{3,4} E_{4,1} - E_{1,2} E_{2,4} E_{3,1} E_{4,3} - E_{1,3} E_{2,4} E_{3,2} E_{4,1} \\
& \phantom{= \ }
   -E_{1,3} E_{2,1} E_{3,4} E_{4,2} - E_{1,4} E_{2,3} E_{3,1} E_{4,2}  -E_{1,4} E_{2,1} E_{3,2} E_{4,3},
\\
 \cC_{4,5}
& = E_{1,2} E_{2,1} E_{3,4} E_{4,3} + E_{1,3} E_{3,1} E_{2,4} E_{4,2} + E_{1,4} E_{4,1} E_{2,3} E_{3,2}.
\end{align*}
In the same way as $ \cC_{3,2} $ and $\cC_{3,3} $, we have 
\begin{align*}
 \cC_{4,2}
& \equiv 
    -(E_{1,2})^2  (E_{3,3}+\tfrac12)(E_{4,4}+\tfrac32) + E_{1,2} (E_{3,3}+\tfrac12)(E_{4,4}+\tfrac32)  E_{1,2}^{\gk} \\ 
& \phantom{= \ } 
   - (E_{2,3})^2 (E_{1,1}-\tfrac32)(E_{4,4}+\tfrac32) + E_{2,3} (E_{1,1}-\tfrac32)(E_{4,4}+\tfrac32) E_{2,3}^{\gk} \\
 & \phantom{= \ }
   - (E_{3,4})^2 (E_{1,1}-\tfrac32)(E_{2,2}-\tfrac12)  + E_{3,4} (E_{1,1}-\tfrac32)(E_{2,2}-\tfrac12) E_{3,4}^{\gk},
\\
 \cC_{4,3}
& \equiv -E_{1,2} E_{2,3} (E_{4,4}+\tfrac32) E_{1,3}^{\gk} - E_{2,3} E_{3,4} (E_{1,1}-\tfrac32) E_{2,4}^{\gk}.
\end{align*}
Since  
$ E_{1,2} E_{2,3} E_{3,4} E_{4,1} \equiv -E_{1,2} E_{2,3} E_{3,4} E_{1,4}^{\gk} $ 
and 
\begin{align*}
& E_{1,2} E_{2,4} E_{3,1} E_{4,3} \equiv E_{1,3} E_{2,4} E_{3,2} E_{4,1} 
 \equiv E_{1,3} E_{2,1} E_{3,4} E_{4,2}   \equiv E_{1,4} E_{2,3} E_{3,1} E_{4,2}  
 \equiv  E_{1,4} E_{2,1} E_{3,2} E_{4,3}  
 \equiv 0,  
\end{align*}  we find
$$ \cC_{4,4} \equiv E_{1,2} E_{2,3} E_{3,4} E_{1,4}^{\gk}. $$
As for $ \cC_{4,5} $, we have 
\begin{align*}
  E_{i,j} E_{j,i} E_{k,l} E_{l,k} 
& = E_{i,j} (E_{i,j}-E_{i,j}^{\gk}) E_{k,l} (E_{k,l}-E_{k,l}^{\gk}) \\
& = (E_{i,j})^2 (E_{k,l})^2 - (E_{i,j})^2 E_{k,l} E_{k,l}^{\gk} -E_{i,j} (E_{k,l})^2 E_{i,j}^{\gk} + E_{i,j}E_{k,l}E_{i,j}^{\gk}E_{k,l}^{\gk}
\end{align*}
for $ \{1,2,3,4\} = \{ i,j,k,l \} $ with $ i<j $ and $ k<l$, 
to get 
\begin{align*}
 \cC_{4,5} \equiv  
  (E_{1,2})^2 (E_{3,4})^2 - (E_{1,2})^2 E_{3,4} E_{3,4}^{\gk} 
 -E_{1,2}(E_{3,4})^2 E_{1,2}^{\gk} + E_{1,2} E_{3,4} E_{1,2}^{\gk} E_{3,4}^{\gk}.
\end{align*}
Thus we are done.
\end{proof}

\begin{lem}
\label{lem:PDE1} Retain the notation in Lemma \ref{lem:Eij_radial} and 
let $ \varphi: V_{(\kappa_1,\kappa_2,\delta_3)} \to {\rm Wh}(\Pi_{\sigma}, \psi_1) $ be a $K$-homomorphism.

\noindent (i)  
For $ l = (l_1,l_2,l_3,l_4,l_{12},l_{13},l_{14}, l_{23}, l_{24}, l_{34}) \in S_{(\kappa_1,\kappa_2,\delta_3)} $, we have 
\begin{align} \label{PDE1_gamma1}
& (\pa_4- \gamma_1)\vp(u_l)(y) = 0.
\end{align}

\noindent (ii)
Assume $ \kappa_1 > \kappa_2 $. 
For $l  = (l_1,l_2,l_3,l_4,l_{12},l_{13},l_{14}, l_{23}, l_{24}, l_{34}) \in S_{(\kappa_1-1, \kappa_2, \delta_3)}$, we have 
\begin{equation*}
\begin{split}
& (\pa_1-\nu_1'-\tfrac{\kappa_1}{2}-1)\vp(u_{l+e_1})(y) + 2\pi \sqrt{-1}y_1 \vp(u_{l+e_2})(y) = 0, 
\end{split} 
\end{equation*}
\begin{equation*}
\begin{split}
& (-\pa_1+\pa_2-\nu_1'-\tfrac{\kappa_1}{2}+ l_1) \vp(u_{l+e_2})(y)   
 + 2\pi \sqrt{-1}y_1 \vp(u_{l+e_1})(y) + 2\pi \sqrt{-1}y_2 \vp(u_{l+e_3})(y) \\
& -l_2 \vp(u_{l-e_2+2e_1})(y) + l_{13} \vp(u_{l+e_1-e_{13}+e_{23}})(y) + l_{14} \vp(u_{l+e_1-e_{14}+e_{24}})(y) \\
& - l_{23} \vp(u_{l+e_1-e_{23}+e_{13}})(y) - l_{24} \vp(u_{l+e_1-e_{24}+e_{14}})(y) = 0, 
\end{split}
\end{equation*}
\begin{equation*}
\begin{split}
& (-\pa_2+\pa_3-\nu_1' +\tfrac{\kappa_1}{2}-l_4) \vp(u_{l+e_3})(y)  
 + 2\pi \sqrt{-1}y_2 \vp(u_{l+e_2})(y) + 2\pi \sqrt{-1}y_3 \vp(u_{l+e_4})(y) \\
&  +l_3 \vp(u_{l-e_3+2e_4})(y) + l_{13} \vp(u_{l+e_4-e_{13}+e_{14}})(y) - l_{14} \vp(u_{l+e_4-e_{14}+e_{13}})(y) \\
& + l_{23} \vp(u_{l+e_4-e_{23}+e_{24}})(y) - l_{24} \vp(u_{l+e_4-e_{24}+e_{23}})(y) = 0, 
\end{split}
\end{equation*}
\begin{equation*}
 (-\pa_3+\pa_4-\nu_1'+\tfrac{\kappa_1}{2}+1) \vp(u_{l+e_4})(y) + 2\pi \sqrt{-1}y_3 \vp(u_{l+e_3})(y) = 0.
\end{equation*}

\noindent (iii)  
Assume $ \kappa_2 \ge 1 $.  
For $ l  = (l_1,l_2,l_3,l_4,l_{12},l_{13},l_{14}, l_{23}, l_{24}, l_{34})  \in S_{(\kappa_1-1, \kappa_2-1, 0)} $, we have 
\begin{equation*} 
 (\pa_2-\nu_1-\nu_2-\tfrac{\kappa_1+\kappa_2}{2}-1) \vp(u_{l+e_{12}})(y) + 2 \pi \sqrt{-1} y_2 \vp(u_{l+e_{13}})(y) = 0,
\end{equation*}
\begin{equation*} 
\begin{split}
& (\pa_1-\pa_2+\pa_3-\nu_1-\nu_2-\tfrac{\kappa_1-\kappa_2}{2}-l_{13}-l_{14}-l_{23}-l_{24}-1) \vp(u_{l+e_{13}})(y) \\
&  +2\pi \sqrt{-1} y_1 \vp(u_{l+e_{23}})(y) + 2\pi \sqrt{-1} y_2 \vp(u_{l+e_{12}})(y) + 2\pi \sqrt{-1} y_3 \vp(u_{l+e_{14}})(y) \\
& +l_2 \vp(u_{l-e_2+e_3+e_{12}})(y) - l_3 \vp(u_{l-e_3+e_2+e_{12}})(y) 
  -l_{13} \vp(u_{l-e_{13}+2e_{12}})(y) + l_{24} \vp(u_{l-e_{24}+e_{12}+e_{34}})(y) = 0,
\end{split}
\end{equation*}
\begin{equation*} 
\begin{split}
& (\pa_1-\pa_3+\pa_4-\nu_1-\nu_2-\tfrac{\kappa_1-\kappa_2}{2}+l_4)\vp(u_{l+e_{14}})(y) 
  + 2 \pi \sqrt{-1} y_1 \vp(u_{l+e_{24}})(y) \\
& + 2\pi \sqrt{-1} y_3 \vp(u_{l+e_{13}})(y)   +l_2 \vp(u_{l-e_2+e_4+e_{12}})(y) + l_3 \vp(u_{l-e_3+e_4+e_{13}})(y) = 0,
\end{split}
\end{equation*}
\begin{equation*} 
\begin{split}
& (-\pa_1+\pa_3-\nu_1-\nu_2+\tfrac{\kappa_1-\kappa_2}{2}-l_4)\vp(u_{l+e_{23}})(y) + 2 \pi \sqrt{-1} y_1 \vp(u_{l+e_{13}})(y) \\
&  + 2\pi \sqrt{-1} y_3 \vp(u_{l+e_{24}})(y)  
  -l_2 \vp(u_{l-e_2+e_4+e_{34}})(y) + l_3 \vp(u_{l-e_3+e_4+e_{24}})(y) = 0,
\end{split}
\end{equation*}
\begin{equation*} 
\begin{split}
& (-\pa_1+\pa_2-\pa_3+\pa_4-\nu_1-\nu_2+\tfrac{\kappa_1-\kappa_2}{2}+l_{13}+l_{14}+l_{23}+l_{24}+1) \vp(u_{l+e_{24}})(y) \\
& +2\pi \sqrt{-1} y_1 \vp(u_{l+e_{14}})(y) + 2\pi \sqrt{-1} y_2 \vp(u_{l+e_{34}})(y) + 2\pi \sqrt{-1} y_3 \vp(u_{l+e_{23}})(y) \\
& +l_2 \vp(u_{l-e_2+e_3+e_{34}})(y) - l_3 \vp(u_{l-e_3+e_2+e_{34}})(y) 
  -l_{13} \vp(u_{l-e_{13}+e_{12}+e_{34}})(y) + l_{24} \vp(u_{l-e_{24}+2e_{34}})(y) = 0,
\end{split}
\end{equation*}
\begin{equation*}
(-\pa_2+\pa_4-\nu_1-\nu_2+\tfrac{\kappa_1+\kappa_2}{2}+1) \vp(u_{l+e_{34}})(y) + 2 \pi \sqrt{-1} y_2 \vp(u_{l+e_{24}})(y) = 0.
\end{equation*}
\end{lem}

\begin{proof}
The equation (\ref{PDE1_gamma1}) is immediate from (\ref{Cap_eigen}) with $ i=1 $, (\ref{eqn:redC1}) and 
Lemma \ref{lem:Eij_radial}.
Let us show the first equation in (ii). 
By (\ref{DiracSchmid1}) 
and 
$ E_{1,k}^{\gp} = 2 E_{1,k}-E_{1,k}^{\gk} $, we have 
\begin{align*}
\sum_{k=1}^4 (R (2E_{1,k}- E_{1,k}^{\gk})  \vp(u_{l+e_k}))(y) - 2\nu_1' \vp(u_{l+e_1})(y) = 0
\end{align*}
for $ l \in S_{(\kappa_1-1, \kappa_2, \delta_3)}. $
Applying Lemmas \ref{lem:Kact_ul_vl} (i) and \ref{lem:Eij_radial}, we find that 
\begin{align} \label{eqn:PDE1-2-1}
\begin{split}
& (2\pa_1-2\nu_1') \vp(u_{l+e_1})(y) + 4\pi \sqrt{-1}y_1 \vp(u_{l+e_2})(y) 
\\
& -(l_2+1)\vp(u_{l+e_1})(y) + l_1 \vp(u_{l-e_1+2e_2})(y) 
   -l_{23}\vp(u_{l+e_2-e_{23}+e_{13}})(y)  \\
& -l_{24}\vp(u_{l+e_2-e_{24}+e_{14}})(y)
   +l_{13}\vp(u_{l+e_2-e_{13}+e_{23}})(y)+l_{14} \vp(u_{l+e_2-e_{14}+e_{24}})(y) 
\\
& -(l_3+1)\vp(u_{l+e_1})(y) + l_1 \vp(u_{l-e_1+2e_3})(y) 
   +l_{23} \vp(u_{l+e_3-e_{23}+e_{12}})(y)  \\
& -l_{34}\vp(u_{l+e_3-e_{34}+e_{14}})(y)
   -l_{12}\vp(u_{l+e_3-e_{12}+e_{23}})(y)+l_{14}\vp(u_{l+e_3-e_{14}+e_{34}})(y) 
\\
& -(l_4+1)\vp(u_{l+e_1})(y) + l_1 \vp(u_{l-e_1+2e_4})(y) 
   +l_{24} \vp(u_{l+e_4-e_{24}+e_{12}})(y)  \\
& -l_{34}\vp(u_{l+e_4-e_{34}+e_{13}})(y)
   -l_{12}\vp(u_{l+e_4-e_{12}+e_{24}})(y)-l_{13}\vp(u_{l+e_4-e_{13}+e_{34}})(y) = 0.
\end{split}
\end{align}
By Lemma \ref{lem:rel_ul} (i), we know 
\begin{align*}
 l_1 \{  \vp(u_{l-e_1+2e_2})(y) +\vp(u_{l-e_1+2e_3})(y) +\vp(u_{l-e_1+2e_4})(y) \}
  = -l_1 \vp(u_{l+e_1})(y). 
\end{align*}
Similarly, Lemma \ref{lem:rel_ul} (ii) implies that 
\begin{align*}
 l_{23} \{ -\vp(u_{l+e_2-e_{23}+e_{13}})(y) + \vp(u_{l+e_3-e_{23}+e_{12}})(y) \} & = -l_{23} \vp(u_{l+e_1})(y), \\
 l_{24} \{ -\vp(u_{l+e_2-e_{24}+e_{14}})(y) + \vp(u_{l+e_4-e_{24}+e_{12}})(y) \} & = -l_{24} \vp(u_{l+e_1})(y), \\
 l_{12} \{ -\vp(u_{l+e_3-e_{12}+e_{23}})(y) - \vp(u_{l+e_4-e_{12}+e_{24}})(y) \} & = -l_{12} \vp(u_{l+e_1})(y), \\
 l_{13} \{ \vp(u_{l+e_2-e_{13}+e_{23}})(y) - \vp(u_{l+e_4-e_{13}+e_{34}})(y)  \} & = -l_{13} \vp(u_{l+e_1})(y), \\
 l_{14} \{ \vp(u_{l+e_2-e_{14}+e_{24}})(y) + \vp(u_{l+e_3-e_{14}+e_{34}})(y) \} & = -l_{14} \vp(u_{l+e_1})(y), \\
 l_{34} \{ -\vp(u_{l+e_3-e_{34}+e_{14}})(y) - \vp(u_{l+e_4-e_{34}+e_{13}})(y) \} & = -l_{34} \vp(u_{l+e_1})(y).
\end{align*}
Then (\ref{eqn:PDE1-2-1}) can be written as 
\begin{align*}
& \{ 2\pa_1-2\nu_1' -(l_1+l_2+l_3+l_4+3) -(l_{12}+l_{13}+l_{14}+l_{23}+l_{24}+l_{34}) \} \vp(u_{l+e_1})(y) 
   + 4\pi \sqrt{-1}y_1 \vp(u_{l+e_2})(y) = 0
\end{align*}
as desired.
The other equations can be similarly confirmed.
\end{proof}

By (\ref{PDE1_gamma1}), 
we can define a function $\hat{\varphi}_l $ on $ (\bR_+)^3 $ by 
\begin{align} \label{eqn:hat_varphi}
 \vp(u_l)(y) = 
  (\sqrt{-1})^{-l_1+l_3-l_{13}+l_{24}} (-1)^{l_2+l_{14}+l_{23}} y_1^{3/2} y_2^2 y_3^{3/2-\kappa_2} y_4^{\gamma_1} \hat{\varphi}_l(y_1,y_2,y_3)
\end{align}
for $l \in S_{(\kappa_1,\kappa_2,\delta_3)}$.
We understand $\hat{\varphi}_l = 0 $ if $ l  \notin S_{(\kappa_1,\kappa_2,\delta_3)} $.
Here is a system of partial differential equations for $ \hat{\vp}_l $.


\begin{prop}
\label{prop:PDE2}
Retain the notation. 

\noindent (i) 
For $ l = (l_1,l_2,l_3,l_4,l_{12},l_{13},l_{14}, l_{23}, l_{24}, l_{34}) \in S_{(\kappa_1,\kappa_2,\delta_3)} $, we have 
\begin{equation} \label{PDE:C2} 
\begin{split}
& \{ \Delta_2 -(2\pi y_1) \mathfrak{K}_{12} - (2\pi y_2) \mathfrak{K}_{23} - (2\pi y_3) \mathfrak{K}_{34} \} \hat{\varphi}_l = 0,
\end{split} 
\end{equation}
\begin{equation} 
 \label{PDE:C3} 
\begin{split}
& \{ \Delta_3 
 + (2\pi y_1) (\pa_2-\gamma_1) \mathfrak{K}_{12}
 + (2\pi y_2) (-\pa_1+\pa_3-\gamma_1-\kappa_2) \mathfrak{K}_{23}
\\
& 
 - (2\pi y_3) \pa_2 \mathfrak{K}_{34}
 + (2\pi y_1)(2\pi y_2) \mathfrak{K}_{13} 
 + (2\pi y_2)(2\pi y_3) \mathfrak{K}_{24} 
  \} \hat{\vp}_l = 0
\end{split}
\end{equation}
and 
\begin{equation} \label{PDE:C4} 
\begin{split}
& [ \Delta_4  
 - (2 \pi y_1) \{ (-\pa_2+\pa_3-\kappa_2)(-\pa_3+\gamma_1+\kappa_2) + (2\pi y_3)^2 \} \mathfrak{K}_{12} 
\\
& 
 - (2\pi  y_2) \pa_1(-\pa_3+\gamma_1+\kappa_2) \mathfrak{K}_{23}
 - (2\pi y_3) \{ \pa_1(-\pa_1+\pa_2) + (2\pi y_1)^2 \} \mathfrak{K}_{34}
\\
& + (2\pi y_1)(2\pi y_2)  (-\pa_3+\gamma_1+\kappa_2) \mathfrak{K}_{13} 
   + (2\pi y_2)(2\pi y_3) \pa_1  \mathfrak{K}_{24} 
\\
& + (2\pi y_1)(2\pi y_2)(2\pi y_3) \mathfrak{K}_{14} 
   + (2\pi y_1)(2\pi y_3) \mathfrak{K}_{12,34}] \hat{\vp}_l = 0,
\end{split}
\end{equation}
where $ \Delta_2 $, $\Delta_3 $ and $\Delta_4 $ are the differential operators defined by 
\begin{equation*}
\begin{split}
 \Delta_2 
&= -\pa_1^2-\pa_2^2-(\pa_3-\kappa_2)^2 + \pa_1\pa_2 + \pa_2(\pa_3-\kappa_2) + \gamma_1(\pa_3-\kappa_2) 
   + (2\pi y_1)^2 + (2\pi y_2)^2 +(2\pi y_3)^2   -\gamma_2,
\end{split}
\end{equation*}
\begin{equation*}
\begin{split}
 \Delta_3 
& = \pa_2(\pa_1-\pa_3+\kappa_2)(\pa_1-\pa_2+\pa_3-\kappa_2)
  -\gamma_1(\pa_1^2+\pa_2^2-\pa_1\pa_2-\pa_2 \pa_3 + \kappa_2 \pa_2) \\
& \phantom{=}  
 + (2\pi y_1)^2(-\pa_2+\gamma_1)
 +(2\pi y_2)^2(\pa_1-\pa_3+\gamma_1+\kappa_2)
 +(2\pi y_3)^2 \pa_2 -\gamma_3
\end{split}
\end{equation*}
and 
\begin{equation*}
\begin{split}
 \Delta_4 
& = \pa_1(\pa_2-\pa_1)(\pa_3-\pa_2-\kappa_2)(-\pa_3+\gamma_1+\kappa_2) 
 + (2\pi y_1)^2(\pa_3-\pa_2-\kappa_2)(-\pa_3+\gamma_1+\kappa_2) 
\\
& \phantom{=}  
  + (2\pi y_2)^2 \pa_1(-\pa_3+\gamma_1+\kappa_2) 
 + (2\pi y_3)^2 \pa_1(\pa_2-\pa_1) 
 + (2\pi y_1)^2(2\pi y_3)^2 - \gamma_4,
\end{split}
\end{equation*}
respectively. 
Here we set 
\begin{align*}
\mathfrak{K}_{12} \hat{\vp}_l
& = l_1 \hat{\varphi}_{l-e_1+e_2} + l_2 \hat{\varphi}_{l-e_2+e_1}  
   + l_{13} \hat{\varphi}_{l-e_{13}+e_{23}} +l_{14} \hat{\varphi}_{l-e_{14}+e_{24}} 
   + l_{23} \hat{\varphi}_{l-e_{23}+e_{13}} +l_{24} \hat{\varphi}_{l-e_{24}+e_{14}},
\\
\mathfrak{K}_{23} \hat{\vp}_l
& = l_2 \hat{\varphi}_{l-e_2+e_3} + l_3 \hat{\varphi}_{l-e_3+e_2}   
    + l_{12} \hat{\varphi}_{l-e_{12}+e_{13}} + l_{13} \hat{\varphi}_{l-e_{13}+e_{12}} 
    + l_{24} \hat{\varphi}_{l-e_{24}+e_{34}} + l_{34} \hat{\varphi}_{l-e_{34}+e_{24}},
\\
\mathfrak{K}_{34} \hat{\vp}_l
& = l_3\hat{\varphi}_{l-e_3+e_4} + l_4 \hat{\varphi}_{l-e_4+e_3}   
     + l_{13} \hat{\varphi}_{l-e_{13}+e_{14}} + l_{14} \hat{\varphi}_{l-e_{14}+e_{13}} 
     + l_{23} \hat{\varphi}_{l-e_{23}+e_{24}} + l_{24} \hat{\varphi}_{l-e_{24}+e_{23}},
\\
\mathfrak{K}_{13} \hat{\vp}_l
& = l_1 \hat{\varphi}_{l-e_1+e_3} + l_3 \hat{\varphi}_{l-e_3+e_1} 
   -l_{12} \hat{\varphi}_{l-e_{12}+e_{23}} + l_{14} \hat{\varphi}_{l-e_{14}+e_{34}}
  +l_{23} \hat{\varphi}_{l-e_{23}+e_{12}} - l_{34} \hat{\varphi}_{l-e_{34}+e_{14}},
\\
\mathfrak{K}_{24} \hat{\vp}_l
& =  l_2 \hat{\varphi}_{l-e_2+e_4} - l_4 \hat{\varphi}_{l-e_4+e_2}  
      + l_{12} \hat{\varphi}_{l-e_{12}+e_{14}}  - l_{14} \hat{\varphi}_{l-e_{14}+e_{12}} 
      - l_{23} \hat{\varphi}_{l-e_{23}+e_{34}} + l_{34} \hat{\varphi}_{l-e_{34}+e_{23}},
\\
\mathfrak{K}_{14} \hat{\vp}_l
& = - l_1\hat{\varphi}_{l-e_1+e_4} -l_4 \hat{\varphi}_{l-e_4+e_1}  
    +l_{12} \hat{\varphi}_{l-e_{12}+e_{24}} + l_{13} \hat{\varphi}_{l-e_{13}+e_{34}}
    +l_{24} \hat{\varphi}_{l-e_{24}+e_{12}} + l_{34} \hat{\varphi}_{l-e_{34}+e_{13}},
\\
 \mathfrak{K}_{12,34} \hat{\vp}_l
& = l_1 l_3 \hat{\varphi}_{l-e_1-e_3+e_2+e_4} + l_1 l_4 \hat{\varphi}_{l-e_1-e_4+e_2+e_3}  
   + l_2 l_3 \hat{\varphi}_{l-e_2-e_3+e_1+e_4} + l_2 l_4 \hat{\varphi}_{l-e_2-e_4+e_1+e_3}  
\\
& \phantom{=} 
  + l_1 (l_{13} \hat{\varphi}_{l-e_1+e_2-e_{13}+e_{14}} + l_{14} \hat{\varphi}_{l-e_1+e_2-e_{14}+e_{13}} 
   + l_{23} \hat{\varphi}_{l-e_1+e_2-e_{23}+e_{24}} +  l_{24} \hat{\varphi}_{l-e_1+e_2-e_{24}+e_{23}} )
\\
& \phantom{=} 
  + l_2 (l_{13} \hat{\varphi}_{l-e_2+e_1-e_{13}+e_{14}} + l_{14} \hat{\varphi}_{l-e_2+e_{23}-e_{14}+e_{13}} 
   + l_{23} \hat{\varphi}_{l-e_2+e_1-e_{23}+e_{24}} + l_{24} \hat{\varphi}_{l-e_2+e_1-e_{24}+e_{23}} )
\\
& \phantom{=} 
   + l_3 (l_{13} \hat{\varphi}_{l-e_3+e_4-e_{13}+e_{23}} + l_{14} \hat{\varphi}_{l-e_3+e_4-e_{14}+e_{24}} 
   + l_{23} \hat{\varphi}_{l-e_3+e_4-e_{23}+e_{13}} + l_{24} \hat{\varphi}_{l-e_3+e_4-e_{24}+e_{14}} )
\\
& \phantom{=} 
  + l_4 (l_{13} \hat{\varphi}_{l-e_4+e_3-e_{13}+e_{23}} + l_{14} \hat{\varphi}_{l-e_4+e_3-e_{14}+e_{24}}    
    + l_{23} \hat{\varphi}_{l-e_4+e_3-e_{23}+e_{13}} + l_{24} \hat{\varphi}_{l-e_4+e_3-e_{24}+e_{14}} )
\\
&  \phantom{=} 
  + l_{13}(l_{14}+l_{23}+1) \hat{\varphi}_{l-e_{13}+e_{24}} 
  +  l_{24}(l_{14}+l_{23}+1) \hat{\varphi}_{l-e_{24}+e_{13}} 
\\
& \phantom{=}  
   + l_{14}(l_{13}+l_{24}+1) \hat{\varphi}_{l-e_{14}+e_{23}} 
   + l_{23} (l_{13}+l_{24}+1) \hat{\varphi}_{l-e_{23}+e_{14}} 
\\
& \phantom{=}  
  + l_{13} l_{24} ( \hat{\varphi}_{l-e_{13}-e_{24}+2e_{14}} + \hat{\varphi}_{l-e_{13}-e_{24}+2e_{23}} )  
  + l_{14} l_{23} ( \hat{\varphi}_{l-e_{14}-e_{23}+2e_{13}} + \hat{\varphi}_{l-e_{14}-e_{23}+2e_{24}} ) 
\\
& \phantom{=} 
 + l_{13}(l_{13}-1) \hat{\varphi}_{l-2e_{13}+e_{14}+e_{23}} + l_{24}(l_{24}-1)\hat{\varphi}_{l-2e_{24}+e_{14}+e_{23}} 
\\
& \phantom{=} 
 + l_{14}(l_{14}-1) \hat{\varphi}_{l-2e_{14}+e_{13}+e_{24}} + l_{23} (l_{23}-1) \hat{\varphi}_{l-2e_{23}+e_{13}+e_{24}}.
\end{align*}

\noindent (ii)
Assume $ \kappa_1 > \kappa_2 $. 
For $l = (l_1,l_2,l_3,l_4,l_{12},l_{13},l_{14}, l_{23}, l_{24}, l_{34}) \in S_{(\kappa_1-1, \kappa_2, \delta_3)}$, we have 
\begin{equation} \label{PDE:DS1}
\begin{split} 
& (\pa_1-\nu_1'-\tfrac{\kappa_1-1}{2})\hat{\vp}_{l+e_1} + (2\pi y_1) \hat{\vp}_{l+e_2} = 0, 
\end{split} 
\end{equation}
\begin{equation} \label{PDE:DS2} 
\begin{split} & (-\pa_1+\pa_2-\nu_1'-\tfrac{\kappa_1-1}{2}+ l_1) \hat{\vp}_{l+e_2}  
   - (2\pi y_1) \hat{\vp}_{l+e_1} + (2\pi y_2) \hat{\vp}_{l+e_3} + l_2 \hat{\vp}_{l-e_2+2e_1} \\
&   + l_{13} \hat{\vp}_{l+e_1-e_{13}+e_{23}} + l_{14} \hat{\vp}_{l+e_1-e_{14}+e_{24}} 
     + l_{23} \hat{\vp}_{l+e_1-e_{23}+e_{13}} + l_{24} \hat{\vp}_{l+e_1-e_{24}+e_{14}} = 0, 
\end{split}
\end{equation}
\begin{equation} \label{PDE:DS3}
\begin{split}
& (-\pa_2+\pa_3-\nu_1' +\tfrac{\kappa_1-1}{2}-\kappa_2-l_4) \hat{\vp}_{l+e_3}
   - (2\pi y_2) \hat{\vp}_{l+e_2}  + (2\pi y_3) \hat{\vp}_{l+e_4} - l_3 \hat{\vp}_{l-e_3+2e_4} \\
&   - l_{13} \hat{\vp}_{l+e_4-e_{13}+e_{14}} - l_{14} \hat{\vp}_{l+e_4-e_{14}+e_{13}} 
   - l_{23} \hat{\vp}_{l+e_4-e_{23}+e_{24}} - l_{24} \hat{\vp}_{l+e_4-e_{24}+e_{23}} = 0, 
\end{split}
\end{equation}
\begin{equation} \label{PDE:DS4}
 (-\pa_3+\gamma_1-\nu_1'+\tfrac{\kappa_1-1}{2}+\kappa_2) \hat{\vp}_{l+e_4} - (2\pi y_3) \hat{\vp}_{l+e_3} = 0.
\end{equation}

\noindent (iii)  
Assume $ \kappa_2 \ge 1 $.  
For $ l = (l_1,l_2,l_3,l_4,l_{12},l_{13},l_{14}, l_{23}, l_{24}, l_{34}) \in S_{(\kappa_1-1, \kappa_2-1, 0)} $, we have 
\begin{equation} \label{PDE:DS12} 
 (\pa_2-\nu_1-\nu_2-\tfrac{\kappa_1+\kappa_2}{2} +1) \hat{\vp}_{l+e_{12}} + (2 \pi y_2) \hat{\vp}_{l+e_{13}}= 0,
\end{equation}
\begin{equation} \label{PDE:DS13} 
\begin{split}
& (\pa_1-\pa_2+\pa_3-\nu_1-\nu_2-\tfrac{\kappa_1+\kappa_2}{2}-l_{13}-l_{14}-l_{23}-l_{24}) \hat{\vp}_{l+e_{13}}  \\
& + (2\pi y_1) \hat{\vp}_{l+e_{23}} - (2\pi y_2) \hat{\vp}_{l+e_{12}}  + (2\pi y_3) \hat{\vp}_{l+e_{14}} \\
& + l_2 \hat{\vp}_{l-e_2+e_3+e_{12}} + l_3 \hat{\vp}_{l-e_3+e_2+e_{12}} 
    +l_{13} \hat{\vp}_{l-e_{13}+2e_{12}} + l_{24} \hat{\vp}_{l-e_{24}+e_{12}+e_{34}} = 0,
\end{split}
\end{equation}
\begin{equation} \label{PDE:DS14} 
\begin{split}
& (\pa_1-\pa_3+\gamma_1-\nu_1-\nu_2-\tfrac{\kappa_1-3\kappa_2}{2}+l_4) \hat{\vp}_{l+e_{14}}  \\
& + (2 \pi y_1) \hat{\vp}_{l+e_{24}} - (2\pi y_3) \hat{\vp}_{l+e_{13}}  
     +l_2 \hat{\vp}_{l-e_2+e_4+e_{12}}  + l_3 \hat{\vp}_{l-e_3+e_4+e_{13}} = 0,
\end{split}
\end{equation}
\begin{equation} \label{PDE:DS23} 
\begin{split}
& (-\pa_1+\pa_3-\nu_1-\nu_2+\tfrac{\kappa_1-3\kappa_2}{2}-l_4) \hat{\vp}_{l+e_{23}} \\
& - (2 \pi y_1) \hat{\vp}_{l+e_{13}} + (2\pi  y_3) \hat{\vp}_{l+e_{24}}  
   -l_2 \hat{\vp}_{l-e_2+e_4+e_{34}} - l_3 \hat{\vp}_{l-e_3+e_4+e_{24}} = 0,
\end{split}
\end{equation}
\begin{equation} \label{PDE:DS24} 
\begin{split}
& (-\pa_1+\pa_2-\pa_3+\gamma_1-\nu_1-\nu_2+\tfrac{\kappa_1+\kappa_2}{2}+l_{13}+l_{14}+l_{23}+l_{24}) \hat{\vp}_{l+e_{24}} \\
& -(2\pi  y_1) \hat{\vp}_{l+e_{14}}  + (2\pi y_2) \hat{\vp}_{l+e_{34}} - (2\pi y_3) \hat{\vp}_{l+e_{23}} \\
&   - l_2 \hat{\vp}_{l-e_2+e_3+e_{34}} - l_3 \hat{\vp}_{l-e_3+e_2+e_{34}} 
     - l_{13} \hat{\vp}_{l-e_{13}+e_{12}+e_{34}} - l_{24} \hat{\vp}_{l-e_{24}+2e_{34}} = 0,
\end{split}
\end{equation}
\begin{equation} \label{PDE:DS34}
(-\pa_2+\gamma_1-\nu_1-\nu_2+\tfrac{\kappa_1+\kappa_2}{2}-1) \hat{\vp}_{l+e_{34}} 
  - (2 \pi y_2) \hat{\vp}_{l+e_{24}} = 0.
\end{equation}
\end{prop}

\begin{proof}
The statement (i) follows from (\ref{Cap_eigen}), Lemmas \ref{lem:Kact_ul_vl}, \ref{lem:Eij_radial} and \ref{lem:Capelli}.
The equations in (ii) and (iii) follow from Lemma \ref{lem:PDE1}.
\end{proof}



\subsection{Reduction of the system of partial differential equations} 
\label{subsec:red_PDE}

In this subsection we give some difference-differential equations which will be used later.

\begin{lem} \label{lem:DS1_l-e1_and_l-e4}
Assume $ \kappa_1 >\kappa_2 $.
For $ l = l_1 e_1 +l_4 e_4 + l_{12} e_{12} + l_{34} e_{34} \in S_{(\kappa_1,\kappa_2, \delta_3)} $, 
we have the following:
\begin{itemize}
\item If $ l_4 \ge 1 $ then we have
\begin{align} 
\label{eqn:DS1:l-e_4+e_3}
 (2\pi y_3) \hat{\vp}_{l-e_4+e_3}
 = (-\pa_3+\gamma_1-\nu_1'+\tfrac{\kappa_1-1}{2}+\kappa_2) \hat{\vp}_{l},
\end{align}
\begin{align}
\label{eqn:DS1:l-e_4+e_2}
\begin{split}
&  (2\pi y_2)(2\pi y_3) \hat{\vp}_{l -e_4+e_2} 
   = \{ (-\pa_2+\pa_3-\nu_1' - \tfrac{\kappa_1 +1}{2} + l_1 )  
      (-\pa_3+\gamma_1-\nu_1'+\tfrac{\kappa_1-1}{2} + \kappa_2) + (2\pi y_3)^2 \}  \hat{\vp}_{l},
\end{split}
\end{align}
\begin{align}
\label{eqn:DS1:l-e_4+e_1}
\begin{split}
&  (2\pi y_1)(2\pi y_2)(2\pi y_3) \hat{\vp}_{l-e_4+e_1} \\
& = \{ (-\pa_1+\pa_2-\nu_1'-\tfrac{\kappa_1+1}{2} +l_1) (-\pa_2+\pa_3-\nu_1' - \tfrac{\kappa_1 +1}{2} +l_1)  
    (-\pa_3+\gamma_1-\nu_1'+\tfrac{\kappa_1 -1}{2}+\kappa_2 ) \\
& \quad  
 + (2\pi y_2)^2(-\pa_3+\gamma_1-\nu_1'+\tfrac{\kappa_1-1}{2} + \kappa_2) 
  + (2\pi y_3)^2  (-\pa_1+\pa_2-\nu_1'-\tfrac{\kappa_1 + 1}{2}  +l_1)\} \hat{\vp}_{l}.
\end{split}
\end{align}
\item If $ l_1 \ge 1 $ then we have
\begin{align}
\label{eqn:DS1:l-e_1+e_2}
 (2\pi y_1) \hat{\vp}_{l-e_1+e_2}= (-\pa_1 + \nu_1' + \tfrac{\kappa_1-1}{2})\hat{\vp}_{l},
\end{align}
\begin{align}
\label{eqn:DS1:l-e_1+e_3}
\begin{split}
& (2\pi y_1) (2\pi y_2) \hat{\vp}_{l  -e_1+e_3} 
 =   \{  (-\pa_1+\pa_2-\nu_1'-\tfrac{\kappa_1-1}{2}+l_1) (\pa_1 - \nu_1' - \tfrac{\kappa_1-1}{2}) + (2\pi y_1)^2 \} 
\hat{\vp}_{l},
\end{split}
\end{align}
\begin{align}
\label{eqn:DS1:l-e_1+e_4}
\begin{split} 
&  (2\pi y_1)(2\pi y_2) (2\pi y_3) \hat{\vp}_{l-e_1+e_4}  \\
& = \{ (-\pa_2+\pa_3-\nu_1' -\tfrac{\kappa_1-1}{2} + l_1) 
    (-\pa_1+\pa_2-\nu_1'-\tfrac{\kappa_1-1}{2}+l_1) (-\pa_1 + \nu_1' + \tfrac{\kappa_1-1}{2}) \\
& \quad - (2\pi y_1)^2 (-\pa_2+\pa_3-\nu_1' -\tfrac{\kappa_1-1}{2} + l_1) 
  + (2\pi y_2)^2 (-\pa_1 + \nu_1' + \tfrac{\kappa_1-1}{2}) \} \hat{\vp}_{l}.
\end{split}
\end{align}
\end{itemize}
\end{lem} 

\begin{proof}
Let $ l = l_1 e_1+l_4 e_4 + l_{12} e_{12} + l_{13} e_{13} + l_{14} e_{14} + l_{23} e_{23} + l_{24} e_{24} + l_{34} e_{34} 
 \in S_{(\kappa_1,\kappa_2,\delta_3)}$.
We apply the equations (\ref{PDE:DS2}), (\ref{PDE:DS3}), (\ref{PDE:DS4}) with   
$ l-e_4\in S_{(\kappa_1-1, \kappa_2, \delta_3)}$ to get 
\begin{align}  \label{PDE:DS2q}
\begin{split}
& (-\pa_1+\pa_2-\nu_1'-\tfrac{\kappa_1-1}{2}+ l_1) \hat{\vp}_{l-e_4+e_2}  
   - (2\pi y_1) \hat{\vp}_{l-e_4+e_1} + (2\pi y_2) \hat{\vp}_{l-e_4+e_3} \\
&  + l_{13} \hat{\vp}_{l-e_4+e_1-e_{13}+e_{23}} + l_{14} \hat{\vp}_{l-e_4+e_1-e_{14}+e_{24}} 
     + l_{23} \hat{\vp}_{l-e_4+e_1-e_{23}+e_{13}} + l_{24} \hat{\vp}_{l-e_4+e_1-e_{24}+e_{14}} = 0, 
\end{split}
\end{align}
\begin{align}  \label{PDE:DS3q}
\begin{split}
& (-\pa_2+\pa_3-\nu_1' +\tfrac{\kappa_1-1}{2}-\kappa_2-(l_4-1)) \hat{\vp}_{l-e_4+e_3}
   - (2\pi y_2) \hat{\vp}_{l-e_4+e_2}  + (2\pi y_3) \hat{\vp}_{l} \\
& - l_{13} \hat{\vp}_{l-e_{13}+e_{14}} - l_{14} \hat{\vp}_{l-e_{14}+e_{13}}   - l_{23} \hat{\vp}_{l-e_{23}+e_{24}} - l_{24} \hat{\vp}_{l-e_{24}+e_{23}} = 0, 
\end{split}
\end{align}
\begin{align}  \label{PDE:DS4q}
 (-\pa_3+\gamma_1-\nu_1'+\tfrac{\kappa_1-1}{2}+\kappa_2) \hat{\vp}_{l} - (2\pi y_3) \hat{\vp}_{l-e_4+e_3} = 0.
\end{align}
The equation (\ref{eqn:DS1:l-e_4+e_3}) is immediate from (\ref{PDE:DS4q}). Since $ (2\pi y_3) (\ref{PDE:DS3q}) $ is equivalent to 
\begin{align*}
& (-\pa_2+\pa_3-\nu_1' - \tfrac{\kappa_1+1}{2} +l_1 ) (2\pi y_3)\hat{\vp}_{l-e_4+e_3} 
   - (2\pi y_2)(2\pi y_3) \hat{\vp}_{l-e_4+e_2}  + (2\pi y_3)^2 \hat{\vp}_{l} \\
& - l_{13} (2\pi y_3) \hat{\vp}_{l-e_{13}+e_{14}} - l_{14}(2\pi y_3) \hat{\vp}_{l-e_{14}+e_{13}} 
   - l_{23} (2\pi y_3) \hat{\vp}_{l-e_{23}+e_{24}} - l_{24} (2\pi y_3)\hat{\vp}_{l-e_{24}+e_{23}} = 0, 
\end{align*}
we get (\ref{eqn:DS1:l-e_4+e_2}) from (\ref{eqn:DS1:l-e_4+e_3}).

Similarly, since the equation $ (2\pi y_2) (2\pi y_3) (\ref{PDE:DS2q}) $ is equivalent to
\begin{align*}
& (-\pa_1+\pa_2-\nu_1'-\tfrac{\kappa_1+1}{2}+ l_1) (2\pi y_2) (2\pi y_3)  \hat{\vp}_{l-e_4+e_2}  
     - (2\pi y_1)(2\pi y_2) (2\pi y_3) \hat{\vp}_{l-e_4+e_1} + (2\pi y_2)^2 (2\pi y_3) \hat{\vp}_{l-e_4+e_3} \\
&  + l_{13} (2\pi y_2) (2\pi y_3)\hat{\vp}_{l-e_4+e_1-e_{13}+e_{23}} 
    + l_{14} (2\pi y_2) (2\pi y_3)\hat{\vp}_{l-e_4+e_1-e_{14}+e_{24}} \\
&  + l_{23} (2\pi y_2) (2\pi y_3)\hat{\vp}_{l-e_4+e_1-e_{23}+e_{13}} 
    + l_{24} (2\pi y_2) (2\pi y_3)\hat{\vp}_{l-e_4+e_1-e_{24}+e_{14}} = 0,
\end{align*}
we can get (\ref{eqn:DS1:l-e_4+e_1}).
The equations (\ref{eqn:DS1:l-e_1+e_2}), (\ref{eqn:DS1:l-e_1+e_3}), (\ref{eqn:DS1:l-e_1+e_4})
can be similarly obtained by the equations (\ref{PDE:DS1}), (\ref{PDE:DS2}), (\ref{PDE:DS3}). 
\end{proof}

\begin{lem} \label{lem:DS2_l-e12_and_l-e34}
Assume $ \kappa_2 \ge 1 $. Set $ c = \nu_1+\nu_2- \tfrac12 \gamma_1 $.
For $ \varepsilon \in \{\pm 1\} $ and $ t \in \bC $, we put
\begin{align*}
 X_{\varepsilon}(t)  
& = (\pa_2-\tfrac{\gamma_1+\kappa_1+\kappa_2}{2} +\varep t) 
     (\pa_1-\pa_2+\pa_3 - \tfrac{\gamma_1+\kappa_1+\kappa_2}{2} + \varep t)
     (\pa_1-\pa_3 + \tfrac{\gamma_1+\kappa_1+\kappa_2}{2}+t )  \\
& \quad      
     - \{ (2\pi y_1)^2 - (2\pi y_3)^2 \}  (\pa_2-\tfrac{\gamma_1+\kappa_1+\kappa_2}{2} +\varep t)   
     + (2\pi y_2)^2 (\pa_1-\pa_3 + \tfrac{\gamma_1+\kappa_1+\kappa_2}{2}+t ).
\end{align*}
For $ l = (\kappa_1-\kappa_2) e_4 + l_{12} e_{12} + l_{34} e_{34} \in S_{(\kappa_1,\kappa_2, \delta_3)} $, we have the following:
\smallskip  

\noindent
(i) If $ l_{12} \ge 1 $ then we have
\begin{align} 
\label{eqn:l-e12+e_13a}
\begin{split}
 (2\pi y_2)  \hat{\vp}_{l-e_{12}+e_{13} } & = -(\pa_2-\tfrac{\gamma_1+\kappa_1+\kappa_2}{2}-c+1) \hat{\vp}_l, 
\end{split}
\\
\label{eqn:l-e12+e_14a}
\begin{split} 
 2(c-1) (2\pi y_2)(2 \pi y_3) \hat{\vp}_{l-e_{12}+e_{14}}  & = X_{-1}(c-1) \hat{\vp}_l, 
\end{split}
\\
\label{eqn:l-e12+e_23a}
\begin{split}
 2(c-1) (2\pi y_1)(2\pi y_2) \hat{\vp}_{l-e_{12}+e_{23}}
& = - X_{1}(-c+1)  \hat{\vp}_l
\end{split}
\end{align}
and
\begin{align} 
\label{eqn:l-e12+e_34a}
\begin{split}
& 2(c-1) (2\pi y_1)(2\pi y_2)^2 (2\pi y_3) \hat{\vp}_{l-e_{12}+e_{34}} \\
& = - \{ (\pa_1-\pa_2+\pa_3 - \tfrac{\gamma_1+\kappa_1+\kappa_2}{2}+c-1)
    (\pa_1-\pa_3 + \tfrac{\gamma_1+\kappa_1+\kappa_2}{2}+c-1) + (2\pi y_3)^2 \} X_{1}(-c+1) \hat{\vp}_l \\
& \quad + (2\pi y_1)^2 \{ X_{-1}(c-1) -2(c-1) 
  (\pa_2-\tfrac{\gamma_1+\kappa_1+\kappa_2}{2}-c+ 1)
  (\pa_1-\pa_2+\pa_3   - \tfrac{\gamma_1+\kappa_1+\kappa_2}{2}+c+1)  \} \hat{\vp}_l.
\end{split}
\end{align}
(ii)
If $ l_{34} \ge 1 $ then we have
\begin{align}
\label{eqn:l-e34+e_24a}
\begin{split}
 (2\pi y_2)  \hat{\vp}_{l-e_{34}+e_{24} } 
& = -(\pa_2-\tfrac{\gamma_1+\kappa_1+\kappa_2}{2}+c+1) \hat{\vp}_l, 
\end{split}
\\
\label{eqn:l-e34+e_23a}
\begin{split} 
2(c+1) (2\pi y_2)(2 \pi y_3) \hat{\vp}_{l-e_{34}+e_{23}} 
& =  - X_{-1}(-c-1) \hat{\vp}_l,
\end{split}
\\
\label{eqn:l-e34+e_14a}
\begin{split}
2 (c+1) (2\pi y_1)(2\pi y_2) \hat{\vp}_{l-e_{34}+e_{14}} 
& =  X_1(c+1)  \hat{\vp}_l
\end{split}
\end{align}
and 
\begin{align}
\label{eqn:l-e34+e_12a}
\begin{split}
& 2(c+1) (2\pi y_1)(2\pi y_2)^2 (2\pi y_3) \hat{\vp}_{l-e_{34}+e_{12}} \\
& = \{ (\pa_1-\pa_2+\pa_3 - \tfrac{\gamma_1+\kappa_1+\kappa_2}{2}-c-1)
     ( \pa_1-\pa_3 + \tfrac{\gamma_1+\kappa_1+\kappa_2}{2}  -c- 1 ) + (2\pi y_3)^2 \} X_{1}(c+1) \hat{\vp}_l \\
& \quad 
  -  (2\pi y_1)^2 \{ X_{-1}(-c-1)  + 2(c+1) (\pa_2-\tfrac{\gamma_1+\kappa_1+\kappa_2}{2}+c+1)
     (\pa_1-\pa_2+\pa_3 - \tfrac{\gamma_1+\kappa_1+\kappa_2}{2}-c+1) \} \hat{\vp}_l.
\end{split}
\end{align}
\end{lem}

\begin{proof}
We abbreviate 
\begin{align*}
 D_{12} & = \pa_2-\tfrac{\gamma_1+\kappa_1+\kappa_2}{2}, &
 D_{13} & = \pa_1-\pa_2+\pa_3 - \tfrac{\gamma_1+\kappa_1+\kappa_2}{2}, &
 D_{14} & = \pa_1-\pa_3 + \tfrac{\gamma_1+\kappa_1+\kappa_2}{2}. 
\end{align*}
Assume $ l_{12} \ge1 $. 
From the equations
(\ref{PDE:DS12}), (\ref{PDE:DS13}), (\ref{PDE:DS14}), (\ref{PDE:DS23}), (\ref{PDE:DS24}), 
for $ l = (\kappa_1-\kappa_2) e_4 + l_{12} e_{12} + l_{34} e_{34} \in S_{(\kappa_1,\kappa_2, \delta_3)}  $, we have 
\begin{align}
\label{PDE:DS12q} 
& (D_{12}-c+1) \hat{\vp}_{l} + (2 \pi y_2) \hat{\vp}_{l-e_{12}+e_{13}}= 0, \\
\label{PDE:DS13q} 
\begin{split}
& (D_{13}-c)  \hat{\vp}_{l-e_{12}+e_{13}}  
 + (2\pi y_1) \hat{\vp}_{l-e_{12} +e_{23}} - (2\pi y_2) \hat{\vp}_{l}  + (2\pi y_3) \hat{\vp}_{l-e_{12}+e_{14}} = 0,
\end{split} 
\\
\label{PDE:DS14q}  
& (D_{14}-c)  \hat{\vp}_{l-e_{12}+e_{14}}
+ (2 \pi y_1) \hat{\vp}_{l-e_{12}+e_{24}} - (2\pi y_3) \hat{\vp}_{l-e_{12}+e_{13}}  = 0,
\\
\label{PDE:DS23q} 
& (-D_{14}-c)  \hat{\vp}_{l-e_{12}+e_{23}} 
 - (2 \pi y_1) \hat{\vp}_{l-e_{12}+e_{13}} + ( 2\pi  y_3) \hat{\vp}_{l-e_{12}+e_{24}}  = 0,
\\
\label{PDE:DS24q} 
& (-D_{13}-c)  \hat{\vp}_{l-e_{12}+e_{24}} 
 - (2 \pi y_1) \hat{\vp}_{l-e_{12}+e_{14}} + ( 2\pi  y_2) \hat{\vp}_{l-e_{12}+e_{34}}  
 - ( 2\pi  y_3) \hat{\vp}_{l-e_{12}+e_{23}}  = 0. 
\end{align}
We apply the operator
$ (D_{14} + t )(2\pi y_2) $ to $ (\ref{PDE:DS13q}) $:
\begin{align*}
&  (D_{13}-c +1)(D_{14}+t) (2\pi y_2) \hat{\vp}_{l-e_{12}+e_{13}} 
  + (2\pi y_1)(2\pi y_2) (D_{14} + t +1) \hat{\vp}_{l-e_{12}+e_{23}} \\
& - (2\pi y_2)^2  (D_{14} + t) \hat{\vp}_{l}
 + (2\pi y_2)(2\pi y_3)  (D_{14} + t-1) \hat{\vp}_{l-e_{12}+e_{14}} = 0.
\end{align*}
By using 
(\ref{PDE:DS14q}), (\ref{PDE:DS23q}) and (\ref{PDE:DS12q}), we  have 
\begin{align*}
& \{ -(D_{12}-c+1) (D_{13}-c +1)(D_{14}+t)  + (2\pi y_1)^2(D_{12}-c+1) - (2\pi y_2)^2  (D_{14}+ t)  - (2\pi y_3)^2 (D_{12}-c+1) \}  \hat{\vp}_l 
 \\
& + (2\pi y_1)(2\pi y_2) (-c +t+1) \hat{\vp}_{l-e_{12}+e_{23}} +   (2\pi y_2)(2\pi y_3) (c+ t-1) \hat{\vp}_{l-e_{12}+e_{14}} 
= 0.
\end{align*}
Substitution $ t = \pm (c-1) $ implies the equations (\ref{eqn:l-e12+e_14a}) and (\ref{eqn:l-e12+e_23a}).
From (\ref{PDE:DS23q}) together with (\ref{eqn:l-e12+e_13a}) and (\ref{eqn:l-e12+e_14a}) we have 
\begin{align*}
 2(c-1)(2 \pi y_1)(2\pi y_2)(2\pi y_3) \hat{\vp}_{l-e_{12}+e_{24}} 
& = \{ -(D_{14}+c-1) X_{1}(-c+1) -2(c-1) (2\pi y_1)^2 (D_{12}-c+1)\} \hat{\vp}_l.
\end{align*}
Since (\ref{PDE:DS24q}) implies 
\begin{align*}
&  2(c-1)  (2\pi y_1)(2\pi y_2)^2(2\pi y_3) \hat{\vp}_{l-e_{12}+e_{34}}\\
& = 2(c-1) \{  (D_{13}+c-1) (2\pi y_1)(2\pi y_2)(2\pi y_3) \hat{\vp}_{l-e_{12}+e_{24}}  \\
& \quad    +  (2\pi y_1)^2 (2\pi y_2)(2\pi y_3) \hat{\vp}_{l-e_{12}+e_{14}} 
    +   (2\pi y_1)(2\pi y_2)(2\pi y_3)^2 \hat{\vp}_{l-e_{12}+e_{23}} \}
\\
& = (D_{13}+c-1) \{ -(D_{14}+c-1) X_{1}(-c+1) -2(c-1) (2\pi y_1)^2 (D_{12}-c+1)\}  \\
& \quad + (2\pi y_1)^2  X_{-1} (c-1) - (2\pi y_3)^2 X_{1} (-c+1),
\end{align*}
we get (\ref{eqn:l-e12+e_34a}).
The case of $ l_{34} \ge 1 $ can be similarly done.
\end{proof}

\begin{prop} \label{prop:PDE_reduction_classone}
Set $ c = \nu_1+\nu_2-\tfrac12 \gamma_1 $.
We assume that 
\begin{align} \label{assumption}
  \begin{cases} 
   \nu_1+\nu_2 \neq \nu_3+\nu_4 & \text{case 1-(iii)}, \\
   \nu_2 \neq \nu_3 & \text{case 2-(ii)}.
  \end{cases}
\end{align}
(i) 
For $ \mu = (\mu_1,\mu_2,\mu_3,\mu_4) \in \bC^4 $, we set 
\begin{align*}
 \mathcal{D}_{2}(\mu) 
& = -(\pa_1^2+\pa_2^2+\pa_3^2)+\pa_1\pa_2+\pa_2\pa_3  + \sigma_1(\mu) \pa_3 
  -\sigma_2(\mu) + (2\pi y_1)^2 + (2\pi y_2)^2 + (2\pi y_3)^2,
\\ 
 \mathcal{D}_{3}(\mu)
& = \pa_2(\pa_1-\pa_3)(\pa_1-\pa_2+\pa_3) - \sigma_1(\mu)(\pa_1^2+\pa_2^2-\pa_1\pa_2-\pa_2\pa_3) -\sigma_3(\mu) \\
& \quad  + (2\pi y_1)^2(-\pa_2+\sigma_1(\mu)) + (2\pi y_2)^2 (\pa_1-\pa_3+\sigma_1(\mu)) + (2\pi y_3)^2 \pa_2,
\\
 \mathcal{D}_{4}(\mu) 
& = \pa_1(\pa_2-\pa_1)(\pa_3-\pa_2)(-\pa_3+\sigma_1(\mu))   
  - \sigma_4(\mu)
 + (2\pi y_1)^2 (\pa_3-\pa_2)(-\pa_3+\sigma_1(\mu)) \\
& \quad + (2\pi y_2)^2 \pa_1(-\pa_3+\sigma_1(\mu))  + (2\pi y_3)^2 \pa_1(\pa_2-\pa_1) + (2\pi y_1)^2 (2\pi y_3)^2.
\end{align*}
Then, for $ l = (\kappa_1-\kappa_2) e_4+l_{12} e_{12} + l_{34} e_{34} \in S_{(\kappa_1,\kappa_2,\delta_3)} $, we have 
\begin{align} \label{eqn:Capelli_reduction}
  \mathcal{D}_i (r) \hat{\vp}_{l} = 0 \quad (i=2,3,4), 
\end{align}
where
\begin{align*}
r = \begin{cases}
   (\nu_1, \nu_2, \nu_3, \nu_4) & \text{case 1-(i)}, \\
   (\nu_1+1, \nu_2, \nu_3, \nu_4) & \text{case 1-(ii)}, \\
   (\nu_1+l_{34}, \nu_2+l_{34}, \nu_3+l_{12}, \nu_4+l_{12}) & \text{case 1-(iii)}, \\  
   (\nu_4+1, \nu_2, \nu_3, \nu_1) & \text{case 1-(iv)}, \\
   (\nu_1+\tfrac{\kappa_1-1}{2}, \nu_1+\tfrac{\kappa_1+1}{2}, \nu_2+l_{34},\nu_3+l_{12}) & \text{case 2}, \\
   (\nu_1+\tfrac{\kappa_1-1}{2}, \nu_1+\tfrac{\kappa_1+1}{2}, \nu_2+\tfrac{\kappa_2-1}{2}, \nu_2+\tfrac{\kappa_2+1}{2})
  & \text{case 3}. \end{cases} 
\end{align*}
(ii) Assume $ \kappa_2 \ge 1 $.
For $ (c_1,c_2) \in \bC^2 $ and $i =1,3 $, we set 
\begin{align*}
  \mathcal{E}_i(c_1,c_2) = (\pa_i-c_1)(\pa_i- c_2) - (2\pi y_i)^2. 
\end{align*}
\begin{itemize}
\item If $ c \neq 1 $ then, 
for $ l =  (\kappa_1-\kappa_2) e_4+l_{12} e_{12} + l_{34} e_{34} \in S_{(\kappa_1-1,\kappa_2-1,0)} $, we have
\begin{align}  
\label{eqn:e12_vs_e14}
 (2\pi y_2)(2\pi y_3)  \hat{\vp}_{l+e_{14}} & = \mathcal{E}_3(a_1,a_2) \hat{\vp}_{l+e_{12}}, 
\\
\label{eqn:e12_vs_e23}
  (2\pi y_1)(2\pi y_2) \hat{\vp}_{l+ e_{23}} & = \mathcal{E}_1(b_1,b_2) \hat{\vp}_{l+e_{12}}
\end{align}
and 
\begin{align} \label{eqn:e12_vs_e34}
\begin{split}
 & (2\pi y_1)( 2\pi y_2)^2 (2\pi y_3) \hat{\vp}_{l+e_{34}}
\\
& = [  \{ (\pa_1-\pa_2+\pa_3-\tfrac{\gamma_1+\kappa_1+\kappa_2}{2}+c-1)
    (\pa_1-\pa_3+\tfrac{\gamma_1+\kappa_1+\kappa_2}{2}+c-1) + (2\pi y_3)^2 \} \mathcal{E}_1(b_1,b_2) 
\\
& \quad + (2\pi y_1)^2 \{ -(\pa_2-\tfrac{\gamma_1+\kappa_1+\kappa_2}{2}-c+1)
    (\pa_1-\pa_2+\pa_3-\tfrac{\gamma_1+\kappa_1+\kappa_2}{2}+c-1) + \mathcal{E}_3(a_1,a_2) \} ] \hat{\vp}_{l+e_{12}}.
\end{split}
\end{align}
\item If $ c \neq -1 $ then, 
for $ l =  (\kappa_1-\kappa_2) e_4+l_{12} e_{12} + l_{34} e_{34} \in S_{(\kappa_1-1,\kappa_2-1,0)} $, we have
\begin{align}  \label{eqn:e34_vs_e14}
 (2\pi y_1)(2\pi y_2)  \hat{\vp}_{l+e_{14}} &= \mathcal{E}_1(a_3,a_4) \hat{\vp}_{l+e_{34}},  
\\  \label{eqn:e34_vs_e23}
 (2\pi y_2)(2\pi y_3)  \hat{\vp}_{l+e_{23}} &= \mathcal{E}_3(b_3,b_4) \hat{\vp}_{l+e_{34}}
\end{align}
and
\begin{align}  \label{eqn:e34_vs_e12}
\begin{split}
& (2\pi y_1)(2\pi y_2)^2 (2\pi y_3) \hat{\vp}_{l+e_{12}} \\
& = [  \{ (\pa_1-\pa_2+\pa_3 - \tfrac{\gamma_1+\kappa_1+\kappa_2}{2}-c-1)
     ( \pa_1-\pa_3 + \tfrac{\gamma_1+\kappa_1+\kappa_2}{2}  -c- 1 ) + (2\pi y_3)^2 \} \mathcal{E}_1(a_3,a_4)  
\\
& \quad  + (2\pi y_1)^2  \{  \mathcal{E}_3(b_3,b_4) -  (\pa_2-\tfrac{\gamma_1+\kappa_1+\kappa_2}{2}+c+1)
     (\pa_1-\pa_2+\pa_3 - \tfrac{\gamma_1+\kappa_1+\kappa_2}{2}-c+1) \} ] \hat{\vp}_{l+e_{34}}.
\end{split}
\end{align}
\end{itemize}
Here
\begin{align*}
 (a_1,a_2,a_3,a_4) & = \begin{cases} 
   (\nu_1+\nu_2+\nu_3+1, \nu_1+\nu_2+\nu_4+1,  \nu_3, \nu_4) & \text{case 1-(iii)}, \\
   (2\nu_1+\nu_2+\kappa_1, \nu_1+\nu_2+\nu_3+\tfrac{\kappa_1+1}{2},  \nu_1+\tfrac{\kappa_1-1}{2}, \nu_3 ) 
  & \text{case 2-(ii)}, \\
   (2\nu_1+\nu_2+\kappa_1+\tfrac{\kappa_2-1}{2}, \nu_1+2\nu_2+\tfrac{\kappa_1-1}{2}+\kappa_2, 
   \nu_1+\tfrac{\kappa_1-1}{2}, \nu_2+\tfrac{\kappa_2-1}{2}) & \text{case 3},
   \end{cases}
\\
 (b_1,b_2,b_3,b_4) & = \begin{cases} 
   (\nu_1,\nu_2, \nu_1+\nu_3+\nu_4+1, \nu_2+\nu_3+\nu_4+1) & \text{case 1-(iii)}, \\
   (\nu_1+\tfrac{\kappa_1-1}{2}, \nu_2,  2\nu_1+\nu_3+\kappa_1 , \nu_1+\nu_2+\nu_3+ \tfrac{\kappa_1+1}{2}) 
  & \text{case 2-(ii)}, \\
   (\nu_1+\tfrac{\kappa_1-1}{2}, \nu_2+\tfrac{\kappa_2-1}{2} ,
   2 \nu_1+\nu_2 +\kappa_1+ \tfrac{\kappa_2-1}{2}, \nu_1+2\nu_2+\tfrac{\kappa_1-1}{2}+\kappa_2 ) & \text{case 3}.
   \end{cases}
\end{align*}
\end{prop}

\begin{proof}Let us show (i). 
For $ l = (\kappa_1-\kappa_2) e_1+l_4  e_4+l_{12} e_{12} + l_{34} e_{34} \in S_{(\kappa_1,\kappa_2,\delta_3)} $, we have 
\begin{align} \label{eqn:K-action_for_edge}
\begin{split}
 \mathfrak{K}_{12} \hat{\vp}_{l} & = 0, 
\\
\mathfrak{K}_{23} \hat{\vp}_{l}
& = l_{12} \hat{\varphi}_{l-e_{12}+e_{13}}  + l_{34} \hat{\varphi}_{l-e_{34}+e_{24}},
\\
\mathfrak{K}_{34} \hat{\vp}_{l} 
& = (\kappa_1-\kappa_2) \hat{\varphi}_{l-e_4+e_3}   
\\
\mathfrak{K}_{13} \hat{\vp}_{l}
& =   -l_{12} \hat{\varphi}_{l-e_{12}+e_{23}} - l_{34} \hat{\varphi}_{l-e_{34}+e_{14}},
\\
\mathfrak{K}_{24} \hat{\vp}_{l}
& =  - (\kappa_1-\kappa_2) \hat{\varphi}_{l-e_4+e_2} 
      + l_{12} \hat{\varphi}_{l-e_{12}+e_{14}} + l_{34} \hat{\varphi}_{l-e_{34}+e_{23}},
\\
\mathfrak{K}_{14} \hat{\vp}_{l}
& =   -(\kappa_1-\kappa_2) \hat{\varphi}_{l-e_4+e_1}  
 +l_{12} \hat{\varphi}_{l-e_{12}+e_{24}}  + l_{34} \hat{\varphi}_{l-e_{34}+e_{13}}, \\
\mathfrak{K}_{12, 34} \hat{\vp}_{l} & = 0.
\end{split}
\end{align}
Then the equation (\ref{PDE:C2}) and Lemmas \ref{lem:DS1_l-e1_and_l-e4} and 
\ref{lem:DS2_l-e12_and_l-e34} tell us that $ D_2 \hat{\vp}  = 0 $ where 
\begin{align*}
 D_{2}
& = \Delta_2 -(\kappa_1-\kappa_2) ( -\pa_3+\gamma_1-\nu_1'+\tfrac{\kappa_1-1}{2}+\kappa_2) \\
& \quad +  l_{12} (\pa_2 - \tfrac{\gamma_1+\kappa_1+\kappa_2}{2}-c+1) 
   + l_{34}   ( \pa_2 - \tfrac{\gamma_1+\kappa_1+\kappa_2}{2} +c+1).
\end{align*}
In view of $ \gamma_1+\kappa_1+\kappa_2 =\sigma_1(r) $, we can see that $ D_2 = \mathcal{D}_{2}(r) $.

\medskip

Let us show $ \mathcal{D}_3(r) \hat{\vp}_l= 0 $. 
When $ c=1$, (\ref{eqn:l-e12+e_14a}) implies that 
\begin{align*} 
&  \{ (\pa_2-\tfrac{\gamma_1+\kappa_1+\kappa_2}{2}) 
    (\pa_1-\pa_2+\pa_3-\tfrac{\gamma_1+\kappa_1+\kappa_2}{2})(\pa_1-\pa_3+\tfrac{\gamma_1+\kappa_1+\kappa_2}{2})  \\
& - (2\pi y_1)^2(\pa_2-\tfrac{\gamma_1+\kappa_1+\kappa_2}{2}) 
    + (2\pi y_2)^2  (\pa_1-\pa_3+\tfrac{\gamma_1+\kappa_1+\kappa_2}{2})  
   + (2\pi y_3)^2 (\pa_2-\tfrac{\gamma_1+\kappa_1+\kappa_2}{2}) \}  \hat{\vp}_l = 0. 
\end{align*}
Because of the identity 
$$
  8\sigma_3(\mu) - 4 \sigma_1(\mu) \sigma_2(\mu) + (\sigma_1(\mu))^3
= (\mu_1+\mu_2-\mu_3-\mu_4)(\mu_1-\mu_2+\mu_3-\mu_4)(\mu_1-\mu_2-\mu_3+\mu_4),
$$
we know that the above differential equation is equivalent to 
$ \{ \mathcal{D}_3(r) - \tfrac{\gamma_1+\kappa_1+\kappa_2}{2} \mathcal{D}_2(r) \} \hat{\vp}_l = 0 $.
Then $ \mathcal{D}_2(r) \hat{\vp}_l = 0 $ implies $ \mathcal{D}_3(r) \hat{\vp}_l = 0 $. The case of $ c = -1 $ is similar.

\medskip 

Let $ c \neq \pm  1 $.
From (\ref{PDE:C3}) and (\ref{eqn:K-action_for_edge}) we have
\begin{align} \label{eqn:C3_red_1}
\begin{split}
& \Delta_3 \hat{\vp}_l
 - (2\pi y_3) \pa_2 (\kappa_1-\kappa_2) \hat{\vp}_{l-e_4+e_3} - (2\pi y_2)(2\pi y_3) (\kappa_1-\kappa_2) \hat{\vp}_{l-e_4+e_2}
\\
& + (2\pi y_2) (-\pa_1+\pa_3-\gamma_1-\kappa_2)  ( l_{12} \hat{\vp}_{l-e_{12}+e_{13}} + l_{34} \hat{\vp}_{l-e_{34}+e_{24}} ) 
\\
&  + l_{12} \{ -(2\pi y_1)(2\pi y_2) \hat{\vp}_{l-e_{12}+e_{23}} + (2\pi y_2)(2\pi y_3)  \hat{\vp}_{l-e_{12}+e_{14}} \}
\\
&  + l_{34} \{ -(2\pi y_1)(2\pi y_2) \hat{\vp}_{l-e_{34}+e_{14}}  + (2\pi y_2)(2\pi y_3)   l_{34} \hat{\vp}_{l-e_{34}+e_{23}} \} = 0. 
\end{split}
\end{align}
Then Lemmas \ref{lem:DS1_l-e1_and_l-e4} and 
\ref{lem:DS2_l-e12_and_l-e34} imply that 
$ D_3 \hat{\vp}_l = 0  $ where 
\begin{align*}
D_3 & =  
 \Delta_3 
 - (\kappa_1-\kappa_2) \pa_2 (-\pa_3+ \sigma_1(r)-\nu_1'-\tfrac{\kappa_1+1}{2} )
\\
&\quad  -  (\kappa_1-\kappa_2) \{  (-\pa_2+\pa_3-\nu_1' - \tfrac{\kappa_1 +1}{2}  )  
      (-\pa_3+ \sigma_1(r)-\nu_1'-\tfrac{\kappa_1+1}{2} ) + (2\pi y_3)^2 \} 
\\
& \quad  
+ (-\pa_1+\pa_3-\sigma_1(r) +\kappa_1) \{ l_{12} (-\pa_2 + \tfrac{\sigma_1(r)}{2}+c-1) +l_{34} 
     ( -\pa_2 + \tfrac{\sigma_1(r)}{2}-c-1) \} 
\\
& \quad + \tfrac{l_{12}}{c-1} 
       \{ (\pa_2-\tfrac{\sigma_1(r)}{2}-c+1) 
    (\pa_1-\pa_2+\pa_3- \tfrac{\sigma_1(r)}{2}-c+1)
   (\pa_1-\pa_3+ \tfrac{\sigma_1(r)}{2}) \\
& \qquad  - (2\pi y_1)^2 (\pa_2- \tfrac{\sigma_1(r)}{2} -c+1) 
    - (2\pi y_2)^2  (\pa_1-\pa_3+ \tfrac{\sigma_1(r)}{2}) 
    - (2\pi y_3)^2 (\pa_2- \tfrac{\sigma_1(r)}{2}-c+1) \}  
\\
& \quad +\tfrac{l_{34}}{c+1} 
    \{ -(\pa_2- \tfrac{\sigma_1(r)}{2} +c+1)
     (\pa_1-\pa_2+\pa_3- \tfrac{\sigma_1(r)}{2} -c+1) 
     (\pa_1-\pa_3+ \tfrac{\sigma_1(r)}{2}) \\ 
& \qquad  + (2 \pi y_1)^2 (\pa_2- \tfrac{\sigma_1(r)}{2} +c+1)
     - (2\pi y_2)^2 (\pa_1-\pa_2+\pa_3-\tfrac{\sigma_1(r)}{2} +c+1) 
- (2\pi y_3)^2 (\pa_2-\tfrac{\sigma_1(r)}{2}+c+1) \}.
\end{align*}
Here we used $ \gamma_1+\kappa_1+\kappa_2 =\sigma_1(r) $. 
Direct computation tells us that 
\begin{align*}
 D_3 & = ( 1+ \tfrac{l_{12}}{c-1}- \tfrac{l_{34}}{c+1} ) \mathcal{D}_3(r) + 
   \{ - \kappa_1 - \tfrac12 \sigma_1(r) ( \tfrac{l_{12}}{c-1}- \tfrac{l_{34}}{c+1}  ) \} \mathcal{D}_2(r).
\end{align*}
Because of the assumption (\ref{assumption}), we know $  1+ \tfrac{l_{12}}{c-1}- \tfrac{l_{34}}{c+1}  \neq 0 $.
Hence we get 
$ \mathcal{D}_3(r) \hat{\vp}_l = 0 $. 

\medskip 

Let us show $ \mathcal{D}_4(r) \hat{\vp}_l = 0$.
From (\ref{PDE:C4}) and (\ref{eqn:K-action_for_edge}) we have
\begin{align} \label{eqn:C4_red_1}
\begin{split}
& [ \Delta_4  \hat{\vp}_l
 - (2\pi  y_2) \pa_1(-\pa_3+\gamma_1+\kappa_2) ( l_{12} \hat{\vp}_{l-e_{12}+e_{13}} + l_{34} \hat{\vp}_{l-e_{34}+e_{24}} ) 
\\
&  - (2\pi y_3) \{ \pa_1(-\pa_1+\pa_2) + (2\pi y_1)^2 \} (\kappa_1-\kappa_2) \hat{\vp}_{l-e_4+e_3}
\\
& + (2\pi y_1)(2\pi y_2)  (-\pa_3+\gamma_1+\kappa_2)( -l_{12} \hat{\vp}_{l-e_{12}+e_{23}} - l_{34} \hat{\vp}_{l-e_{34}+e_{14}} )
\\
&  - (2\pi y_2)(2\pi y_3) \pa_1 (\kappa_1-\kappa_2) \hat{\vp}_{l-e_4+e_2} 
     + (2\pi y_2)(2\pi y_3) \pa_1 (l_{12}\hat{\vp}_{l-e_{12}+e_{14}} + l_{34} \hat{\vp}_{l-e_{34}+e_{23}} ) \\ 
& - (2\pi y_1)(2\pi y_2)(2\pi y_3) (\kappa_1-\kappa_2) \hat{\vp}_{l-e_4+e_1} 
    + (2\pi y_1)(2\pi y_2)(2\pi y_3) ( l_{12} \hat{\vp}_{l-e_{12}+e_{24}} + l_{34} \hat{\vp}_{l-e_{34}+e_{13}} ) = 0. 
\end{split}
\end{align}
From the equations (\ref{PDE:DS23}) and (\ref{PDE:DS24}), we know
\begin{align}
\label{PDE:DS23r}
& (-D_{14}-c)  \hat{\vp}_{l-e_{34}+e_{23}} 
 - (2 \pi y_1) \hat{\vp}_{l-e_{34}+e_{13}} + ( 2\pi  y_3) \hat{\vp}_{l-e_{34}+e_{24}}  = 0,
\\
\label{PDE:DS24r}
& ( -D_{13}  -c) \hat{\vp}_{l-e_{34}+e_{24}} - (2\pi y_1) \hat{\vp}_{l-e_{34}+e_{14}}
  + (2\pi y_2) \hat{\vp}_{l} - (2\pi y_3) \hat{\vp}_{l-e_{34}+e_{23}} = 0.
\end{align}
By computing 
\begin{align*}
& (\ref{eqn:C4_red_1}) - \tfrac{\gamma_1-\kappa_1+\kappa_2}{4} (\ref{eqn:C3_red_1})
+  (2\pi y_2)(2\pi y_3) (-l_{12} (\ref{PDE:DS14q}) + l_{34} (\ref{PDE:DS23r}) )  \\
&  + (2\pi y_2)  (\pa_3-\tfrac{3\gamma_1+\kappa_1+3\kappa_2}{4}) (-l_{12} (\ref{PDE:DS13q}) + l_{34} (\ref{PDE:DS24r})),
\end{align*}
we arrive at the equation 
\begin{align*}
& \{  \Delta_4  - \tfrac{\gamma_1-\kappa_1+\kappa_2}{4}   \Delta_3 + \kappa_2 (2\pi y_2)^2 (\pa_3- \tfrac{3\gamma_1+\kappa_1+3\kappa_2}{4}) \} \hat{\vp}_l 
\\
& + (\kappa_1-\kappa_2) \{ -\pa_1(-\pa_1+\pa_2) - (2\pi y_1)^2  + \tfrac{\gamma_1-\kappa_1+\kappa_2}{4}  \pa_2 \} 
      (2\pi y_3) \hat{\vp}_{l-e_4+e_3} \\
& + (\kappa_1-\kappa_2) (  -\pa_1 + \tfrac{\gamma_1-\kappa_1+\kappa_2}{4}  ) (2\pi y_2)(2\pi y_3)\hat{\vp}_{l-e_4+e_2} 
   - (\kappa_1-\kappa_2) (2\pi y_1)(2\pi y_2)(2\pi y_3)\hat{\vp}_{l-e_4+e_1} 
\\
&  + l_{12} (c-1)(2\pi y_2)(2\pi y_3) \hat{\vp}_{l-e_{12}+e_{14}} - l_{34} (c+1) (2\pi y_2)(2\pi y_3) \hat{\vp}_{l-e_{34}+e_{23}} 
\\
& + l_{12} \{ -\pa_1(-\pa_3+\gamma_1+\kappa_2) - \tfrac{\gamma_1-\kappa_1+\kappa_2}{4} 
       (-\pa_1+\pa_3-\gamma_1-\kappa_2) + (2\pi y_3)^2 \\
&  \quad  -(\pa_3- \tfrac{3\gamma_1+\kappa_1+3\kappa_2}{4})
       (\pa_1-\pa_2+\pa_3-\tfrac{\gamma_1+\kappa_1+\kappa_2}{2}-c+1) \} (2\pi y_2) 
                 \hat{\vp}_{l-e_{12}+e_{13}} \\
& + l_{34} \{ -\pa_1(-\pa_3+\gamma_1+\kappa_2) -\tfrac{\gamma_1-\kappa_1+\kappa_2}{4}  
     (-\pa_1+\pa_3-\gamma_1-\kappa_2) + (2\pi y_3)^2 \\
&  \quad  -(\pa_3- \tfrac{3\gamma_1+\kappa_1+3\kappa_2}{4})
       (\pa_1-\pa_2+\pa_3-\tfrac{\gamma_1+\kappa_1+\kappa_2}{2}+c+1) \} (2\pi y_2) 
                \hat{\vp}_{l-e_{34}+e_{24}}  = 0.
\end{align*}
In view of Lemmas \ref{lem:DS1_l-e1_and_l-e4} and 
\ref{lem:DS2_l-e12_and_l-e34}, we know $ D_4 \hat{\vp}_l  = 0 $ where
\begin{align*}
 D_4
& = \Delta_4  - \tfrac{\sigma_1(r)-2\kappa_1}{4}  \Delta_3 + 
    \kappa_2 (2\pi y_2)^2 (\pa_3- \tfrac{3\sigma_1(r) -2 \kappa_1}{4}) 
\\
& \quad + (\kappa_1-\kappa_2) \{ -\pa_1(-\pa_1+\pa_2) - (2\pi y_1)^2  + \tfrac{\sigma_1(r)-2\kappa_1}{4}  \pa_2 \} 
    (-\pa_3+\sigma_1(r) -\nu_1' -\tfrac{\kappa_1+1}{2})\\
& \quad + (\kappa_1-\kappa_2) (  -\pa_1 + \tfrac{\sigma_1(r)-2\kappa_1}{4}   ) \{ (-\pa_2+\pa_3-\nu_1' - \tfrac{\kappa_1 +1}{2}  )  
      (-\pa_3+\sigma_1(r) -\nu_1' -\tfrac{\kappa_1+1}{2} ) + (2\pi y_3)^2 \}
\\
& \quad  - (\kappa_1-\kappa_2)
   \{ (-\pa_1+\pa_2-\nu_1'-\tfrac{\kappa_1+1}{2} ) (-\pa_2+\pa_3-\nu_1' - \tfrac{\kappa_1 +1}{2} )  
      (-\pa_3+ \sigma_1(r) -\nu_1' -\tfrac{\kappa_1+1}{2}) \\
& \qquad  
 + (2\pi y_2)^2(-\pa_3+ \sigma_1(r) -\nu_1' -\tfrac{\kappa_1+1}{2}) 
 + (2\pi y_3)^2  (-\pa_1+\pa_2-\nu_1'-\tfrac{\kappa_1 + 1}{2}  )\}
\\
&  \quad + \tfrac12 l_{12} \{ (\pa_2-\tfrac{\sigma_1(r)}{2} -c+1)(\pa_1-\pa_2+\pa_3-\tfrac{\sigma_1(r)}{2} -c+1)(\pa_1-\pa_3+\tfrac{\sigma_1(r)}{2}+c-1)
\\
& \qquad -(2\pi y_1)^2 (\pa_2-\tfrac{\sigma_1(r)}{2} -c+1) 
  +(2\pi y_2)^2 (\pa_1-\pa_3+\tfrac{\sigma_1(r)}{2}+c-1)
   + (2\pi y_3)^2 (\pa_2-\tfrac{\sigma_1(r)}{2} -c+1) \}\\
&
 \quad + \tfrac12 l_{34}  \{ (\pa_2-\tfrac{\sigma_1(r)}{2} +c+1)
  (\pa_1-\pa_2+\pa_3-\tfrac{\sigma_1(r)}{2} +c+1)
  (\pa_1-\pa_3+\tfrac{\sigma_1(r)}{2}-c-1) \\
& \qquad - (2\pi y_1)^2 (\pa_2-\tfrac{\sigma_1(r)}{2} +c+1) 
   + (2\pi y_2)^2 (\pa_1-\pa_3+\tfrac{\sigma_1(r)}{2}-c-1)
   + (2\pi y_3)^2  (\pa_2-\tfrac{\sigma_1(r)}{2} +c+1) \} 
\\
& \quad + l_{12} \{ -\pa_1(-\pa_3+\sigma_1(r) - \kappa_1) - \tfrac{\sigma_1(r)-2\kappa_1}{4} 
      (-\pa_1+\pa_3 -\sigma_1(r) + \kappa_1) + (2\pi y_3)^2 \\
& \qquad   -(\pa_3- \tfrac{3\sigma_1(r) -2 \kappa_1}{4})(\pa_1-\pa_2+\pa_3-\tfrac{\sigma_1(r)}{2}-c+1) \} 
     ( -\pa_2 + \tfrac{\sigma_1(r)}{2} +c-1) \\
& \quad + l_{34} \{ -\pa_1(-\pa_3+\sigma_1(r)-\kappa_1) - \tfrac{\sigma_1(r)-2\kappa_1}{4}  (-\pa_1+\pa_3 -\sigma_1(r) + \kappa_1) + (2\pi y_3)^2 \\
&  \qquad  -(\pa_3- \tfrac{3\sigma_1(r) -2 \kappa_1}{4})(\pa_1-\pa_2+\pa_3-\tfrac{\sigma_1(r)}{2}+c+1) \} 
       ( -\pa_2 + \tfrac{\sigma_1(r)}{2} -c-1),
\end{align*}
Since we can confirm the identity
\begin{align*}
 D_4 & = \mathcal{D}_4(r) - \tfrac{ \sigma_1(r) +2 \kappa_1-2\kappa_2}{4} \mathcal{D}_3(r)
   + \{ \tfrac12(c-1) l_{12} - \tfrac12(c+1) l_{34} 
     + (\kappa_1-\kappa_2)(\nu_1' + \tfrac12) + \tfrac14 \kappa_1 \sigma_1(r)  \} \mathcal{D}_2(r)
\end{align*}
by case by case argument, 
$ {\mathcal D}_3(r) \hat{\vp}_l =  {\mathcal D}_2(r) \hat{\vp}_l = 0 $ implies $  {\mathcal D}_4(r) \hat{\vp}_l = 0 $
as desired.

\medskip

Let us prove the equations in (ii). 
For $ \varep \in \{ \pm 1\} $, $ t \in \bC $ and $ \mu= (\mu_1,\mu_2,\mu_3, \mu_4) \in \bC^4$, 
the following identity holds:
\begin{align} \label{eqn:DS_vs_C3_1}
\begin{split}
& (\pa_2-\tfrac{\sigma_1(\mu)}{2}+ \varep t)(\pa_1-\pa_2+\pa_3-\tfrac{\sigma_1(\mu)}{2}+ \varep  t)(\pa_1-\pa_3+\tfrac{\sigma_1(\mu)}{2}+t) \\
& - (2\pi y_1)^2 (\pa_2-\tfrac{\sigma_1(\mu)}{2} - \varep t) + (2\pi y_2)^2 (\pa_1-\pa_3+\tfrac{\sigma_1(\mu)}{2}+t) 
    + (2\pi y_3)^2 (\pa_2-\tfrac{\sigma_1(\mu)}{2} -\varep  t) 
\\
& = \mathcal{D}_3(\mu) + (-\tfrac{\sigma_1(\mu)}{2}+t) \mathcal{D}_2(\mu)
+ \mathcal{R}(t,\mu)   +
  \begin{cases}
   2t \, \mathcal{E}_1(\mu_1, \tfrac{\sigma_1(\mu)}{2} -t-\mu_1) & \text{if $\varep = 1$}, \\
   2t \, \mathcal{E}_3(\sigma_1(\mu)-\mu_1, \tfrac{\sigma_1(\mu)}{2} +t+\mu_1) 
   & \text{if $\varep = -1$}
  \end{cases}
\end{split}
\end{align}
with 
\begin{align*}
 \mathcal{R}(t,\mu) = (t+\tfrac{ \mu_1+ \mu_2-\mu_3-\mu_4}{2})  (t+\tfrac{ \mu_1- \mu_2+\mu_3-\mu_4}{2})
      (t+\tfrac{ \mu_1- \mu_2-\mu_3+\mu_4}{2}).
\end{align*}
Let $ l = (\kappa_1-\kappa_2)e_4+l_{12} e_{12} + l_{34} e_{34} \in S_{(\kappa_1,\kappa_2,\delta_3)} $
with $ l_{12} \ge 1$. 
We use  
(\ref{eqn:DS_vs_C3_1}) with $ (t,\varep) = (c-1, -1) $ and 
\begin{align*}
 \mu= \begin{cases}
   (\nu_3+1,\nu_4+1,\nu_1,\nu_2) & \text{case 1-(iii)}, \\
   (\nu_1+\tfrac{\kappa_1+1}{2}, \nu_1+\tfrac{\kappa_1-1}{2}, \nu_2,\nu_3+1) & \text{case 2-(ii)}, \\
   (\nu_1+\tfrac{\kappa_1+1}{2}, \nu_1+\tfrac{\kappa_1-1}{2}, \nu_2+\tfrac{\kappa_2+1}{2}, \nu_2+\tfrac{\kappa_2-1}{2}) &
  \text{case 3}.
\end{cases}
\end{align*}
Since $ \mathcal{R}(t,\mu) = 0 $ and 
$ \mathcal{D }_2(\mu) \hat{\vp}_{l}= \mathcal{D}_{3}(\mu)\hat{\vp}_l = 0 $, we know
$$ X_{-1}(c-1) \hat{\vp}_l =  2(c-1) \mathcal{E}_3(a_1,a_2)  \hat{\vp}_{l}. $$ 
Thus the equation (\ref{eqn:l-e12+e_14a}) implies (\ref{eqn:e12_vs_e14}).
Similarly we can show 
$$ X_1(-c+1) \hat{\vp}_l = 2(-c+1) \mathcal{E}_1(b_1,b_2) \hat{\vp}_{l} $$ to get 
(\ref{eqn:e12_vs_e23}) and (\ref{eqn:e12_vs_e34}) 
from (\ref{eqn:l-e12+e_23a}) and  (\ref{eqn:l-e12+e_34a}), respectively.
We can similarly show 
(\ref{eqn:e34_vs_e14}), (\ref{eqn:e34_vs_e23}) and (\ref{eqn:e34_vs_e12}).
\end{proof}

\begin{rem} \label{rem:PDE_does_not_characterize}
Let us explain the assumption (\ref{assumption}) for the case 1-(iii).
For $ \sigma = \chi_{(\nu_1,1)} \boxtimes \chi_{(\nu_2,1)} \boxtimes \chi_{(\nu_3,0)} \boxtimes \chi_{(\nu_4,0)} $
and 
$ \sigma' = \chi_{(\nu_3,1)} \boxtimes \chi_{(\nu_4,1)} \boxtimes \chi_{(\nu_1,0)} \boxtimes \chi_{(\nu_2,0)}$,
two $P_0$-principal series representations ${\Pi}_{\sigma} $ and $ \Pi_{\sigma'} $ 
have the same minimal $K$-type $ \tau_{(1,1,0)} $.
When $ \nu_1+\nu_2  = \nu_3 + \nu_4 $, we know that the equations in Proposition \ref{prop:P1111_DS} (i), (iii) 
for $ \Pi_{\sigma} $ and $\Pi_{\sigma'} $ are the same. 
This means that 
the system in Proposition \ref{prop:PDE2} can not characterize Whittaker functions
in case 1-(ii) with $ \nu_1+\nu_2 = \nu_3+\nu_4$. 
\end{rem}

\begin{prop} \label{lem:dim_PDE_kappa_2neq1} {\rm (}\cite[Theorem 3.3]{Hashizume}{\rm )}
Retain the notation in Proposition \ref{prop:PDE_reduction_classone}.
Let $ {\rm Sol}(\mu) $ be the space of smooth functions $f$ on $ (\bR_+)^3 $ satisfying 
\begin{align*}
 \mathcal{D}_i(\mu) f(y_1,y_2,y_3) =  0 \ \ \ (i=2,3,4).
\end{align*}
Then we have $ \dim_{\bC} {\rm Sol}(\mu) \le 24 $.
\end{prop}

\begin{proof} (cf. \cite[Lemma 1.1]{HIM})
For $ f \in  {\rm Sol}(\mu)  $, we define the functions $ f_i $  $(0 \le i \le 23) $ by 
\begin{align*}
 f_{j_1+4j_2+ 12j_3}(y_1,y_2,y_3) & = \pa_{1}^{j_1} \pa_2^{j_2} \pa_3^{j_3}   f(y_1,y_2,y_3)
 & 
\end{align*}
with $0 \le j_1 \le 3$, $0 \le j_2 \le 2$,  $ 0\le j_3 \le 1$. 
Since 
\begin{align*}
\mathcal{D}_{2}(\mu) 
&= -\pa_3^2+\pa_2\pa_3-\pa_2^2+\pa_1\pa_2-\pa_1^2  
+[\text{differential operators of order lower than $2$}],\\
\mathcal{D}_{3}(\mu)-\pa_2\mathcal{D}_{2}(\mu)
&= 
\pa_2^3
-2\pa_1\pa_2^2
+2\pa_1^2\pa_2
+[\text{differential operators of order lower than $3$}],\\
\mathcal{D}_{4}(\mu)-\pa_1\mathcal{D}_{3}(\mu)+\pa_1^2\mathcal{D}_{2}(\mu)
&=
-\pa_1^4+[\text{differential operators of order lower than $4$}],
\end{align*}
the equalities $\mathcal{D}_{2}(\mu)f=0$, $(\mathcal{D}_{3}(\mu)-\pa_2\mathcal{D}_{2}(\mu))f=0$ 
and $(\mathcal{D}_{4}(\mu)-\pa_1\mathcal{D}_{3}(\mu)+\pa_1^2\mathcal{D}_{2}(\mu))f=0$ imply that 
there exists some polynomial functions $ m_{i,j}^k (y_1,y_2,y_3) $
$ (0 \le i,j \le 23, 1 \le k \le 3) $ such that 
\begin{align*}
 \pa_k f_{i} (y_1,y_2,y_3) & =  \sum_{0 \le j \le 23} m_{i,j}^k(y_1,y_2,y_3) f_j (y_1,y_2,y_3)  \quad (k=1,2,3).
\end{align*}
Applying \cite[Theorems B.8 and B.9]{Knapp_002}, we have  $ \dim_{\bC} {\rm Sol}(\mu) \le 24 $.
\end{proof}


\section{Explicit formulas of Whittaker functions} 
\label{sec:EF}

In this section we give explicit formulas of moderate growth Whittaker functions.
As in our previous work \cite{HIM}, we give Mellin-Barnes integral representations of Whittaker functions.
We give a moderate growth solution of the system in Proposition \ref{prop:PDE2}
under the assumption (\ref{assumption}).
In \S \ref{subsec:Jacquet}, we remove the assumption (\ref{assumption}) by considering the relation
between our solutions and Jacquet integrals.

Let us explain the outline of our approach to the system in Proposition \ref{prop:PDE2}.
We have shown the function 
$ \hat{\varphi}_{(\kappa_1-\kappa_2)e_4+ l_{12} e_{12} + l_{34} e_{34}}  $ belongs to the space $ {\rm Sol}(r) $. 
In Proposition \ref{prop:classone_MB}, we describe the space ${\rm Sol}^{\rm mg}(\mu) $ 
consisting of moderate growth functions in $ {\rm Sol}(\mu) $.   
Since $ {\rm Sol}^{\rm mg}(\mu) $ is the space of class one Whittaker functions, 
Proposition \ref{prop:classone_MB} is a rephrase of the results in \cite{Ishii_Stade_001} and \cite{Stade_002}. 
But we give more direct proof here. 
With the aid of the equations (\ref{eqn:e12_vs_e34}) and (\ref{eqn:e34_vs_e12}), we can get explicit formula of 
$ \hat{\varphi}_{(\kappa_1-\kappa_2)e_4+ l_{12} e_{12} + l_{34} e_{34}}  $ in Lemma \ref{lem:EF_step1}. 

To determine $ \hat{\varphi}_l $ for all $l$, 
we use difference-differential equations given in Propositions \ref{prop:PDE2} and \ref{prop:PDE_reduction_classone} (ii), 
and relations among $ \hat{\varphi}_l $ coming from the relations between the generators 
$ \{ u_l \mid l \in S_{(\kappa_1,\kappa_2,\delta_3)} \} $. 
From Lemma \ref{lem:rel_ul}, we know the following:
\begin{itemize}
\item 
When $ \kappa_1 > \kappa_2 > 0 $, for $ l \in S_{(\kappa_1-2,\kappa_2-1,\delta_3)} $, we have 
\begin{gather} \label{eqn:vp_gen_relation_2}
 \sum_{1 \le j \le 4, \, j \neq i} (-1)^{j}\hat{\vp}_{l+e_j+e_{ij}} = 0 \qquad (1 \le i \le 4), 
\\  
 \label{eqn:vp_gen_relation_3}
 \hat{\vp}_{l+e_i+e_{jk}}- \hat{\vp}_{l+e_j+e_{ik}} + \hat{\vp}_{l+e_k+e_{ij}} = 0 \qquad (1 \le i<j<k \le 4). 
\end{gather}
\item 
When $ \kappa_2 \ge 2 $, for $ l \in S_{(\kappa_1-2,\kappa_2-2,\delta_3)} $, we have
\begin{gather} \label{eqn:vp_gen_relation_4}
 \hat{\vp}_{l+2e_{12}} - \hat{\vp}_{l+2e_{13}} + \hat{\vp}_{l+2e_{14}} = 0, 
 \hat{\vp}_{l+2e_{13}} - \hat{\vp}_{l+2e_{23}} - \hat{\vp}_{l+2e_{34}} = 0, 
\\
\label{eqn:vp_gen_relation_5}
  \hat{\vp}_{l+e_{12}+e_{34}} - \hat{\vp}_{l+e_{13}+e_{24}} + \hat{\vp}_{l+e_{14}+e_{23}} = 0.
\end{gather}
\end{itemize}

\medskip

Here is more precise strategy for each case.

\begin{itemize}
\item Case 1-(i) ($ \kappa_1 = \kappa_2 = 0 $):  Straightforward by Proposition \ref{prop:classone_MB}. 
\item Cases 1-(ii), (iv) and 2-(i) ($ \kappa_1 > \kappa_2 = 0 $):
\begin{itemize} 
\item Determine $ \hat{\varphi}_{\kappa_1 e_4} $ by Proposition \ref{prop:classone_MB}.
\item Determine $ \hat{\varphi}_{l_1 e_1 + l_4 e_4} $ by the equation (\ref{eqn:DS1:l-e_4+e_1}).
\item Determine $ \hat{\varphi}_{l} $ by the equations (\ref{PDE:DS1}) and (\ref{PDE:DS4}).
\end{itemize}  
\item Cases 1-(iii)  ($ \kappa_1 = \kappa_2 = 1 $):
\begin{itemize}
\item Determine $ \hat{\varphi}_{ l_{12} }$ and $ \hat{\vp}_ { e_{34} } $ by 
Proposition \ref{prop:classone_MB} and the equations (\ref{eqn:e12_vs_e34}) and (\ref{eqn:e34_vs_e12}).
\item Determine $ \hat{\vp}_{e_{14}} $ and $ \hat{\vp}_{e_{23}} $ by the equations 
(\ref{eqn:e12_vs_e14}), (\ref{eqn:e12_vs_e23}), 
(\ref{eqn:e34_vs_e14}) and (\ref{eqn:e34_vs_e23}).
\item Determine $ \hat{\vp}_{l_{13}} $ and  $\hat{\vp}_{e_{24}} $ by the equations (\ref{PDE:DS12}) and (\ref{PDE:DS34}).
\end{itemize}
\item Cases 2-(i) ($ \kappa_1 > \kappa_2 =1 $):
\begin{itemize}
\item Determine $ \hat{\varphi}_{ (\kappa_1-1) e_4 + e_{12} }$ and  $ \hat{\varphi}_{ (\kappa_1-1) e_4 + e_{34} }$ 
by 
Proposition \ref{prop:classone_MB} and the equations (\ref{eqn:e12_vs_e34}) and (\ref{eqn:e34_vs_e12}).
\item Determine $ \hat{\varphi}_{l_1 e_1 + l_4 e_4 + l_{12} e_{12} + l_{34} e_{34}} $ by the equation (\ref{eqn:DS1:l-e_4+e_1}).
\item Determine $ \hat{\varphi}_{l_1 e_1 + l_2 e_2 +l_3 e_3 + l_4 e_4 + l_{12} e_{12} + l_{13} e_{13} + l_{24} e_{24} + l_{34} e_{34}} $ by the equations (\ref{PDE:DS1}), (\ref{PDE:DS4}), (\ref{PDE:DS12}) and (\ref{PDE:DS34}).
\item Determine $ \hat{\varphi}_{l} $ by the relations (\ref{eqn:vp_gen_relation_2}) and (\ref{eqn:vp_gen_relation_3}).
\end{itemize}
\item Case 3 with $ \kappa_1 > \kappa_2 \ge  2 $:
\begin{itemize}
\item Determine $ \hat{\varphi}_{ (\kappa_1-\kappa_2) e_4 + l_{12} e_{12} + l_{34} e_{34} } $ by Proposition \ref{prop:classone_MB}  and the equations (\ref{eqn:e12_vs_e34}) and (\ref{eqn:e34_vs_e12}).
\item Determine  $ \hat{\varphi}_{l_1 e_1 + l_4 e_4 + l_{12} e_{12} + l_{34} e_{34} } $ by the equation (\ref{eqn:DS1:l-e_4+e_1}).
\item Determine  $ \hat{\varphi}_{l_1 e_1 + l_2 e_2 +l_3 e_3 + l_4 e_4 + l_{12} e_{12} + l_{13} e_{13} + l_{24} e_{24} + l_{34} e_{34}} $ by the equations (\ref{PDE:DS1}), (\ref{PDE:DS4}), (\ref{PDE:DS12}) and (\ref{PDE:DS34}).
\item Determine $ \hat{\varphi}_l $ by the relations (\ref{eqn:vp_gen_relation_2}) and (\ref{eqn:vp_gen_relation_3}).
\end{itemize}
\item Case 3 with $ \kappa_1 = \kappa_2 \ge 2  $:
\begin{itemize}
\item Determine  $ \hat{\varphi}_{ l_{12} e_{12} + l_{34} e_{34} } $ by Proposition \ref{prop:classone_MB} and the equations
(\ref{eqn:e12_vs_e34}) and (\ref{eqn:e34_vs_e12}).
\item Determine  $ \hat{\varphi}_{ l_{12} e_{12} + e_{14}+ l_{34} e_{34} } $ and $ \hat{\varphi}_{ l_{12} e_{12} + e_{23} + l_{34} e_{34} } $ 
by the equations (\ref{eqn:e12_vs_e14}) and (\ref{eqn:e34_vs_e23}).
\item Determine  $ \hat{\varphi}_{ l_{12} e_{12} +l_{13} e_{13} + l_{24} e_{24} +  l_{34} e_{34} } $, 
$ \hat{\varphi}_{ l_{12} e_{12} +l_{13} e_{13} + e_{14}+ l_{24} e_{24} +  l_{34} e_{34} } $ and 
$ \hat{\varphi}_{ l_{12} e_{12} +l_{13} e_{13} + e_{23} + l_{24} e_{24} + l_{34} e_{34} } $ by 
the equations  (\ref{PDE:DS12}) and (\ref{PDE:DS34}).
\item Determine  $ \hat{\varphi}_l $ by the relations (\ref{eqn:vp_gen_relation_4}) and (\ref{eqn:vp_gen_relation_5}).
\end{itemize}
\end{itemize}


\medskip 

Retain the notation in Propositions \ref{prop:PDE_reduction_classone} and \ref{lem:dim_PDE_kappa_2neq1}.
We define a homomorphism 
\begin{align}  \label{eqn:inj_hom_Sol}
 {\rm Hom}_K( V_{(\kappa_1,\kappa_2,\delta_3)}, {\rm Wh}(\Pi_{\sigma}, \psi_1)) \ni \varphi \mapsto
 \begin{cases}
   \hat{\varphi}_{(\kappa_1-\kappa_2)e_4+ \kappa_2 e_{12} } \in {\rm Sol}(r)  & \text{if $c \neq 1 $}, \\
   \hat{\varphi}_{(\kappa_1-\kappa_2)e_4+ \kappa_2 e_{34} } \in {\rm Sol}(r)  & \text{if $c = 1 $} \\
 \end{cases}
\end{align}
of $ \bC $-vector spaces by (\ref{eqn:hat_varphi}).
Because we can determine $ \hat{\vp}_l $ from 
$ \hat{\varphi}_{(\kappa_1-\kappa_2)e_4+ \kappa_2 e_{12} } $ ($c \neq 1$) or 
$ \hat{\varphi}_{(\kappa_1-\kappa_2)e_4+ \kappa_2 e_{34} } $ ($c \neq -1$),  
the map (\ref{eqn:inj_hom_Sol}) is injective.
Since $ \dim_{\bC} {\rm Hom}_K (V_{(\kappa_1,\kappa_2,\delta_3)}, H(\sigma)_K) = 1 $ and $ \Pi_{\sigma} $ is irreducible,
we have 
$$
 \dim_{\bC} {\rm Hom}_K(V_{(\kappa_1,\kappa_2,\delta_3)}, {\rm Wh}(\Pi_{\sigma}, \psi_1) ) = 
  \dim_{\bC} \mathcal{I}_{\Pi_{\sigma}, \psi_1} = 24 \ge \dim_{\bC}  {\rm Sol}(r).
$$
Here the last inequality follows from Proposition \ref{lem:dim_PDE_kappa_2neq1}. 
Then we know the map (\ref{eqn:inj_hom_Sol}) is isomorphism, and hence
the map
\begin{align} 
\label{eqn:inj_sol}
 {\rm Hom}_K (V_{(\kappa_1,\kappa_2,\delta_3)}, {\rm Wh}(\Pi_{\sigma}, \psi_1)^{\rm mg}) 
 \ni \varphi \mapsto 
 \begin{cases}
   \hat{\varphi}_{(\kappa_1-\kappa_2)e_4+ \kappa_2 e_{12} } \in {\rm Sol}^{\rm mg} (r)  & \text{if $c \neq 1 $}, \\
   \hat{\varphi}_{(\kappa_1-\kappa_2)e_4+ \kappa_2 e_{34} } \in {\rm Sol}^{\rm mg} (r)  & \text{if $c = 1 $} \\
 \end{cases} 
\end{align}
induced from (\ref{eqn:inj_hom_Sol}) is also isomorphism.
This argument means that our system in Proposition \ref{prop:PDE2} together with relations of generators given in 
Lemma \ref{lem:rel_ul} characterize Whittaker functions. 
More precisely we have the following:

\begin{thm}  \label{thm:whittaker_isom}
Assume (\ref{assumption}).
Let $ {\rm Sol}(\Pi_{\sigma}, \psi_1) $ be the space smooth solution $\{ \hat{\vp}_l \mid l \in  S_{(\kappa_1,\kappa_2,\delta_3)} \} $
of the system in Proposition \ref{prop:PDE2} satisfying the relations 
(\ref{eqn:vp_gen_relation_2}), (\ref{eqn:vp_gen_relation_3}), (\ref{eqn:vp_gen_relation_4}) and (\ref{eqn:vp_gen_relation_5}).
For $ \varphi \in {\rm Hom}_K (V_{(\kappa_1,\kappa_2,\delta_3)}, {\rm Wh}(\Pi_{\sigma}, \psi_1))  $, 
the map
$$ \varphi \mapsto \{  (\sqrt{-1})^{l_1-l_3+l_{13}-l_{24}} (-1)^{l_2+l_{14}+l_{23}} y_1^{-3/2} y_2^{-2} y_3^{-3/2+\kappa_2} y_4^{-\gamma_1}  \varphi(u_l)(y) \mid l \in S_{(\kappa_1,\kappa_2, \delta_3)} \} $$
gives the isomorphisms of $ \bC $-vector spaces
\begin{align*}
 {\rm Hom}_K (V_{(\kappa_1,\kappa_2,\delta_3)}, {\rm Wh}(\Pi_{\sigma}, \psi_1))  & \cong {\rm Sol}(\Pi_{\sigma}, \psi_1),
\\
 {\rm Hom}_K (V_{(\kappa_1,\kappa_2,\delta_3)}, {\rm Wh}(\Pi_{\sigma}, \psi_1)^{\rm mg}) 
&  \cong {\rm Sol}(\Pi_{\sigma}, \psi_1)^{\rm mg},
\end{align*}
where 
$ y = {\rm diag}(y_1y_2y_3y_4, y_2y_3y_4, y_3y_4, y_4) \in A$ and 
${\rm Sol}(\Pi_{\sigma}, \psi_1)^{\rm mg} $ is the subspace of $ {\rm Sol}(\Pi_{\sigma}, \psi_1)$  
consisting of moderate growth functions.
\end{thm}

\subsection{The space $ {\rm Sol}^{\rm mg}(\mu) $}

We give Mellin-Barnes integral representations of moderate growth functions in ${\rm Sol}(\mu) $. 
As in \cite{Ishii_Stade_001}, 
we first express our solutions in terms of Mellin-Barnes kernel of the class one principal series Whittaker functions on 
$ {\rm GL}(3,\bR)$. 
For $ t_1,t_2 \in \bC $ 
and $\mu = (\mu_1,\mu_2,\mu_3,\mu_4) \in \bC^4 $, set 
\begin{align*}
 V'(t_1,t_2)
& =     
  \frac{ \GR(t_1+\mu_2)\GR(t_1+\mu_3)\GR(t_1+\mu_4) 
  \GR(t_2+\mu_3+\mu_4) \GR(t_2+\mu_2+\mu_4) \GR(t_2+\mu_2+\mu_3)}
 {\GR(t_1+t_2+\mu_2+\mu_3+\mu_4)}.
\end{align*}
For $ s_1, s_2, s_3, \alpha_1,\alpha_2 \in \bC $ satisfying 
$ {\rm Re}(s_1+\mu_1) > 0 $,
$ {\rm Re}(s_1+\mu_i + \alpha_1) > 0 $ $(2 \le i \le 4) $, 
$ {\rm Re}(s_2+\mu_1+\mu_i+\alpha_1) > 0$ $(2 \le i \le 4) $,
$ {\rm Re}(s_2+\mu_i +\mu_j+\alpha_2) > 0 $ $(2 \le i \le j \le 4)$,
$ {\rm Re}(s_3+\mu_1+\mu_i+\mu_j + \alpha_2) > 0 $ $(2 \le i \le j \le 4)$, 
$ {\rm Re}(s_3+\mu_2+\mu_3+\mu_4) > 0 $,
and a polynomial $ P = P(s_1,s_2,s_3,t_1,t_2) $ on $ \bC^5 $, we define 
\begin{align} \label{eqn:def_V}
\begin{split}
& V(s_1,s_2,s_3; \alpha_1,\alpha_2;  P) 
\\
& =  \frac{1}{(4\pi \sqrt{-1})^2} \int_{t_2} \int_{t_1}
     \GR(s_1+\mu_1) \GR(s_1-t_1+\alpha_1)   \GR(s_2-t_1+\mu_1+\alpha_1) \GR(s_2-t_2+\alpha_2)
\\
& \times 
\GR(s_3-t_2+\mu_1+\alpha_2) \GR(s_3+\mu_2+\mu_3+\mu_4) 
   P (s_1,s_2,s_3,t_1,t_2) V'(t_1,t_2) \,dt_1 dt_2. 
\end{split}
\end{align}
Here the path $ \int_{t_i} $ $(i=1,2)$ is the vertical line from $ {\rm Re}(t_i) - \sqrt{-1} \infty $ to
$ {\rm Re}(t_i) + \sqrt{-1} \infty $ with the real part
\begin{gather*}
 \max \{  -{\rm Re}(\mu_2), -{\rm Re}(\mu_3), -{\rm Re}(\mu_4) \} 
 < {\rm Re}(t_1) < \min \{ {\rm Re}(s_1+\alpha_1), {\rm Re}(s_2+\mu_1+\alpha_1) \},
\\
 \max \{  -{\rm Re}(\mu_3+\mu_4), -{\rm Re}(\mu_2+\mu_4), -{\rm Re}(\mu_2+\mu_3) \} 
 < {\rm Re}(t_2) < \min \{ {\rm Re}(s_2+\alpha_2), {\rm Re}(s_3+\mu_1+\alpha_2) \}.
\end{gather*}

\medskip 

Let 
\begin{align*}
  Q_2' (t_1,t_2) &= -t_1^2-t_2^2+t_1t_2 - (\mu_2+\mu_3+\mu_4) t_2- (\mu_2\mu_3+\mu_3 \mu_4+\mu_4\mu_2),
\\
 Q_3'(t_1,t_2) & = -t_1 (t_2+\mu_2+\mu_3+\mu_4) (t_1-t_2) - \mu_2 \mu_3\mu_4.
\end{align*}
Then, from the formula (\ref{eqn:gamma_FE}),  we have 
\begin{align} 
\label{eqn:GL3_PDE_compati_C2}
&  Q_2'(t_1,t_2) V'(t_1,t_2) + (2\pi)^2 \{ V'(t_1+2,t_2) + V'(t_1,t_2+2) \} = 0,
\\
\label{eqn:GL3_PDE_compati_C3}
& Q_3'(t_1,t_2) V'(t_1,t_2)  + (2\pi)^2 \{ (t_2+\mu_2+\mu_3+\mu_4) V'(t_1+2,t_2) -t_1 V'(t_1,t_2+2) \} = 0. 
\end{align}
These relations are nothing but the compatibility with the system of partial differential equations 
satisfied by class one Whittaker functions on ${\rm GL}(3,\bR)$.

To express our solution in terms ${\rm GL}(2,\bR) $-Whittaker functions as in \cite{Stade_002}, 
for $ m \in \bZ $, we put
\begin{align} \label{eqn:def_Um}
\begin{split}
 U_m(s_1,s_2,s_3;\mu) 
& = \GR(s_1+\mu_1) \GR(s_1+\mu_2) \GR(s_2 + \mu_1+\mu_2- m) \GR(s_2+\mu_3+\mu_4+m) \\
& \quad \times 
  \GR(s_3+\mu_1+\mu_3+\mu_4) \GR(s_3+\mu_2+\mu_3+\mu_4) \\
& \quad \times \frac{1}{4\pi \sqrt{-1}} \int_q  
  \frac{\GR(s_1-q+m) \GR(s_2-q+\mu_1) \GR(s_2-q+\mu_2)}
  { \GR(s_1+s_2-q+\mu_1+\mu_2) \GR(s_2+s_3-q+\mu_1+\mu_2+\mu_3+\mu_4)} 
\\
& \quad  \times \GR(s_3-q+\mu_1+\mu_2-m) \GR(q+\mu_3) \GR(q+\mu_4) \,dq.
\end{split}
\end{align} 
See \S \ref{sec:Barnes} for the assumption for $ s_i, \mu_i, m $ and the path of integration.

\begin{prop} $($\cite[Theorem 12]{Ishii_Stade_001}, \cite[Theorem 3.1]{Stade_002}$)$
\label{prop:classone_MB}
Retain the notation in Proposition \ref{prop:PDE_reduction_classone}.  \smallskip 

\noindent
(i) Let 
\begin{align*}
 f^{\rm mg}(y_1,y_2,y_3) 
& = \frac{1}{(4\pi \sqrt{-1})^3} \int_{s_3}\int_{s_2}\int_{s_1} V_0(s_1,s_2, s_3) 
 \, y_1^{-s_1} y_2^{-s_2} y_3^{-s_3}  \, ds_1 ds_2 ds_3
\end{align*}
with $ V_0(s_1,s_2, s_3) = V(s_1,s_2,s_3; 0,0; 1) $. 
Here the path $ \int_{s_i} $ $(i=1,2,3)$ is the vertical line from $ {\rm Re}(s_i) - \sqrt{-1} \infty $ to
$ {\rm Re}(s_i) + \sqrt{-1} \infty $ with the sufficiently large real part.
More precisely, 
$ {\rm Re}(s_1+\mu_i) > 0 $ $(1\le i \le 4) $, 
$ {\rm Re}(s_2+\mu_i+\mu_j)>0 $ $(1 \le i<j \le 4) $
and 
$ {\rm Re}(s_3+\mu_i+\mu_j+\mu_k)> 0 $ $(1 \le i<j <k \le 4) $.
Then $ f_{}^{\rm mg} $ is a moderate growth function in the space $ {\rm Sol}(\mu) $.  \smallskip 

\noindent
(ii)
We have 
\begin{align} \label{eqn:V0=U0}
 V(s_1,s_2,s_3; 0,0,1)  = U_0(s_1,s_2,s_3;\mu_1,\mu_2,\mu_3,\mu_4) 
\end{align}
and
\begin{align} \label{eqn:U0_Weyl}
 U_0(s_1,s_2,s_3; \mu_{w(1)}, \mu_{w(2)},  \mu_{w(3)},  \mu_{w(4)})
 = U_0(s_1,s_2,s_3; \mu_1,\mu_2,\mu_3,\mu_4) \qquad ( w \in {\mathfrak S}_4 ). 
\end{align} 
\end{prop}

\begin{proof} 
Since the Stirling's formula implies that $ f^{\rm mg} $ is of moderate growth,  
it is enough to show $ {f}^{\rm mg} \in {\rm Sol}(\mu) $. 
We set 
\begin{align*}
 Q_2(s_1,s_2,s_3) 
& = -s_1^2-s_2^2-s_3^2 + s_1 s_2  + s_2 s_3  -  \s_1(\mu)  s_3  - \s_2(\mu),
\\
 Q_3(s_1,s_2,s_3)
& = (-s_2)(-s_1+s_3)(-s_1+s_2-s_3)   -\s_1(\mu)(s_1^2+s_2^2-s_1s_2-s_2s_3)-\s_3(\mu),
\\
 Q_4(s_1,s_2,s_3) 
& = (-s_1)(s_1-s_2)(s_2-s_3)(s_3+\s_1(\mu)) -\s_4(\mu)
\end{align*}
and define functions $ X_iV_0 $ ($i=2,3,4$) on $ \bC^3 $ by 
$$  (X_i V_0)(s_1,s_2,s_3)  =  Q_i(s_1,s_2,s_3) V_0(s_1,s_2,s_3)  + (Z_i V_0)(s_1,s_2,s_3), $$
where $ Z_i V_0 =  (Z_i V_0)(s_1,s_2,s_3) $ is given by 
\begin{align*}
Z_2 V_0 & =  (2\pi)^2 \{ V_0( s_1+2,s_2,s_3 ) + V_0( s_1,s_2+2,s_3 ) + V_0( s_1,s_2,s_3+2 ) \},
\\
Z_3 V_0 & = (2\pi)^2 \{ (s_2+ \s_1(\mu))  V_{0}(s_1+2,s_2,s_3) +  (-s_1+s_3+\s_1(\mu)) V_{0}(s_1,s_2+2,s_3) 
   + (-s_2) V_{0}(s_1,s_2,s_3+2) \},   
\\
Z_4 V_0 & = 
    (2\pi)^2 \{ (s_2-s_3)(s_3+\s_1(\mu)) V_0(s_1+2,s_2,s_3) + (-s_1)(s_3 + \s_1(\mu))  V_0(s_1,s_2+2,s_3) \\
  & \quad   +  (-s_1)(s_1-s_2) V_0(s_1,s_2,s_3+2) \} 
  + (2\pi)^4 V_0(s_1+2,s_2,s_3+2).
\end{align*}
Then our task is to confirm that the Mellin-Barnes kernel of $ {D}_i(\mu) f^{\rm mg} $ vanishes, that is, 
$$
 (X_i V_0)(s_1,s_2,s_3) = 0 \quad (i=2, 3,4).
$$
In view of (\ref{eqn:gamma_FE}) we know
\begin{align*}
 (X_i V_0)(s_1,s_2,s_3) & = V(s_1,s_2,s_3; 0,0;  P_i),
\end{align*}
where $ P_i = P_i(s_1,s_2,s_3,t_1,t_2) $ is given by
\begin{align*}
 P_2 
&= Q_2(s_1,s_2,s_3) +  (s_1+\mu_1)(s_1-t_1) + (s_2-t_1+\mu_1)(s_2-t_2) 
      + (s_3-t_2+\mu_1)(s_3+\mu_2+\mu_3+\mu_4),
\\
 P_3 & =  Q_3(s_1,s_2,s_3) +
 (s_2+\s_1(\mu)) (s_1+\mu_1)(s_1-t_1)   + (-s_1+s_3+\s_1(\mu)) (s_2-t_1+\mu_1)(s_2-t_2) \\
& \quad + (-s_2) (s_3-t_2+\mu_1)(s_3+\mu_2+\mu_3+\mu_4),
\\
 P_4
& =  Q_4(s_1,s_2,s_3) + (s_2-s_3)(s_3+\s_1(\mu)) (s_1+\mu_1)(s_1-t_1)  + (-s_1)(s_3 + \s_1(\mu))(s_2-t_1+\mu_1)(s_2-t_2) \\
& \quad + (-s_1)(s_1-s_2)  (s_3-t_2+\mu_1)(s_3+\mu_2+\mu_3+\mu_4)  + (s_1+\mu_1)(s_1-t_1) (s_3-t_2+\mu_1)(s_3+\mu_2+\mu_3+\mu_4).
\end{align*}
Define $ P_i' = P_i'(s_1,s_2,s_3,t_1,t_2) $ $ (i=2,3) $ by 
\begin{align*}
 P_2'
& = Q_2'(t_1,t_2) + (s_1-t_1) (s_2-t_1+\mu_1) + (s_2-t_2)(s_3-t_2+\mu_1), 
\\
 P_3'   
& = Q_3'(t_1,t_2) 
 + (t_2+\mu_2+\mu_3+\mu_4)(s_1-t_1 ) (s_2-t_1+\mu_1)  -t_1 (s_2-t_2)(s_3-t_2+\mu_1).
\end{align*}
Because of the identity $ P_2  =P_2' $ and (\ref{eqn:gamma_FE}), we have 
\begin{align*}
\begin{split}
 V(s_1,s_2,s_3;  0,0; P_2) & = V(s_1,s_2,s_3; 0,0; Q_2')  + (2\pi)^2 V(s_1,s_2,s_3;  2,0 ; 1)  + (2\pi)^2 V(s_1,s_2,s_3;  0,2; 1). 
\end{split}
\end{align*}
If we substitute $ t_1 \to t_1 +2 $ and $ t_2 \to t_2+2 $ in the second and the third terms, respectively, 
then (\ref{eqn:GL3_PDE_compati_C2}) implies that $ X_2 V_0  = 0 $.
Similarly, the identities 
$ P_3 = P_3' + \mu_1 P_2' $ and $ P_4  = \mu_1 P_3' $ 
together with (\ref{eqn:GL3_PDE_compati_C2}) and (\ref{eqn:GL3_PDE_compati_C3})
lead $ X_3 V_0 = X_4 V_0 = 0$ as desired.

\medskip 

Let us show (ii).
By Lemma \ref{lem:Barnes1st}, $ V'(t_1,t_2) $ can be written as  
\begin{align*}
&  \GR(t_1+\mu_2 ) \GR(t_2+\mu_3+\mu_4 )  \cdot
 \frac{1}{4\pi \sqrt{-1}} \int_q \GR(t_1-q) \GR(t_2-q+\mu_2) \GR(q+\mu_3) \GR(q+\mu_4) \,dq.
\end{align*}
We substitute the above expression for $ V'(t_1,t_2) $ into (\ref{eqn:def_V}) and 
use Lemma \ref{lem:Barnes1st} for the inegrations $ \int_{t_1}$ and $ \int_{t_2} $ to get (\ref{eqn:V0=U0}).
Since $ V_0(s_1,s_2,s_3) $ is invariant under the change of $ \mu_2,  \mu_3 $ and $ \mu_4 $, and 
$ U_0(s_1,s_2,s_3; \mu_1,\mu_2,\mu_3,\mu_4) = U_0(s_1,s_2,s_3; \mu_2,\mu_1,\mu_3,\mu_4) $ 
from the definition, 
we know that  
(\ref{eqn:V0=U0}) implies (\ref{eqn:U0_Weyl}).
\end{proof}

\subsection{Auxiliary Lemma}
\label{subsec:aux_lem_U}

The following lemma will be used to determine $ \hat{\vp}_l $ 
by using the differential equations in Lemma \ref{lem:DS1_l-e1_and_l-e4} and Proposition \ref{prop:PDE_reduction_classone} (ii).

\begin{lem} \label{lem:VvsU}
Retain the notation in Propositions \ref{prop:classone_MB}.
For $ \mu = (\mu_1,\mu_2,\mu_3, \mu_4) \in \bC^4 $, let
\begin{align*}
\begin{split}
 {U}'(s_1,s_2,s_3; \mu) 
&= \GR(s_1+\mu_1) \GR(s_1+\mu_2) \GR (s_2+\mu_1+\mu_2-1)  \GR(s_2+\mu_3+\mu_4+1) 
\\
& \quad \times  \GR(s_3+\mu_1+\mu_3+\mu_4+1) \GR(s_3+\mu_2+\mu_3+\mu_4+1)
\\
& \quad  \times \frac{1}{4\pi \sqrt{-1}} \int_q 
  \frac{ \GR(s_1-q) \GR(s_2-q+\mu_1-1) \GR(s_2-q+\mu_2-1) }{ \GR(s_1+s_2-q+\mu_1+\mu_2-1) \GR(s_2+s_3-q+\mu_1+\mu_2+\mu_3+\mu_4) } 
\\
& \quad  \times \GR(s_3-q+\mu_1+\mu_2-1) \GR(q+\mu_3) \GR(q+\mu_4) \,dq.
\end{split}
\end{align*}
(i) For $ (c_1,c_2) \in \bC^2$ and $ i =1,3 $, we set 
\begin{align*}
 E_i(c_1,c_2) U_0(s_1,s_2,s_3; \mu) & =
 \begin{cases} (s_1+c_1)(s_1+c_2) U_0(s_1,s_2,s_3; \mu) - U_0(s_1+2,s_2,s_3; \mu) & \text{if $ i = 1$}, \\
   (s_3+c_1)(s_3+c_2) U_0(s_1,s_2,s_3; \mu) - U_0(s_1,s_2,s_3+2; \mu) & \text{if $ i = 3.$} \end{cases}
\end{align*}
Let $ \mu' = (\mu_3+1,\mu_4+1, \mu_1,\mu_2) $ and $ \mu'' = (\mu_1+1,\mu_2+1,\mu_3,\mu_4) $.
We have
\begin{align}  
\label{eqn:Mellin_12vs23} 
\begin{split}
&  E_1(\mu_1,\mu_2) U_0(s_1,s_2,s_3; \mu')
  = U'(s_1+1,s_2+1,s_3,\mu''), 
\end{split}
\\
\label{eqn:Mellin_34vs23} 
\begin{split}
&  E_3(\mu_1+\mu_3+\mu_4+1,\mu_2+\mu_3+\mu_4+1) U_0(s_1,s_2,s_3; \mu'')
 = U'(s_1,s_2+1,s_3+1,\mu''), 
\end{split}
\\
\label{eqn:Mellin_12vs14}
\begin{split} 
&E_3(\mu_1+\mu_2+\mu_3+1, \mu_1+\mu_2+\mu_4+1) U_0(s_1,s_2,s_3; \mu')
 = U'(s_1,s_2+1,s_3+1; \mu'),  
\end{split}
\\
\label{eqn:Mellin_34vs14}
\begin{split} 
& E_1(\mu_3, \mu_4) U_0(s_1,s_2,s_3; \mu'')
 = U'(s_1+1,s_2+1,s_3, \mu'), 
\end{split}
\end{align}
\begin{align} \label{eqn:Mellin_12vs34}
\begin{split}
& (2\pi)^{-2} (-s_1+s_2-s_3-\mu_3-\mu_4-2)(-s_1+s_3+\mu_1+\mu_2) E_1(\mu_1,\mu_2) U_0(s_1,s_2,s_3; \mu')
\\
&+ (2\pi)^{-2} (-s_1+s_2-s_3-\mu_3-\mu_4-2)(s_2+\mu_1+\mu_2)  U_0(s_1+2,s_2,s_3; \mu' )
\\
& + E_1(\mu_1,\mu_2) U_0(s_1,s_2,s_3+2; \mu' )
  + E_3(\mu_1+\mu_2+\mu_3+1, \mu_1+\mu_2+\mu_4+1) U_0(s_1+2,s_2,s_3; \mu' )
\\
& = U_0(s_1+1,s_2+2,s_3+1; \mu''), 
\end{split}
\end{align}
\begin{align} \label{eqn:Mellin_34vs12}
\begin{split}
& (2\pi)^{-2} (-s_1+s_2-s_3-\mu_1-\mu_2-2)(-s_1+s_3+\mu_3+\mu_4) E_1(\mu_3,\mu_4) U_0(s_1,s_2,s_3; \mu'')
\\
&+ (2\pi)^{-2} (-s_1+s_2-s_3-\mu_1-\mu_2-2)(s_2+\mu_3+\mu_4)  U_0(s_1+2,s_2,s_3; \mu'' )
\\
& + E_1(\mu_3,\mu_4) U_0(s_1,s_2,s_3+2; \mu' )
  + E_3(\mu_1+\mu_3+\mu_4+1, \mu_2+\mu_3+\mu_4+1) U_0(s_1+2,s_2,s_3; \mu'' )
\\
& = U_0(s_1+1,s_2+2,s_3+1; \mu').
\end{split}
\end{align}
(ii) We have 
\begin{align} \label{eqn:Um3}
\begin{split}
&  (2\pi)^{-3} (s_1-s_2-\mu_1+m-1) (s_2-s_3-\mu_1+m-1) (s_3+\mu_2+\mu_3+\mu_4) U_{m-1}(s_1,s_2,s_3;\mu) \\
& +(2\pi)^{-1} (s_3+\mu_2+\mu_3+\mu_4) U_{m-1}(s_1,s_2+2,s_3;\mu) \\
&   +(2\pi )^{-1} (s_1-s_2-\mu_1+m-1) U_{m-1}(s_1,s_2,s_3+2;\mu) \\
& = U_m(s_1+1, s_2+1, s_3+1; \mu_1-1, \mu_2+1, \mu_3, \mu_4). 
\end{split}
\end{align}
\end{lem}

\begin{proof}
Let us show the identities in (i). 
For $ \alpha_i,\beta_j,\gamma_k \in \bC $ $(1 \le i,j \le 6, 1 \le k \le 2)$, we put 
\begin{align*}
& U(s_1,s_2,s_3; \mu_1,\mu_2,\mu_3,\mu_4; \alpha_1, \alpha_2, \alpha_3, \alpha_4, \alpha_5, \alpha_6;  \beta_1,\beta_2,
 \beta_3,\beta_4,\beta_5,\beta_6; \gamma_1, \gamma_2) \\
& = \GR(s_1+\mu_1+\alpha_1) \GR(s_1+\mu_2+\alpha_2) 
    \GR(s_2+\mu_1+\mu_2+\alpha_3) \\
& \quad \times  \GR(s_2+\mu_3+\mu_4+\alpha_4)
 \GR(s_3+\mu_1+\mu_3+\mu_4+ \alpha_5) \GR(s_3+\mu_2+\mu_3+\mu_4+\alpha_6) 
\\
& \quad \times \frac{1}{4 \pi \sqrt{-1}} \int_q
  \frac{\GR(s_1-q+\beta_1) \GR(s_2-q+\mu_1+\beta_2) \GR(s_2-q+\mu_2+\beta_3) }{ 
  \GR(s_1+s_2-q+\mu_1+\mu_2+\gamma_1)}
\\
& \quad \times \frac{ \GR(s_3-q+\mu_1+\mu_2+\beta_4) \GR(q+\mu_3+\beta_5) \GR(q+\mu_4+\beta_6) }{ 
   \GR(s_2+s_3+\mu_1+\mu_2+\mu_3+\mu_4+ \gamma_2) }\,dq.
\end{align*}
Let us show (\ref{eqn:Mellin_12vs23}). 
In view of (\ref{eqn:U0_Weyl}) we know
 $ E_1(\mu_1,\mu_2) U_0(s_1,s_2,s_3; \mu') = 
 E_1(\mu_1,\mu_2) U_0(s_1,s_2,s_3; \mu_1,\mu_2,\mu_3+1,\mu_4+1)  $. 
Then (\ref{eqn:gamma_FE}) implies that  
\begin{align*}
 E_1(\mu_1,\mu_2) U_0(s_1,s_2,s_3; \mu')
& = U(s_1,s_2,s_3; \mu_1,\mu_2,\mu_3,\mu_4; 2,2,0,2,2,2; 0,0,0,0,1,1; 0,2) \\
& \quad - U(s_1,s_2,s_3; \mu_1,\mu_2,\mu_3,\mu_4; 2,2,0,2,2,2; 2,0,0,0,1,1; 2,2) \\
& = U(s_1,s_2,s_3; \mu_1,\mu_2,\mu_3,\mu_4; 2,2,2,2,2,2; 0,0,0,0,1,1; 2,2) \\
& = U(s_1,s_2,s_3; \mu_1,\mu_2,\mu_3,\mu_4; 2,2,2,2,2,2; 1.1,1,1,0,0; 3,3) 
\end{align*}
as desired. 
We can similarly show
(\ref{eqn:Mellin_34vs23}), (\ref{eqn:Mellin_12vs14}) and (\ref{eqn:Mellin_34vs14}).

\medskip 

To show (\ref{eqn:Mellin_12vs34}), we rewrite (\ref{eqn:Mellin_12vs14}).
By Lemma \ref{lem:Barnes1st}, we know that  $U'(s_1,s_2+1,s_3+1; \mu' )$ can be written as  
\begin{align*}
&  \frac{1}{(4\pi \sqrt{-1})^3} \int_{t_2} \int_{t_1} \int_q 
   \GR(s_1+\mu_3+1)  \GR(s_1-t_1) \GR(s_2-t_1+\mu_3+1) \GR(s_2-t_2) 
\\
& \times  
  \GR(s_3-t_2+\mu_3+1) \GR(s_3+\mu_1+\mu_2+\mu_4+3)  \GR(t_1+\mu_4+1)
   \GR(t_2+\mu_1+\mu_2+2)  \\
& \times 
   \GR(t_1-q)  \GR(t_2-q+\mu_4+1) \GR(q+\mu_1) \GR(q+\mu_2) \,dq dt_1dt_2.
\end{align*}
In view of Lemma \ref{lem:Barnes1st} and (\ref{eqn:gamma_FE}), we have
\begin{align*}
&  \GR(t_1+\mu_4+1) \GR(t_2+\mu_1+\mu_2+2)  \cdot  \frac{1}{4\pi \sqrt{-1}} \int_q
    \GR(t_1-q)  \GR(t_2-q+\mu_4+1) \GR(q+\mu_1) \GR(q+\mu_2) \,dq 
\\
& = \GR(t_1+\mu_1) \GR(t_2+\mu_2+\mu_4+1) \\
&  \quad \times \sum_{i=0,1} 
   \frac{1}{4\pi \sqrt{-1}} \int_q
    \GR(t_1-q)  \GR(t_2-q+\mu_1+2-2i) \GR(q+\mu_2+2i) \GR(q+\mu_4+1) \,dq.
\end{align*}
Then we find 
\begin{align*}
 U'(s_1,s_2+1,s_3+1; \mu' )
= 
 \sum_{i=0}^1
  U(s_1,s_2,s_3; \mu_1,\mu_3,\mu_2,\mu_4; 0,1,1,1,3,2; 0,2-2i,1, 3-2i,  2i,1; 1, 4-2i)
\end{align*}
from Lemma \ref{lem:Barnes1st}. 
Further, we use 
\begin{align*}
& \frac{\GR(s_1+\mu_3+1)\GR(s_2+\mu_1+\mu_3+1)\GR(s_1-q)\GR(s_2-q+\mu_1+2-2i)}{ \GR(s_1+s_2-q+\mu_1+\mu_3+1)}
\\
& = \sum_{j_1=0}^{1-i} \frac{1}{4\pi \sqrt{-1}} \int_{t_1} 
  \GR(s_1-t_1) \GR(s_2-t_1+\mu_1+2-2i-2j_1) \GR(t_1-q+2j_1) \GR(t_1+\mu_3+1) \,dt_1
\end{align*}
and 
\begin{align*}
& \frac{ \GR(s_2+\mu_2+\mu_4+1) \GR(s_3+\mu_1+\mu_2+\mu_4+3) \GR(s_2-q+\mu_3+1) \GR(s_3-q+\mu_1+\mu_3+3-2i)}{ 
 \GR(s_2+s_3+\mu_1+\mu_2+\mu_3+\mu_4+4-2i) } 
\\
& = \sum_{j_2=0}^{i} 
  \frac{1}{4\pi \sqrt{-1}} \int_{t_2} \GR(s_2-t_2) \GR(s_3-t_2+\mu_1+2-2j_2 ) \GR(t_2-q+\mu_3+1) \GR(t_2+\mu_2+\mu_4+1+2j_2) \,dt_2.
\end{align*}
Then $ U'(s_1,s_2+1,s_3+1; \mu') $ becomes
\begin{align*}
& \GR(s_1+\mu_1) \GR(s_3+\mu_2+\mu_3+\mu_4+2) 
    \sum_{i=0}^1 \sum_{j_1=0}^{1-i} \sum_{j_2=0}^{i}
\\
& \times  \frac{1}{(4\pi \sqrt{-1})^3} \int_{t_2} \int_{t_1} \int_q
   \GR(s_1-t_1) \GR(s_2-t_1+\mu_1+2-2i-2j_1) \GR(t_1-q+2j_1) \GR(t_1+\mu_3+1)
\\
& \times 
 \GR(s_2-t_2) \GR(s_3-t_2+\mu_1+2-2j_2 ) \GR(t_2-q+\mu_3+1) \GR(t_2+\mu_2+\mu_4+1+2j_2) 
\\
& \times \GR(q+\mu_2+2i) \GR(q+\mu_4+1) \,dq dt_1dt_2.
\end{align*}
We can collect three terms $ (i,j_1,j_2)= (0,0,0), (0,1,0), (1,0,0) $. Then the above is expressed as
\begin{align*}
& \GR(s_1+\mu_1) \GR(s_3+\mu_2+\mu_3+\mu_4+2)    
\cdot  \frac{1}{(4\pi \sqrt{-1})^3} \int_{t_2} \int_{t_1} \int_q
   \GR(s_1-t_1) \GR(s_2-t_1+\mu_1)  \\
& \times \GR(s_2-t_2) \GR(t_1-q) \GR(t_2-q+\mu_3+1) \GR(t_1+\mu_3+1) \GR(q+\mu_4+1)
\\
& \times 
 \{ (2\pi)^{-1} (s_2+\mu_1+\mu_2)  \GR(s_3-t_2+\mu_1+2) \GR(t_2+\mu_2+\mu_4+1) \GR(q+\mu_2) 
\\
& \quad + \GR(s_3-t_2+\mu_1) \GR(t_2+\mu_2+\mu_4+3) \GR(q+\mu_2+2) \}
  \,dq dt_1dt_2.
\end{align*}
Since 
\begin{align*}
& \GR(t_1+\mu_3+1) \GR(t_2+\mu_2+\mu_4+1+2i) \\
& \times \frac{1}{4\pi \sqrt{-1}} \int_q \GR(t_1-q)\GR(t_2-q+\mu_3+1) \GR(q+\mu_2+2i) \GR(q+\mu_4+1) \,dq
\\
& = \GR(t_1+\mu_2+2i) \GR(t_2+\mu_3+\mu_4+2) 
\\
& \times \frac{1}{4\pi \sqrt{-1}} \int_q \GR(t_1-q+1) \GR(t_2-q+\mu_2+1+2i) \GR(q+\mu_3) \GR(q+\mu_4) \,dq
\end{align*}
for $ i=0,1 $, after the integration over $ t_1 $ and $ t_2$, we reach the expression 
\begin{align*}
 U'(s_1,s_2+1,s_3+1; \mu') 
& = \sum_{i=0}^1 U(s_1,s_2,s_3; \mu_1,\mu_2,\mu_3,\mu_4; 0,2i, 2,2,4-2i,2; 1,1,1+2i,3,0,0; 1+2i, 5).
\end{align*}
Hence the left hand side of (\ref{eqn:Mellin_12vs34}) becomes
\begin{align*} 
& (2\pi)^{-1} (-s_1+s_2-s_3-\mu_3-\mu_4-2) \\
& \times \{ (2\pi)^{-1} (-s_1+s_3+\mu_1+\mu_2) U(s_1,s_2,s_3; \mu_1,\mu_2,\mu_3,\mu_4; 2,2,2,2,2,2; 1.1,1,1,0,0; 3,3)
\\
& \quad + (2\pi)^{-1} (s_2+\mu_1+\mu_2) U(s_1,s_2,s_3; \mu_1,\mu_2,\mu_3,\mu_4; 2,2,0,2,2,2; 3.1,1,1,0,0; 3,3) \}
\\
& + U(s_1,s_2,s_3; \mu_1,\mu_2,\mu_3,\mu_4; 2,2,2,2,4,4; 1.1,1,3,0,0; 3,5) \\
& + \sum_{i=0}^1 U(s_1,s_2,s_3; \mu_1,\mu_2,\mu_3,\mu_4; 0,2i, 2,2,4-2i,2; 1,1,1+2i,3,0,0; 1+2i, 5).
\end{align*}
The terms in the bracket $ \{ \cdots \} $ can be written as
\begin{align*}
& (2\pi)^{-1} (-s_1+s_3+\mu_1+\mu_2) U(s_1,s_2,s_3; \mu_1,\mu_2,\mu_3,\mu_4; 2,2,2,2,2,2; 1.1,1,1,0,0; 3,3)
\\
&  + (2\pi)^{-1} (s_2+\mu_1+\mu_2) U(s_1,s_2,s_3; \mu_1,\mu_2,\mu_3,\mu_4; 2,2,0,2,2,2; 3.1,1,1,0,0; 3,3)
\\
& = (2\pi)^{-1} \{ (-s_1+s_3+\mu_1+\mu_2) + (s_1-q) \} U(s_1,s_2,s_3; \mu_1,\mu_2,\mu_3,\mu_4; 2,2,2,2,2,2; 1.1,1,1,0,0; 3,3)
\\
& = U(s_1,s_2,s_3; \mu_1,\mu_2,\mu_3,\mu_4; 2,2,2,2,2,2; 1.1,1,3,0,0; 3,3).
\end{align*}
Then the identity 
\begin{align*}
& (-s_1+s_2-s_3-\mu_3-\mu_4-2) (s_1+s_2-q+\mu_1+\mu_2+3)(s_2+s_3-q+\mu_1+\mu_2+\mu_3+\mu_4+3)
\\
& + (s_3+\mu_1+\mu_3+\mu_4+2)(s_3+\mu_2+\mu_3+\mu_4+2)(s_1+s_2-q+\mu_1+\mu_2+3) 
\\
& + (s_3+\mu_1+\mu_3+\mu_4+2)(s_1-q+1)(s_2-q+\mu_1+1) 
\\
& + (s_1+\mu_2+2)(s_2+s_3-q+\mu_1+\mu_2+\mu_3+\mu_4+3)(s_1-q+1) 
\\
& = (s_2-q+\mu_1+1)(s_2-q+\mu_2+1)(s_2+\mu_1+\mu_2+2)  
\end{align*}
implies that the left hand side of (\ref{eqn:Mellin_12vs34}) becomes
\begin{align*}
  U(s_1,s_2,s_3; \mu_1,\mu_2,\mu_3,\mu_4; 2,2,4,2,2,2; 1.3,3,3,0,0; 5,5) 
  = U_0(s_1+1,s_2+2,s_3+1; \mu'') 
\end{align*}
as desired.
The identity (\ref{eqn:Mellin_34vs12}) follows from (\ref{eqn:Mellin_12vs34}) and (\ref{eqn:U0_Weyl}).

\medskip

We show (\ref{eqn:Um3}).
By the definition of $ U_m(s_1,s_2,s_3;\mu) $,  
the left hand side of (\ref{eqn:Um3}) can be written as 
\begin{align*}
& \GR(s_1+\mu_1) \GR(s_1+\mu_2) \GR(s_2 + \mu_1+\mu_2- m+1) \GR(s_2+\mu_3+\mu_4+m-1) \\
& \times \GR(s_3+\mu_1+\mu_3+\mu_4) \GR(s_3+\mu_2+\mu_3+\mu_4+2) \\
& \times \frac{1}{4\pi \sqrt{-1}} \int_q  
  \frac{\GR(s_1-q+m-1) \GR(s_2-q+\mu_1) \GR(s_2-q+\mu_2)}
  { \GR(s_1+s_2-q+\mu_1+\mu_2+2) } 
\\
& \times 
   \frac{ \GR(s_3-q+\mu_1+\mu_2-m+1) \GR(q+\mu_3) \GR(q+\mu_4) }
   { \GR(s_2+s_3-q+\mu_1+\mu_2+\mu_3+\mu_4+2) } \cdot (2\pi)^{-4}
\\
& \times 
 \{ (s_1-s_2-\mu_1+m-1)(s_2-s_3-\mu_1 + m-1) 
   (s_1+s_2-q+\mu_1+\mu_2) (s_2+s_3-q+\mu_1+\mu_2+\mu_3+\mu_4) 
\\
& \quad +  (s_2+\mu_1+\mu_2-m+1)(s_2+\mu_3+\mu_4+m-1) (s_2-q+\mu_1)(s_2-q+\mu_2)
\\
& \quad +(s_1-s_2-\mu_1+m-1) (s_3+\mu_1+\mu_3+\mu_4) 
  (s_3-q+\mu_1+\mu_2-m+1)(s_1+s_2-q+\mu_1+\mu_2) \} \,dq.
\end{align*}
Since the term in the bracket $\{ \cdots \} $ is factorized as 
$$ 
(s_1+\mu_2) (s_1-q+m-1)(s_2-q+\mu_2)(s_2+\mu_3+\mu_4+m-1), 
$$
the left hand side of (\ref{eqn:Um3})  becomes $ U(s_1,s_2,s_3; \mu_1,\mu_2, \mu_3,\mu_4; 0,2,-m+1,m+1,0,2; 1,0,0,-m+1,0,0; 2,2). $
Thus we are done.
\end{proof}


\subsection{Explicit formulas of $ \hat{\vp}_{l_1 e_1+ l_4 e_4+ l_{12} e_{12} + l_{34} e_{34}} $}
\label{subsec:EF_step1}

In this subsection we determine $ \hat{\vp}_{l_1 e_1 + l_4e_4+ l_{12}e_{12} + l_{34} e_{34}} $.

\begin{lem} \label{lem:EF_step1}
Retain the notation in Propositions \ref{prop:PDE_reduction_classone} and \ref{prop:classone_MB}.
For $ l = (\kappa_1-\kappa_2) e_4 + l_{12} e_{12} + l_{34} e_{34} \in S_{(\kappa_1,\kappa_2,\delta_3)}$, we set
\begin{align*}
 \hat{\varphi}_{l}^{\rm mg} (y_1,y_2,y_3)
 & = \frac{1}{(4\pi \sqrt{-1})^3} \int_{s_3} \int_{s_2} \int_{s_1} \widehat{V}_l(s_1,s_2,s_3) 
    \, y_1^{-s_1} y_2^{-s_2} y_3^{-s_3} \,ds_1ds_2ds_3,
\end{align*}
where
\begin{align*}
 \widehat{V}_l (s_1,s_2,s_3) = U_0(s_1,s_2,s_3; r).
\end{align*}
Then $ \hat{\varphi}_{l}^{\rm mg} $ is a moderate growth function satisfying 
(\ref{eqn:Capelli_reduction}), (\ref{eqn:e12_vs_e34}) and (\ref{eqn:e34_vs_e12}).
\end{lem}

\begin{proof}
From (\ref{eqn:Capelli_reduction}) and Proposition \ref{prop:classone_MB}, 
we have 
\begin{align*}
& \widehat{V}_{ (\kappa_1-\kappa_2) e_4 + l_{12} e_{12} + l_{34} e_{34}} (s_1,s_2,s_3) \\
& = 
\begin{cases}
 C_{l_{12}} U_0(s_1,s_2,s_3; \nu_1+l_{34} ,\nu_2+l_{34} ,\nu_3+l_{12},\nu_4+l_{12}) & \text{case 1-(ii)}, \\
 C_{l_{12}} U_0(s_1,s_2,s_3; \nu_1+\tfrac{\kappa_1-1}{2}, \nu_1+\tfrac{\kappa_1+1}{2}, \nu_2+l_{34}, \nu_3+l_{12}) 
   & \text{case 2-(ii)}, \\
 C_{l_{12}} U_0(s_1,s_2,s_3; \nu_1+\tfrac{\kappa_1-1}{2},  \nu_1+\tfrac{\kappa_1+1}{2},
   \nu_2+\tfrac{\kappa_2-1}{2},  \nu_2+\tfrac{\kappa_2+1}{2})
  & \text{case 3}
\end{cases}
\end{align*}
for some constants $ C_{l_{12}} $ $(0 \le l_{12} \le \kappa_2) $ depending on $l_{12}$. 
In view of the identities (\ref{eqn:Mellin_12vs34}) and (\ref{eqn:Mellin_34vs12}), 
the equations (\ref{eqn:e12_vs_e34}) and (\ref{eqn:e34_vs_e12}) imply that 
$ C_{0} = \cdots = C_{\kappa_2} $.  
\end{proof}

\begin{lem} \label{lem:EF_step2}
For $ l = l_1 e_1+ l_4 e_4 + l_{12} e_{12} + l_{34} e_{34} \in S_{(\kappa_1,\kappa_2,\delta_3)}$
let $ \hat{\vp}_{l} $ be the function determined by the function 
$ \hat{\vp}_{ l' }= 
  \hat{\vp}_{l' }^{\rm mg} $ 
$( l' =  (\kappa_1-\kappa_2)e_4+   l_{12} e_{12} + l_{34} e_{34} \in S_{(\kappa_1,\kappa_2,\delta_3)} )$
in Lemmas \ref{lem:EF_step1}
and  by the equation (\ref{eqn:DS1:l-e_4+e_1}) when $ \kappa_1 > \kappa_2 $.
Then we have  
\begin{align*}
 \hat{\vp}_l (y_1,y_2,y_3)  
& = \frac{1}{(4\pi \sqrt{-1})^3} \int_{s_3} \int_{s_2} \int_{s_1} \widehat{V}_l(s_1,s_2,s_3) 
    \, y_1^{-s_1} y_2^{-s_2} y_3^{-s_3} \,ds_1ds_2ds_3,
\end{align*}
where
\begin{align*}
& \widehat{V}_{l}(s_1,s_2,s_3)  = \begin{cases}
   U_{l_1}(s_1,s_2,s_3; \nu_1+l_4+l_{34}, \nu_2+l_1+l_{34}, \nu_3+l_{12}, \nu_4+l_{12}) & \text{cases 1-(i),(ii),(iii)}, \\
   U_{l_1}(s_1,s_2,s_3; \nu_4+l_4, \nu_2+l_1, \nu_3, \nu_1) & \text{case 1-(iv)}, \\
   U_{l_1}(s_1,s_2,s_3; \nu_1+\tfrac{\kappa_1-1}{2}, \nu_1+\tfrac{\kappa_1+1}{2}, \nu_2+l_{34}, \nu_3+l_{12}) & \text{case 2}, \\
  U_{l_1}(s_1,s_2,s_3; \nu_1+\tfrac{\kappa_1-1}{2}, \nu_1+\tfrac{\kappa_1+1}{2}, 
         \nu_2+\tfrac{\kappa_2-1}{2}, \nu_2+\tfrac{\kappa_2+1}{2}) & \text{case 3}.
 \end{cases}
\end{align*}
\end{lem}

\begin{proof}
Since the case $ l_1 = 0$ is done in Lemma \ref{lem:EF_step1}.
let $ l_1 \ge  1$.
The equation (\ref{eqn:DS1:l-e_4+e_1}) tells us that  
\begin{align*}
 \widehat{V}_{l}(s_1+1,s_2+1,s_3+1) 
& = (2\pi)^{-3} (s_1 -s_2-\nu_1'-\tfrac{\kappa_1+1}{2} +l_1-1) (s_2-s_3-\nu_1' - \tfrac{\kappa_1 +1}{2} +l_1-1)  
\\
& \quad \times 
      (s_3+\gamma_1-\nu_1'+\tfrac{\kappa_1 -1}{2}+\kappa_2 ) \widehat{V}_{l-e_1+e_4} (s_1,s_2,s_3)
\\
& \quad 
 + (2\pi)^{-1} (s_3+\gamma_1-\nu_1'+\tfrac{\kappa_1-1}{2} + \kappa_2)  \widehat{V}_{l-e_1+e_4}(s_1,s_2+2,s_3)
\\
& \quad 
  + (2\pi)^{-1}  (s_1-s_2-\nu_1'-\tfrac{\kappa_1 + 1}{2}  +l_1-1)  \widehat{V}_{l-e_1+e_4}(s_1,s_2,s_3+2).
\end{align*}
If we notice $ U_{m}(s_1,s_2,s_3; \mu_1, \mu_2, \mu_3,\mu_4) =  U_{m}(s_1,s_2,s_3; \mu_2, \mu_1, \mu_3,\mu_4) $, 
then our claim is a consequence of (\ref{eqn:Um3}) and induction on $l_1$.
\end{proof}

\begin{rem} \label{rem:EFstep1}
Using the duplication formula (\ref{eqn:gamma_dup}), we can rewrite 
our formulas in Lemma \ref{lem:EF_step2} for cases 2 and 3 as follows.
\begin{itemize}
\item Case 2: 
For $ l = l_1 e_1 + l_4 e_4 + l_{12} e_{12} + l_{34} e_{34} \in S_{(\kappa_1, \kappa_2, \delta_3)} $, we have
\begin{align} \label{eqn:step1_case2_dup}  
\begin{split} 
 \widehat{V}_{l} (s_1,s_2,s_3) 
& =  \GC(s_1+\nu_1+\tfrac{\kappa_1-1}{2}) \GR(s_2+2\nu_1+\kappa_1-l_1) 
\\
& \quad \times  \GR(s_2+\nu_2+\nu_3+\kappa_2+l_1) \GC(s_3+\nu_1+\nu_2+\nu_3+\tfrac{\kappa_1-1}{2}+\kappa_2) 
\\
& \quad \times \frac{1}{4\pi \sqrt{-1}} \int_q 
  \frac{ \GR(s_1-q+l_1) \GC(s_2-q+\nu_1+\tfrac{\kappa_1-1}{2}) 
 }{ \GR(s_1+s_2-q+2\nu_1+\kappa_1) }
\\
& \quad \times 
  \frac{  \GR(s_3-q+2\nu_1+\kappa_1-l_1) \GR(q+\nu_2+l_{34}) \GR(q+\nu_3+l_{12} )}{ \GR(s_2+s_3-q+2\nu_1+\nu_2+\nu_3+\kappa_1+\kappa_2) } \,dq.
\end{split}
\end{align}
\item Case 3:
For $ l = l_1 e_1 + l_4 e_4 + l_{12} e_{12} + l_{34} e_{34} \in S_{(\kappa_1, \kappa_2, \delta_3)} $, we have
\begin{align} \label{eqn:step1_case3_dup} 
\begin{split}
& \widehat{V}_l(s_1,s_2,s_3) \\
& = \GC(s_1+\nu_1+\tfrac{\kappa_1-1}{2}) \GR(s_2+2\nu_1+\kappa_1-l_1)  
 \GR(s_2+2\nu_2+\kappa_2+l_1) \GC(s_3+\nu_1+2\nu_2+\tfrac{\kappa_1-1}{2}+\kappa_2) 
\\
& \quad \times \frac{1}{4\pi \sqrt{-1}} \int_q 
  \frac{ \GR(s_1-q+l_1) \GC(s_2-q+\nu_1+\tfrac{\kappa_1-1}{2}) 
    \GR(s_3-q+2\nu_1+\kappa_1-l_1) \GC(q+\nu_2+\tfrac{\kappa_2-1}{2})}{ \GR(s_1+s_2-q+2\nu_1+\kappa_1)
   \GR(s_2+s_3-q+2\nu_1+2\nu_2+\kappa_1+\kappa_2) }\,dq.
\end{split}
\end{align}
\end{itemize}
\end{rem}

\subsection{Explicit formulas of $ \hat{\vp}_{l} $}
\label{subsec:EF_step3}

In this subsection we determine $ \hat{\vp}_l $ for all $ l \in S_{(\kappa_1,\kappa_2,\delta_3)} $. 
If the moderate growth solution of the system in Proposition \ref{prop:PDE2} is given by 
\begin{align*}
 \hat{\varphi}_l(y_1,y_2,y_3) 
 = \frac{1}{(4\pi \sqrt{-1})^3} \int_{s_3} \int_{s_2} \int_{s_1} \widehat{V}_l(s_1,s_2,s_3) y_1^{-s_1} y_2^{-s_2} y_3^{-s_3} 
 \,ds_1ds_2ds_3
\end{align*}
then the equations (\ref{PDE:DS1}), (\ref{PDE:DS4}), (\ref{PDE:DS12}), (\ref{PDE:DS34}) imply that 
\begin{align*}
 \widehat{V}_{l}(s_1,s_2,s_3) &=
 \begin{cases}
   (2\pi)^{-1} (s_1+\nu_1'+\tfrac{\kappa_1-1}{2}-1) \widehat{V}_{l-e_2+e_1}(s_1-1,s_2,s_3) & \text{if $ l_2 \ge 1 $}, \\
   (2\pi)^{-1} (s_3 +\gamma_1-\nu_1'+\tfrac{\kappa_1-1}{2}+\kappa_2-1) 
          \widehat{V}_{l-e_3+e_4}(s_1,s_2,s_3-1) & \text{if $ l_3 \ge 1 $}, \\
  (2\pi)^{-1} (s_2+ \nu_1+\nu_2+\tfrac{\kappa_1+\kappa_2}{2}-2) 
          \widehat{V}_{l-e_{13}+e_{12}}(s_1,s_2-1,s_3) & \text{if $l_{13} \ge 1 $}, \\
  (2\pi)^{-1} (s_2 + \gamma_1 - \nu_1- \nu_2+\tfrac{\kappa_1+\kappa_2}{2}-2) 
          \widehat{V}_{l-e_{24}+e_{34}}(s_1,s_2-1,s_3)  & \text{if $l_{24} \ge 1 $}
\end{cases}
\end{align*}
for $ l  = (l_1,l_2,l_3,l_4, l_{12},l_{13},l_{14},l_{23},l_{24},l_{34} )  \in S_{(\kappa_1,\kappa_2,\delta_3)} $. 
Therefore we know
\begin{align} \label{eqn:hatV_DS}
\begin{split}
  \widehat{V}_{l}(s_1,s_2,s_3) 
& = (2\pi)^{-l_2-l_3-l_{13}-l_{24}} (s_1+\nu_1'+\tfrac{\kappa_1-1}{2}-l_2)_{l_2} 
     (s_3+\gamma_1-\nu_1'+\tfrac{\kappa_1-1}{2}+\kappa_2-l_3)_{l_3}
\\
& \quad \times
   (s_2+\nu_1+\nu_2+\tfrac{\kappa_1+\kappa_2}{2}-1-l_{13}-l_{24})_{l_{13}} 
   (s_2 +\gamma_1 -\nu_1 - \nu_2+\tfrac{\kappa_1+\kappa_2}{2}-1-l_{24})_{l_{24}} 
\\
& \quad \times \widehat{V}_{(l_1+l_2,0,0,l_3+l_4, l_{12}+l_{13},0, l_{14},l_{23},0, l_{24}+l_{34})}(s_1-l_2,s_2-l_{13}-l_{24} ,s_3-l_3).
\end{split}
\end{align}
Thus we are done in the case of $ \kappa_2 = 0 $ (cases 1-(ii), (iv), 2-(i)). 
In the following we consider the remaing cases.

\bigskip

$\bullet $ {\bf Case 1-(iii) $(\kappa_1  = \kappa_2 =1) $}.

\begin{lem} \label{lem:EF_kappa1=kappa2=1} (Case1-(iii))
For $  l = (0,0,0,0, l_{12}, l_{13}, l_{14},l_{23}, l_{24}, l_{34} )  \in S_{(1,1, 0)} $
let $ \hat{\vp}_l $ be the functions determined by the functions $ \hat{\vp}_{l' } $
$(l' =  l_{12} e_{12} + l_{34} e_{34} \in S_{(1,1,0)}) $
in Lemma \ref{lem:EF_step2}, and 
by the equations (\ref{PDE:DS12}), (\ref{PDE:DS34}), 
(\ref{eqn:e12_vs_e14}), (\ref{eqn:e12_vs_e23}), 
(\ref{eqn:e34_vs_e14}) and (\ref{eqn:e34_vs_e23}).
Then we have
\begin{align*}
 \hat{\vp}_l (y_1,y_2,y_3)  
& = \frac{1}{(4\pi \sqrt{-1})^3} \int_{s_3} \int_{s_2} \int_{s_1} \widehat{V}_l(s_1,s_2,s_3) 
    \, y_1^{-s_1} y_2^{-s_2} y_3^{-s_3} \,ds_1ds_2ds_3,
\end{align*}
where
\begin{align*}
& \widehat{V}_{ (0,0,0,0,l_{12},l_{13},0,l_{23}, 0,0) }(s_1,s_2,s_3) \\
& = \GR(s_1+\nu_1+l_{23}) \GR(s_1+\nu_2+l_{23}) 
\GR(s_2 + \nu_1+\nu_2+l_{13}+l_{23}) \GR(s_2+\nu_3+\nu_4+l_{12}+1) \\
& \quad \times 
   \GR(s_3+\nu_1+\nu_3+\nu_4 +2) \GR(s_3+\nu_2+\nu_3+\nu_4 +2) \\
& \quad \times \frac{1}{4\pi \sqrt{-1}} \int_q  
  \frac{\GR(s_1-q+l_{12}+l_{13}) \GR(s_2-q+\nu_1+l_{12}) }
  { \GR(s_1+s_2-q+\nu_1+\nu_2+l_{12}+l_{23}) } \\
& \quad \times 
  \frac{ \GR(s_2-q+\nu_2+l_{12}) \GR(s_3-q+\nu_1+\nu_2+1) \GR(q+\nu_3) \GR(q+\nu_4) }
  {\GR(s_2+s_3-q+\nu_1+\nu_2+\nu_3+\nu_4+2+l_{12})} \,dq, 
\\
& \widehat{V}_{(0,0,0,0, 0,0,l_{14}, 0, l_{24}, l_{34})}(s_1,s_2,s_3) \\
& = \GR(s_1+\nu_3+l_{14}) \GR(s_1+\nu_4+l_{14}) 
 \GR(s_2 + \nu_3+\nu_4+l_{14}+l_{24}) \GR(s_2+\nu_1+\nu_2+l_{34}+1) \\
& \quad \times 
   \GR(s_3+\nu_1+\nu_2+\nu_3 +2) \GR(s_3+\nu_1+\nu_2+\nu_4 +2) \\
& \quad \times \frac{1}{4\pi \sqrt{-1}} \int_q  
  \frac{\GR(s_1-q+l_{24}+l_{34}) \GR(s_2-q+\nu_3+l_{34}) }
  { \GR(s_1+s_2-q+\nu_3+\nu_4+l_{14}+l_{34}) } 
\\
& \quad \times 
  \frac{ \GR(s_2-q+\nu_4+l_{34}) \GR(s_3-q+\nu_3+\nu_4+1) \GR(q+\nu_1) \GR(q+\nu_2) }
  {\GR(s_2+s_3-q+\nu_1+\nu_2+\nu_3+\nu_4+2+l_{34})}  \,dq.
\end{align*}
\end{lem}

\begin{proof}
As for $ \widehat{V}_{e_{14}} $ and $ \widehat{V}_{e_{23}} $, 
compatibilities with the equations (\ref{eqn:e12_vs_e14}), (\ref{eqn:e12_vs_e23}), 
(\ref{eqn:e34_vs_e14}) and (\ref{eqn:e34_vs_e23}) follow from 
the identities (\ref{eqn:Mellin_12vs14}), (\ref{eqn:Mellin_12vs23}), (\ref{eqn:Mellin_34vs14})
and (\ref{eqn:Mellin_34vs23}), respectively. 
(\ref{eqn:hatV_DS}) implies our formulas for $ \widehat{V}_{e_{13}} $ and $ \widehat{V}_{e_{24}} $.
\end{proof}

\medskip

$\bullet $ {\bf  Case 2-(ii) $(\kappa_1  > \kappa_2 =1)$}.

\begin{lem} \label{lem:EF_kappa1>kappa2=1} (Case 2-(ii))
For $  l = (l_1, l_2, l_3, l_4, l_{12}, l_{13}, l_{14},l_{23}, l_{24}, l_{34} )  \in S_{(\kappa_1,1, 0)} $
let $ \hat{\vp}_l $ be the functions determined by the functions $ \hat{\vp}_{l' } $
$(l' = l_1 e_1 + l_4 e_4 + l_{12} e_{12} + l_{34} e_{34} \in S_{(\kappa_1,1,0)}) $
in Lemma \ref{lem:EF_step2}, and 
by the equations (\ref{PDE:DS1}), (\ref{PDE:DS4}), (\ref{PDE:DS12}), (\ref{PDE:DS34}), 
(\ref{eqn:vp_gen_relation_2}) and (\ref{eqn:vp_gen_relation_3}).
Then we have
\begin{align*}
 \hat{\vp}_l (y_1,y_2,y_3)  
& = \frac{1}{(4\pi \sqrt{-1})^3} \int_{s_3} \int_{s_2} \int_{s_1} \widehat{V}_l(s_1,s_2,s_3) 
    \, y_1^{-s_1} y_2^{-s_2} y_3^{-s_3} \,ds_1ds_2ds_3,
\end{align*}
where
\begin{align} \label{eqn:EF_kappa1>kappa2=1}
\begin{split}
& \widehat{V}_{ l }(s_1,s_2,s_3) \\
& =  (2\pi)^{-l_{13}-l_{24}} 
   (s_2+\nu_1+\nu_2+\tfrac{\kappa_1-1}{2}-l_{13}-l_{24} )_{l_{13}} 
   (s_2 +\nu_1+\nu_3 + \tfrac{\kappa_1-1}{2}-l_{24})_{l_{24}}
\\
& \quad \times 
\Gamma_{\bC}(s_1+\nu_1+\tfrac{\kappa_1-1}{2}) 
      \Gamma_{\bR}(s_2+2\nu_1+ l_3+l_4 + l_{12} + l_{34} ) \\
& \quad \times 
      \Gamma_{\bR}(s_2+\nu_2+\nu_3+ l_1+l_2+ l_{12} + l_{34}  ) 
      \Gamma_{\bC}(s_3+\nu_1+\nu_2+\nu_3+\tfrac{\kappa_1+1}{2}) 
\\
& \quad  \times \sum_{ i_{14} =0}^{ l_{14} } \sum_{i_{23} = 0}^{  l_{23} }  
 \frac{1}{4\pi \sqrt{-1}} \int_q 
  \frac{\Gamma_{\bR}(s_1-q+l_1 + i_{14} + i_{23} )\Gamma_{\bC}(s_2-q+\nu_1+\tfrac{\kappa_1-1}{2} -l_{13}-l_{24} ) 
   }{\Gamma_{\bR}(s_1+s_2-q+2\nu_1+l_1+l_3 + l_4  + l_{12} + l_{34} + i_{14} + i_{23} )   } 
\\
& \quad \times \frac{ \Gamma_{\bR}(s_3-q+2\nu_1+ 1+ l_4 + l_{14} + l_{23} - i_{14}- i_{23}  )  }{
  \Gamma_{\bR}(s_2+s_3-q+2\nu_1+\nu_2+\nu_3+1+ l_1+l_2+l_4 + l_{12} +l_{14} + l_{23} + l_{34}  -i_{14} -i_{23} )}\\
& \quad  \times 
   \Gamma_{\bR}(q+\nu_2 + l_{23} + l_{24} + l_{34} +i_{14} -i_{23} )  \Gamma_{\bR}(q+\nu_3 + l_{12} + l_{13}+ l_{14}-i_{14} +i_{23} ) \,dq. 
\end{split}
\end{align}
\end{lem}

\begin{proof}
By (\ref{eqn:step1_case2_dup}) and (\ref{eqn:hatV_DS}) we know that (\ref{eqn:EF_kappa1>kappa2=1}) is 
true for $ l = (l_1, l_2, l_3, l_4, l_{12}, l_{13},0,0,l_{24},l_{34}) $.
Let us determine $ \widehat{V}_{l_1e_1+l_2 e_2+l_3 e_3+l_4 e_4+ e_{14} }(s_1,s_2,s_3)  $. 
Since $ \kappa_1 > \kappa_2  $ we know $ l_i \ge 1 $ for some $ 1 \le i \le 4 $.
Assume $ l_4 \ge 1 $. 
The equation (\ref{eqn:vp_gen_relation_2}) with $ i=1 $ implies 
$ \hat{\vp}_{l_1e_1+l_2e_2+l_3e_3+l_4e_4+e_{14}} 
  =  - \hat{\vp}_{l_1e_1+(l_2+1)e_2+l_3e_3+(l_4-1)e_4+e_{12}} 
  + \hat{\vp}_{l_1e_1+ l_2 e_2+(l_3+1)e_3+(l_4-1)e_4+e_{13}}. 
$
Then we have  
\begin{align*}
& \widehat{V}_{l_1e_1+l_2e_2+l_3e_3+l_4e_4+e_{14}}(s_1,s_2,s_3) \\
& =  -\widehat{V}_{l_1e_1+(l_2+1)e_2+l_3e_3+(l_4-1)e_4+e_{12}}(s_1,s_2,s_3) 
  + \widehat{V}_{l_1e_1+ l_2 e_2+(l_3+1)e_3+(l_4-1)e_4+e_{13}}(s_1,s_2,s_3)
\\
& = -
      \Gamma_{\bC}(s_1+\nu_1+\tfrac{\kappa_1-1}{2}) 
      \Gamma_{\bR}(s_2+2\nu_1+l_3+l_4) 
      \Gamma_{\bR}(s_2+\nu_2+\nu_3+l_1+l_2+2) 
      \Gamma_{\bC}(s_3+\nu_1+\nu_2+\nu_3+\tfrac{\kappa_1+1}{2}) 
\\
& \quad \times \frac{1}{4\pi \sqrt{-1}} \int_q 
  \frac{\Gamma_{\bR}(s_1-q+l_1)\Gamma_{\bC}(s_2-q+\nu_1+\tfrac{\kappa_1-1}{2})  \Gamma_{\bR}(s_3-q+2\nu_1+\kappa_1 +l_4) 
   \Gamma_{\bR}(q+\nu_2 ) \Gamma_{\bR}(q+\nu_3+1 )
   }{ \Gamma_{\bR}(s_1+s_2-q+2\nu_1+l_1+l_3+l_4)   
  \Gamma_{\bR}(s_2+s_3-q+2\nu_1+\nu_2+\nu_3+l_1+l_2+l_4+2)} \,dq
\\
& \quad + (2\pi)^{-1} 
   (s_2+\nu_1+\nu_2+\tfrac{\kappa_1-3}{2})
      \Gamma_{\bC}(s_1+\nu_1+\tfrac{\kappa_1-1}{2}) 
      \Gamma_{\bR}(s_2+2\nu_1+l_3+l_4) \\
& \quad \times 
      \Gamma_{\bR}(s_2+\nu_2+\nu_3+l_1+l_2) 
      \Gamma_{\bC}(s_3+\nu_1+\nu_2+\nu_3+\tfrac{\kappa_1+1}{2}) 
\\
& \quad \times \frac{1}{4\pi \sqrt{-1}} \int_q 
  \frac{\Gamma_{\bR}(s_1-q+l_1)\Gamma_{\bC}(s_2-q+\nu_1+\tfrac{\kappa_1-3}{2}) \Gamma_{\bR}(s_3-q+2\nu_1+\kappa_1+l_4) 
  \Gamma_{\bR}(q+\nu_2 ) \Gamma_{\bR}(q+\nu_3+1 )
   }{\Gamma_{\bR}(s_1+s_2-q+2\nu_1+l_1+l_3+l_4)  \Gamma_{\bR}(s_2+s_3-q+2\nu_1+\nu_2+\nu_3+l_1+l_2+l_4)  } \,dq.
\end{align*}
In view of (\ref{eqn:gamma_FE}) and the identity
\begin{align*}
& -(s_2+\nu_2+\nu_3+l_1+l_2)(s_2-q+\nu_1+\tfrac{\kappa_1-3}{2}) \\
&  +(s_2+\nu_1+\nu_2+\tfrac{\kappa_1-3}{2}) (s_2+s_3-q+2\nu_1+\nu_2+\nu_3+l_1+l_2+l_4)
\\
& = (s_2-q+\nu_1+\tfrac{\kappa_1-3}{2})(s_3-q+2\nu_1+l_4) + (q+\nu_2)(s_2+s_3-q+2\nu_1+\nu_2+\nu_3+l_1+l_2+l_4), 
\end{align*}
we get 
\begin{align*}
& \widehat{V}_{l_1e_1+l_2e_2+l_3e_3+l_4e_4+e_{14}}(s_1,s_2,s_3) \\
& =
\Gamma_{\bC}(s_1+\nu_1+\tfrac{\kappa_1-1}{2}) 
      \Gamma_{\bR}(s_2+2\nu_1+l_3+l_4) 
      \Gamma_{\bR}(s_2+\nu_2+\nu_3+l_1+l_2) 
      \Gamma_{\bC}(s_3+\nu_1+\nu_2+\nu_3+\tfrac{\kappa_1+1}{2}) 
\\
& \quad \times \frac{1}{4\pi \sqrt{-1}} \int_q 
  \biggl\{ 
  \frac{\Gamma_{\bR}(s_1-q+l_1)\Gamma_{\bC}(s_2-q+\nu_1+\tfrac{\kappa_1-1}{2}) \Gamma_{\bR}(s_3-q+2\nu_1+l_4+2) 
   \Gamma_{\bR}(q+\nu_2 ) \Gamma_{\bR}(q+\nu_3+1 )
   }{\Gamma_{\bR}(s_1+s_2-q+2\nu_1+l_1+l_3+l_4)   \Gamma_{\bR}(s_2+s_3-q+2\nu_1+\nu_2+\nu_3+l_1+l_2+l_4+2) } 
\\
& \quad + 
  \frac{\Gamma_{\bR}(s_1-q+l_1)\Gamma_{\bC}(s_2-q+\nu_1+\tfrac{\kappa_1-3}{2}) \Gamma_{\bR}(s_3-q+2\nu_1+l_4) 
  \Gamma_{\bR}(q+\nu_2+2 ) \Gamma_{\bR}(q+\nu_3+1 )
   }{\Gamma_{\bR}(s_1+s_2-q+2\nu_1+l_1+l_3+l_4)  \Gamma_{\bR}(s_2+s_3-q+2\nu_1+\nu_2+\nu_3+l_1+l_2+l_4) }
 \biggr\} \,dq
\\
& = 
\Gamma_{\bC}(s_1+\nu_1+\tfrac{\kappa_1-1}{2}) 
      \Gamma_{\bR}(s_2+2\nu_1+l_3+l_4)
      \Gamma_{\bR}(s_2+\nu_2+\nu_3+l_1+l_2) 
      \Gamma_{\bC}(s_3+\nu_1+\nu_2+\nu_3+\tfrac{\kappa_1+1}{2}) 
\\
& \quad \times \sum_{i=0}^1  \frac{1}{4\pi \sqrt{-1}} \int_q 
  \frac{\Gamma_{\bR}(s_1-q+l_1+i)\Gamma_{\bC}(s_2-q+\nu_1+\tfrac{\kappa_1-1}{2}) 
   }{\Gamma_{\bR}(s_1+s_2-q+2\nu_1+l_1+l_3+l_4 + i )  }  
\\
& \quad \times \frac{ 
  \Gamma_{\bR}(s_3-q+2\nu_1+l_4+2-i ) 
   \Gamma_{\bR}(q+\nu_2+i ) \Gamma_{\bR}(q+\nu_3+1-i )
   }{ \Gamma_{\bR}(s_2+s_3-q+2\nu_1+\nu_2+\nu_3+l_1+l_2+l_4+2-i) } \,dq
\end{align*}
as desired.
In the same way we can determine $ \widehat{V}_{l_1e_1+l_2 e_2+l_3 e_3+l_4 e_4+ e_{23} }(s_1,s_2,s_3) $
from (\ref{eqn:vp_gen_relation_3}) with $ (i,j,k) =(2,3,4) $. 
We can similarly show that 
(\ref{eqn:EF_kappa1>kappa2=1}) is compatible with other relations in (\ref{eqn:vp_gen_relation_2})  
and (\ref{eqn:vp_gen_relation_3}). 
\end{proof}

\medskip

$\bullet $ {\bf Case 3 with $ \kappa_1  > \kappa_2  \ge 2 $}.

\begin{lem} \label{lem:EF_kappa1>kappa2>1} (Case 3) 
Assume $ \kappa_1 > \kappa_2 \ge 2 $.
For $  l = (l_1, l_2, l_3, l_4, l_{12}, l_{13}, l_{14},l_{23}, l_{24}, l_{34} )  \in S_{(\kappa_1,\kappa_2, 0)} $
let $ \hat{\vp}_l $ be the functions determined from the functions $ \hat{\vp}_{l' } $
$(l' = l_1 e_1 + l_4 e_4 + l_{12} e_{12} + l_{34} e_{34} \in S_{(\kappa_1,\kappa_2 ,0)}) $
in Lemma \ref{lem:EF_step2}, and 
by the equations (\ref{PDE:DS1}), (\ref{PDE:DS4}), (\ref{PDE:DS12}), (\ref{PDE:DS34}), 
(\ref{eqn:vp_gen_relation_2}) and (\ref{eqn:vp_gen_relation_3}).
Then we have
\begin{align*}
 \hat{\vp}_l (y_1,y_2,y_3)  
& = \frac{1}{(4\pi \sqrt{-1})^3} \int_{s_3} \int_{s_2} \int_{s_1} \widehat{V}_l(s_1,s_2,s_3) 
    \, y_1^{-s_1} y_2^{-s_2} y_3^{-s_3} \,ds_1ds_2ds_3,
\end{align*}
where
\begin{align} \label{eqn:EF_kappa1>kappa2>1}
\begin{split}
 \widehat{V}_{ l } (s_1,s_2,s_3) 
& = (2\pi)^{-l_{13}-l_{24}} 
   (s_2+\nu_1+\nu_2+\tfrac{\kappa_1+\kappa_2}{2}-1-l_{13}-l_{24} )_{l_{13} + l_{24} } \\
& \quad \times 
   \GC (s_1+\nu_1+\tfrac{\kappa_1-1}{2})
   \GR (s_2+2\nu_1+ l_3+l_4+l_{12}+l_{34} )  \\
& \quad \times 
   \GR (s_2+2\nu_2 +l_1+l_2 +l_{12}+l_{34} ) 
       \GC(s_3+\nu_1+2\nu_2+\tfrac{\kappa_1-1}{2} + \kappa_2 ) \\
& \quad \times   \sum_{i = 0}^{ l_{14}+l_{23} } \binom{l_{14}+l_{23}}{i}  
  \frac{1}{4\pi \sqrt{-1}} \int_q 
  \frac{\GR(s_1-q+l_1+i) \GC (s_2-q+\nu_1+\tfrac{\kappa_1-1}{2}-l_{13}-l_{24} )}{ 
  \GR(s_1+s_2-q+2\nu_1+ l_1+l_3+l_4 + l_{12} + l_{34} + i) } 
\\
& \quad  \times
 \frac{ \GR(s_3-q+2\nu_1+\kappa_2 + l_4  + l_{14}+l_{23}-i) \GC(q+\nu_2+\tfrac{\kappa_2-1}{2} )}
{ \GR(s_2+s_3-q+2\nu_1+2\nu_2+\kappa_2 +l_1+l_2+l_4 +l_{12}+  l_{14}+l_{23} +l_{34} -i )} \,dq.
\end{split}
\end{align}
\end{lem}

\begin{proof}
Our proof is similar to Lemma \ref{lem:EF_kappa1>kappa2=1}.
By (\ref{eqn:step1_case3_dup}) and (\ref{eqn:hatV_DS}) we know that the formula (\ref{eqn:EF_kappa1>kappa2>1}) is true 
for $ l = (l_1,l_2,l_3,l_4,l_{12}, l_{13}, 0,0, l_{24}, l_{34}) $. 
By using (\ref{eqn:vp_gen_relation_2}) and (\ref{eqn:vp_gen_relation_3}), 
we show (\ref{eqn:EF_kappa1>kappa2>1}) by induction on $ l_{14} + l_{23} $.
Let $ l_{14} + l_{23} \ge 1 $.
As in the proof of Lemma \ref{lem:EF_kappa1>kappa2=1}, if $ l_4 \ge 1 $, 
(\ref{eqn:vp_gen_relation_2}) with $ i=1 $ and (\ref{eqn:vp_gen_relation_3}) with $ (i,j,k) = (2,3,4) $
tell us 
\begin{align*}
 \widehat{V}_l(s_1,s_2, s_3) 
 = \begin{cases}
  -\widehat{V}_{l-e_4-e_{14} +e_2 +e_{12}} (s_1,s_2, s_3) + \widehat{V}_{l-e_4-e_{14} +e_3 +e_{13}} (s_1,s_2, s_3)  
  & \text{if $l_{14} \ge 1$}, \\
  -\widehat{V}_{l-e_4-e_{23} +e_2 +e_{34}} (s_1,s_2, s_3) + \widehat{V}_{l-e_4-e_{23} +e_3 +e_{24}} (s_1,s_2, s_3)  
  & \text{if $l_{23} \ge 1$}.
\end{cases}
\end{align*}
The hypothesis of induction implies that the above can be written as 
\begin{align} \label{eqn:kappa1>kappa2=1_step1}
\begin{split}
& -(2\pi)^{-l_{13}-l_{24}} 
   (s_2+\nu_1+\nu_2+\tfrac{\kappa_1+\kappa_2}{2}-l_{13}-l_{24}-1 )_{l_{13} + l_{24} } \\
& \quad \times 
  \GC(s_1+\nu_1+\tfrac{\kappa_1-1}{2}) \GR (s_2+2\nu_1+ l_3+l_4+l_{12}+l_{34} )  \\
& \quad \times \GR(s_2+2\nu_2 +l_1+l_2 +l_{12}+l_{34}+2 )  \GC (s_3+\nu_1+2\nu_2+\tfrac{\kappa_1-1}{2} + \kappa_2 ) \\
& \quad  \times  
  \sum_{i = 0}^{  l_{14}+l_{23}-1} \binom{l_{14}+l_{23}-1 }{i} 
  \frac{1}{4\pi \sqrt{-1}} \int_q 
  \frac{ \GR(s_1-q+l_1+i) \GC(s_2-q+\nu_1+\tfrac{\kappa_1-1}{2}-l_{13}-l_{24} )  }{  
       \GR(s_1+s_2-q+2\nu_1+ l_1+l_3+l_4 + l_{12} + l_{34} + i) } 
\\
& \quad \times
 \frac{ \GR(s_3-q+2\nu_1+\kappa_2 + l_4 + l_{14}+l_{23}-i-2) \GC(q+\nu_2+\tfrac{\kappa_2-1}{2} )}
{ \GR(s_2+s_3-q+2\nu_1+2\nu_2+\kappa_2 +l_1+l_2+l_4 +l_{12} + l_{14}+l_{23} + l_{34}-i )} \,dq
\\
& + (2\pi)^{-l_{13}-l_{24}-1} 
   (s_2+\nu_1+\nu_2+\tfrac{\kappa_1+\kappa_2}{2}-l_{13}-l_{24} -2)_{l_{13} + l_{24}+1 } \\
& \quad  \times 
  \GC(s_1+\nu_1+\tfrac{\kappa_1-1}{2}) \GR (s_2+2\nu_1+ l_3+l_4+l_{12}+l_{34} )  \\
& \quad  \times \GR(s_2+2\nu_2 +l_1+l_2 +l_{12}+l_{34} ) \GC (s_3+\nu_1+2\nu_2+\tfrac{\kappa_1-1}{2} + \kappa_2 ) \\
& \quad \times    \sum_{i = 0}^{  l_{14}+l_{23}-1} \binom{l_{14}+l_{23}-1 }{i}  
  \frac{1}{4\pi \sqrt{-1}} \int_q 
  \frac{ \GR (s_1-q+l_1+i) \GC (s_2-q+\nu_1+\tfrac{\kappa_1-1}{2}-l_{13}-l_{24} -1)
    }{  \GR(s_1+s_2-q+2\nu_1+ l_1+l_3+l_4 + l_{12} + l_{34} + i) } 
\\
& \quad  \times
 \frac{ \GR (s_3-q+2\nu_1+\kappa_2 + l_4  + l_{14}+l_{23}-i-2)   \GC(q+\nu_2+\tfrac{\kappa_2-1}{2} )}
{ \GR (s_2+s_3-q+2\nu_1+2\nu_2+\kappa_2 +l_1+l_2+l_4 +l_{12} + l_{14}+l_{23} + l_{34} -i-2)} \,dq.
\end{split}
\end{align}
In view of (\ref{eqn:gamma_FE}) and the identity 
\begin{align*}
& -(s_2+2\nu_2+ l_1+l_2+l_{12}+l_{34} ) (s_2-q+\nu_1+\tfrac{\kappa_1-1}{2}-l_{13}-l_{24} -1)
\\
&  + ( s_2+\nu_1+\nu_2+\tfrac{\kappa_1+\kappa_2}{2}-l_{13}-l_{24} -2) \\
& \times 
 (s_2+s_3-q+2\nu_1+2\nu_2+\kappa_2 +l_1+l_2+l_4 +l_{12}+ l_{14}+l_{23}+ l_{34} - i-2)
\\
& = ( s_2-q+\nu_1 + \tfrac{\kappa_1-1}{2} -l_{13}-l_{24}-1 ) 
     ( s_3 - q + 2\nu_1  + \kappa_2+l_4 + l_{14}+l_{23}-i-2) 
\\
& \quad +  ( q+\nu_2+\tfrac{\kappa_2-1}{2} ) 
 (s_2+s_3-q+2\nu_1+2\nu_2+\kappa_2 +l_1+l_2+l_4 +l_{12} + l_{14}+l_{23}+ l_{34}-i-2),
\end{align*}
we know that (\ref{eqn:kappa1>kappa2=1_step1}) becomes
\begin{align*}
& (2\pi)^{-l_{13}-l_{24}} 
   (s_2+\nu_1+\nu_2+\tfrac{\kappa_1+\kappa_2}{2}-l_{13}-l_{24}-1 )_{l_{13} + l_{24} } 
  \GC(s_1+\nu_1+\tfrac{\kappa_1-1}{2})
  \\
& \times 
  \GR (s_2+2\nu_1+ l_3+l_4+l_{12}+l_{34} ) 
  \GR(s_2+2\nu_2 +l_1+l_2 +l_{12}+l_{34} ) 
 \GC(s_3+\nu_1+2\nu_2+\tfrac{\kappa_1-1}{2} + \kappa_2 ) \\
& \times   \sum_{i=0}^{  l_{14}+l_{23}-1} \binom{l_{14}+l_{23}-1}{i} 
  \frac{1}{4\pi \sqrt{-1}} \int_q 
  \biggl\{ 
  \frac{\GR(s_1-q+l_1+i) \GC(s_2-q+\nu_1+\tfrac{\kappa_1-1}{2}-l_{13}-l_{24} )
    }{  \GR(s_1+s_2-q+2\nu_1+ l_1+l_3+l_4 + l_{12} + l_{34} + i) } 
\\
& \times
 \frac{ \GR(s_3-q+2\nu_1+\kappa_2 + l_4  +l_{14}+l_{23} -i)   \GC(q+\nu_2+\tfrac{\kappa_2-1}{2} )}
{ \GR(s_2+s_3-q+2\nu_1+2\nu_2+\kappa_2 +l_1+l_2+l_4 +l_{12} +l_{14}+l_{23} + l_{34} -i)} \\
& + 
  \frac{\GR(s_1-q+l_1+i) \GC(s_2-q+\nu_1+\tfrac{\kappa_1-1}{2}-l_{13}-l_{24} -1)
    }{  \GR(s_1+s_2-q+2\nu_1+ l_1+l_3+l_4 + l_{12} + l_{34} + i) } 
\\
& \times
 \frac{ \GR(s_3-q+2\nu_1+\kappa_2 + l_4 + l_{14}+l_{23}-i-2)   \GC(q+\nu_2+\tfrac{\kappa_2+1}{2})}
{ \GR(s_2+s_3-q+2\nu_1+2\nu_2+\kappa_2 +l_1+l_2+l_4 +l_{12} +l_{14}+l_{23} +  l_{34} -i-2)} \biggr\} \,dq.
\end{align*}
We substitute 
$ (q,i) \to (q,i) $ and 
$ (q,i) \to (q-1, i-1) $ in the first and the second terms in the bracket $\{ \cdots \} $, respectively. 
Then the formula $ \binom{l_{14}+l_{23}-1}{i} + \binom{l_{14}+l_{23}-1}{i-1} = \binom{l_{14}+l_{23}}{i} $ implies our assertions.
We can similarly show that 
(\ref{eqn:EF_kappa1>kappa2>1}) is compatible with other relations in (\ref{eqn:vp_gen_relation_2})  
and (\ref{eqn:vp_gen_relation_3}). 
\end{proof}

\medskip

$\bullet $ {\bf Case 3 with $ \kappa_1  = \kappa_2 \ge 2 $}.

\begin{lem}  \label{lem:kappa1=kappa_2_ge_2} (Case 3)
Assume $ \kappa_1 = \kappa_2 \ge 2 $.
For $  l = (0, 0, 0, 0, l_{12}, l_{13}, l_{14},l_{23}, l_{24}, l_{34} )  \in S_{(\kappa_2,\kappa_2, 0)} $
let $ \hat{\vp}_l $ be the functions determined from the functions $ \hat{\vp}_{l' } $
$(l' =  l_{12} e_{12} + l_{34} e_{34} \in S_{(\kappa_2,\kappa_2 ,0)}) $
in Lemma \ref{lem:EF_step2}, and 
by the equations (\ref{PDE:DS12}), (\ref{PDE:DS34}), (\ref{eqn:e12_vs_e14}), (\ref{eqn:e34_vs_e23}), 
(\ref{eqn:vp_gen_relation_4}) and (\ref{eqn:vp_gen_relation_5}).
Then we have
\begin{align*}
 \hat{\vp}_l (y_1,y_2,y_3)  
& = \frac{1}{(4\pi \sqrt{-1})^3} \int_{s_3} \int_{s_2} \int_{s_1} \widehat{V}_l(s_1,s_2,s_3) 
    \, y_1^{-s_1} y_2^{-s_2} y_3^{-s_3} \,ds_1ds_2ds_3,
\end{align*}
where
\begin{align} \label{eqn:Vl_kappa1=kappa2}
\begin{split}
\widehat{V}_{ l } (s_1,s_2,s_3) 
& = (2\pi)^{-l_{13}-l_{24}} 
   (s_2+\nu_1+\nu_2+\kappa_2-l_{13}-l_{24}-1 )_{l_{13} + l_{24} } \Gamma_{\bC}(s_1+\nu_1+\tfrac{\kappa_1-1}{2})
\\
& \quad  \times 
    \Gamma_{\bR} (s_2+2\nu_1+l_{12}+l_{34} )  
 \Gamma_{\bR}(s_2+2\nu_2 +l_{12}+l_{34} ) 
 \Gamma_{\bC}(s_3+\nu_1+2\nu_2+\tfrac{\kappa_1-1}{2} + \kappa_2 ) \\
& \quad \times    \sum_{i=0}^{ l_{14}+l_{23}} \binom{l_{14}+l_{23}}{i} 
  \frac{1}{4\pi \sqrt{-1}} \int_q 
  \frac{\Gamma_{\bR}(s_1-q+i)\Gamma_{\bC}(s_2-q+\nu_1+\tfrac{\kappa_1-1}{2}-l_{13}-l_{24} ) }{  
  \Gamma_{\bR}(s_1+s_2-q+2\nu_1 + l_{12} + l_{34} + i) } 
\\
& \quad \times
 \frac{ \Gamma_{\bR}(s_3-q+2\nu_1+\kappa_2  + l_{14}+l_{23}-i)   \Gamma_{\bC}(q+\nu_2+\tfrac{\kappa_2-1}{2} ) }
{ \Gamma_{\bR}(s_2+s_3-q+2\nu_1+2\nu_2+\kappa_2  +l_{12} + l_{14}+ l_{23} +  l_{34} -i )} \,dq.
\end{split}
\end{align}
\end{lem}

\begin{proof}
We first confirm (\ref{eqn:Vl_kappa1=kappa2}) for 
$ l = l_{12} e_{12} + e_{14} + l_{34} e_{34}, l_{12} e_{12} + e_{23} + l_{34} e_{34} \in S_{(\kappa_2,\kappa_2,0)} $.
From the equations (\ref{eqn:e12_vs_e14}) and (\ref{eqn:e34_vs_e23}) we have 
\begin{align*}
& \widehat{V}_{  l_{12}  e_{12} + e_{14} + l_{34} e_{34}}(s_1,s_2,s_3)
 = \widehat{V}_{l_{12} e_{12} + e_{23} + l_{34} e_{34}} (s_1,s_2,s_3) 
\\
& = (2\pi)^{-2} (s_3+2\nu_1+\nu_2+ \tfrac{3\kappa_2-3}{2} )(s_3+\nu_1+2\nu_2+\tfrac{3\kappa_2-3}{2})
   \widehat{V}_{ (l_{12}+1)  e_{12} + l_{34} e_{34}} (s_1,s_2-1,s_3-1) 
\\
& \quad  - \widehat{V}_{  (l_{12}+1) e_{12} + l_{34} e_{34}} (s_1,s_2-1,s_3+1).
\end{align*}
In view of the expression (\ref{eqn:step1_case3_dup}) the above can be written as
\begin{align} \label{eqn:Vl_kappa1=kappa2_step1}
\begin{split}
& \GC(s_1+\nu_1+\tfrac{\kappa_2-1}{2})  \GR (s_2+2\nu_1+\kappa_2-1) 
 \GR (s_2+2\nu_2+\kappa_2-1) \GC(s_3+\nu_1+2\nu_2+\tfrac{3\kappa_2-1}{2} )
\\
& \times \frac{1}{4\pi \sqrt{-1}} \int_q
  \frac{ \GR(s_1-q) \GC(s_2-q+\nu_1+\frac{\kappa_2-3}{2}) \GR(s_3-q+2\nu_1+\kappa_2-1) \GC(q+\nu_2+\frac{\kappa_2-1}{2}) }
 {\GR(s_1+s_2-q+2\nu_1+\kappa_2-1) \GR(s_2+s_3-q+2\nu_1+2\nu_2+2\kappa_2) }
\\
& \times \{ (s_3+2\nu_1+\nu_2+\tfrac{3\kappa_2-3}{2})(s_2+s_3-q+2\nu_1+2\nu_2+2\kappa_2-2) \\
& \quad 
 - (s_3+\nu_1+2\nu_2+\tfrac{3\kappa_2-1}{2}) (s_3-q+2\nu_1+\kappa_2-1)
   \} \, dq.
\end{split}
\end{align}
Since the bracket $ \{ \cdots \} $ in the above can be written as 
\begin{align*}
 (s_3-q+2\nu_1+\kappa_2-1)(s_2-q+\nu_1+ \tfrac{\kappa_2-3}{2}) 
+ (q+\nu_2+\tfrac{\kappa_2-1}{2})(s_2+s_3-q+2\nu_1+2\nu_2+2\kappa_2-2),
\end{align*}
we know (\ref{eqn:Vl_kappa1=kappa2_step1}) becomes
\begin{align*}
& \GC(s_1+\nu_1+\tfrac{\kappa_2-1}{2})  \GR (s_2+2\nu_1+\kappa_2-1) 
 \GR (s_2+2\nu_2+\kappa_2-1) \GC(s_3+\nu_1+2\nu_2+\tfrac{3\kappa_2-1}{2} )
\\
& \times \biggl\{ \frac{1}{4\pi \sqrt{-1}} \int_q
  \frac{ \GR(s_1-q) \GC(s_2-q+\nu_1+\frac{\kappa_2-1}{2}) \GR(s_3-q+2\nu_1+\kappa_2+1) \GC(q+\nu_2+\frac{\kappa_2-1}{2}) }
 {\GR(s_1+s_2-q+2\nu_1+\kappa_2-1) \GR(s_2+s_3-q+2\nu_1+2\nu_2+2\kappa_2) } \,dq
\\
& + \frac{1}{4\pi \sqrt{-1}} \int_q
  \frac{ \GR(s_1-q) \GC(s_2-q+\nu_1+\frac{\kappa_2-3}{2}) \GR(s_3-q+2\nu_1+\kappa_2-1) \GC(q+\nu_2+\frac{\kappa_2+1}{2}) }
 {\GR(s_1+s_2-q+2\nu_1+\kappa_2-1) \GR(s_2+s_3-q+2\nu_1+2\nu_2+2\kappa_2-2) } \,dq \biggr\}
\end{align*}
as desired.
Then, by (\ref{eqn:hatV_DS}), we know that (\ref{eqn:Vl_kappa1=kappa2}) holds when $ l_{14}+ l_{23} = 0,1 $.

Let us show (\ref{eqn:Vl_kappa1=kappa2}) by induction on $ l_{14}+l_{23} $.
Assume $ l_{14} + l_{23} \ge 2 $. 
From (\ref{eqn:vp_gen_relation_4}) and (\ref{eqn:vp_gen_relation_5}), we know
\begin{align*}
 \hat{\vp}_l
& = \begin{cases}
 -\hat{\vp}_{ l +e_{12} -e_{14} - e_{23} + e_{34}} + \hat{\vp}_{l+e_{13}-e_{14}-e_{23}+e_{24} } 
   & \text{if $ l_{14} \ge1 $ and $l_{23}  \ge 1 $}, \\
 -\hat{\vp}_{ l +2e_{12} -2e_{14} } + \hat{\vp}_{l+2e_{13}-2e_{14}  }   & \text{if $ l_{14} \ge 2 $}, \\
 -\hat{\vp}_{ l  - 2e_{23} + 2e_{34}} + \hat{\vp}_{l+ 2e_{13}-2e_{23} }  & \text{if $ l_{23}  \ge 2 $} 
\end{cases}
\end{align*}
for $ l = (0,0,0,0,l_{12}, l_{13}, l_{14}, l_{23}, l_{24}, l_{34}) \in S_{(\kappa_2,\kappa_2,0)} $.
Then the hypothesis of induction imply that 
$$
  \widehat{V}_{l}(s_1,s_2,s_3) = \Gamma_{\bC}(s_1+\nu_1+\tfrac{\kappa_2-1}{2})  
  \Gamma_{\bC}(s_3+\nu_1+2\nu_2+\tfrac{3\kappa_2-1}{2} ) \cdot  (V_0 + V_1) 
$$
with  
\begin{align*}
V_p & 
= (-1)^{1-p} (2\pi)^{-l_{13}-l_{24} -2p } 
   (s_2+\nu_1+\nu_2+\kappa_2-l_{13}-l_{24}-1-2p )_{l_{13} + l_{24}+2p }
\\
& \quad \times  \Gamma_{\bR} (s_2+2\nu_1+l_{12}+l_{34}+2-2p ) \Gamma_{\bR}(s_2+2\nu_2 +l_{12}+l_{34}+2-2p )\\
& \quad \times \sum_{i=0}^{l_{14}+l_{23}-2} \binom{l_{14}+l_{23}-2}{i} 
  \frac{1}{4\pi \sqrt{-1}} \int_q 
  \frac{\Gamma_{\bR}(s_1-q+i)\Gamma_{\bC}(s_2-q+\nu_1+\tfrac{\kappa_2-1}{2}-l_{13}-l_{24}-2p )
    }{  \Gamma_{\bR}(s_1+s_2-q+2\nu_1 + l_{12} + l_{34} +i+2-2p) } 
\\
& \quad \times 
 \frac{ \Gamma_{\bR}(s_3-q+2\nu_1+\kappa_2  + l_{14}+l_{23}-i-2)   \Gamma_{\bC}(q+\nu_2+\tfrac{\kappa_2-1}{2} )}
{ \Gamma_{\bR}(s_2+s_3-q+2\nu_1+2\nu_2+\kappa_2  +l_{12} + l_{14} + l_{23} + l_{34} -i-2p)} \,dq
\end{align*}
for $ p=0,1 $.
In view of the identities
\begin{align*}
&  \GR(s_2+2\nu_1+l_{12}+l_{34}+2) \GR(s_2+2\nu_2 +l_{12}+l_{34} +2) \\
& = (2\pi)^{-2}  \GR(s_2+2\nu_1+l_{12}+l_{34}) \GR(s_2+2\nu_2 +l_{12}+l_{34} ) \\
& \times 
\{ (s_1+s_2-q+2\nu_1+l_{12}+l_{34}+i) (s_2+s_3-q+2\nu_1+2\nu_2+\kappa_2+l_{12} + l_{14}+l_{23}+ l_{34}-i-2  )\\
& \quad  -(s_1+s_2-q+2\nu_1+l_{12}+l_{34}+i)  (s_3-q+2\nu_1+\kappa_2+ l_{14}+l_{23}- i -2) \\
& \quad   - (s_1-q+i) (s_2+s_3-q+2\nu_1+2\nu_2+\kappa_2+l_{12} + l_{14}+l_{23}+l_{34} -i-2 )  \\
&\quad  +  (s_1-q+i)(s_3-q+2\nu_1+\kappa_2+ l_{14}+l_{23} -i-2) \}
\end{align*}
and 
\begin{align*}
& (2\pi)^{-l_{13}-l_{24}-2}    (s_2+\nu_1+\nu_2+\kappa_2-l_{13}-l_{24} -3)_{l_{13} + l_{24}+2 } 
\\
& = (2\pi)^{-l_{13}-l_{24}}    (s_2+\nu_1+\nu_2+\kappa_2-l_{13}-l_{24}-1 )_{l_{13} + l_{24}} \\
& \quad  \times \sum_{k=0}^2 \binom{2}{k} 
   \frac{ \GC(s_2-q+\nu_1+\tfrac{\kappa_2-1}{2}-l_{13}-l_{24}-k) }{ \GC(s_2-q+\nu_1+\tfrac{\kappa_2-1}{2}-l_{13}-l_{24}-2) }
   \frac{ \GC(q+\nu_2 + \tfrac{\kappa_2-1}{2} +k) }{ \GC(q+\nu_2 + \tfrac{\kappa_2-1}{2} -2)},
\end{align*}
we find that 
\begin{align*}
V_0 + V_1
&= (2\pi)^{-l_{13}-l_{24}} 
   (s_2+\nu_1+\nu_2+\kappa_2-l_{13}-l_{24}-1 )_{l_{13} + l_{24} } \\
& \quad \times  \Gamma_{\bR} (s_2+2\nu_1+l_{12}+l_{34})   \Gamma_{\bR}(s_2+2\nu_2 +l_{12}+l_{34} ) 
\\
&  \quad  \times (V_{0; 0,0}+V_{0;  0,1}+V_{0; 1,0}+ V_{0; 1,1} + V_{1; 0} + V_{1; 1} + V_{1; 2}),
\end{align*}
where
\begin{align*}
  V_{0;  k_1,k_2}
& = (-1)^{k_1+k_2+1} \sum_{i=0}^{l_{14}+l_{23}-2} \binom{l_{14}+l_{23}-2}{i} 
\\
& \quad \times 
  \frac{1}{4\pi \sqrt{-1}} \int_q 
  \frac{\Gamma_{\bR}(s_1-q+i +2k_1 )\Gamma_{\bC}(s_2-q+\nu_1+\tfrac{\kappa_2-1}{2}-l_{13}-l_{24} )
    }{  \Gamma_{\bR}(s_1+s_2-q+2\nu_1 + l_{12} + l_{34} +i +2k_1) } 
\\
& \quad \times
 \frac{ \Gamma_{\bR}(s_3-q+2\nu_1+\kappa_2  + l_{14}+l_{23}-i-2 + 2k_2) \Gamma_{\bC}(q+\nu_2+\tfrac{\kappa_2-1}{2} )}
{ \Gamma_{\bR}(s_2+s_3-q+2\nu_1+2\nu_2+\kappa_2  +l_{12} +l_{14}+ l_{23}+ l_{34} -i-2+ 2k_2 )} \,dq
\end{align*}
for $ 0 \le k_1, k_2 \le 1 $, and 
\begin{align*}
V_{1; k}  & = \binom{2}{k}
 \sum_{i=0}^{ l_{14}+l_{23}-2} \binom{l_{14}+l_{23}-2}{i} 
  \cdot   \frac{1}{4\pi \sqrt{-1}} \int_q 
  \frac{\Gamma_{\bR}(s_1-q+i)\Gamma_{\bC}(s_2-q+\nu_1+\tfrac{\kappa_2-1}{2}-l_{13}-l_{24}-k)
    }{  \Gamma_{\bR}(s_1+s_2-q+2\nu_1 + l_{12} + l_{34} + i) } 
\\
& \quad \times
 \frac{ \Gamma_{\bR}(s_3-q+2\nu_1+\kappa_2  + l_{14}+l_{23}-i-2) \Gamma_{\bC}(q+\nu_2+\tfrac{\kappa_2-1}{2}+k )}
{ \Gamma_{\bR}(s_2+s_3-q+2\nu_1+2\nu_2+\kappa_2  +l_{12} + l_{14}+l_{23}+ l_{34}-i-2 )} \,dq
\end{align*}
for $ 0 \le k \le 2 $. 
Then we know $ V_{0; 0,0} + V_{1; 0} = 0 $ and $ V_{0; 1,1} +V_{1; 2} = 0 $.
We substitute $ i  \to i-2 $ and $ (q,i) \to (q-1,i-1) $ in $ V_{0;1,0} $ and $ V_{1;1} $, respectively. 
Therefore  
the formula $ \binom{n-2}{i} + 2\binom{n-2}{i-1} + \binom{n-2}{i-2} = \binom{n}{i} $ leads our assertion.
\end{proof}

\subsection{Relations with the Jacquet integrals}
\label{subsec:Jacquet}

We discuss here the relation between the Jacquet integral and 
the moderate growth solution $\hat{\varphi}_l$ in the previous subsections. 
As a result, we can remove the assumption (\ref{assumption}) from our main theorems.

First, we recall the definition of the Jacquet integral and its properties.  
Let 
\begin{align*}
&\widehat{\sigma} = \chi_{(\hat{\nu}_1,\hat{\delta}_1)}\boxtimes \chi_{(\hat{\nu}_2,\hat{\delta}_2)} 
\boxtimes \chi_{(\hat{\nu}_3,\hat{\delta}_3)}\boxtimes \chi_{(\hat{\nu}_4,\hat{\delta}_4)}&
&\text{with}\quad  \hat{\nu}=(\hat{\nu}_1,\hat{\nu}_2,\hat{\nu}_3,\hat{\nu}_4)\in \bC^4,\quad 
\hat{\delta}=(\hat{\delta}_1,\hat{\delta}_2,\hat{\delta}_3,\hat{\delta}_4) \in \{0,1\}^4, 
\end{align*}
and consider the principal series representation $(\Pi_{\widehat{\sigma}},H(\widehat{\sigma}))$ of $G$. 
Then we may regard $H(\widehat{\sigma})_K$ as the space of $K$-finite smooth functions $f$ on $K$ satisfying 
\begin{align*}
&f(mk)=m_1^{\hat{\delta}_1}m_2^{\hat{\delta}_2}m_3^{\hat{\delta}_3}m_4^{\hat{\delta}_4}f(k)&
&(m=\diag (m_1,m_2,m_3,m_4)\in {K}\cap M_{(1,1,1,1)},\ k\in {K}),
\end{align*}
and we note that the space $H(\widehat{\sigma})_K$ does not depend on $\hat{\nu}$. 
If $\hat{\nu}$ satisfies $\mathrm{Re}(\hat{\nu}_1)>\mathrm{Re}(\hat{\nu}_2)>\mathrm{Re}(\hat{\nu}_3)>\mathrm{Re}(\hat{\nu}_4)$, 
for $f\in H(\widehat{\sigma})_K$, 
we define the Jacquet integral $\cJ_{\widehat{\sigma}}(f)$ by the convergent integral 
\begin{align*}
&\cJ_{\widehat{\sigma}} (f)(g):=\int_{N}f_{\hat{\nu}} (wxg)\psi_1(x)^{-1}dx\quad (g\in G)&\text{with}\quad 
& w = \begin{pmatrix} & & & 1 \\ & & 1 & \\ & 1 &  & \\ 1 & & & \end{pmatrix},
\end{align*}
where $f_{\hat{\nu}}$ is a smooth function on $G$ defined by 
\begin{align*}
&f_{\hat{\nu}} (xyk)=y_1^{\hat{\nu}_1+\frac{3}{2}}y_2^{\hat{\nu}_1+\hat{\nu}_2+2}
y_3^{\hat{\nu}_1+\hat{\nu}_2+\hat{\nu}_3+\frac{3}{2}}y_4^{\hat{\nu}_1+\hat{\nu}_2+\hat{\nu}_3+\hat{\nu}_4}f(k)\\
&(x\in N,\ y=\diag (y_1y_2y_3y_4,y_2y_3y_4,y_3y_4,y_4)\in A,\ k\in K)
\end{align*}
with the Iwasawa decomposition $G=NAK$. 
By \cite[Theorem 15.4.1]{Wallach_003}, we know that 
$\cJ_{\widehat{\sigma}}(f)(g)$ has the holomorphic continuation 
to whole $\hat{\nu} \in \bC^4$ for every $f\in H(\widehat{\sigma})_K$ and $g\in G$. 
Furthermore, this extends $\cJ_{\widehat{\sigma}}$ 
to all $\hat{\nu} \in \bC^4$ as a nonzero $G$-homomorphism in 
${\cI}_{\Pi_{\widehat{\sigma}} ,\psi_1 }^{\mathrm{mg}}$.

In this subsection, we do not always assume that $\Pi_\sigma$ is irreducible. 
For each cases 1, 2 and 3 introduced in \S \ref{subsec;system_PDE}, 
we define the symbols ${p_0}$, $\nu$ and specify the parameters $\hat{\nu}$, $\hat{\delta}$ of $\widehat{\sigma}$ as follows: 
\begin{itemize}
\item Case 1: ${p_0}=4$, $\nu =(\nu_1,\nu_2,\nu_3,\nu_4)$, $\hat{\nu}=\nu$, and $\hat{\delta}=(\delta_1,\delta_2,\delta_3,\delta_4)$.
\smallskip 
\item Case 2: ${p_0}=3$, $\nu =(\nu_1,\nu_2,\nu_3)$, $\hat{\nu}=(\nu_1+\tfrac{\kappa_1-1}{2},\nu_1-\tfrac{\kappa_1-1}{2},\nu_2,\nu_3)$, 
 and $\hat{\delta}=(\delta_1,0,\delta_2,\delta_3)$.
\smallskip 
\item Case 3: ${p_0}=2$, $\nu =(\nu_1,\nu_2)$, 
$\hat{\nu}=(\nu_1+\tfrac{\kappa_1-1}{2},\nu_1-\tfrac{\kappa_1-1}{2},\nu_2+\tfrac{\kappa_2-1}{2},\nu_2-\tfrac{\kappa_2-1}{2})$, 
 and $\hat{\delta}=(\delta_1,0,\delta_2,0)$.
\smallskip 
\end{itemize} 
Then we can define $\cJ_{\sigma}\in {\cI}_{\Pi_\sigma ,\psi_1 }^{\mathrm{mg}}$ by 
$\cJ_{\sigma}:=\cJ_{\widehat{\sigma}}\circ \rI_\sigma$, 
where $\rI_\sigma \colon H(\sigma)\to H(\widehat{\sigma})$ is the embedding defined by 
$\rI_\sigma =\id_{H(\sigma)}$, (\ref{eqn:P211_embed}) and (\ref{eqn:P22_embed}) 
respectively for cases 1, 2 and 3. 
Let $\hat{\eta}_\sigma \colon V_{(\kappa_1,\kappa_2,\delta_3)}\to H(\sigma)_K$ be the $K$-embedding 
defined in \S \ref{subsec:P_1111_ps}, \S \ref{subsec:P_211_gps} and \S \ref{subsec:P22} 
respectively for cases 1, 2 and 3. 
Here we note that $H(\sigma)_K$ and $\hat{\eta}_\sigma$ do not depend on $\nu$. 
By the properties of the Jacquet integral $\cJ_{\widehat{\sigma}}$, we obtain the following lemma. 

\begin{lem}\label{lem:Jacquet_property}
Retain the notation. \smallskip

\noindent (i) The function $\cJ_{\sigma}(f)(g)$ of $\nu \in \bC^{p_0}$ is entire 
for every $f\in H(\sigma)_K$ and $g\in G$. 
In particular, 
the function $\cJ_{\sigma}(\hat{\eta}_\sigma (v))(g)$ of $\nu \in \bC^{p_0}$ is entire 
for every $v\in V_{(\kappa_1,\kappa_2,\delta_3)}$ and $g\in G$. \smallskip 

\noindent (ii) We have $\cJ_{\sigma}\neq 0$. In particular, we have $\cJ_{\sigma}\circ \hat{\eta}_\sigma \neq 0$ if $\Pi_\sigma$ is irreducible. 
\smallskip

\noindent (iii) We have ${\cI}_{\Pi_{\sigma} ,\psi_1 }^{\mathrm{mg}}=\bC \cJ_{\sigma}$. In particular, we have 
$\Hom_{K}(V_{(\kappa_1,\kappa_2,\delta_3)},{\mathrm{Wh}}(\Pi_\sigma ,\psi_1 )^{\mathrm{mg}})=\bC \cJ_{\sigma}\circ \hat{\eta}_\sigma $. 
\end{lem}
\begin{proof}
The statement (i) follows immediately from the entireness of the function 
$\cJ_{\widehat{\sigma}}(f)(g)$ of $\hat{\nu} \in \bC^4$ for every $f\in H(\widehat{\sigma})_K$ and $g\in G$. 
Since the quotient $H(\widehat{\sigma})_K/\rI_\sigma (H(\sigma)_K)$ is not large in the sense of Vogan \cite{Vogan_001}, 
we have 
\begin{align*}
&\Hom_{(\g_\bC ,K)}\bigl(H(\widehat{\sigma})_K/\rI_\sigma (H(\sigma)_K) ,C^\infty (N\backslash G;\psi_1)_K\bigr)=\{0\}&
\end{align*}
by the result of Matumoto \cite[Corollary 2.2.2, Theorem 6.2.1]{Matumoto_001}. 
Hence, $\cJ_{\widehat{\sigma}}\neq 0$ implies $\cJ_{\sigma}=\cJ_{\widehat{\sigma}}\circ \rI_\sigma \neq 0$, 
and we obtain the statement (ii). By $0\neq \cJ_{\sigma}\in {\cI}_{\Pi_{\sigma} ,\psi_1 }^{\mathrm{mg}}$ and (\ref{eqn:dimWh_gps}), 
we obtain the statement (iii). 
\end{proof}

We define an open dense subset $\Omega_0$ of $\bC^{p_0}$ by 
\begin{align*}
&\Omega_0:=\begin{cases}
\{s=(s_1,s_2,s_3,s_4)\in \bC^4\mid s_i-s_j\not\in \frac{1}{2}\bZ \ \ (1\leq i<j\leq 4),\quad s_1+s_2\neq s_3+s_4\}& 
\text{case 1}, \\
\{s=(s_1,s_2,s_3)\in \bC^3\mid s_i-s_j\not\in \frac{1}{2}\bZ \ \ (1\leq i<j\leq 3)\}& \text{case 2}, \\
\{s=(s_1,s_2)\in \bC^3\mid s_1-s_2\not\in \frac{1}{2}\bZ \}& \text{case 3}.
\end{cases}
\end{align*}
By defintion, we note that (\ref{assumption}) holds if $\nu \in \Omega_0$. 
Furthermore, by the result of Speh \cite[\S 2]{Speh_001} (see \cite{Speh_Vogan_001} for more general result), 
we know that $\Pi_\sigma$ is irreducible if $\nu \in \Omega_0$. 
Therefore, if $\nu \in \Omega_0$, our system of partial differential equations characterize 
Whittaker functions for $(\Pi_\sigma ,\psi_1 )$ at the minimal $K$-type.

Based on the Iwasawa decomposition $G=NAK$, for general $\nu \in \bC^{p_0}$, 
we define a $K$-homomorphism $\varphi_\sigma \colon V_{(\kappa_1,\kappa_2,\delta_3)}\to C^\infty (N\backslash G;\psi_1)$ by 
the equalities 
\begin{align*}
&\varphi_\sigma (v)(xyk)=\psi_1 (x)\varphi_\sigma (\tau_{(\kappa_1,\kappa_2,\delta_3)} (k)v)(y)&
&(v\in V_{(\kappa_1,\kappa_2,\delta_3)},\ x\in {N},\ y\in {A},\ k\in {K})
\end{align*}
and 
\begin{align}
\label{eqn:def_varphi_sigma}
\begin{aligned}
\varphi_\sigma (u_l)(y) = 
&(\sqrt{-1})^{-l_1+l_3-l_{13}+l_{24}} (-1)^{l_2+l_{14}+l_{23}} y_1^{3/2} y_2^2 y_3^{3/2-\kappa_2} y_4^{\gamma_1}\\
&\times \frac{1}{(4\pi \sqrt{-1})^3} \int_{s_3} \int_{s_2} \int_{s_1} \widehat{V}_l(s_1,s_2,s_3) y_1^{-s_1} y_2^{-s_2} y_3^{-s_3} 
 \,ds_1ds_2ds_3
\end{aligned}
\end{align}
for $y=\diag (y_1y_2y_3y_4,y_2y_3y_4,y_3y_4,y_4)\in A$ and 
$l=(l_1,l_2,l_3,l_4, l_{12},l_{13},l_{14},l_{23},l_{24},l_{34}) \in S_{(\kappa_1,\kappa_2,\delta_3)}$. 
Here $\widehat{V}_l(s_1,s_2,s_3)$ is the Mellin-Barnes kernel of $\hat{\varphi}_l$ determined in the previous subsections. 
Then, if $\nu \in \Omega_0$, the arguments in the previous subsections imply that $\varphi_\sigma$ is a unique element of 
$\Hom_K(V_{(\kappa_1,\kappa_2,\delta_3)},{\rm Wh}(\Pi_{\sigma}, \psi_1)^{\mathrm{mg}})$ up to scalar multiple.

For the well-definedness of $\varphi_\sigma$ for general $\nu \in \bC^{p_0}$, 
we need to confirm that the definition 
(\ref{eqn:def_varphi_sigma}) is compatible with the relations of $u_l$ ($l\in S_{(\kappa_1,\kappa_2,\delta_3)}$) in 
Lemma \ref{lem:rel_ul}, that is, the functions $\varphi_\sigma (u_l)(y)$ ($l\in S_{(\kappa_1,\kappa_2,\delta_3)}$) of $y\in A$ defined 
by (\ref{eqn:def_varphi_sigma}) satisfy the following relations: 
\begin{itemize}
\item When $\kappa_1-\kappa_2\geq 2$, 
for $y\in A$ and $l\in S_{(\kappa_1-2,\kappa_2,\delta_3)}$, we have 
\begin{align*}
&\varphi_\sigma (u_{l+2e_1})(y)+\varphi_\sigma (u_{l+2e_2})(y)+\varphi_\sigma (u_{l+2e_3})(y)+\varphi_\sigma (u_{l+2e_4})(y)=0.
\end{align*}

\item When $\kappa_1>\kappa_2>0$, for $y\in A$ and $l\in S_{(\kappa_1-2,\kappa_2-1,\delta_3)}$, we have 
\begin{align*}
&\sum_{1\leq j\leq 4,\, j\neq i}
\sgn (j-i)\varphi_\sigma (u_{l+e_j+e_{ij}})(y)=0&
&(1\leq i\leq 4),
\end{align*}
and 
\begin{align*}
& \varphi_\sigma (u_{l+e_i+e_{jk}})(y)-\varphi_\sigma (u_{l+e_j+e_{ik}})(y)+\varphi_\sigma (u_{l+e_k+e_{ij}})(y)=0 &  
&(1\leq i<j<k\leq 4).
\end{align*}

\item When $\kappa_2\geq 2$, 
for $y\in A$ and $l\in S_{(\kappa_1-2,\kappa_2-2,\delta_3)}$, we have 
\begin{align*}
&\sum_{1\leq k\leq 4,\, k\not\in \{i,j\}}
\sgn ((k-i)(k-j))\varphi_\sigma (u_{l+e_{ik}+e_{jk}})(y)=0&& 
(1\leq i,j\leq 4),
\end{align*}
and 
\begin{align*}
&\varphi_\sigma (u_{l+e_{12}+e_{34}})(y)-\varphi_\sigma (u_{l+e_{13}+e_{24}})(y)+\varphi_\sigma (u_{l+e_{14}+e_{23}})(y)=0. 
\end{align*}

\end{itemize}
By the definition (\ref{eqn:def_varphi_sigma}), we note that $\varphi_\sigma (u_l)(y)$ is an entire function of $\nu \in \bC^{p_0}$ 
for every $l\in S_{(\kappa_1,\kappa_2,\delta_3)}$ and $y\in A$. 
Since the above relations hold for $\nu \in \Omega_0$, they also hold for all $\nu \in \bC^{p_0}$ by the analytic continuation. 
Hence, $\varphi_\sigma$ is well-defined for all $\nu \in \bC^{p_0}$.

\begin{lem}\label{lem:varphi_property}
Retain the notation.\smallskip 

\noindent (i) The function $\varphi_\sigma (v)(g)$ of $\nu \in \bC^{p_0}$ is entire 
for every $v\in V_{(\kappa_1,\kappa_2,\delta_3)}$ and $g\in G$. \smallskip 

\noindent (ii) We have $\varphi_\sigma (u_{(\kappa_1-\kappa_2)e_4+ \kappa_2 e_{34} })|_A \neq 0$. 
\smallskip 

\noindent (iii) If $\nu \in \Omega_0$, we have 
$\Hom_{K}(V_{(\kappa_1,\kappa_2,\delta_3)},{\mathrm{Wh}}(\Pi_\sigma ,\psi_1 )^{\mathrm{mg}})=\bC \varphi_\sigma$. 
\end{lem}
\begin{proof}
The statement (i) follows from the definition of $\varphi_\sigma$. 
By the Mellin inversion formula (see \cite[Lemma 8.4]{HIM}), Lemma \ref{lem:EF_step1} and 
Barnes' second lemma (Lemma \ref{lem:Barnes2nd}), 
for $s_1,s_2\in \bC$ with the sufficiently large real parts, we have 
\begin{equation}
\label{eqn:pf_varphi_property_001}
\begin{aligned}
&\int_{0}^{\infty} \! \int_{0}^{\infty} \! \int_{0}^{\infty} 
\varphi_\sigma (u_{(\kappa_1-\kappa_2)e_4+ \kappa_2 e_{34} })(\hat{y})
\, y_1^{s_1-\frac32} y_2^{s_2-2} y_3^{s_1+s_2+\kappa_2-\frac32}
\frac{dy_1}{y_1} \frac{dy_2}{y_2} \frac{dy_3}{y_3}\\
&=\widehat{V}_{(\kappa_1-\kappa_2)e_4+ \kappa_2 e_{34} }(s_1,s_2,s_1+s_2)=U_0(s_1,s_2,s_1+s_2;r_1,r_2,r_3,r_4)\\
& =\frac{\bigl(\prod_{i=1}^4\GR(s_1+r_i)\bigr)\bigl(\prod_{1\leq i<j\leq 4}\GR(s_2+r_i+r_j)\bigr)}
{\GR(2s_2+r_1+r_2+r_3+r_4)},
\end{aligned}
\end{equation}
where $\hat{y} = {\rm diag}(y_1y_2y_3, y_2y_3, y_3,1) \in A$ and 
\begin{align*}
(r_1,r_2,r_3,r_4) = \begin{cases}
   (\nu_1+\kappa_1, \nu_2+\kappa_2, \nu_3, \nu_4) & \text{cases 1-(i), (ii), (iii)}, \\
   (\nu_4+1, \nu_2, \nu_3, \nu_1) & \text{case 1-(iv)}, \\
   (\nu_1+\tfrac{\kappa_1-1}{2}, \nu_1+\tfrac{\kappa_1+1}{2}, \nu_2+\kappa_2,\nu_3) & \text{case 2}, \\
   (\nu_1+\tfrac{\kappa_1-1}{2}, \nu_1+\tfrac{\kappa_1+1}{2}, \nu_2+\tfrac{\kappa_2-1}{2}, \nu_2+\tfrac{\kappa_2+1}{2})
  & \text{case 3}. \end{cases} 
\end{align*}
Hence, the integrand of the left hand side of (\ref{eqn:pf_varphi_property_001}) is not the zero function, 
and we obtain the statement (ii). 
The statement (iii) has already been proved above. 
\end{proof}

By Lemma \ref{lem:varphi_property} (ii), (iii), 
if $\nu \in \Omega_0$, there is a unique $C(\sigma )\in \bC$ such that 
$\cJ_{\sigma}\circ \hat{\eta}_\sigma =C(\sigma )\varphi_\sigma$. 
The following proposition allows us to remove the assumption (\ref{assumption}) 
from our main theorems.

\begin{prop}
\label{prop:rel_Jacquet_varphi}
Retain the notation. 
Then $C(\sigma )$ is extended to whole $\nu \in \bC^{p_0}$ as an entire function of $\nu$, 
and the equality $\cJ_{\sigma}\circ \hat{\eta}_\sigma =C(\sigma )\varphi_\sigma$ holds for all $\nu \in \bC^{p_0}$. 
Furthermore, if $\Pi_\sigma$ is irreducible, we have $C(\sigma )\neq 0$ and 
$\Hom_{K}(V_{(\kappa_1,\kappa_2,\delta_3)},{\mathrm{Wh}}(\Pi_\sigma ,\psi_1 )^{\mathrm{mg}})=\bC \varphi_\sigma$. 
\end{prop}
\begin{proof}
As a notation in this proof only, for $\nu \in \bC^{p_0}$, the symbol $\sigma_\nu$ denotes $\sigma$ corresponding to $\nu $. 
Take $s\in \bC^{p_0}$ arbitrarily. 
By Lemma \ref{lem:varphi_property} (ii), 
we can choose $y_s\in A$ so that $\varphi_{\sigma_s} (u_{(\kappa_1-\kappa_2)e_4+ \kappa_2 e_{34} })(y_s) \neq 0$. 
Since $\varphi_{\sigma_\nu} (u_{(\kappa_1-\kappa_2)e_4+ \kappa_2 e_{34} })(y_s)$ is a continuous function of $\nu$, 
we can also choose a open neighborhood $\Omega (s,y_s)$ of $s$ so that 
$\varphi_{\sigma_\nu}(u_{(\kappa_1-\kappa_2)e_4+ \kappa_2 e_{34} })(y_s)\neq 0$ ($\nu \in \Omega (s,y_s)$). 
For $\nu \in \Omega (s,y_s)$, we set 
\begin{align*}
&C_s(\sigma_\nu ):=\frac{\cJ_{\sigma_\nu }(\hat{\eta}_{\sigma_\nu}(u_{(\kappa_1-\kappa_2)e_4+ \kappa_2 e_{34} }))(y_s)}
{\varphi_{\sigma_\nu}(u_{(\kappa_1-\kappa_2)e_4+ \kappa_2 e_{34} })(y_s)}. 
\end{align*}
By Lemmas \ref{lem:Jacquet_property} (i) and \ref{lem:varphi_property} (i), 
we note that $C_s(\sigma_\nu )$ is holomorphic on $\Omega (s,y_s)$ as a function of $\nu$. 
Since $\cJ_{\sigma_\nu}\circ \hat{\eta}_{\sigma_\nu}=C(\sigma_\nu )\varphi_{\sigma_\nu}$ for $\nu \in \Omega_0$, we have 
\begin{align}
\label{eqn:pf_rel_Jacquet_varphi_001}
&\cJ_{\sigma_\nu}(\hat{\eta}_{\sigma_\nu}(v))(g)=C(\sigma_\nu )\varphi_{\sigma_\nu}(v)(g)&
&(v\in V_{(\kappa_1,\kappa_2,\delta_3)},\ g\in G,\ \nu \in \Omega_0)
\end{align}
and 
\begin{align}
\label{eqn:pf_rel_Jacquet_varphi_002}
&C_s(\sigma_\nu )=C(\sigma_\nu )&&(\nu \in \Omega_0\cap \Omega (s,y_s)). 
\end{align}
Since we can take $s\in \bC^{p_0}$ arbitrarily and $\Omega_0$ is an open dense subset of $\bC^{p_0}$, 
the equality (\ref{eqn:pf_rel_Jacquet_varphi_002}) extends $C(\sigma_\nu)$ to whole $\nu \in \bC^{p_0}$ as an entire function of $\nu$. 
By the analytic continuation of the both sides of (\ref{eqn:pf_rel_Jacquet_varphi_001}), we know that 
$\cJ_{\sigma_\nu}\circ \hat{\eta}_{\sigma_\nu}=C(\sigma_\nu )\varphi_{\sigma_\nu}$ holds for all $\nu \in \bC^{p_0}$. 
Furthermore, if $\Pi_{\sigma_\nu}$ is irreducible, Lemma \ref{lem:Jacquet_property} (ii), (iii) imply that $C(\sigma_\nu )\neq 0$ and 
$\Hom_{K}(V_{(\kappa_1,\kappa_2,\delta_3)},{\mathrm{Wh}}(\Pi_{\sigma_\nu } ,\psi_1 )^{\mathrm{mg}})=\bC \varphi_{\sigma_\nu }$. 
\end{proof}


\subsection{Explicit formulas of the minimal $K$-type Whittaker functions}
\label{subsec:EF_main}

Thanks to the argument of previous subsections, we arrive at explicit formulas of Whittaker functions which is a main result of this paper. 
As in Theorem \ref{thm:whittaker_isom}, let 
$$ y = {\rm diag}(y_1y_2y_3y_4, y_2y_3y_4, y_3y_4, y_4) \in A. $$ 
See \S \ref{sec:Barnes} for the paths of integrations.


\begin{thm} \label{thm:EF1}
Let $ \sigma = \chi_{(\nu_1,\delta_1)} \boxtimes \chi_{(\nu_2,\delta_2)} 
  \boxtimes \chi_{(\nu_3,\delta_3)} \boxtimes \chi_{(\nu_4,\delta_4)} $ 
with $ \nu_1,\nu_2,\nu_3,\nu_4 \in \bC $, $ \delta_1,\delta_2,\delta_3,\delta_4 \in \{0,1\} $ 
and $ \delta_1 \ge \delta_2 \ge \delta_3 \ge \delta_4 $ 
such that $ \Pi_{\sigma} $ is irreducible.

\noindent
(i) When  $ \delta_1 = \delta_2 = \delta_3 = \delta_4 $, there exists a $K$-homomorphism 
$$ \vp_{\sigma}: V_{(0,0,\delta_1)} \to {\rm Wh}(\Pi_{\sigma}, \psi_1)^{\rm mg} $$
whose radial part is given by 
\begin{align*}
\varphi_{\sigma}(u_{\bf 0})(y) 
& = y_1^{3/2} y_2^2 y_3^{3/2} y_4^{\nu_1+\nu_2+\nu_3+\nu_4} \\
& \quad \times
  \frac{1}{(4\pi \sqrt{-1})^3} \int_{s_3} \int_{s_2} \int_{s_1} 
   V_{\sigma, \bf 0}(s_1,s_2,s_3) \, y_1^{-s_1} y_2^{-s_2} y_3^{-s_3} 
  \,ds_1ds_2ds_3
\end{align*}
with
\begin{align*}
& V_{\sigma,  \bf 0}(s_1,s_2,s_3) \\
& = 
   \GR(s_1+\nu_1) \GR(s_1+\nu_2) \GR(s_2 + \nu_1+\nu_2) \GR(s_2+\nu_3+\nu_4) 
   \GR(s_3+\nu_1+\nu_3+\nu_4) \GR(s_3+\nu_2+\nu_3+\nu_4) \\
& \quad \times \frac{1}{4\pi \sqrt{-1}} \int_q  
  \frac{\GR(s_1-q) \GR(s_2-q+\nu_1) \GR(s_2-q+\nu_2) \GR(s_3-q+\nu_1+\nu_2) \GR(q+\nu_3) \GR(q+\nu_4)}
  { \GR(s_1+s_2-q+\nu_1+\nu_2)  \GR(s_2+s_3-q+\nu_1+\nu_2+\nu_3+\nu_4) } \,dq.
\end{align*} 
\noindent (ii)
When $ (\delta_1,\delta_2,\delta_3,\delta_4) = (1,0,0,0) $, 
there exists a $K$-homomorphism 
$$ \vp_{\sigma}: V_{(1,0,0)} \to {\rm Wh}(\Pi_{\sigma}, \psi_1)^{\rm mg} $$
whose radial part is given by 
\begin{align*}
\varphi_{\sigma}(u_{l})(y) 
& = y_1^{3/2} y_2^2 y_3^{3/2} y_4^{\nu_1+\nu_2+\nu_3+\nu_4} \cdot (\sqrt{-1})^{-l_1+l_3} (-1)^{l_2} 
\\
&\quad  \times \frac{1}{(4\pi \sqrt{-1})^3} \int_{s_3} \int_{s_2} \int_{s_1}
   V_{\sigma,  l}(s_1,s_2,s_3) \, y_1^{-s_1} y_2^{-s_2} y_3^{-s_3} 
  \,ds_1ds_2ds_3
\end{align*}
with $  l = (l_1, l_2, l_3, l_4, 0,0,0,0,0,0) \in S_{(1,0,0)} $. Here 
\begin{align*}
& 
V_{\sigma,  l}(s_1,s_2,s_3) \\
& = 
  \GR(s_1+\nu_1+l_2+l_3+l_4) \GR(s_1+\nu_2+l_1) \GR(s_2 + \nu_1+\nu_2+l_3+l_4)\GR(s_2+\nu_3+\nu_4+l_1+l_2)
\\
&\quad  \times   
   \GR(s_3+\nu_1+\nu_3+\nu_4 +l_4) \GR(s_3+\nu_2+\nu_3+\nu_4 +l_1+l_2+l_3) \\
&\quad  \times \frac{1}{4\pi \sqrt{-1}} \int_q  
  \frac{\GR(s_1-q+l_1) \GR(s_2-q+\nu_1+l_3+l_4) }
  { \GR(s_1+s_2-q+\nu_1+\nu_2+l_1+l_3+l_4) } 
\\
&\quad  \times 
  \frac{ \GR(s_2-q+\nu_2+l_1+l_2)\GR(s_3-q+\nu_1+\nu_2+l_4)  \GR(q+\nu_3) \GR(q+\nu_4)  }
  {\GR(s_2+s_3-q+\nu_1+\nu_2+\nu_3+\nu_4+l_1+l_2+l_4)} \,dq.
\end{align*} 
\noindent (iii)
When $ (\delta_1,\delta_2,\delta_3,\delta_4) = (1,1,0,0) $, 
there exists a $K$-homomorphism 
$$ \vp_{\sigma}: V_{(1,1,0)} \to {\rm Wh}(\Pi_{\sigma}, \psi_1)^{\rm mg} $$
whose radial part is given by 
\begin{align*}
\varphi_{\sigma}(u_{l})(y) 
& = y_1^{3/2} y_2^2 y_3^{3/2} y_4^{\nu_1+\nu_2+\nu_3+\nu_4} 
   \cdot (\sqrt{-1})^{-l_{13}+l_{24}} (-1)^{l_{14}+l_{23}} 
\\
& \quad \times \frac{1}{(4\pi \sqrt{-1})^3} \int_{s_3} \int_{s_2} \int_{s_1}
   V_{\sigma,  l}(s_1,s_2,s_3) \, y_1^{-s_1} y_2^{-s_2} y_3^{-s_3} 
  \,ds_1ds_2ds_3
\end{align*}
with $  l = (0,0,0,0, l_{12},l_{13},l_{14},l_{23},l_{24},l_{34}) \in S_{(1,1,0)} $.
Here  
\begin{align*}
& V_{\sigma, (0,0,0,0, l_{12}, l_{13}, 0, l_{23},0,0)  }(s_1,s_2,s_3) \\
& = \GR(s_1+\nu_1+l_{23}) \GR(s_1+\nu_2+l_{23}) 
\GR(s_2 + \nu_1+\nu_2+l_{13}+l_{23}) \GR(s_2+\nu_3+\nu_4+l_{12}+1) \\
& \quad \times 
   \GR(s_3+\nu_1+\nu_3+\nu_4 +1) \GR(s_3+\nu_2+\nu_3+\nu_4 +1) \\
& \quad \times \frac{1}{4\pi \sqrt{-1}} \int_q  
  \frac{\GR(s_1-q+l_{12}+l_{13}) \GR(s_2-q+\nu_1+l_{12}) }
  { \GR(s_1+s_2-q+\nu_1+\nu_2+l_{12}+l_{23}) } \\
& \quad \times 
  \frac{ \GR(s_2-q+\nu_2+l_{12}) \GR(s_3-q+\nu_1+\nu_2) \GR(q+\nu_3) \GR(q+\nu_4) }
  {\GR(s_2+s_3-q+\nu_1+\nu_2+\nu_3+\nu_4+l_{12} +1)} \,dq, 
\\
& V_{\sigma, (0,0,0,0, 0,0,l_{14}, 0, l_{24}, l_{34})}(s_1,s_2,s_3) \\
& = \GR(s_1+\nu_3+l_{14}) \GR(s_1+\nu_4+l_{14}) 
 \GR(s_2 + \nu_3+\nu_4+l_{14}+l_{24}) \GR(s_2+\nu_1+\nu_2+l_{34}+1) \\
& \quad \times 
   \GR(s_3+\nu_1+\nu_2+\nu_3 +1) \GR(s_3+\nu_1+\nu_2+\nu_4 +1) \\
& \quad \times \frac{1}{4\pi \sqrt{-1}} \int_q  
  \frac{\GR(s_1-q+l_{24}+l_{34}) \GR(s_2-q+\nu_3+l_{34}) }
  { \GR(s_1+s_2-q+\nu_3+\nu_4+l_{14}+l_{34}) } 
\\
& \quad \times 
  \frac{ \GR(s_2-q+\nu_4+l_{34}) \GR(s_3-q+\nu_3+\nu_4) \GR(q+\nu_1) \GR(q+\nu_2) }
  {\GR(s_2+s_3-q+\nu_1+\nu_2+\nu_3+\nu_4+l_{34} +1)}  \,dq.
\end{align*} 
\noindent (iv)
When $ (\delta_1,\delta_2,\delta_3,\delta_4) = (1,1,1,0) $, there exists a $K$-homomorphism 
$$ \vp_{\sigma}: V_{(1,0,1)} \to {\rm Wh}(\Pi_{\sigma}, \psi_1)^{\rm mg} $$
whose radial part is given by 
\begin{align*}
\varphi_{\sigma}(u_{l})(y) 
& = y_1^{3/2} y_2^2 y_3^{3/2} y_4^{\nu_1+\nu_2+\nu_3+\nu_4} \cdot (\sqrt{-1})^{-l_1+l_3} (-1)^{l_2} 
\\
& \quad \times \frac{1}{(4\pi \sqrt{-1})^3} \int_{s_3} \int_{s_2} \int_{s_1}
   V_{\sigma,  l}(s_1,s_2,s_3) \, y_1^{-s_1} y_2^{-s_2} y_3^{-s_3} 
   \,ds_1ds_2ds_3
\end{align*}
with $  l = (l_1, l_2, l_3, l_4, 0,0,0,0,0,0) \in S_{(1,0,1)} $. Here
\begin{align*}
& V_{\sigma, l}(s_1,s_2,s_3) \\
& = \GR(s_1+\nu_4+l_2+l_3+l_4) \GR(s_1+\nu_2+l_1) 
  \GR(s_2 + \nu_2+\nu_4+l_3+l_4) \GR(s_2+\nu_1+\nu_3+l_1+l_2) \\
& \quad \times 
   \GR(s_3+\nu_1+\nu_3+\nu_4 +l_4) \GR(s_3+\nu_1+\nu_2+\nu_3 +l_1+l_2+l_3) \\
& \quad \times \frac{1}{4\pi \sqrt{-1}} \int_q  
  \frac{\GR(s_1-q+l_1) \GR(s_2-q+\nu_4+l_3+l_4) }
  { \GR(s_1+s_2-q+\nu_2+\nu_4+l_1+l_3+l_4) } 
\\
& \quad \times 
  \frac{ \GR(s_2-q+\nu_2+l_1+l_2) \GR(s_3-q+\nu_2+\nu_4+l_4)\GR(q+\nu_1) \GR(q+\nu_3)  }
  {\GR(s_2+s_3-q+\nu_1+\nu_2+\nu_3+\nu_4+l_1+l_2+l_4)}  \,dq.
\end{align*} 
\end{thm}


\begin{thm} \label{thm:EF2}
Let $ \sigma = D_{(\nu_1,\kappa_1)} \boxtimes \chi_{(\nu_2,\delta_2)} 
  \boxtimes \chi_{(\nu_3,\delta_3)}  $ 
with $ \nu_1,\nu_2,\nu_3\in \bC $, $ \kappa_1 \in \bZ_{\ge 2} $, 
$ \delta_2, \delta_3 \in \{0,1\} $ and 
$ \delta_2 \ge \delta_3 $ 
such that $ \Pi_{\sigma} $ is irreducible.

\noindent (i)
When $ \delta_2 = \delta_3 $,
there exists a $K$-homomorphism 
$$ \vp_{\sigma}: V_{(\kappa_1,0,\delta_3)} \to {\rm Wh}(\Pi_{\sigma}, \psi_1)^{\rm mg} $$
whose radial part is given by  
\begin{align*}
 \varphi_{\sigma}(u_{l})(y) 
& = y_1^{3/2} y_2^2 y_3^{3/2} y_4^{2\nu_1+\nu_2+\nu_3} \cdot (\sqrt{-1})^{-l_1+l_3} (-1)^{l_2} 
\\
& \quad \times \frac{1}{(4\pi \sqrt{-1})^3} \int_{s_3} \int_{s_2} \int_{s_1}
   V_{\sigma, l}(s_1,s_2,s_3) \, y_1^{-s_1} y_2^{-s_2} y_3^{-s_3} 
  \,ds_1ds_2ds_3
\end{align*}
with $  l = (l_1, l_2, l_3, l_4, 0,0,0,0,0,0) \in S_{(\kappa_1,0,\delta_3)} $. Here 
\begin{align*}
&  V_{\sigma, l}(s_1,s_2,s_3) \\
&
 =  \GC(s_1+\nu_1+\tfrac{\kappa_1-1}{2})  \GR(s_2 +2\nu_1+l_3+l_4)  
  \GR(s_2+\nu_2+\nu_3+l_1+l_2)
   \GC(s_3+\nu_1+\nu_2+\nu_3 + \tfrac{\kappa_1-1}{2})  \\
& \quad  \times \frac{1}{4\pi \sqrt{-1}} \int_q  
  \frac{\GR(s_1-q+l_1) \GC(s_2-q+\nu_1+\tfrac{\kappa_1-1}{2})\GR(s_3-q+2\nu_1+l_4) \GR(q+\nu_2) \GR(q+\nu_3)  }
  { \GR(s_1+s_2-q+2\nu_1+l_1+l_3+l_4) \GR(s_2+s_3-q+2\nu_1+\nu_2+\nu_3+l_1+l_2+l_4)} \,dq.
\end{align*} 
\noindent (ii)
When $ (\delta_2, \delta_3)=(1,0) $,
there exists a $K$-homomorphism 
$$ \vp_{\sigma}: V_{(\kappa_1,1,0)} \to {\rm Wh}(\Pi_{\sigma}, \psi_1)^{\rm mg} $$
whose radial part is given by  
\begin{align*}
 \varphi_{\sigma}(u_{l})(y) 
& = y_1^{3/2} y_2^2 y_3^{3/2} y_4^{2\nu_1+\nu_2+\nu_3} \cdot (\sqrt{-1})^{-l_1+l_3-l_{13}+l_{24}} (-1)^{l_2+l_{14}+l_{23}} 
\\
& \quad  \times \frac{1}{(4\pi \sqrt{-1})^3} \int_{s_3} \int_{s_2} \int_{s_1}
   V_{\sigma, l}(s_1,s_2,s_3) \, y_1^{-s_1} y_2^{-s_2} y_3^{-s_3} 
  \,ds_1ds_2ds_3
\end{align*}
with $  l = (l_1, l_2, l_3, l_4, l_{12},l_{13},l_{14},l_{23},l_{24},l_{34}) \in S_{(\kappa_1,1,0)} $. 
Here 
\begin{align*}
&  V_{\sigma, l}(s_1,s_2,s_3) \\
& = (2\pi)^{-l_{13}-l_{24}} (s_2+\nu_1+\nu_2+\tfrac{\kappa_1-1}{2} -l_{13}-l_{24})_{l_{13}} 
   (s_2+\nu_1+\nu_3+\tfrac{\kappa_1-1}{2}-l_{24})_{l_{24}} 
\\
& \quad \times 
   \GC(s_1+\nu_1+\tfrac{\kappa_1-1}{2}) \GR(s_2+2\nu_1+l_3+l_4+l_{12}+l_{34} ) 
\\
& \quad  \times \GR(s_2+\nu_2+\nu_3+l_1+l_2+l_{12}+l_{34}) \GC(s_3+\nu_1+\nu_2+\nu_3+\tfrac{\kappa_1-1}{2})   \\
& \quad  \times \sum_{i_{14}=0}^{l_{14}} \sum_{i_{23}=0}^{l_{23}}  
   \frac{1}{4\pi \sqrt{-1}} \int_q  \frac{\GR(s_1-q+l_1+i_{14}+i_{23}  ) \GC(s_2-q+\nu_1+\tfrac{\kappa_1-1}{2} -l_{13}-l_{24})  }
   { \GR(s_1+s_2-q+2\nu_1+l_1+l_3+l_4+l_{12}+l_{34}+i_{14}+i_{23} ) } \\
& \quad  \times \frac{ \GR(s_3-q+2\nu_1+l_4+l_{14}+l_{23}-i_{14}-i_{23} )}{ \GR(s_2+s_3-q+2\nu_1+\nu_2+\nu_3 + l_1+l_2+l_4+l_{12}+l_{14}+l_{23}+l_{34}-i_{14}-i_{23} ) } 
\\
&  \quad  \times \GR(q+\nu_2+ l_{23}+l_{24}+l_{34}+i_{14}-i_{23} )
     \GR(q+\nu_3+l_{12}+l_{13}+l_{14}-i_{14}+i_{23} ) \,dq.
\end{align*}
\end{thm}


\begin{thm} \label{thm:EF3}
Let $ \sigma = D_{(\nu_1,\kappa_1)} \boxtimes D_{(\nu_2,\kappa_2)} $ 
with $ \nu_1,\nu_2  \in \bC $, $\kappa_1, \kappa_2 \in \bZ_{\ge 2} $ and $ \kappa_1 \ge \kappa_2 $ 
such that $ \Pi_{\sigma} $ is irreducible.
There exists a $K$-homomorphism 
$$ \vp_{\sigma}: V_{(\kappa_1,\kappa_2,0)} \to {\rm Wh}(\Pi_{\sigma}, \psi_1)^{\rm mg} $$
whose radial part is given by  
\begin{align*}
 \varphi_{\sigma}(u_{l})(y) 
& = y_1^{3/2} y_2^2 y_3^{3/2} y_4^{2\nu_1+2\nu_2} \cdot (\sqrt{-1})^{-l_1+l_3-l_{13}+l_{24}} (-1)^{l_2+l_{14}+l_{23}} 
\\
&  \quad \times \frac{1}{(4\pi \sqrt{-1})^3} \int_{s_3} \int_{s_2} \int_{s_1}
   V_{\sigma, l}(s_1,s_2,s_3) \, y_1^{-s_1} y_2^{-s_2} y_3^{-s_3} 
  \,ds_1ds_2ds_3
\end{align*}
with $  l = (l_1, l_2, l_3, l_4, l_{12},l_{13},l_{14},l_{23},l_{24},l_{34}) \in S_{(\kappa_1,\kappa_2,0)} $. 
Here
\begin{align*}
& V_{\sigma, l}(s_1,s_2,s_3)
\\
& =(2\pi)^{-l_{13}-l_{24}}
    (s_2+\nu_1+\nu_2+\tfrac{\kappa_1+\kappa_2}{2}-l_{13}-l_{24}-1)_{l_{13}+l_{24}}  
    \GC(s_1+\nu_1+\tfrac{\kappa_1-1}{2} )\\
& \quad \times  
  \GR(s_2+2\nu_1 + l_3+l_4 +l_{12} + l_{34} )
 \GR(s_2+2\nu_2+   l_1+l_2 + l_{12} + l_{34} ) \GC(s_3+\nu_1+2\nu_2+\tfrac{\kappa_1-1}{2})  
\\ 
&  \quad \times  \sum_{i=0}^{l_{14}+l_{23}}  \binom{l_{14}+l_{23}}{i}  
 \frac{1}{4\pi \sqrt{-1}} \int_q 
   \frac{\GR(s_1-q+l_1+i) \GC(s_2-q + \nu_1+ \tfrac{\kappa_1-1}{2} -l_{13}-l_{24} )}{ 
      \GR(s_1+s_2-q+2\nu_1+l_1+l_3+l_4+l_{12}+l_{34}+i)} 
\\
&  \quad \times \frac{ \GR(s_3-q+2\nu_1 + l_4+l_{14}+l_{23}-i) \GC(q+ \nu_2+\tfrac{\kappa_2-1}{2} )  }{ 
   \GR(s_2+s_3-q+ 2\nu_1 + 2\nu_2 +l_1+l_2 +l_4 +l_{12} + l_{14} + l_{23} + l_{34} -i) }\,dq.
\end{align*}
\end{thm}



\section{Test vectors for archimedean Bump-Friedberg integrals}
\label{sec:BF_int}

Bump and Friedberg \cite{Bump_Friedberg_001} gave a zeta integral containing two complex variables 
which interpolates the standard and the exterior square $L$-functions on ${\rm GL}(n) $ simultaneously.
The aim of this section is to give test vectors for this zeta integrals.
The case of principal series is done in \cite{Ishii_001}.

\subsection{$L$- and $\varepsilon$-factors for the standard and the exterior square $L$-functions}
\label{subsec:L_and_varepsilon}

We recall the theory of finite dimensional semisimple representations of the Weil group $W_\bR$. 
The Weil group $W_{\bR}$ for the real field $\bR$ is given by 
$W_{\bR} ={\bC}^\times \cup (\bC^\times \,\mathtt{j})
\subset {\bH}^\times$. 
Here we regard $W_\bR$ as a subgroup of the multiplicative group ${\bH}^\times$ of the Hamilton quaternion algebra 
${\bH} =\bC \oplus ({\bC}\, \mathtt{j})$ ($\mathtt{j}^2=-1$ and $\mathtt{j}z\mathtt{j}^{-1}=\overline{z}$ for $z\in \bC$). 
We recall the irreducible representations of $W_\bR$ and the corresponding $L$- and $\varepsilon$-factors. 
\begin{enumerate}
\item 
{\textit{Characters.}} 
We define characters $\phi^\delta_\nu \ (\nu \in \bC ,\delta \in \{0,1\} )$ 
of $W_\bR$ by 
\begin{align*}
\phi^\delta_\nu (z)&=|z|^{2\nu} \ \ \ (z \in \bC),&
\phi^\delta_\nu (\mathtt{j})&=(-1)^\delta .
\end{align*}
We define the corresponding $L$- and $\varepsilon$-factors by 
\begin{align*}
 L(s,\phi^\delta_\nu ) 
& =\GR (s+\nu +\delta ),
&
\varepsilon (s,\phi^\delta_\nu ,\psi_{\bR} )&=(\sqrt{-1})^\delta.
\end{align*}
\item 
{\textit{Two dimensional representations.}} 
We define two dimensional representations 
$\phi_{\nu ,\kappa}\colon W_\bR \to {\rm GL}(2,\bC)\ 
(\nu \in \bC ,\kappa \in \bZ_{\geq 0})$ by 
\begin{align*}
\phi_{\nu ,\kappa}(re^{\sqrt{-1} \theta })
& =\left(
\begin{array}{cc}
r^{2\nu} e^{-\sqrt{-1} \kappa \theta}&0\\
0&r^{2\nu} e^{\sqrt{-1} \kappa \theta}
\end{array}
\right) 
\ \ \ (r\in \bR_+,\ \theta \in \bR),\\
\phi_{\nu ,\kappa }(\mathtt{j})
& =\left(
\begin{array}{cc}
0&(-1)^\kappa \\
1&0
\end{array}
\right).
\end{align*}
The representation $\phi_{\nu ,\kappa }$ is 
irreducible for $\kappa >0$, while $\phi_{\nu ,0}$ is equivalent to 
a sum of two characters $\phi_{\nu }^0$ and $\phi_{\nu }^1$. 
For $\phi_{\nu ,\kappa }\ (\kappa >0)$, we define the corresponding 
$L$- and $ \varepsilon $-factors by 
\begin{align*}
 L(s,\phi_{\nu ,\kappa } ) & =\GC (s+\nu+\tfrac{\kappa}{2} ), 
& \varepsilon (s,\phi_{\nu ,\kappa } ,\psi_{\bR} )
&=(\sqrt{-1} )^{\kappa +1}.
\end{align*}
\end{enumerate}
The set of equivalence classes of 
irreducible representations of $W_\bR$ is 
exhausted by 
\[
\Sigma_\bR 
=\left\{\phi_{\nu }^{\delta}\mid 
\nu \in \bC ,\ \delta \in \{0,1\} \right\} \cup 
\left\{\phi_{\nu ,\kappa }\mid 
\nu \in \bC ,\ \kappa  \in \bZ_{\geq 1}\right\}.
\]
For a finite dimensional semisimple representation $\phi$ of $W_\bR$, 
we define the corresponding $L$- and  $ \varepsilon$-factors by 
\begin{align*}
L(s,\phi ) &=\prod_{i=1}^mL(s,\phi_i ), 
& \varepsilon(s,\phi,\psi_{\bR}) = \prod_{i=1}^m \varepsilon (s,\phi_i, \psi_{\bR}),
\end{align*}
where
$\phi \simeq \bigoplus_{i=1}^m\phi_i$ is the irreducible decomposition of $\phi$.

\medskip 

The local Langlands correspondence on ${\rm GL}(n,\bR) $ is  
a bijection between the set of infinitesimal 
equivalence classes of irreducible admissible representations 
of ${\rm GL}(n,\bR)$ and the set of equivalence classes of $n$-dimensional 
semisimple representations of $W_\bR$. 
For an irreducible admissible representation $\Pi $ of ${\rm GL}(n,\bR)$, 
the corresponding representation $\phi[\Pi] $ of $W_\bR$ is called Langlands parameter of $ \Pi. $  
See \cite{Knapp_001} (cf. \cite[section 9.1]{HIM}) for the precise. 
We define the local $L$- and $ \varepsilon$-factors 
$ L(s,\Pi) $ and $ \varepsilon (s,\Pi, \psi_{\bR}) $
for the standard $L$-function by   
\begin{align*}
 L(s,\Pi) & = L(s, \phi[\Pi]), & 
 \varepsilon(s,\Pi,\psi_{\bR}) & = \varepsilon(s, \phi[\Pi],\psi_{\bR}). 
\end{align*}
Let $ \wedge^2: {\rm GL}(n,\bC) \to {\rm GL}(\tfrac{n(n-1)}{2},\bC) $ be the exterior square representation.  
We define the local $L$- and $ \varepsilon$-factors
$L(s,\Pi,\wedge^2) $ and  $ \varepsilon(s,\Pi,\wedge^2, \psi_{\bR}) $ for the exterior square $L$-function by 
\begin{align*}
 L(s,\Pi, \wedge^2)  & = L(s, \wedge^2(\phi[\Pi])), & 
 \varepsilon(s,\Pi, \wedge^2, \psi_{\bR}) & = \varepsilon(s, \wedge^2(\phi[\Pi]), \psi_{\bR}).
\end{align*}

\medskip 

Let us describe $L$- and $ \varepsilon$-factors for the irreducible $ P_{\mn} $-principal series representation $ \Pi_{\sigma}  $.
The Langlands parameter $ \phi[\Pi_{\sigma}] $ is given by 
\begin{align*}
 \phi[\Pi_{\sigma}] 
 \cong \begin{cases}
   \oplus_{1 \le i \le 4} \phi_{\nu_i}^{\delta_i}
  & \text{if $\mn=(1,1,1,1) $}, \\
  \phi_{\nu_1, \kappa_1-1} \oplus \phi_{\nu_2}^{\delta_2} \oplus \phi_{\nu_3}^{\delta_3} 
  & \text{if $\mn=(2,1,1) $}, \\
  \phi_{\nu_1, \kappa_1-1} \oplus \phi_{\nu_2, \kappa_2-1} & \text{if $\mn=(2,2) $}.
\end{cases}
\end{align*}
We know
\begin{align*}
& \wedge^2  (\phi[ \Pi_{\sigma}])  \cong \begin{cases}
   \oplus_{1 \le i<j \le 4} \phi_{\nu_i+\nu_j}^{|\delta_i-\delta_j|}
  & \text{if $ \mn =(1,1,1,1) $}, \\
  \phi_{\nu_1+\nu_2, \kappa_1-1} \oplus \phi_{\nu_1+\nu_3, \kappa_1-1}  
  \oplus \phi_{2\nu_1}^{\delta_1} \oplus \phi_{\nu_2 + \nu_3}^{|\delta_2-\delta_3|} 
  & \text{if $\mn=(2,1,1)$}, \\
  \phi_{\nu_1+\nu_2, |\kappa_1-\kappa_2|} \oplus \phi_{\nu_1+\nu_2, \kappa_1+\kappa_2-2}
  \oplus \phi_{2\nu_1}^{\delta_1} \oplus \phi_{2\nu_2}^{\delta_2}  & \text{if $\mn=(2,2)$}.
\end{cases}
\end{align*}
Here the integers $ \delta_{1} \in \{0,1 \} $ ($\mn=(2,1,1)$) and  
$\delta_1, \delta_2  \in \{0,1 \}$ ($\mn=(2,2)$) are defined by 
$ \delta_i \equiv \kappa_i \pmod 2 $ $(i=1, 2)$ as in \S \ref{subsec;system_PDE}.
Then we have the following.
\begin{itemize}
\item
When $ \sigma = \chi_{(\nu_1,\delta_1)} \boxtimes \chi_{(\nu_2,\delta_2)} \boxtimes  
 \chi_{(\nu_3,\delta_3)} \boxtimes  \chi_{(\nu_4,\delta_4)} $, we have 
\begin{align*}
 L(s,\Pi_{\sigma}) 
& = \textstyle \prod_{1 \le i \le 4} \GR(s+\nu_i+\delta_i), 
\\
 L(s,\Pi_{\sigma}, \wedge^2) 
& = \textstyle \prod_{1 \le i<j \le 4} \GR(s + \nu_i + \nu_j + |\delta_i-\delta_j| ),
\\
  \varepsilon(s,\Pi_{\sigma}, \psi_{\bR}) 
& = (\sqrt{-1})^{\sum_{1 \le i \le 4} \delta_i}, 
\\
  \varepsilon(s,\Pi_{\sigma}, \wedge^2, \psi_{\bR}) 
& = (\sqrt{-1})^{\sum_{1 \le i<j \le 4} |\delta_i-\delta_j| }.
\end{align*}
\item
When $ \sigma = D_{(\nu_1,\kappa_1)} \boxtimes \chi_{(\nu_2,\delta_2)} \boxtimes  
 \chi_{(\nu_3,\delta_3)} $, we have 
\begin{align*}
 L(s,\Pi_{\sigma}) 
& = \GC(s+\nu_1+\tfrac{\kappa_1-1}{2}) \GR(s+\nu_2+\delta_2)\GR(s+\nu_3+\delta_3), 
\\
 L(s,\Pi_{\sigma}, \wedge^2) 
& = \GC(s+\nu_1+\nu_2+\tfrac{\kappa_1-1}{2})\GC(s+\nu_1+\nu_3+\tfrac{\kappa_1-1}{2}) \\
& \quad \times 
 \GR(s+2\nu_1+\delta_1)  \GR(s+\nu_2+\nu_3+ |\delta_2-\delta_3|),
\\
  \varepsilon(s,\Pi_{\sigma}, \psi_{\bR}) 
& = (\sqrt{-1})^{\kappa_1+\delta_2+\delta_3}, 
\\
  \varepsilon(s,\Pi_{\sigma}, \wedge^2, \psi_{\bR}) 
& = (\sqrt{-1})^{2\kappa_1+ \delta_1+|\delta_2-\delta_3|}.
\end{align*}
\item
When $ \sigma = D_{(\nu_1,\kappa_1)} \boxtimes D_{(\nu_2,\kappa_2)} $, we have 
\begin{align*}
 L(s,\Pi_{\sigma}) 
& = \GC(s+\nu_1+\tfrac{\kappa_1-1}{2}) \GC(s+\nu_2+\tfrac{\kappa_2-1}{2}) , 
\\
 L(s,\Pi_{\sigma}, \wedge^2) 
& = \GC(s+\nu_1+\nu_2+\tfrac{|\kappa_1-\kappa_2|}{2})\GC(s+\nu_1+\nu_2+\tfrac{\kappa_1+\kappa_2-2}{2})  
 \GR(s+2\nu_1+\delta_1)\GR(s+2\nu_2+\delta_2), 
\\
  \varepsilon(s,\Pi_{\sigma}, \psi_{\bR}) 
& = (\sqrt{-1})^{\kappa_1+\kappa_2}, 
\\
  \varepsilon(s,\Pi_{\sigma}, \wedge^2, \psi_{\bR}) 
& = (\sqrt{-1})^{2\kappa_1+ \delta_1+\delta_{2}}.
\end{align*}
\end{itemize}


\subsection{Archimedean Bump-Friedberg integrals}
\label{subsec:BF_int}

We recall the archimedean part of Bump-Friedberg integrals. 
As in \S \ref{subsec:group_algebra},
we set $ G_2 = {\rm GL}(2,\bR) $,  
$ N_2 = \{ ( \begin{smallmatrix} 1 & x_{1,2} \\ 0 & 1 \end{smallmatrix}) \mid x_{1,2} \in \bR \} $, $ K_2 = {\rm O}(2) $
and $ \rk_{\theta}^{(2)} = 
 \left(\begin{smallmatrix} \cos \theta & \sin \theta \\  
  -\sin \theta & \cos \theta \end{smallmatrix} \right) \in K_2. $ 
We define an embedding $ \tilde{\iota} : G_2 \times G_2 \to G$ by 
\begin{align*}
 \left( g_1 = \begin{pmatrix} a_1 & b_1 \\ c_1 & d_1 \end{pmatrix}, 
        g_2 = \begin{pmatrix} a_2 & b_2 \\ c_2 & d_2 \end{pmatrix} \right) 
  \mapsto 
  \tilde{\iota}(g_1,g_2) := \begin{pmatrix} 
    a_1 & & b_1 & \\ & a_2 & & b_2 \\ c_1 & & d_1 & \\ & c_2 & & d_2 
   \end{pmatrix}.       
\end{align*}
Let $ {\mathcal S}(\bR^2) $ be the space of Schwartz-Bruhat functions on $ \bR^2 $. 
For $ s_1,s_2 \in \bC $, $ \Phi \in {\mathcal S}(\bR^2) $ and $ W \in {\rm Wh}(\Pi_{\sigma}, \psi_{1})^{\rm mg} $, 
we consider the following archimedean zeta integral:
\begin{align*}
 Z(s_1,s_2, W,\Phi) 
& = \int_{N_2 \backslash G_2 } \int_{ N_2 \backslash G_2} 
    W( \tilde{\iota}(g_1,g_2) ) \Phi((0,1)g_2) 
   |\det g_1|^{s_1-\frac12} |\det g_2|^{-s_1+s_2+\frac12} \,d\dot{g}_1 d\dot{g}_2.
\end{align*}
Here 
$ d\dot{g}$ is the right $ G_2 $-invariant measure on 
$ N_2 \backslash G_2 $ normalized so that 
\begin{align*}
  \int_{N_2 \backslash G_2} f(g) \,d\dot{g} 
&  = \int_{0}^{\infty} \! \int_{0}^{\infty} \int_{K_2}
    f \bigl( {\rm diag}( y_1y_2, y_2) k \bigr) \, dk 
    \dfrac{2dy_1}{y_1^2} \dfrac{2dy_2}{y_2}
\\
& = 2 \sum_{ \varepsilon \in \{\pm 1\} } 
      \int_{0}^{\infty} \! \int_{0}^{\infty} \! \int_{0}^{2\pi}  
     f \bigl( {\rm diag}( y_1y_2, y_2) \, {\rm diag}( \varepsilon,1) \, \rk_{\theta}^{(2)} \bigr) \,  
     \dfrac{d\theta}{2\pi} \dfrac{dy_1}{y_1^2} \dfrac{dy_2}{y_2}
\end{align*}
for any compactly supported continuous function $f$ on $ N_2 \backslash G_2. $
Here $ dk $ is the normalized Haar measure on $K_2 $ such that $ \int_{K_2} dk = 1$.
For $ \Phi \in \mathcal{S}(\bR^2) $ we define the Fourier transform $ \widehat{\Phi} $ of $ \Phi $ by 
\begin{align*}
 \widehat{\Phi}(x_1,x_2) 
 = \int_{\bR^2} \Phi(y_1,y_2) \psi_{\bR} (x_1y_1+x_2 y_2) \,dy_1 dy_2.
\end{align*} 
For $ a,b \in \bZ_{\ge 0} $ we put
\begin{align*}
   \Phi_{(a,b)}(x_1,x_2)  = (-\sqrt{-1}x_1+x_2)^a (\sqrt{-1}x_1+x_2)^b \exp \{ -\pi (x_1^2+x_2^2) \}.
\end{align*}
We write $ g_i \in N_2 \backslash G_2 $ $(i=1,2)$ as 
$$ 
  g_i = \begin{pmatrix} y_{i1} y_{i2} & \\ & y_{i2} \end{pmatrix}
        \begin{pmatrix} \varepsilon_i & \\ & 1 \end{pmatrix}  \rk_{\theta_i}^{(2)}
$$
with $ y_{i1}, y_{i2} \in \bR_+ $, $ \varepsilon_i \in \{\pm 1\} $ and 
$ 0 \le \theta_i < 2\pi $, 
to find that    
\begin{align*} 
& Z(s_1,s_2,W,\Phi_{(a,b)}) 
\\
& = 2^2 
     \sum_{ \varepsilon_1,\varepsilon_2 \in \{\pm 1 \}} 
     \int_{0}^{\infty} \! \int_{0}^{\infty} \! \int_{0}^{\infty} \! \int_{0}^{\infty} \!
     \int_{0}^{2\pi} \! \int_{0}^{2\pi} 
    W \bigl( {\rm diag}(y_{11}y_{12}, y_{21}y_{22}, y_{12}, y_{22}) 
   \, {\rm diag} (\varepsilon_1, \varepsilon_2, 1, 1) 
   \, \tilde{\iota} (\rk_{\theta_1}^{(2)}, \rk_{\theta_2}^{(2)})  \bigr)
\\
& \phantom{=}  \times \Phi_{(a,b)} (-y_{22}\sin \theta_2, y_{22} \cos \theta_2) 
  \, (y_{11}y_{12}^2)^{s_1-\frac12} (y_{21} y_{22}^2)^{-s_1+s_2+\frac12} 
  \,  \frac{d\theta_1}{2\pi} \frac{d\theta_2}{2\pi}
  \frac{dy_{11}}{y_{11}^2}  \frac{dy_{12}}{y_{12}}  \frac{dy_{21}}{y_{21}^2} \frac{dy_{22}}{y_{22}}.
\end{align*}
Now we change the variables $ (y_{11},y_{12},y_{21},y_{22}) \to (y_1,y_2,y_3,y_4) $ by
\begin{align*}
 y_1 &= \frac{y_{11} y_{12}}{y_{21} y_{22}}, 
&  y_2 &= \frac{y_{21} y_{22}}{y_{12}}, 
& y_3 &= \frac{y_{12}}{y_{22}}, 
& y_4 &= y_{22}
\end{align*}
and write 
\begin{align*}
 y & = {\rm diag}(y_1y_2y_3y_4, y_2y_3y_4, y_3y_4,y_4) \in A, & 
 \hat{y} & = {\rm diag}(y_1y_2y_3, y_2y_3, y_3,1) \in A.
\end{align*}
Then we have
\begin{align*}
& Z(s_1,s_2,W,\Phi_{(a,b)} ) \\
& = 2^2 
    \sum_{ \varepsilon_1,\varepsilon_2 \in \{\pm 1 \}} 
    \int_{0}^{\infty} \! \int_{0}^{\infty} \! \int_{0}^{\infty} \! \int_{0}^{\infty} \!
    \int_{0}^{2\pi} \! \int_{0}^{2\pi} 
  W \bigl(  {\rm diag} (\varepsilon_1, \varepsilon_2, 1, 1) 
   \, y \, \tilde{\iota}( \rk_{\theta_1}^{2)} ,\rk_{\theta_2}^{(2)}) \bigr) 
\\
& \quad \times y_1^{s_1-\frac32} y_2^{s_2-2} y_3^{s_1+s_2-\frac32} y_4^{2s_2+a+b} 
  \exp(-\pi y_4^2) \exp\{ \sqrt{-1}(a-b) \theta_2\} 
\,  \frac{d\theta_1}{2\pi} \frac{d\theta_2}{2\pi}
  \frac{dy_1}{y_1} \frac{dy_2}{y_2} \frac{dy_3}{y_3} \frac{dy_4}{y_4}.
\end{align*}
Since 
\begin{align*}
  W \bigl( {\rm diag} (\varepsilon_1, \varepsilon_2, 1, 1) 
   \,y\, \tilde{\iota}( \rk_{\theta_1}^{2)} ,\rk_{\theta_2}^{(2)}) \bigr)
= y_4^{\gamma_1} 
  W \bigl( {\rm diag} (\varepsilon_1, \varepsilon_2, 1, 1) 
   \, \hat{y} \, \tilde{\iota}( \rk_{\theta_1}^{2)} ,\rk_{\theta_2}^{(2)}) \bigr), 
\end{align*}
the formula 
\begin{align*}
 \int_{0}^{\infty} \exp(-\pi x^2) \, x^s \frac{dx}{x} 
= 2^{-1} \GR(s)
\end{align*}
implies that 
\begin{align*}
 Z(s_1,s_2, W, \Phi_{(a,b)})
& = 2 \GR(2s_2+\gamma_1+a+b) 
   \sum_{0 \le i \le 3} 
   \int_{0}^{\infty} \! \int_{0}^{\infty} \! \int_{0}^{\infty} \! \int_0^{2\pi} \! \int_0^{2\pi}
   W \bigl( m_i \, \hat{y} \, \tilde{\iota} (\rk_{\theta_1}^{(2)}, \rk_{\theta_2}^{(2)} ) \bigr)\\
& \quad \times 
   \exp\{ \sqrt{-1}(a-b) \theta_2\}  
   \, y_1^{s_1-\frac32} y_2^{s_2-2} y_3^{s_1+s_2-\frac32}
  \frac{d\theta_1}{2\pi} \frac{d\theta_2}{2\pi}
  \frac{dy_1}{y_1} \frac{dy_2}{y_2} \frac{dy_3}{y_3}.
\end{align*}
Here we defined $ m_i $ $ (0\le i \le 3) $ by 
\begin{align*}
 m_0 &= 1_4, 
& m_1 &= {\rm diag}(-1,1,1,1),
& m_2 &= {\rm diag}(1,-1,1,1),  
& m_3 &= m_1 m_2.
\end{align*}

\begin{lem} \label{lem:W'}
Retain the notation.
Let
\begin{align*}
& W'(\hat{y}) = 2^{-2}  
  \sum_{0 \le i \le 3} 
  \int_0^{2\pi} \! \int_0^{2\pi}
   W \bigl( m_i \, \hat{y} \, \tilde{\iota}(\rk_{\theta_1}^{(2)}, \rk_{\theta_2}^{(2)} ) \bigr) 
   \exp\{ \sqrt{-1}(a-b) \theta_2\} \, \frac{d\theta_1}{2\pi} \frac{d\theta_2}{2\pi}.
\end{align*}
Assume that $W'$ admits the Mellin-Barnes integral representaiton
$$
 W'( \hat{y} ) = y_1^{3/2} y_2^{2} y_3^{3/2} \cdot \dfrac{1}{(4\pi \sqrt{-1})^3} 
  \int_{s_3} \int_{s_2} \int_{s_1} 
  V'(s_1,s_2,s_3) \,y_1^{-s_1} y_2^{-s_2} y_3^{-s_3} \,ds_1 ds_2 ds_3,
$$
where the path of integration
$ \int_{s_i} $ is the vertical line from 
$ {\rm Re}(s_i)-\sqrt{-1}\infty $ to $ {\rm Re}(s_i) + \sqrt{-1} \infty $ with sufficiently large real part
to keep the poles of $ V'(s_1,s_2,s_3) $ on its left. 
Then we have 
\begin{align*}
 Z(s_1,s_2, W, \Phi_{(a.b)})  =  \GR(2s_2+\gamma_1+a+b) V'(s_1,s_2,s_1+s_2).
\end{align*}
\end{lem}

\begin{proof} 
It is immediate from Mellin inversion
(see \cite[Lemma 8.4]{HIM}).
\end{proof}

We  use the following formula to evaluate archimedean zeta integrals.

\begin{lem} \label{lem:Saal}
For $ a,b,c \in \bC $ and $ m \in \bZ_{\ge 0} $
such that $ {\rm Re}(a), {\rm Re}(b)>0 $ and $ {\rm Re}(c) > 2m $, it holds that 
\begin{align*} 
\begin{split}
& \sum_{0 \le j \le m} \binom{m}{j} \dfrac{\GR(a+2j) \GR(b+2j) \GR(c-2j) }{ \GR(a+b+c+2j-2m)} 
 =    \dfrac{ \GR(a) \GR(b) \GR(a+c) \GR(b+c) \GR(c-2m) }{ \GR(a+b+c) \GR(a+c-2m) \GR(b+c-2m)}.
\end{split}
\end{align*}
\end{lem}

\begin{proof}
Using the equalities  
\begin{align} \label{eqn:Pochhammer} 
 (z)_j & = \frac{\GR(2z+2j)}{ \pi^{-j} \GR(2z)} = (-1)^j \frac{ \pi^j \GR(2-2z)}{ \GR(2-2z-2j)} 
& (z \in \bC, \, j \in \bZ_{\ge 0}), 
\end{align}
and 
$ \binom{m}{j} =  (-1)^j  (-m)_j / j!$,
we have 
\begin{align*} 
\begin{split}
& \sum_{0 \le j \le m} \binom{m}{j} \dfrac{\GR(a+2j) \GR(b+2j) \GR(c-2j) }{ \GR(a+b+c+2j-2m)} 
   = \frac{\GR(a) \GR(b) \GR(c) }{ \GR(a+b+c-2m)}
   \, {}_3F_2\left(\begin{array}{c}
-m,\ \frac{a}{2}, \ \frac{b}{2} \smallskip \\
1-\frac{c}{2}, \ \frac{a+b+c}{2}-m 
\end{array};1\right)
\end{split}
\end{align*}
with the generalized hypergeometric series 
\begin{align*}
{}_3F_2\left(\begin{array}{c}
a_1,\ a_2, \ a_3 \\
b_1, \ b_2
\end{array};z\right)=\sum_{j=0}^\infty \frac{(a_1)_j(a_2)_j(a_3)_j}{(b_1)_j(b_2)_j}
\frac{z^j}{j!}.
\end{align*}
Applying the Saalsch\"utz's theorem (\cite[2.2 (1)]{Bailey})
\begin{align*}
{}_3F_2\left(\begin{array}{c}
-m,\ a, \ b \\
c, \ 1+a+b-c-m \end{array};1\right)
& = \frac{ (c-a)_m (c-b)_m}{(c)_m (c-a-b)_m} & (m \in \bZ_{\ge 0})
\end{align*}
and (\ref{eqn:Pochhammer}), we obtain the assertion.
\end{proof}


\subsection{Contragredient Whittaker functions}
\label{subsec:contra_Wh}

Let $ \widetilde{\Pi}_{\sigma} $ be the contragredient representation of $ \Pi_{\sigma} $. 
We have 
\begin{align*}
  {\rm Wh}(\widetilde{\Pi}_{\sigma}, \psi_{-1})^{\rm mg} = 
 \{ \widetilde{W} \mid W \in {\rm Wh}(\Pi_{\sigma}, \psi_1)^{\rm mg}  \},
\end{align*}
where the contragredient Whittaker function $ \widetilde{W} $ is defined by 
\begin{align*}
 \widetilde{W}(g) &:= W( w \,{}^t\!g^{-1}), & w = \begin{pmatrix} & & & 1 \\ & & 1 & \\ & 1 &  & \\ 1 & & & \end{pmatrix}.
\end{align*}
Let $ \tilde{\sigma} $ be a representation of $ M_{\mn} $ defined by replacing $ \nu_i $ to $-\nu_i $:
\begin{align*}
 \tilde{\sigma}
& = \begin{cases}
    \chi_{(-\nu_1,\delta_1)} \boxtimes \chi_{(-\nu_2,\delta_2)} \boxtimes 
    \chi_{(-\nu_3,\delta_3)} \boxtimes \chi_{(-\nu_4,\delta_4)} &  \text{if $\mn=(1,1,1,1) $}, \\
   D_{(-\nu_1, \kappa_1)} \boxtimes \chi_{(-\nu_2,\delta_2)} \boxtimes \chi_{(-\nu_3,\delta_3)} & \text{if $\mn=(2,1,1)$}, \\
   D_{(-\nu_1, \kappa_1)} \boxtimes D_{(-\nu_2, \kappa_2)} & \text{if $\mn=(2,2)$}.
    \end{cases}
\end{align*}
Then we know $ \widetilde{\Pi}_{\sigma} \cong \Pi_{\tilde{\sigma}} $
and the following:

\begin{prop} 
\label{prop:contra_radial}
For the $K$-homomorphisms $ \vp_{\sigma}: V_{(\kappa_1,\kappa_2,\delta_3)} \to {\rm Wh}(\Pi_{\sigma}, \psi_{1})^{\rm mg} $ 
given in Theorems \ref{thm:EF1}, \ref{thm:EF2} and \ref{thm:EF3}, 
let $ \widetilde{\varphi}_{\sigma}: V_{(\kappa_1,\kappa_2,\delta_3)} 
\to {\rm Wh}(\widetilde{\Pi}_{\sigma}, \psi_{-1})^{\rm mg} $ 
be a $ K$-homomorphism defined by
\begin{align*}
 \widetilde{\varphi}_{\sigma} (v)(g) & = \varphi_{\sigma}(v) ( w \,{}^t\!g^{-1})
 & (v \in V_{(\kappa_1,\kappa_2,\delta_3)}).  
\end{align*}
Then the radial part of $ \widetilde{\varphi}_{\sigma} $ is given by 
\begin{align*}
  \widetilde{\varphi}_{\sigma} (u_l)(y) 
& = y_1^{3/2} y_2^2 y_3^{3/2} y_4^{-\gamma_1} 
     \cdot (\sqrt{-1})^{l_2-l_4+l_{13}-l_{24}} (-1)^{\kappa_2 + l_3+l_{14}+l_{23}} 
\\
&  \quad \times \frac{1}{(4\pi \sqrt{-1})^3} \int_{s_3} \int_{s_2} \int_{s_1}
   V_{\tilde{\sigma}, l}(s_1,s_2,s_3) \, y_1^{-s_1} y_2^{-s_2} y_3^{-s_3} 
  \,ds_1ds_2ds_3
\end{align*}
with $  l = (l_1, l_2, l_3, l_4, l_{12},l_{13},l_{14},l_{23},l_{24},l_{34}) \in S_{(\kappa_1,\kappa_2,\delta_3)} $. 
\end{prop}

\begin{proof}
For $  l = (l_1, l_2, l_3, l_4, l_{12},l_{13},l_{14},l_{23},l_{24},l_{34}) \in S_{(\kappa_1,\kappa_2,\delta_3)} $, 
if we set 
\begin{align*}
  \tilde{ l} = (l_4, l_3, l_2, l_1, l_{34},l_{24},l_{23},l_{14},l_{13},l_{12}) \in S_{(\kappa_1,\kappa_2,\delta_3)}, 
\end{align*}
then we have 
\begin{align*}
  \widetilde{\varphi}_{\sigma} (u_l)(y) 
&  =  \varphi_{{\sigma}}(u_l)(w\,{}^t y^{-1} w^{-1} \cdot w ) \\
&  = \varphi_{\sigma} (\tau_{(\kappa_1,\kappa_2,\delta_3)}(w) u_l) (  w\,{}^t y^{-1} w^{-1} ) \\
&  = (-1)^{\kappa_2} \varphi_{\sigma} ( u_{\tilde{l}} ) 
   \left(  {\rm diag} ( y_4^{-1}, (y_3y_4)^{-1}, (y_2y_3y_4)^{-1}, (y_1y_2y_3y_4)^{-1} ) \right).
\end{align*}
Since the explicit formulas of $ V_{\sigma, l}(s_1,s_2,s_3) $ 
in Theorems \ref{thm:EF1}, \ref{thm:EF2} and \ref{thm:EF3} tell us that 
\begin{align*}
 V_{\sigma, \tilde{l}}(s_3-\gamma_1,  s_2-\gamma_1, s_1-\gamma_1)
& = V_{\tilde{\sigma}, l}(s_1,s_2,s_3),
\end{align*}
our claim follows.
\end{proof}

\begin{lem}
\label{lem:contra}
Retain the notation. Let  $ v \in V_{(\kappa_1,\kappa_2,\delta_3)} $. \\
\noindent 
(i) If $ W = \varphi_{\sigma}(v) $, then we have 
\begin{align*}
   \widetilde{W}(gk) & = 
 \widetilde{\varphi}_{\sigma} ( \tau_{(\kappa_1, \kappa_2, \delta_3)}(k) v)(g)
\end{align*}
for  $ (g,k) \in G\times K$, and 
\begin{align*}
  \widetilde{W}(m_py) & = \widetilde{\varphi}_{\sigma} ( \tau_{(\kappa_1, \kappa_2, \delta_3)}(m_p) v)(y)
\end{align*}
for $ 0 \le p \le 3 $ and  $y \in A $. \medskip 

\noindent
(ii)
If $ W = R(E_{i,j}^{\gp}) \varphi_{\sigma}(v) $ $(1\le i < j \le 4) $, then we have 
\begin{align*}
   \widetilde{W}(gk) & = - R( {\rm Ad}(k) (E_{i,j}^{\gp}))
 \widetilde{\varphi}_{\sigma} ( \tau_{(\kappa_1, \kappa_2, \delta_3)}(k) v)(g) 
\end{align*}
for $ (g,k) \in G\times K $, and
\begin{align*}
  \widetilde{W}(m_py) & = - R( {\rm Ad}(m_p) (E_{i,j}^{\gp}))
 \widetilde{\varphi}_{\sigma} ( \tau_{(\kappa_1, \kappa_2, \delta_3)}(m_p) v)(y), 
\\
   \widetilde{W}(y)
 &  = \begin{cases}
   4 \pi \sqrt{-1} y_i  \widetilde{\varphi}_{\sigma} (v) (y)
    +  \widetilde{\varphi}_{\sigma} (\tau_{(\kappa_1,\kappa_2,\delta_3)}( E_{i,j}^{\gk}) v) (y)
   & \text{if $ j=i+1 $}, \\
    \widetilde{\varphi}_{\sigma} (\tau_{(\kappa_1,\kappa_2,\delta_3)}( E_{i,j}^{\gk}) v) (y)
   & \text{if $ j>i+1 $} \end{cases}
\end{align*}
for $ 0 \le p \le 3 $ and $ y = {\rm diag}(y_1y_2y_3y_4, y_2y_3y_4, y_3y_4, y_4) \in A $.
\end{lem}

\begin{proof}
Our statement (i) is obvious.
Let us show (ii). 
We have
\begin{align*}
 \widetilde{W}(gk) 
&= W( w \, {}^t\! g^{-1} k) \\
& = \frac{d}{dt}\biggl|_{t=0} \varphi_{\sigma}(v) \big( w\, {}^t\! g^{-1} k \exp(t E_{i,j}^{\gp} ) k^{-1} \cdot k \big) \\
& = \frac{d}{dt}\biggl|_{t=0} \varphi_{\sigma}(\tau_{(\kappa_1,\kappa_2,\delta_3)}(k)v) 
    \big( w\, {}^t ( g k \exp( -t E_{i,j}^{\gp}) k^{-1} )^{-1}   \big) \\
& = - \frac{d}{dt}\biggl|_{t=0} \varphi_{\sigma}(\tau_{(\kappa_1,\kappa_2,\delta_3)}(k)v) 
    \big( w\, {}^t ( g  k \exp( t E_{i,j}^{\gp}) k^{-1} )^{-1}   \big) \\
& = - \frac{d}{dt}\biggl|_{t=0}  \widetilde{\varphi}_{\sigma}  (\tau_{(\kappa_1,\kappa_2,\delta_3)}(k)v) 
    \big(  g  k \exp( t E_{i,j}^{\gp}) k^{-1}   \big) \\
& = - R( {\rm Ad}(k) (E_{i,j}^{\gp}))
    \widetilde{\varphi}_{\sigma}  ( \tau_{(\kappa_1, \kappa_2, \delta_3)}(k) v)(g).
\end{align*}
Our claim for $ \widetilde{W}(m_p y) $ readily follows from the above. 
As for the radial part, in view of $ E_{i,j}^{\gp} = 2E_{j,i} + E_{i,j}^{\gk} $, we have 
\begin{align*}
 \widetilde{W}(y) 
&= 2 \frac{d}{dt} \biggl|_{t=0} 
    \varphi_{\sigma} (v) \bigl(w\,{}^t y^{-1} \exp (t E_{j,i}) \bigr) 
   +  \widetilde{\varphi}_{\sigma}  ( \tau_{(\kappa_1,\kappa_2,\delta_3)} (E_{i,j}^{\gk}) v)(y)  \\
& = 2 \frac{d}{dt} \biggl|_{t=0} 
    \varphi_{\sigma} (v) \bigl(w\,{}^t ( y \exp (-t E_{i,j}) )^{-1} \bigr) 
   +  \widetilde{\varphi}_{\sigma}  ( \tau_{(\kappa_1,\kappa_2,\delta_3)} (E_{i,j}^{\gk}) v)(y) \\
& = -2 \frac{d}{dt} \biggl|_{t=0} 
     \widetilde{\varphi}_{\sigma}  (v) (  y \exp (t E_{i,j}) )
   +  \widetilde{\varphi}_{\sigma} ( \tau_{(\kappa_1,\kappa_2,\delta_3)} (E_{i,j}^{\gk}) v)(y) \\
& = -2 R(E_{i,j})  \widetilde{\varphi}_{\sigma}  (v)(y) 
  +  \widetilde{\varphi}_{\sigma} ( \tau_{(\kappa_1,\kappa_2,\delta_3)} (E_{i,j}^{\gk}) v)(y)
\end{align*}
to get the assertion.
\end{proof}


\subsection{Test vectors and main results}
\label{subsec:test_and_main}

To compute $ W'(\hat{y}) $ in Lemma \ref{lem:W'}, 
we consider the actions of $ E_{1,3}^{\gk} $,  $ E_{2,4}^{\gk} $ 
and $ m_i $ $(i=0,1,2,3) $ on $ U(\gp_{\bC}) \otimes V_{\lambda} $.
For 
$ w = \sum_{i} c_i E_{p_i,q_i}^{\gp} \otimes u_{l_i} \in \gp_{\bC} \otimes V_{\lambda} $
$(c_i \in \bC, \, {l_i} \in S_{\lambda}) $
we write
$ \overline{w} = \sum_{i} \overline{c_i} E_{p_i,q_i}^{\gp} \otimes u_{l_i}. $

\begin{lem}[{\cite[Proposition 2.4]{Ishii_001}}]
 \label{lem:test1}
For $ \delta_1,\delta_2,\delta_3,\delta_4 \in \{0,1\} $ with 
$ \delta_1 \ge \delta_2 \ge \delta_3 \ge \delta_4 $,
we set $ \lambda = (\delta_1-\delta_4, \delta_2-\delta_3, \delta_3) $.
\begin{itemize}
\item[(a)] 
When $ (\delta_1,\delta_2,\delta_3,\delta_4) = (0,0,0,0) $, if we set 
$$ w = u_{{\bf 0}} \in V_{\lambda}, $$
then we have
\begin{align*}
  \tau_{\lambda}(E_{1,3}^{\gk}) w& = \tau_{\lambda}(E_{2,4}^{\gk}) w = 0, 
&  
 \tau_{\lambda}(m_i) w & = w \quad (0 \le i \le 3). 
\end{align*}
\item[(b)] 
When $ (\delta_1,\delta_2,\delta_3,\delta_4) = (1,0,0,0) $, if we set 
$$ w = u_{e_2} + \sqrt{-1} u_{e_4} \in V_{\lambda}, $$
then we have
\begin{align*}
 \tau_{\lambda}(E_{1,3}^{\gk}) w &= 0, &
 \tau_{\lambda}(E_{2,4}^{\gk}) w &= \sqrt{-1} w, & 
 \tau_{\lambda}(m_i) w & = 
  \begin{cases} w  & \mbox{ if } i=0,1, \\  -\overline{w} & \mbox{ if } i=2,3. \end{cases}
\end{align*}
\item[(c)] 
When $ (\delta_1,\delta_2,\delta_3,\delta_4) = (1,1,0,0) $, if we set
$$ 
 w = E_{1,2}^{\gp} \otimes u_{e_{12}} - E_{2,3}^{\gp} \otimes u_{e_{23}}
     +E_{3,4}^{\gp} \otimes u_{e_{34}} + E_{1,4}^{\gp} \otimes u_{e_{14}} \in 
    \gp_{\bC} \otimes V_{\lambda},
$$ 
then we have 
\begin{align*}
 ({\rm ad}\otimes \tau_{\lambda})(E_{1,3}^{\gp}) w 
& = ({\rm ad}\otimes \tau_{\lambda})(E_{2,4}^{\gp}) w = 0, &
 ({\rm Ad}\otimes \tau_{\lambda})(m_i) w = w \quad (0 \le i \le 3).
\end{align*}
\item[(d)]  
When $ (\delta_1,\delta_2,\delta_3,\delta_4) = (1,1,1,0) $, if we set 
$$
 w = (-E_{3,4}^{\gp}+\sqrt{-1} E_{2,3}^{\gp}) \otimes u_{e_1} + (E_{1,4}^{\gp} - \sqrt{-1} E_{1,2}^{\gp}) \otimes u_{e_3}
  \in \gp_{\bC} \otimes V_{\lambda},
$$
then we have 
\begin{gather*}
 ({\rm ad} \otimes \tau_{\lambda})(E_{1,3}^{\gk}) w = 0,  \quad 
 ({\rm ad} \otimes \tau_{\lambda})(E_{2,4}^{\gk}) w = \sqrt{-1} w,  \quad 
  ({\rm Ad} \otimes \tau_{\lambda})(m_i) w = 
  \begin{cases} w  & \mbox{ if } i=0,1, \\  -\overline{w} & \mbox{ if } i=2,3. \end{cases}
\end{gather*}
\item[(e)] 
When $ (\delta_1,\delta_2,\delta_3,\delta_4) = (1,1,1,1) $, if we set 
$$
 w = (E_{1,4}^{\gp} E_{2,3}^{\gp}-E_{1,2}^{\gp} E_{3,4}^{\gp} ) \otimes u_{\bf 0} \in U(\gp_{\bC}) \otimes V_{\lambda},
$$
then we have 
\begin{align*}
 ({\rm ad}\otimes \tau_{\lambda})(E_{1,3}^{\gp}) w 
= ({\rm ad}\otimes \tau_{\lambda})(E_{2,4}^{\gp}) w = 0, \quad 
 ({\rm Ad}\otimes \tau_{\lambda})(m_i) w = w \quad (0 \le i \le 3).
\end{align*}
\end{itemize}
\end{lem}

\begin{proof} 
Use Lemma \ref{lem:Kact_ul_vl} (i) and the formulas
\begin{align*}
 {\rm ad}(E_{1,3}^{\gk}) E_{p,q}^{\gp}
 &  = -\delta_{1,p} E_{3,q}^{\gp} - \delta_{1,q} E_{3,p}^{\gp} 
  + \delta_{3,p} E_{1,q}^{\gp} + \delta_{3,q} E_{1,p}^{\gp},
\\ 
 {\rm ad}(E_{2,4}^{\gk}) E_{p,q}^{\gp}
 &  = -\delta_{2,p} E_{4,q}^{\gp} - \delta_{2,q} E_{4,p}^{\gp} 
  + \delta_{4,p} E_{2,q}^{\gp} + \delta_{4,q} E_{2,p}^{\gp}, 
\\
 {\rm Ad}(m_i) E_{p,q}^{\gp}
 & = (-1)^{\delta_{i,p} + \delta_{i,q}} E_{p,q}^{\gp} \qquad (i=1,2) 
\end{align*}
for $ 1 \le p,q \le 4. $
\end{proof}

\begin{lem} \label{lem:test2}
For $ \kappa_1 \in \bZ_{\ge 2} $ and 
$ \delta_2, \delta_3 \in \{0,1\} $ with $ \delta_2 \ge \delta_3 $, 
we set
$ \lambda = (\kappa_1, \delta_2-\delta_3, \delta_3) $.
For $ 1 \le p,q,r \le 4 $ we define $ w_0, w_p, w_{p,q}, w_{pq}, w_{p,qr} \in V_{\lambda} $ 
as follows:
\begin{itemize}
\item
When $ \kappa_1 $ is even and $ \delta_2 = \delta_3 $, we put 
\begin{align*}
w_{0} 
:=  \rrq_\cR \bigl( (( \xi_1)^2 + (\xi_3)^2)^{\kappa_1/2} \bigr),  \quad 
 w_{p,q} 
: = \rrq_\cR \bigl( ( (\xi_1)^2 + (\xi_3)^2)^{(\kappa_1-2)/2} \xi_{p} \xi_{q} \bigr).
\end{align*}
\item 
When $ \kappa_1 $ is even and $ (\delta_2,\delta_3)=(1,0) $, we put 
\begin{align*}
w_{p, qr} := \rrq_\cR \bigl( ((\xi_1)^2 + (\xi_3)^2)^{(\kappa_1-2)/2} \xi_{p}\xi_{qr} \bigr). 
\end{align*}
\item
When $ \kappa_1 $ is odd and $ \delta_2 = \delta_3 $, we put 
\begin{align*}
w_{p} 
& := \rrq_\cR \bigl( ( (\xi_1)^2 + (\xi_3)^2 \bigr)^{(\kappa_1-1)/2} \xi_{p}).
\end{align*}
\item 
When $ \kappa_1 $ is odd and $ (\delta_2,\delta_3)=(1,0) $, we put 
\begin{align*}
 w_{pq}
& :=  \rrq_\cR \bigl(( (\xi_1)^2 + (\xi_3)^2)^{(\kappa_1-1)/2} \xi_{pq} \bigr).
\end{align*}
\end{itemize}
Then we have the following.
\begin{itemize}
\item[(a)]
When $ \kappa_1 $ is even and $ (\delta_2,\delta_3) = (0,0) $, if we set 
$$ w = w_0 \in V_{\lambda}, $$
then we have
\begin{align*}
  \tau_{\lambda}(E_{1,3}^{\gk}) w & = \tau_{\lambda}(E_{2,4}^{\gk}) w = 0, 
 &      
 \tau_{\lambda}(m_i) w & = w \quad (0 \le i \le 3). 
\end{align*}
\item[(b)]
When $ \kappa_1 $ is odd and $ (\delta_2,\delta_3) = (0,0) $, if we set 
$$ w = w_2+\sqrt{-1} w_4 \in V_{\lambda}, $$
then we have
\begin{align*}
 \tau_{\lambda}(E_{1,3}^{\gk}) w &= 0, & 
 \tau_{\lambda}(E_{2,4}^{\gk}) w &= \sqrt{-1} w, &
 \tau_{\lambda}(m_i) w & = 
  \begin{cases} w  & \mbox{ if } i=0,1, \\  -\overline{w} & \mbox{ if } i=2,3. \end{cases}
\end{align*}
\item[(c)] 
When $ \kappa_1 $ is even and $ (\delta_2,\delta_3) = (1,1) $, if we set 
$$ 
 w = E_{1,2}^{\gp} \otimes w_{3,4} - E_{2,3}^{\gp} \otimes w_{1,4}
     +E_{3,4}^{\gp} \otimes w_{1,2} - E_{1,4}^{\gp} \otimes w_{2,3} \in 
    \gp_{\bC} \otimes V_{\lambda},
$$ 
then we have
\begin{align*}
  ({\rm ad} \otimes \tau_{\lambda})(E_{1,3}^{\gk}) w 
&= ({\rm ad} \otimes \tau_{\lambda})(E_{2,4}^{\gk}) w = 0, 
& 
 ({\rm Ad} \otimes \tau_{\lambda})(m_i) w &= w \quad (0 \le i \le 3). 
\end{align*}
\item[(d)] 
When $ \kappa_1 $ is odd and $ (\delta_2,\delta_3) = (1,1) $, if we set 
$$
 w = (-E_{3,4}^{\gp}+\sqrt{-1} E_{2,3}^{\gp}) \otimes w_1 + (E_{1,4}^{\gp} - \sqrt{-1} E_{1,2}^{\gp}) w_3
  \in \gp_{\bC} \otimes V_{\lambda},
$$
then we have 
\begin{gather*}
 ({\rm ad} \otimes \tau_{\lambda})(E_{1,3}^{\gk}) w = 0, \quad  
 ({\rm ad} \otimes \tau_{\lambda})(E_{2,4}^{\gk}) w = \sqrt{-1} w, \quad 
 ({\rm Ad} \otimes \tau_{\lambda})(m_i) w  = 
  \begin{cases} w  & \mbox{ if } i=0,1, \\  -\overline{w} & \mbox{ if } i=2,3. \end{cases}
\end{gather*}
\item[(e)] 
When $ \kappa_1 $ is even and $ (\delta_2,\delta_3) = (1,0) $, if we set  
$$
 w = - \sqrt{-1} w_{2,24} + w_{4,24} \in V_{\lambda}, 
$$
then we have 
\begin{align*}
 \tau_{\lambda}(E_{1,3}^{\gk}) w &= 0, & 
 \tau_{\lambda}(E_{2,4}^{\gk}) w &= \sqrt{-1} w, &
 \tau_{\lambda}(m_i) w & = 
  \begin{cases} w  & \mbox{ if } i=0,1, \\  -\overline{w} & \mbox{ if } i=2,3. \end{cases}
\end{align*}
\item[(f)]
When $ \kappa_1 $ is odd and $ (\delta_2,\delta_3) = (1,0) $, if we set 
$$ 
 w = E_{1,2}^{\gp} \otimes w_{12} - E_{2,3}^{\gp} \otimes w_{23}
     +E_{3,4}^{\gp} \otimes w_{34} + E_{1,4}^{\gp} \otimes w_{14} \in 
    \gp_{\bC} \otimes V_{\lambda},
$$ 
then we have
\begin{align*}
  ({\rm ad} \otimes \tau_{\lambda})(E_{1,3}^{\gk}) w 
&= ({\rm ad} \otimes \tau_{\lambda})(E_{2,4}^{\gk}) w = 0, 
& 
 ({\rm Ad} \otimes \tau_{\lambda})(m_i) w &= w \quad (0 \le i \le 3). 
\end{align*}
\end{itemize}
\end{lem}

\begin{proof}
Since  
$ \cT(E_{1,3}^{\gk}) (( (\xi_1)^2 + (\xi_3)^2)^n ) = \cT(E_{2,4}^{\gk}) (( (\xi_1)^2 + (\xi_3)^2)^n )  = 0 $
$ (n\in \bZ_{\ge 0}) $,
we know that 
\begin{align*}
 \tau_{(\kappa_1,0,\delta_3)}(E_{1,3}^{\gk}) w_0 &= \tau_{(\kappa_1,0,\delta_3)}(E_{2,4}^{\gk}) w_0  = 0, \\
 \tau_{(\kappa_1,0,\delta_3)}(E_{1,3}^{\gk}) w_p &= \delta_{3,p} w_1 - \delta_{1,p} w_3, \\
 \tau_{(\kappa_1,0,\delta_3)}(E_{2,4}^{\gk}) w_p &= \delta_{4,p} w_2 - \delta_{2,p} w_4, \\
 \tau_{(\kappa_1,0,\delta_3)}(E_{1,3}^{\gk}) w_{p,q}
 &= \delta_{3,p} w_{1,q} -\delta_{1,p} w_{3,q} + \delta_{3,q} w_{p,1} - \delta_{1,q} w_{p,3}, \\
 \tau_{(\kappa_1,0,\delta_3)}(E_{2,4}^{\gk}) w_{p,q}
 &= \delta_{4,p} w_{2,q} -\delta_{2,p} w_{4,q} + \delta_{4,q} w_{p,2} - \delta_{2,q} w_{p,4}, \\
 \tau_{(\kappa_1,1,0)}(E_{1,3}^{\gk}) w_{2,24} &= \tau_{(\kappa_1,1,0)}(E_{1,3}^{\gk}) w_{4,24}  = 0, \\ 
 \tau_{(\kappa_1,1,0)}(E_{2,4}^{\gk}) w_{2,24} &=  -w_{4,24}, \ \ 
 \tau_{(\kappa_1,1,0)}(E_{2,4}^{\gk}) w_{4,24}  = w_{2,24}, \\ 
 \tau_{(\kappa_1,1,0)}(E_{1,3}^{\gk}) w_{pq} 
 & = \delta_{3,p} w_{1q} - \delta_{1,p} w_{3q} + \delta_{3,q} w_{p1} - \delta_{1,q} w_{p3},	\\
 \tau_{(\kappa_1,1,0)}(E_{2,4}^{\gk}) w_{pq} 
 & = \delta_{4,p} w_{2q} - \delta_{2,p} w_{4q} + \delta_{4,q} w_{p2} - \delta_{2,q} w_{p4}.
\end{align*}
Note that $ w_{p,q}= w_{q,p} $ and $ w_{pq} = -w_{qp}. $ 
Combined with the formulas in the proof of Lemma \ref{lem:test1}, we can get the assertion.
\end{proof}

\begin{lem} \label{lem:test3}
For $ \kappa_1, \kappa_2 \in \bZ_{\ge 2} $ with $ \kappa_1 \ge \kappa_2 $, 
we set $ \lambda = (\kappa_1, \kappa_2,0) $.
For $ 1 \le p,q \le 4 $ we define $ w_0, w_p, w_{pq} \in V_{\lambda} $ as follows:
\begin{itemize}
\item
When $ \kappa_1 - \kappa_2 $ is even,  we put
\begin{align*}
 w_0 
 := \rrq_\cR \bigl( ((\xi_1)^2+ (\xi_3)^2)^{(\kappa_1-\kappa_2)/2} \xi_{24}^{\kappa_2} \bigr),
\quad w_{pq} 
:= \rrq_\cR \bigl( ((\xi_1)^2+ (\xi_3)^2)^{(\kappa_1-\kappa_2)/2} \xi_{pq} \xi_{24}^{\kappa_2-1} \bigr).
\end{align*}
\item
When $ \kappa_1 - \kappa_2$ is odd, we put
\begin{align*}
w_p
& := \rrq_\cR \bigl(  ((\xi_1)^2+(\xi_3)^2)^{(\kappa_1-\kappa_2-1)/2} \xi_p \xi_{24}^{\kappa_2} \bigr).
\end{align*}
\end{itemize}
Then we have the following.
\begin{itemize}
\item[(a)] 
When $ \kappa_1 $ and $ \kappa_2 $ are even, if we set
$$ w = w_0 \in V_{\lambda}, $$
then we have 
\begin{align*}
  \tau_{\lambda}(E_{1,3}^{\gk}) w & = \tau_{\lambda}(E_{2,4}^{\gk}) w = 0, 
& \tau_{\lambda}(m_i) w = w \quad (0 \le i \le 3). 
\end{align*}
\item[(b)] 
When $ \kappa_1 $ and $ \kappa_2 $ are odd, if we set
$$ 
 w = E_{1,2}^{\gp} \otimes w_{{12}} - E_{2,3}^{\gp} \otimes w_{{23}}
     +E_{3,4}^{\gp} \otimes w_{{34}} + E_{1,4}^{\gp} \otimes w_{{14}} \in 
    \gp_{\bC} \otimes V_{\lambda},
$$ 
then we have
\begin{align*}
  ({\rm ad}\otimes\tau_{\lambda})(E_{1,3}^{\gk}) w & 
  = ({\rm ad}\otimes\tau_{\lambda})(E_{2,4}^{\gk}) w = 0,  &
 ({\rm Ad} \otimes \tau_{\lambda})(m_i) w = w \quad (0 \le i \le 3). 
\end{align*}
\item[(c)]
When $ \kappa_1-\kappa_2 $ is odd, if we set
$$ 
 w =  w_2 + \sqrt{-1} w_4 \in V_{\lambda},
$$ 
then we have
\begin{align*}
 \tau_{\lambda}(E_{1,3}^{\gk}) w &= 0, & 
 \tau_{\lambda}(E_{2,4}^{\gk}) w &= \sqrt{-1} w, &
 \tau_{\lambda}(m_i) w & = 
  \begin{cases} w  & \mbox{ if } i=0,1, \\  (-1)^{\kappa_1} \overline{w} & \mbox{ if } i=2,3. \end{cases}
\end{align*}
\end{itemize}
\end{lem}

\begin{proof}
Note that 
\begin{align*}
 \tau_{(\kappa_1,\kappa_2,0)}(E_{1,3}^{\gk}) w_p 
& = \delta_{3,p} w_1 - \delta_{1,p} w_3,
\\
\tau_{(\kappa_1,\kappa_2,0)}(E_{2,4}^{\gk}) w_p 
& = \delta_{4,p} w_2 - \delta_{2,p} w_4,
\\
 \tau_{(\kappa_1,\kappa_2,0)}(E_{1,3}^{\gk}) w_{pq} 
& =   \delta_{3,p} w_{1q} - \delta_{1,p} w_{3q} + \delta_{3,q} w_{p1} - \delta_{1,q} w_{p3},
\\
  \tau_{(\kappa_1,\kappa_2,0)}(E_{2,4}^{\gk}) w_{pq} 
& =   \delta_{4,p} w_{2q} - \delta_{2,p} w_{4q} + \delta_{4,q} w_{p2} - \delta_{2,q} w_{p4}.
\end{align*}
\end{proof}

Now we can state our main results for archimedean zeta integrals.

\begin{thm}[{\cite[Theorem 4.1]{Ishii_001}}] \label{thm:BF1} 
Let $ \sigma = \chi_{(\nu_1,\delta_1)} \boxtimes \chi_{(\nu_2,\delta_2)} 
  \boxtimes \chi_{(\nu_3,\delta_3)} \boxtimes \chi_{(\nu_4,\delta_4)} $ 
with $ \nu_1,\nu_2,\nu_3,\nu_4 \in \bC $, 
$ \delta_1,\delta_2,\delta_3,\delta_4 \in \{0,1\} $ and 
$ \delta_1 \ge \delta_2 \ge \delta_3 \ge \delta_4 $ 
such that $ \Pi_{\sigma} $ is irreducible.
Let $ \varphi_{\sigma}: V_{(\delta_1-\delta_4, \delta_2-\delta_3, \delta_3)} \to {\rm Wh}(\Pi_{\sigma}, \psi_1)^{\rm mg} $
be the $K$-homomorphism given in Theorem \ref{thm:EF1}.
We define $ W \in {\rm Wh}(\Pi_{\sigma}, \psi_1)^{\rm mg} $ and $ \Phi \in \mathcal{S}(\bR^2) $ 
as follows.
\begin{itemize}
\item Case 1-(a): 
When $ (\delta_1,\delta_2,\delta_3,\delta_4) = (0,0,0,0) $, we set
$$ W = \varphi_{\sigma}(u_{\bf 0}). $$
\item Case 1-(b): 
When $ (\delta_1,\delta_2,\delta_3,\delta_4) = (1,0,0,0) $, we set
$$ W = (-\sqrt{-1}) \varphi_{\sigma}( u_{e_2} +\sqrt{-1} u_{e_4}). $$
\item Case 1-(c): 
When $ (\delta_1,\delta_2,\delta_3,\delta_4) = (1,1,0,0) $, we set
\begin{align*}
 W & = (-\sqrt{-1}) (4\pi)^{-1}
    \{ R(E_{1,2}^{\gp}) \vp_{\sigma}(u_{e_{12}})  - R( E_{2,3}^{\gp}) \vp_{\sigma}(u_{e_{23}}) 
   + R(E_{3,4}^{\gp}) \vp_{\sigma}(u_{e_{34}}) + R(E_{1,4}^{\gp}) \vp_{\sigma}( u_{e_{14}}) \}.
\end{align*}
\item Case 1-(d): 
When $ (\delta_1,\delta_2,\delta_3,\delta_4) = (1,1,1,0) $, we set
\begin{align*}
 W & = (-\sqrt{-1}) (4\pi)^{-1}
    \{ R(-E_{3,4}^{\gp}+\sqrt{-1} E_{2,3}^{\gp}) \vp_{\sigma}(u_{e_1})  
  + R(E_{1,4}^{\gp} - \sqrt{-1} E_{1,2}^{\gp}) \vp_{\sigma}(u_{e_3}) \}.
\end{align*}
\item Case 1-(e):
When $ (\delta_1,\delta_2,\delta_3,\delta_4) = (1,1,1,1) $, we set
$$
  W = (4 \pi)^{-2} R(E_{1,4}^{\gp} E_{2,3}^{\gp}-E_{1,2}^{\gp} E_{3,4}^{\gp}) \vp_{\sigma}(u_{\bf 0}).
$$
\end{itemize}
Here $ R $ is the right differential.
Further we choose $ b \in \{0,1\} $ such that 
$ b \equiv \delta_1+\delta_2+\delta_3+ \delta_4 \pmod{2}$,  
and set $ \Phi = \Phi_{(0,b)} $.
Then we have 
\begin{align*}
  Z(s_1,s_2,W,\Phi) &= L(s_1,\Pi_{\sigma}) L(s_2,\Pi_{\sigma}, \wedge^2),
\\
 Z(1-s_1,1-s_2,\widetilde{W}, \widehat{\Phi}) & = 
 \varepsilon(s_1,\Pi_{\sigma},\psi_{\bR}) \varepsilon(s_2, \Pi_{\sigma},\wedge^2, \psi_{\bR}) 
 L(1-s_1, \widetilde{\Pi}_{\sigma}) L(1-s_2, \widetilde{\Pi}_{\sigma}, \wedge^2).  
\end{align*}
\end{thm}


\begin{thm} \label{thm:BF2}
Let $ \sigma = D_{(\nu_1,\kappa_1)} \boxtimes \chi_{(\nu_2,\delta_2)} 
  \boxtimes \chi_{(\nu_3,\delta_3)}  $ 
with $ \nu_1,\nu_2,\nu_3\in \bC $, $ \kappa_1 \in \bZ_{\ge 2} $, 
$ \delta_2, \delta_3 \in \{0,1\} $ and 
$ \delta_2 \ge \delta_3 $ 
such that $ \Pi_{\sigma} $ is irreducible.
We define $ \delta_1 \in \{0, 1\} $ by $ \delta_1 \equiv \kappa_1 \pmod 2 $.
Let $ \varphi_{\sigma} : V_{(\kappa_1, \delta_2-\delta_3, \delta_3)} \to {\rm Wh}(\Pi_{\sigma}, \psi_1)^{\rm mg} $
be the $K$-homomorphism given in Theorem \ref{thm:EF2}.
Under the notation in Lemma \ref{lem:test2}, 
we define $ W \in {\rm Wh}(\Pi_{\sigma}, \psi_1)^{\rm mg} $ and $ \Phi \in \mathcal{S}(\bR^2) $ 
as follows.
\begin{itemize}
\item
Case 2-(a): 
When $ (\delta_1, \delta_2,\delta_3) = (0, 0,0) $, we set
$$ W  = (-\sqrt{-1})^{\kappa_1} \vp_{\sigma} (w_0).   $$
\item 
Case 2-(b): 
When $ (\delta_1, \delta_2,\delta_3) = (1, 0,0) $, we set
$$ W = (-\sqrt{-1})^{\kappa_1} \vp_{\sigma} (w_2+\sqrt{-1} w_4). $$
\item 
Case 2-(c): 
When $ (\delta_1, \delta_2,\delta_3) = (0,1,1) $, we set
\begin{align*}
 W & = (-\sqrt{-1})^{\kappa_1} (4\pi)^{-1}
     \{ R(E_{1,2}^{\gp}) \vp_{\sigma}(w_{3,4}) - R(E_{2,3}^{\gp}) \vp_{\sigma} (w_{1,4})
   +R(E_{3,4}^{\gp}) \vp_{\sigma}(w_{1,2}) - R(E_{1,4}^{\gp}) \vp_{\sigma}(w_{2,3}) \}.
\end{align*}
\item 
Case 2-(d): 
When $ (\delta_1, \delta_2,\delta_3) = (1, 1,1) $, we set
\begin{align*}
 W & =  (-\sqrt{-1})^{\kappa_1} (4\pi)^{-1}
 \{ R(-E_{3,4}^{\gp}+\sqrt{-1} E_{2,3}^{\gp}) \vp_{\sigma}( w_1)  
   + R(E_{1,4}^{\gp} - \sqrt{-1} E_{1,2}^{\gp}) \vp_{\sigma}(w_3) \}.
\end{align*}
\item 
Case 2-(e):
When $ (\delta_1, \delta_2,\delta_3) = (0, 1,0) $, we set
\begin{align*}
 W & = (-\sqrt{-1})^{\kappa_1} \vp_{\sigma} (- \sqrt{-1} w_{2,24} + w_{4,24}).
\end{align*}
\item 
Case 2-(f): 
When $ (\delta_1, \delta_2,\delta_3) = (1,1,0) $, we set
\begin{align*} 
 W & =  (-\sqrt{-1})^{\kappa_1} (4\pi)^{-1}  \{ R(E_{1,2}^{\gp}) \vp_{\sigma}(w_{12}) - R(E_{2,3}^{\gp}) \vp_{\sigma}(w_{23}) 
   +R(E_{3,4}^{\gp}) \vp_{\sigma}(w_{34}) +R(E_{1,4}^{\gp}) \vp_{\sigma}(w_{14}) \}.
\end{align*}
\end{itemize}
Here $ R $ is the right differential.
Further we choose $ b \in \{0,1\} $ such that 
$ b \equiv \delta_1+\delta_2+\delta_3 \pmod{2}$, 
and set $ \Phi = \Phi_{(0,b)} $. Then we have
\begin{align*}
  Z(s_1,s_2,W,\Phi) &= L(s_1,\Pi_{\sigma}) L(s_2,\Pi_{\sigma}, \wedge^2),
\\
 Z(1-s_1,1-s_2,\widetilde{W}, \widehat{\Phi}) & = 
 \varepsilon(s_1,\Pi_{\sigma},\psi_{\bR}) \varepsilon(s_2, \Pi_{\sigma},\wedge^2, \psi_{\bR}) 
 L(1-s_1, \widetilde{\Pi}_{\sigma}) L(1-s_2, \widetilde{\Pi}_{\sigma}, \wedge^2).  
\end{align*}
\end{thm}


\begin{thm} \label{thm:BF3}
Let $ \sigma = D_{(\nu_1,\kappa_1)} \boxtimes D_{(\nu_2,\kappa_2)} $ 
with $ \nu_1,\nu_2  \in \bC $, $ \kappa_1, \kappa_2 \in \bZ_{\ge 2} $
and $ \kappa_1 \ge \kappa_2 $ 
such that $ \Pi_{\sigma} $ is irreducible.
We define $ \delta_1, \delta_2 \in \{0,1\} $ by $ \delta_1 \equiv \kappa_1 $,  $\delta_2 \equiv \kappa_2 \pmod 2$. 
Let $ \varphi_{\sigma} : V_{(\kappa_1, \kappa_2, 0)} \to {\rm Wh}(\Pi_{\sigma}, \psi_1)^{\rm mg} $
be the $K$-homomorphism given in Theorem \ref{thm:EF3}.
Under the notation in Lemma \ref{lem:test3}, 
we define $ W \in {\rm Wh}(\Pi_{\sigma}, \psi_1)^{\rm mg} $ and $ \Phi \in \mathcal{S}(\bR^2) $ 
as follows.
\begin{itemize}
\item 
Case 3-(a): 
When $(\delta_1,\delta_2)=(0,0) $, we set
\begin{align*}
 W = (-\sqrt{-1})^{\kappa_1} \vp_{\sigma}(w_0).
\end{align*}
\item 
Case 3-(b):
When $(\delta_1,\delta_2)=(1,1) $, we set
\begin{align*}
 W & = (-\sqrt{-1})^{\kappa_1}  (4\pi)^{-1} \{R( E_{1,2}^{\gp}) \vp_{\sigma}(w_{12}) - R(E_{2,3}^{\gp}) \vp_{\sigma}(w_{23}) 
      +R( E_{3,4}^{\gp} ) \vp_{\sigma}(w_{34}) +R( E_{1,4}^{\gp}) \vp_{\sigma}(w_{14}) \}.
\end{align*}
\item
Case 3-(c):
When $\delta_1 \neq \delta_2 $, we set
\begin{align*}
  W=  (-\sqrt{-1})^{\kappa_1+\delta_2} \vp_{\sigma}(w_2 + \sqrt{-1} w_4).
\end{align*}
\end{itemize}
Here $ R $ is the right differential.
Further we choose $ b \in \{0,1\} $ such that 
$ b \equiv \kappa_1+\kappa_2 \pmod{2} $, 
and set $ \Phi = \Phi_{(0,b)} $. Then we have
\begin{align*}
  Z(s_1,s_2,W,\Phi) &= L(s_1,\Pi_{\sigma}) L(s_2,\Pi_{\sigma}, \wedge^2),
\\
 Z(1-s_1,1-s_2,\widetilde{W}, \widehat{\Phi}) & = 
 \varepsilon(s_1,\Pi_{\sigma},\psi_{\bR}) \varepsilon(s_2, \Pi_{\sigma},\wedge^2, \psi_{\bR}) 
  L(1-s_1, \widetilde{\Pi}_{\sigma}) L(1-s_2, \widetilde{\Pi}_{\sigma}, \wedge^2).  
\end{align*}
\end{thm}

In the next three subsections we prove Theorems \ref{thm:BF1}, \ref{thm:BF2} and \ref{thm:BF3}. 


\subsection{Proof of Theorem \ref{thm:BF1}} 
\label{subsec:BF_case1}

By Lemma \ref{lem:test1}, we know that
\begin{align*}
W'(\hat{y}) 
& = 2^{-2} \sum_{0 \le i \le 3} 
  \int_0^{2\pi} \! \int_0^{2\pi}
   W \bigl( m_i \, \hat{y} \, \tilde{\iota}(\rk_{\theta_1}^{(2)}, \rk_{\theta_2}^{(2)} ) \bigr) 
   \exp(- \sqrt{-1} b\, \theta_2) \, \frac{d\theta_1}{2\pi} \frac{d\theta_2}{2\pi}
\\
& = 
\begin{cases}
   W(\hat{y}) & \mbox{cases 1-(a), (c), (e)}, \\
   \varphi_{\sigma} (u_{e_4})(\hat{y}) & \mbox{case 1-(b)}, \\
   (4\pi)^{-1} \{ R(E_{2,3}^{\gp}) \varphi_{\sigma}(u_{e_1})(\hat{y}) 
    -R(E_{1,2}^{\gp}) \varphi_{\sigma}(u_{e_3})(\hat{y}) \} & \mbox{case 1-(d)}.
 \end{cases}
\end{align*}
Then Lemma \ref{lem:W'} implies that  
\begin{align} \label{eqn:lem_W'_case1}
 Z(s_1,s_2,W, \Phi_{(0,b)}) 
= \GR(2s_2+\gamma_1+b) V'(s_1,s_2,s_1+s_2).
\end{align}
Here we use the notation $ \gamma_1 = \nu_1+\nu_2+\nu_3+\nu_4. $
We express the Mellin-Barnes kernel $ V'(s_1,s_2,s_3) $ of $   W(\hat{y}) $ 
in terms of $ V_{\sigma, l}(s_1,s_2,s_3) $ given in Theorem \ref{thm:EF1}, and compute the zeta integral
by using the following lemma.

\begin{lem} \label{lem:BFlem1}
For $ s_1,s_2 \in \bC $ and $ \mu = (\mu_1,\mu_2,\mu_3,\mu_4) \in \bC^4 $, we set 
\begin{align*}
 A(a_1,a_2,a_3,a_4,a_5,a_6,a_7;\mu) 
& = \GR(2s_2+\mu_1+\mu_2+\mu_3+\mu_4+a_1) \GR(s_1+\mu_1+a_2) \GR(s_1+\mu_2+a_3)   
\\
& \quad \times 
  \GR(s_2+\mu_1+\mu_2+a_4) \GR(s_2+\mu_3+\mu_4+a_5) \\
& \quad \times \GR(s_1+s_2+\mu_1+\mu_3+\mu_4+a_6) \GR(s_1+s_2+\mu_2+\mu_3+\mu_4+a_7),
\end{align*}
\begin{align*}
 B_0(b_1,b_2,b_3,b_4,b_5,b_6; b_7,b_8; \mu) 
& = \frac{\GR(s_1-q+b_1)\GR(s_2-q+\mu_1+b_2) \GR(s_2-q+\mu_2+b_3) }
   {\GR(s_1+s_2-q+\mu_1+\mu_2+b_7) \GR(s_1+2s_2-q+\mu_1+\mu_2+\mu_3+\mu_4+b_8)} 
\\
& \quad \times \GR(s_1+s_2-q+\mu_1+\mu_2+b_4) \GR(q+\mu_3+b_5) \GR(q+\mu_4+b_6)
\end{align*}
and 
\begin{align*}
 {B}(b_1,b_2,b_3,b_4,b_5,b_6; b_7,b_8;\mu) 
& =\frac{1}{4\pi \sqrt{-1}} \int_q   B_0(b_1,b_2,b_3,b_4,b_5,b_6;b_7, b_8;\mu) \,dq.
\end{align*}
When $ b_4 = b_7$, 
$ B_0(b_1,b_2,b_3,b_4,b_5,b_6; b_7,b_8; \mu)  $ and 
$ {B}(b_1,b_2,b_3,b_4,b_5,b_6; b_7,b_8; \mu)  $ are denoted
by $ B_0(b_1,b_2,b_3, b_5,b_6; b_8; \mu)  $ and $ B(b_1,b_2,b_3, b_5,b_6; b_8; \mu), $ respectively.
If $ b_8 = b_1+b_2+b_3+b_5+b_6 $, $ a_6 = b_1+b_2+b_5+b_6$ and 
$ a_7 = b_1 + b_3+ b_5 +b_6 $, then we have 
\begin{align*}
& A(a_1,a_2,a_3,a_4,a_5,a_6,a_7;\mu)   {B} (b_1,b_2,b_3,b_5,b_6; b_8;\mu)
\\
& = \frac{\GR(2s_2+\mu_1+\mu_2+\mu_3+\mu_4+a_1)}{\GR(2s_2+\mu_1+\mu_2+\mu_3+\mu_4+b_2+b_3+b_5+b_6)}
\\
& \quad \times L_1(a_2, a_3,b_1+b_5,b_1+b_6;\mu) L_2(a_4,b_2+b_5,b_2+b_6,b_3+b_5,b_3+b_6, a_5;\mu),
\end{align*}
where 
\begin{align*}
 L_1(\alpha_1,\alpha_2,\alpha_3,\alpha_4;\mu) &= \prod_{1 \le i \le 4}  \GR(s_1+ \mu_i+\alpha_i ), \\
 L_2(\alpha_{1,2},\alpha_{1,3},\alpha_{1,4},\alpha_{2,3},\alpha_{2,4}, \alpha_{3,4};\mu) &= \prod_{1 \le i <j \le 4} 
  \GR(s_2+\mu_i + \mu_j + \alpha_{i,j}).
\end{align*}
\end{lem}

\begin{proof}
It is immediate from Lemma \ref{lem:Barnes2nd}.
\end{proof}

Let us compute the zeta integral $ Z(s_1,s_2,W,\Phi). $
We write $ \nu = (\nu_1,\nu_2,\nu_3,\nu_4) $ and $ \nu' = (\nu_4,\nu_2,\nu_3,\nu_1) $.  
\medskip 

\noindent 
\underline{\bf Case 1-(a):} \
$ V'(s_1,s_2,s_3) = V_{\sigma, {\bf 0}}(s_1,s_2,s_3) $, 
(\ref{eqn:lem_W'_case1}), Theorem \ref{thm:EF1} (i) and Lemma \ref{lem:BFlem1} with $ \mu = \nu $ imply that 
\begin{align*}
 Z(s_1,s_2, W, \Phi) 
& = \GR(2s_2+\gamma_1) V_{\sigma, \bf 0}(s_1,s_2,s_1+s_2) \\
& = A(0,0,0,0,0,0,0;\nu) {B}(0,0,0,0,0; 0;\nu) \\
& = L_1(0,0,0,0;\nu) L_2(0,0,0,0,0,0;\nu).
\end{align*}

\medskip  

\noindent 
\underline{\bf Case 1-(b):} \
$ V'(s_1,s_2,s_3) = V_{\sigma, e_4}(s_1,s_2,s_3)$, 
(\ref{eqn:lem_W'_case1}),
Theorem \ref{thm:EF1} (ii) and Lemma \ref{lem:BFlem1} with $ \mu= \nu $ imply that 
\begin{align*}
 Z(s_1,s_2, W, \Phi) 
& = \GR(2s_2+\gamma_1+1) V_{\sigma, e_4}(s_1,s_2,s_1+s_2) \\
& = A(1,1,0,1,0,1,0;\nu) B(0,1,0,0,0; 1;\nu) \\
& = L_1(1,0,0,0;\nu) L_2(1,1,1,0,0,0;\nu).
\end{align*}

\medskip 

\noindent 
\underline{\bf Case 1-(c):} \
By using $ E_{i,j}^{\gp} = 2E_{i,j} - E_{i,j}^{\gk} $, we know 
\begin{align*}
 W'(\hat{y})
& = (4\pi \sqrt{-1})^{-1} 
\{ R(2E_{1,2}) \vp_{\sigma}(u_{e_{12}})(\hat{y}) 
   -\vp_{\sigma}(\tau_{(1,1,0)}(E_{1,2}^{\gk})u_{e_{12}})(\hat{y}) \\
 & \quad -R(2E_{2,3}) \vp_{\sigma}(u_{e_{23}})(\hat{y}) 
   +\vp_{\sigma}(\tau_{(1,1,0)}(E_{2,3}^{\gk})u_{e_{23}})(\hat{y}) \\
 & \quad +R(2E_{3,4}) \vp_{\sigma}(u_{e_{34}})(\hat{y})   
   -\vp_{\sigma}(\tau_{(1,1,0)}(E_{3,4}^{\gk})u_{e_{34}})(\hat{y}) \\
 & \quad +R(2E_{1,4}) \vp_{\sigma}(u_{e_{14}})(\hat{y}) 
   -\vp_{\sigma}(\tau_{(1,1,0)}(E_{1,4}^{\gk})u_{e_{14}})(\hat{y}) \}.
\end{align*} 
In view of Lemma \ref{lem:Eij_radial} and $ \tau_{(1,1,0)}(E_{i,j}^{\gk}) u_{e_{ij}} = 0$, we know 
\begin{align*}
 W'(\hat{y}) 
  = y_1 \vp_{\sigma}(u_{e_{12}})(\hat{y}) - y_2 \vp_{\sigma}(u_{e_{23}})(\hat{y})
   + y_3 \vp_{\sigma}(u_{e_{34}})(\hat{y}),
\end{align*}
that is, 
\begin{align*}
 V'(s_1,s_2,s_3) = 
 V_{\sigma, e_{12}}(s_1+1,s_2,s_3) + V_{\sigma, e_{23}} (s_1,s_2+1,s_3) + V_{\sigma, e_{34}} (s_1,s_2,s_3+1).
\end{align*}
Then (\ref{eqn:lem_W'_case1}) and Theorem \ref{thm:EF1} (iii) imply that  
\begin{align*}
  Z(s_1,s_2, W, \Phi) 
& = \GR(2s_2+\gamma_1)
   \{ V_{\sigma, e_{12}}(s_1+1,s_2, s_1+s_2) 
\\
&  \quad + V_{\sigma, e_{23}} (s_1,s_2+1,s_1+s_2) + V_{\sigma, e_{34}} (s_1,s_2,s_1+s_2+1) \} \\
& = A(0,1,1,0,2,1,1;\nu)  { B}(2,1,1,0,0,0; 2,2;\nu) \\
& \quad  + A(0,1,1,2,2,1,1;\nu) { B}(0,1,1,0,0,0; 2,2;\nu) \\
& \quad + A(0,1,1,2,0,1,1;\nu) { B}(0,1,1,0,0; 2;\nu).
\end{align*}
Here we used the expression
$ V_{\sigma, e_{34}} (s_1,s_2,s_3) = U_0(s_1,s_2,s_3-1; \nu_1+1,\nu_2+1,\nu_3,\nu_4) $ (Lemma \ref{lem:VvsU} (ii)).
In view of (\ref{eqn:gamma_FE}), we know
\begin{align*}
& A(0,1,1,0,2,1,1;\nu)B_0(2,1,1,0,0,0; 2,2;\nu) + A(0,1,1,2,2,1,1;\nu) B_0(0,1,1,0,0,0; 2,2;\nu) 
\\
& \phantom{=} + A(0,1,1,2,0,1,1;\nu)  B_0(0,1,1,0,0; 2;\nu)
\\
& = {(2\pi)}^{-1}  \{ (s_2+\nu_1+\nu_2) + (s_1- q) \} A(0,1,1,0,2,1,1;\nu)B_0(0,1,1,0,0,0; 2,2;\nu) \\
& \phantom{=}
    + A(0,1,1,2,0,1,1;\nu)  B_0(0,1,1,0,0; 2;\nu)
\\
& = \{ A(0,1,1,0,2,1,1;\nu) + A(0,1,1,2,0,1,1;\nu)  \} B_0(0,1,1,0,0; 2;\nu)
\\
& = (2\pi)^{-1} (2s_2+\gamma_1) \cdot A(0,1,1,0,0,1,1;\nu)  B_0(0,1,1,0,0; 2;\nu).
\end{align*}
Then Lemma \ref{lem:BFlem1} with $ \mu = \nu $ leads us that 
\begin{align*}
  Z(s_1,s_2, W, \Phi) 
& =(2\pi)^{-1} (2s_2+\gamma_1) 
  A(0,1,1,0,0,1,1;\nu) { B}(0,1,1,0,0; 2;\nu) \\
& =  \frac{2s_2+\gamma_1}{2\pi} \cdot \frac{\GR(2s_2+\gamma_1)}{\GR(2s_2+\gamma_1+2)} 
   \cdot L_1(1,1,0,0;\nu) L_2(0,1,1,1,1,0;\nu) 
\\
& = L_1(1,1,0,0;\nu) L_2(0,1,1,1,1,0;\nu).
\end{align*}

\medskip

\noindent 
\underline{\bf Case 1-(d):} \
As is the case 1-(c), we know
\begin{align*}
 W'(\hat{y})
& = (4\pi)^{-1} \{ 
      R(2E_{2,3}) \vp_{\sigma}(u_{e_1})(\hat{y}) - \vp_{\sigma}(\tau_{(1,0,1)}(E_{2,3}^{\gk}) u_{e_1})(\hat{y}) \\
&\quad   
   - R(2E_{1,2}) \vp_{\sigma}(u_{e_3})(\hat{y}) + \vp_{\sigma}(\tau_{(1,0,1)}(E_{1,2}^{\gk}) u_{e_3})(\hat{y}) \}
\\
& = \sqrt{-1} y_2 \vp_{\sigma}(u_{e_1})(\hat{y}) - \sqrt{-1} y_1 \vp_{\sigma}(u_{e_3})(\hat{y}),
\end{align*}
that is,
$$   V'(s_1,s_2,s_3) = V_{\sigma, e_1}(s_1,s_2+1,s_3) + V_{\sigma, e_3}(s_1+1,s_2,s_3). $$
Then (\ref{eqn:lem_W'_case1}) and Theorem \ref{thm:EF1} (iv) imply that
\begin{align*}
Z(s_1,s_2, W, \Phi) 
& =  A(1,0,1,1,2,0,1;\nu') { B}(1,1,2,0,0,0; 2,2;\nu') \\
&\quad    + A(1,2,1,1,0,0,1 ;\nu')  { B}(1,1,0,0,0,0; 2,0 ;\nu').
\end{align*}
In view of 
\begin{align*}
&  A(1,0,1,1,2,0,1;\nu') B_0(1,1,2,0,0,0; 2,2;\nu')  + A(1,2,1,1,0,0,1 ;\nu') B_0(1,1,0,0,0,0; 2,0 ;\nu')
\\
& = A(1,0,1,1, 0,0,1;\nu') B_0(1,1,0,0,0,0; 2,2;\nu') \\
& \quad \times (2\pi)^{-2} \{ (s_2+\nu_3+\nu_1)(s_2-q+\nu_2)  + (s_1+\nu_4)(s_1+2s_2-q+\gamma_1)\} 
\\
& = A(1,0,1,1, 0,0,1;\nu') B_0(1,1,0,0,0,0; 2,2;\nu')  \\
& \quad \times 
(2\pi)^{-2} (s_1+s_2+\nu_1+\nu_3+\nu_4) (s_1+s_2-q+\nu_4+\nu_2)  
\\
& = A(1,0,1,1, 0,2,1 ;\nu') B_0(1,1,0,0,0; 2 ;\nu'),
\end{align*}
Lemma \ref{lem:BFlem1} with $ \mu = \nu'  $ leads us that 
\begin{align*}
Z(s_1,s_2,W,\Phi) 
& = L_1(0,1,1,1;\nu' ) L_2(1,1,1,0,0,0 ;\nu') 
= L_1(1,1,1,0; \nu) L_2(0,0,0,1,1,1; \nu).
\end{align*}

\medskip

\noindent 
\underline{\bf Case 1-(e):} \ Since 
\begin{align*}
  W'(\hat{y}) = (4 \pi \sqrt{-1})^{-2} R(4 E_{1,2} E_{3,4}) \vp_{\sigma}(u_{\bf 0})(\hat{y})
     = y_1 y_3 \vp_{\sigma}(u_{\bf 0})(\hat{y}),
\end{align*} 
we have 
$  V'(s_1,s_2,s_3) = V_{\bf 0}(s_1+1, s_2, s_3+1). $
Then we can get the assertion as is the case 1-(a).


\medskip 

As for the contragredient zeta integral $ Z(s_1,s_2,\widetilde{W},\widehat{\Phi}) $,
our claim follows from 
Proposition \ref{prop:contra_radial}, Lemma \ref{lem:contra} and 
$ \widehat{\Phi}_{(0,b)} = (\sqrt{-1})^{b} \Phi_{(0,b)} $ $ (b \in \bZ_{\ge 0}) $.
See also \cite[\S 4.4]{Ishii_001}.

\subsection{Proof of Theorem \ref{thm:BF2}} 
\label{subsec:BF_case2}

By Lemma \ref{lem:test2}, we know that
\begin{align*}
 W'(\hat{y})  
& = 2^{-2} \sum_{0 \le i \le 3} 
  \int_0^{2\pi} \! \int_0^{2\pi}
   W \bigl( m_i \, \hat{y} \, \tilde{\iota}(\rk_{\theta_1}^{(2)}, \rk_{\theta_2}^{(2)} ) \bigr) 
   \exp(- \sqrt{-1} b\, \theta_2) \, \frac{d\theta_1}{2\pi} \frac{d\theta_2}{2\pi}
\\
&  = \begin{cases} 
   W(\hat{y}) & \mbox{cases 2-(a),(c),(f)}, \\
   (-\sqrt{-1})^{\kappa_1-1} \varphi_{\sigma}(w_4)(\hat{y}) & \mbox{case 2-(b)}, \\
   (-\sqrt{-1})^{\kappa_1-1} (4\pi)^{-1} 
     \{ R(E_{2,3}^{\gp}) \varphi_{\sigma}(w_1)(\hat{y}) 
    - R(E_{1,2}^{\gp}) \varphi_{\sigma}(w_3) (\hat{y})\} & \mbox{case 2-(d)}, \\ 
  (-\sqrt{-1})^{\kappa_1+1} \varphi_{\sigma}(w_{2,24}') (\hat{y}) & \mbox{case 2-(e)}.
   \end{cases} 
\end{align*}
Then Lemma \ref{lem:W'} implies that  
\begin{align} \label{eqn:lem_W'_case2}
 Z(s_1,s_2,W, \Phi_{(0,b)}) 
= \GR(2s_2+ \gamma_1+b) V'(s_1,s_2,s_1+s_2).
\end{align}
Here we use the notation $ \gamma_1 = 2\nu_1+\nu_2+\nu_3. $ 

\begin{lem} \label{lem:BFlem2}
For $ s_1,s_2 \in \bC$, and 
$ a_i, b_j, c_k,d_l \in  \bC$ $(1 \le i,k, l\le 5, 1\le j \le 7) $, 
we set
\begin{align*}
 A(a_1,a_2,a_3,a_4,a_5) 
&= \GR(2s_2+\gamma_1+a_1) \GC(s_1+\nu_1+\tfrac{\kappa_1-1}{2} +a_2) \GR(s_2+2\nu_1+\kappa_1-2j+a_3) 
\\
& \quad \times  \GR(s_2+\nu_2+\nu_3+2j+a_4)  
   \GC(s_1+s_2+\nu_1+\nu_2+\nu_3+\tfrac{\kappa_1-1}{2}+a_5), 
\end{align*}
\begin{align*}
B_0(b_1,b_2,b_3,b_4,b_5; b_6,b_7)
& = \frac{ \GR(s_1-q+2j+b_1) \GC(s_2-q+\nu_1+\tfrac{\kappa_1-1}{2}+b_2) \GR(s_1+s_2-q+2\nu_1+b_3) }{ 
   \GR(s_1+s_2-q+2\nu_1+\kappa_1+b_6) \GR(s_1+2s_2-q+2\nu_1+\nu_2+\nu_3+2j+b_7)} \\
 & \quad \times \GR(q+\nu_2+b_4) \GR(q+\nu_3+b_5),
\end{align*}
\begin{align*}
 C(c_1,c_2,c_3,c_4,c_5) 
& = \GC(s_1+\nu_1+\tfrac{\kappa_1-1}{2}+c_1)
 \GC(s_1+s_2+\nu_1+\nu_2+\nu_3+\tfrac{\kappa_1-1}{2} +c_2) \\
& \quad \times \GR(s_2+\nu_2+\nu_3+c_3) 
   \GR(2s_2+2\nu_1+\nu_2+\nu_3+\kappa_1+c_4)  \GR(s_2+2\nu_1+c_5),
\end{align*}
\begin{align*}
 D_0(d_1,d_2,d_3,d_4;d_5) 
& = \frac{\GR(s_1-q+d_1)\GC(s_2-q+\nu_1+\tfrac{\kappa_1-1}{2}+d_2) \GR(q+\nu_2+d_3)\GR(q+\nu_3+d_4) }
 {  \GR(s_1+2s_2-q+2\nu_1+\nu_2+\nu_3+\kappa_1+d_5)}
\end{align*}
and 
\begin{align*}
{B}(b_1,b_2,b_3,b_4,b_5; b_6,b_7)
& =   \frac{1}{4\pi \sqrt{-1}} \int_q B_0(b_1,b_2,b_3,b_4,b_5; b_6,b_7) \,dq,
\\
  {D}(d_1,d_2,d_3,d_4; d_5) 
&= 
  \frac{1}{4\pi \sqrt{-1}} \int_q D_0(d_1,d_2,d_3,d_4; d_5)\, dq.
\end{align*}

\medskip 

\noindent 
(i)
Let $\delta \in \bZ_{\ge 0} $ such that $ \kappa_1-\delta \in 2\bZ_{\ge 0} $. 
If $ a_3+a_4 +b_1+\delta = b_7 $, $ a_3+b_1+\delta = b_3 $ and $ a_3+b_1 = b_6 $, then we have 
\begin{align*}
&  \sum_{0 \le j \le \frac{\kappa_1-\delta}{2}} \binom{ \frac{\kappa_1-\delta}{2}}{j} 
    A(a_1,a_2,a_3,a_4,a_5) B(b_1,b_2,b_3,b_4,b_5; b_6,b_7) \\
& = \frac{ \GR(2s_2+\gamma_1+a_1)}{\GR(2s_2+\gamma_1+a_3+a_4+\delta)} 
   \cdot 
   C(a_2, a_5, a_4, a_3+a_4, a_3+\delta) 
   D(b_1, b_2, b_4, b_5;  a_3+a_4+b_1).
\end{align*}
\noindent
(ii)
If $ d_5 = d_1+2d_2+d_3+d_4 $, $ c_2 = d_1+d_2+d_3+d_4 $ and 
$ c_4 = 2d_2+d_3+d_4 $,
then we have 
\begin{align*}
  C(c_1,c_2,c_3,c_4,c_5) {D}(d_1,d_2,d_3,d_4; d_5)
& = L_1(c_1, d_1+d_3, d_1+d_4) L_2(d_2+d_3, d_2+d_4, c_5, c_3),
\end{align*}
where
\begin{align*}
 L_1(\alpha_1,\alpha_2,\alpha_3)
& = \GC(s_1+\nu_1+\tfrac{\kappa_1-1}{2}+\alpha_1) \GR(s_1+\nu_2+\alpha_2) \GR(s_1+\nu_3+\alpha_3), 
\\
 L_2(\beta_1,\beta_2,\beta_3,\beta_4) 
& = \GC(s_2+\nu_1+\nu_2+\tfrac{\kappa_1-1}{2}+\beta_1) 
   \GC(s_2+\nu_1+\nu_3+\tfrac{\kappa_1-1}{2}+\beta_2) \\ 
& \quad \times \GR(s_2+2\nu_1+\beta_3)\GR(s_2+\nu_2+\nu_3+\beta_4).
\end{align*}
\end{lem}

\begin{proof}
The first claim (i) follows from Lemma \ref{lem:Saal}. 
The latter is immediate from Lemma \ref{lem:Barnes2nd} together with the duplication formula 
(\ref{eqn:gamma_dup}). 
\end{proof}

\medskip

\noindent 
\underline{\bf Case 2-(a):} \ 
Since $  W'(\hat{y}) = W(\hat{y}) $ can be written as 
\begin{align*}
  (-\sqrt{-1})^{\kappa_1} \varphi_{\sigma} \bigl(  \rrq_\cR  ((\xi_1)^2+(\xi_3)^2)^{\kappa_1/2} \bigr) (y) 
= (-\sqrt{-1})^{\kappa_1} \sum_{0 \le j \le \frac{\kappa_1}{2} } \binom{\frac{\kappa_1}{2}}{j}
    \varphi_{\sigma} \bigl(  \rrq_\cR  ( (\xi_1)^j (\xi_3)^{\kappa_1-2j} ) \bigr) (y), 
\end{align*}
we have
\begin{align*}
 V'(s_1,s_2,s_3)
& = (-\sqrt{-1})^{\kappa_1} \!
 \sum_{0 \le j  \le \frac{\kappa_1}{2}} \binom{\frac{\kappa_1}{2}}{j}
  (\sqrt{-1})^{-2j + \kappa_1-2j} 
     V_{\sigma, 2je_1+(\kappa_1-2j)e_3} (s_1,s_2,s_3)
\\
& =  
  \sum_{0 \le j  \le \frac{\kappa_1}{2}} \binom{\frac{\kappa_1}{2}}{j} V_{\sigma, 2je_1+(\kappa_1-2j)e_3} (s_1,s_2,s_3).
\end{align*}
Then (\ref{eqn:lem_W'_case2}), Theorem \ref{thm:EF2} (i) and \ref{lem:BFlem2} imply that 
\begin{align*}
 Z(s_1,s_2, W, \Phi)
& = \GR(2s_2+\gamma_1) 
      \sum_{0 \le j  \le \frac{\kappa_1}{2}} \binom{ \frac{\kappa_1}{2} }{j} V_{\sigma, 2je_1+(\kappa_1-2j)e_3}(s_1,s_2,s_1+s_2) \\
& = \sum_{0 \le j  \le \frac{\kappa_1}{2}} \binom{ \frac{\kappa_1}{2} }{j} 
      A(0, 0, 0, 0, 0) B(0, 0, 0, 0, 0; 0, 0) \\
& = C(0,0,0,0,0) D (0,0,0,0 ; 0) \\
& = L_1(0,0,0) L_2(0,0,0,0).
\end{align*}

\medskip  

\noindent 
\underline{\bf Case 2-(b):} \ 
We can see 
$$ V'(s_1,s_2,s_3) =  
\sum_{0 \le j \le \frac{\kappa_1-1}{2}} \binom{\frac{\kappa_1-1}{2}}{j} 
  V_{\sigma, 2je_1+(\kappa_1-2j-1)e_3+ e_4} (s_1,s_2,s_3). $$
Then (\ref{eqn:lem_W'_case2}), Theorem \ref{thm:EF2} (i) and \ref{lem:BFlem2} imply that 
\begin{align*}
 Z(s_1,s_2, W, \Phi) 
& = \GR(2s_2+\gamma_1+1) \sum_{0 \le j \le \frac{\kappa_1-1}{2}} 
   \binom{\frac{\kappa_1-1}{2}}{j} V_{\sigma, 2je_1+(\kappa_1-2j-1)e_3+ e_4} (s_1,s_2,s_1+s_2)
\\
& =  \sum_{0 \le j \le \frac{\kappa_1-1}{2}}  \binom{\frac{\kappa_1-1}{2}}{j} 
     A(1, 0, 0, 0, 0) B(0, 0, 1, 0, 0; 0, 1) 
\\
& =  C(0,0,0,0,1) D(0,0,0,0; 0) 
\\
& = L_1(0, 0, 0) L_2(0, 0, 1, 0).
\end{align*}

\medskip 

\noindent 
\underline{\bf Case 2-(c):} \ 
As is the case 1-(c), we know 
\begin{align*}
 W'(\hat{y}) 
& = 
  (-\sqrt{-1})^{\kappa_1} (4\pi)^{-1} \{ R(2E_{1,2}) \varphi_{\sigma}( w_{3,4} ) (\hat{y}) 
      - \varphi_{\sigma} (\tau_{(\kappa_1,0,1)}(E_{1,2}^{\gk})  w_{3,4}  ) (\hat{y}) \\
& \qquad \qquad - R(2E_{2,3}) \varphi_{\sigma} (w_{1,4})(\hat{y}) 
  + \varphi(\tau_{(\kappa_1,0,1)}(E_{2,3}^{\gk})w_{1,4}) (\hat{y})  \\
& \qquad \qquad + R(2E_{3,4}) \varphi_{\sigma}(w_{1,2})(\hat{y}) 
   - \varphi_{\sigma} (\tau_{(\kappa_1,0,1)}(E_{3,4}^{\gk})w_{1,2}) (\hat{y})  \\
& \qquad \qquad  - R(2E_{1,4}) \varphi_{\sigma}(w_{2,3})(\hat{y}) 
  + \varphi_{\sigma} (\tau_{(\kappa_1,0,1)}(E_{1,4}^{\gk})w_{2,3}) (\hat{y})  \}.
\end{align*}
Since we can see that 
$  \tau_{(\kappa_1,0,1)}(E_{1,2}^{\gk})w_{3,4} = \tau_{(\kappa_1,0,1)}(E_{1,4}^{\gk})  w_{2,3}  $
and 
\begin{align*}
  \tau_{(\kappa_1,0,1)}(E_{2,3}^{\gk})  w_{1,4}  
& =  - \tau_{(\kappa_1,0,1)}(E_{3,4}^{\gk})  w_{1,2} 
\\
& = \sum_{0 \le j \le \frac{\kappa_1-2}{2}} 
    \binom{\frac{\kappa_1-2}{2}}{j} (\kappa_1-2j-2) u_{(2j+1)e_1+e_2+ (\kappa_1-2j-3)e_3+e_4},
\end{align*}
we have 
\begin{align*}
 W'(\hat{y} ) 
& = (-\sqrt{-1})^{\kappa_1}  
     \sum_{0 \le j \le \frac{\kappa_1-2}{2} } \binom{\frac{\kappa_1-2}{2}}{j} 
       \{  \sqrt{-1} y_1 \varphi_{\sigma} ( u_{2je_1+(\kappa_1-2j-1)e_3 + e_4})(\hat{y})  \\
& \qquad - \sqrt{-1} y_2  \varphi_{\sigma} ( u_{(2j+1)e_1+(\kappa_1-2j-2)e_3 + e_4})(\hat{y})   
     + \sqrt{-1} y_3  \varphi_{\sigma} ( u_{(2j+1)e_1+e_2+(\kappa_1-2j-2)e_3 })(\hat{y})   \\
& \qquad + (2\pi)^{-1} (\kappa_1-2j-2) \varphi_{\sigma} (u_{(2j+1)e_1+e_2+(\kappa_1-2j-3)e_3+e_4} )(\hat{y}) \}, 
\end{align*}
that is, 
\begin{align*}
V'(s_1,s_2,s_3)
& = \sum_{0 \le j \le \frac{\kappa_1-2}{2}} \binom{\frac{\kappa_1-2}{2}}{j} 
    \{ V_{\sigma,  2je_1+(\kappa_1-2j-1)e_3 + e_4}(s_1+1,s_2,s_3)   \\
& \qquad  + V_{\sigma, (2j+1)e_1+(\kappa_1-2j-2)e_3 + e_4}(s_1, s_2+1, s_3)  
   +  V_{\sigma,  (2j+1)e_1+e_2+ (\kappa_1-2j-2)e_3}(s_1, s_2, s_3+1)   \\
& \qquad -  (2\pi)^{-1} (\kappa_1-2j-2) V_{\sigma, (2j+1)e_1+e_2+(\kappa_1-2j-3)e_3+e_4} (s_1,s_2,s_3) \}.
\end{align*}
Then (\ref{eqn:lem_W'_case2}) and Theorem \ref{thm:EF2} (i) imply that 
\begin{align*}
 Z(s_1,s_2, W, \Phi)
& = \sum_{0 \le j \le \frac{\kappa_1-2}{2}} \binom{\frac{\kappa_1-2}{2}}{j}  
  \{ A(0, 1, 0, 0, 0) B(1, 0,1,0,0; 1, 1) \\
& \quad 
    + A(0, 0,0,2,0) B(1,1,1,0,0; 1,3) + A(0,0, -2,2,1) B(1,0,1,0,0; -1,3) 
\\
& \quad - (2\pi)^{-1} (\kappa_1-2j-2)   A(0,0,-2,2,0) B(1,0,1,0,0; -1,3) \}.
\end{align*}
In view of 
\begin{align*}
& A(0, 1, 0, 0, 0) B_0(1, 0,1,0,0; 1, 1)  + A(0, 0,0,2,0) B_0(1,1,1,0,0; 1,3) \\
& \quad - (2\pi)^{-1} (\kappa_1-2j-2)  A(0,0,-2,2,0) B_0(1,0,1,0,0; -1,3) \\
& = A(0, 0, -2, 0, 0) B_0(1, 0, 1, 0, 0; 1, 3) \cdot   {(2\pi)^{-3}}  
\\
& \quad \times 
 \{ (s_1+\nu_1+\tfrac{\kappa_1-1}{2})(s_2+2\nu_1+\kappa_1-2j-2)  (s_1+2s_2-q+2\nu_1+\nu_2+\nu_3+2j+1) 
\\
& \qquad 
+ (s_2+2\nu_1+\kappa_1-2j-2)(s_2+\nu_2+\nu_3+2j)(s_2-q+\nu_1+\tfrac{\kappa_1-1}{2}) \\
& \qquad -(\kappa_1-2j-2)(s_2+\nu_2+\nu_3+2j)(s_1+s_2-q+2\nu_1+\kappa_1-1) \} 
\\
& = A(0, 0, -2, 0, 0) B_0(1, 0, 1, 0, 0; 1, 3) \cdot {(2\pi)^{-3}}
\\
& \quad \times  \{ (s_1+\nu_1+\tfrac{\kappa_1-1}{2})(s_1+s_2-q+2\nu_1+1)(s_2+2\nu_1+\kappa_1^2j-2)
\\
& \qquad + (s_2+2\nu_1) (s_2+\nu_2+\nu_3+2j)(s_1+s_2-q+2\nu_1+\kappa_1-1) \} 
\\
& =  A(0,1,0,0,0) B_0(1,0,3,0,0; 1,3) 
 + (2\pi)^{-1} (s_1+2\nu_1)\, A(0,0,-2,2,0) B_0(1,0,1,0,0; -1,3),
\end{align*}
we can use Lemma \ref{lem:BFlem2} (i) to get
\begin{align*}
& Z(s_1,s_2, W, \Phi) 
\\
& = \sum_{0 \le j \le \frac{\kappa_1-2}{2}} \binom{\frac{\kappa_1-2}{2}}{j}  
  \{ A(0,0, -2,2,1) B(1,0,1,0,0; -1,3) +A(0,1,0,0,0) B(1,0,3,0,0; 1,3)
\\
& \quad  
    + (2\pi)^{-1} (s_1+2\nu_1 ) A(0,0,-2,2,0) B(1,0,1,0,0; -1,3)\} 
\\
& = \frac{\GR(2s_2+\gamma_1)}{\GR(2s_2+\gamma_1+2)} 
 \Bigl\{ C(0,1,2,0,0)  + C(1,0,0,0,2) + \frac{s_1+2\nu_1}{2\pi} \, C(0,0,2,0,0) \Bigr\}
  D (1,0,0,0; 1).
\end{align*}
Since 
\begin{align*}
& C(0,1,2,0,0) + C(1,0,0,0,2) + (2\pi)^{-1} (s_1+2\nu_1) \, C(0,0,2,0,0) 
= (2\pi)^{-1} (2s_2+\gamma_1) \, C(0,1,0,0,0),   
\end{align*}
Lemma \ref{lem:BFlem2} (ii) leads us that 
\begin{align*}
 Z(s_1,s_2, W, \Phi) 
& = C(0,1,0,0,0) {D}(1,0,0,0; 1) \\
& = L_1(0,1,1) L_2(0,0,0,0).
\end{align*}

\medskip

\noindent 
\underline{\bf Case 2-(d):} \ 
As is the case 1-(d), we know 
\begin{align*}
W'(\hat{y}) 
& = (-\sqrt{-1})^{\kappa_1-1} (4 \pi)^{-1} 
  \{R(2E_{2,3}) \varphi( w_1 )(\hat{y}) - \varphi(\tau_{(\kappa_1,0,1)}(E_{2,3}^{\gk})w_1 ) (\hat{y}) \\
& \qquad 
  - R(2E_{1,2}) \varphi( w_3 )(\hat{y}) + \varphi(\tau_{(\kappa_1,0,1)}(E_{1,2}^{\gk})w_3) (\hat{y}) \}.
\end{align*}
Since
\begin{align*}
& \tau_{(\kappa_1,0,1)}(E_{2,3}^{\gk})w_1 = - \tau_{(\kappa_1,0,1)}(E_{1,2}^{\gk})w_3
 =  \sum_{0 \le j \le \frac{\kappa_1-1}{2}} 2j \binom{ \tfrac{\kappa_1-1}{2}}{j} u_{(2j-1) e_1+(\kappa_1-2j)e_3+e_2},
\end{align*}
we have
\begin{align*}
 W'(\hat{y}) 
& =  (-\sqrt{-1})^{\kappa_1-1} 
  \sum_{0 \le j \le \frac{\kappa_1-1}{2}}  \binom{ \frac{\kappa_1-1}{2}}{j}  
 \{ \sqrt{-1}y_2  \varphi( u_{(2j+1)e_1 + (\kappa_1-2j-1) e_3} )(\hat{y}) 
\\
& \quad 
   - \sqrt{-1} y_1  \varphi( u_{2j e_1 + (\kappa_1-2j) e_3}) (\hat{y})
  - (2\pi)^{-1} (2j)   \varphi(u_{(2j-1) e_1+(\kappa_1-2j)e_3+e_2})(\hat{y})  \},
\end{align*}
that is,
\begin{align*}
 V'(s_1,s_2,s_3) 
& = \sum_{0 \le j \le \frac{\kappa_1-1}{2}}  \binom{ \frac{\kappa_1-1}{2}}{j} 
    \{ V_{\sigma, (2j+1)e_1 + (\kappa_1-2j-1) e_3}(s_1,s_2+1,s_3) \\
& \quad + V_{\sigma, 2j e_1 + (\kappa_1-2j) e_3} (s_1+1, s_2, s_3)  
   - (2\pi)^{-1} (2j)  V_{\sigma, (2j-1) e_1+e_2+ (\kappa_1-2j)e_3} (s_1,s_2,s_3)   \}.
\end{align*}
Then (\ref{eqn:lem_W'_case2}) and Theorem \ref{thm:EF2} (i) imply that 
\begin{align*}
 Z(s_1,s_2, W, \Phi)
& = \sum_{0 \le j \le \frac{\kappa_1-1}{2}}  \binom{ \frac{\kappa_1-1}{2}}{j}
 \{ A(1, 0,0,2,0) B(1,1,0,0,0; 1,2) \\
& \quad + A(1,1,0,0,0) B(1,0,0,0,0; 1,0)  - (2\pi)^{-1} (2j) A(1,0,0,0,0) B(-1,0,0,0,0;-1,0) \}.
\end{align*}
In view of 
\begin{align*}
&  A(1, 0,0,2,0) B_0(1,1,0,0,0; 1,2) + A(1,1,0,0,0) B_0(1,0,0,0,0; 1,0) \\
& -  (2\pi)^{-1} (2j)  A(1,0,0,0,0)B_0(-1,0,0,0,0;-1,0) 
\\
& = 
 A(1,0,0,0,0) B_0(-1,0,0,0,0; 1,2) \cdot {(2\pi)^{-3}} 
\\
& \quad \times  \{ (s_2+\nu_2+\nu_3+2j) (s_1-q+2j-1)(s_2-q+\nu_1+\tfrac{\kappa_1-1}{2}) 
\\
& \qquad + (s_1+\nu_1+\tfrac{\kappa_1-1}{2})(s_1-q+2j-1)(s_1+2s_2-q+2\nu_1+\nu_2+\nu_3+2j) 
\\
& \qquad -2j(s_1+s_2-q+2\nu_1+\kappa_1-1)(s_1+2s_2-q+2\nu_1+\nu_2+\nu_3+2j) \} 
\\
& =  A(1,0,0,0,0) B_0(-1,0,0,0,0; 1,2) \cdot  {(2\pi)^{-3}}
\\
& \quad \times \{ (s_1-q-1)(s_1+s_2-q+2\nu_1+\kappa_1-1)(s_1+2s_2-q+2\nu_1+\nu_2+\nu_3+2j) 
\\
& \qquad -(s_1-q+2j-1)(s_2-q+\nu_1+\tfrac{\kappa_1-1}{2})(s_1+s_2-q+2\nu_1) \} 
\\
& = 
     (2\pi)^{-1}(s_1-q-1) \, A(1,0,0,0,0) B_0(-1,0,0,0,0; -1,0)  - A(1,0,0,0,0) B_0(1,1,2,0,0; 1,2), 
\end{align*}
we can use Lemma \ref{lem:BFlem2} (i) to get 
\begin{align*}
& Z(s_1,s_2, W, \Phi) =  C(0,0,0,0,1) \{  {D}(1,0,0,0; -1) - {D}(1,1,0,0; 1) \}.
\end{align*}
Since
\begin{align*}
& C(0,0,0,0,1)   \{ D_0(1,0,0,0; -1) - D_0(1,1,0,0; 1) \} 
\\
& = C(0,0,0,0,1) D_0(1,0,0,0; 1) 
\cdot (2\pi)^{-1}
   \{ (s_1+2s_2-q+2\nu_1+\nu_2+\nu_3+\kappa_1-1) - (s_2-q+\nu_1+\tfrac{\kappa_1-1}{2}) \} 
\\
& = C(1,0,0,0,1)  D_0(1,0,0,0; 1),
\end{align*}
Lemma \ref{lem:BFlem2} (ii) leads us that 
$$ Z(s_1,s_2, W, \Phi) = L_1(0,1,1) L_2(0,0,1,0). $$

\medskip

\noindent 
\underline{\bf Case 2-(e):} \ 
We know 
\begin{align*}
 V'(s_1,s_2,s_3) 
& = \sum_{0 \le j  \le \frac{\kappa_1-2}{2}} 
   \binom{ \frac{\kappa_1-2}{2}}{j} V_{\sigma, 2je_1+e_2+(\kappa_1-2j-2)e_3+e_{24}}(s_1,s_2,s_3).
\end{align*}
Then (\ref{eqn:lem_W'_case2}), Theorem \ref{thm:EF2} (ii) and Lemma \ref{lem:BFlem2} implies that 
\begin{align*}
 Z(s_1,s_2, W, \Phi)
& = (2\pi)^{-1} (s_2+\nu_1+\nu_3+\tfrac{\kappa_1-3}{2})   
    \sum_{0 \le j \le \frac{\kappa_1-2}{2}} \binom{ \frac{\kappa_1-2}{2}}{j} 
   A(1, 0, -2, 1, 0) { B}(0,-1, 0, 1, 0; -2, 1) 
\\
& = (2\pi)^{-1} (s_2+\nu_1+\nu_3+\tfrac{\kappa_1-3}{2})   
   C(0,0,1,-1,0) { D}(0,-1,1,0; -1) 
\\
& = (2\pi)^{-1} (s_2+\nu_1+\nu_3+\tfrac{\kappa_1-3}{2}) \, L_1(0,1,0) L_2(0,-1,0,1) 
\\
& = L_1(0,1,0) L_2(0,0,0,1).
\end{align*}

\medskip

\noindent 
\underline{\bf Case 2-(f):} \ 
As is the case 1-(c), we know
\begin{align*}
 W'(\hat{y}) 
& = (-\sqrt{-1})^{\kappa_1} (4\pi)^{-1} \{ R(2E_{1,2}) \varphi( w_{12} )(\hat{y}) 
   - \varphi(\tau_{(\kappa_1,0,1)}(E_{1,2}^{\gk})  w_{12} ) (\hat{y}) \\
& \quad - R(2E_{2,3}) \varphi(w_{23})(\hat{y}) + \varphi(\tau_{(\kappa_1,0,1)}(E_{2,3}^{\gk})w_{23}) (\hat{y})  \\
& \quad + R(2E_{3,4}) \varphi(w_{34})(\hat{y}) - \varphi(\tau_{(\kappa_1,0,1)}(E_{3,4}^{\gk})w_{34}) (\hat{y})  \\
& \quad  + R(2E_{1,4}) \varphi(w_{14})(\hat{y}) - \varphi(\tau_{(\kappa_1,0,1)}(E_{1,4}^{\gk})w_{14}) (\hat{y})  \}.
\end{align*} 
Since 
\begin{align*}
& -\tau_{(\kappa_1,0,1)}(E_{1,2}^{\gk}) w_{12}
   + \tau_{(\kappa_1,0,1)}(E_{2,3}^{\gk}) w_{23} 
   - \tau_{(\kappa_1,0,1)}(E_{3,4}^{\gk} ) w_{34} -\tau_{(\kappa_1,0,1)}(E_{1,4}^{\gk})w_{14}
\\
& = \sum_{0 \le j \le \frac{\kappa_1-1}{2}} \binom{\frac{\kappa_1-1}{2}}{ j} 
     \{ 2j u_{(2j-1)e_1+e_2+(\kappa_1-2j-1)e_3+e_{12}} 
  +  (\kappa_1-2j-1) u_{2je_1+e_2+(\kappa_1-2j-2)e_3+e_{23}}
\\
& \qquad + (\kappa_1-2j-1) u_{2je_1+(\kappa_1-2j-2)e_3+e_4+e_{23}}  
   + 2j u_{(2j-1)e_1+(\kappa_1-2j-1)e_3+e_4+e_{14}} \},
\end{align*}
the relations
\begin{align*}
 u_{l+e_2+e_{12}} + u_{l+e_4+e_{14}} &= - u_{l+e_3+e_{13}}, & 
 u_{l+e_4+e_{34}} &= u_{l+e_2+e_{23}} + u_{l+e_1+e_{13}} 
\end{align*}
for $ l \in S_{(\kappa_1-1, 0,0)} $, and $ (j+1) \binom{m}{j+1} =(m-j) \binom{m}{j} $ for $ 0 \le j \le m-1 $
imply that 
\begin{align*}
& -\tau_{(\kappa_1,0,1)}(E_{1,2}^{\gk}) w_{12}
   + \tau_{(\kappa_1,0,1)}(E_{2,3}^{\gk}) w_{23} 
   - \tau_{(\kappa_1,0,1)}(E_{3,4}^{\gk} ) w_{34} -\tau_{(\kappa_1,0,1)}(E_{1,4}^{\gk})w_{14}
\\  
& = \sum_{0 \le j \le \frac{\kappa_1-1}{2}}  2(\kappa_1-2j-1) 
   \binom{\frac{\kappa_1-1}{2}}{j}  u_{2j e_1+e_2+ (\kappa_1-2j-2)e_3+e_{23}}.
\end{align*}
Then we get 
\begin{align*}
V'(s_1,s_2,s_3)
& = \sum_{0 \le j \le \frac{\kappa_1-1}{2}} \binom{\frac{\kappa_1-1}{2}}{j} 
\{  V_{\sigma, 2je_1+(\kappa_1-2j-1)e_3+e_{12}} (s_1+1,s_2,s_3) \\
& \quad  + V_{\sigma, 2je_1+(\kappa_1-2j-1)e_3+e_{23}} (s_1,s_2+1,s_3) 
  +   V_{\sigma, 2je_1+(\kappa_1-2j-1)e_3+e_{34}} (s_1,s_2,s_3+1) 
\\
& \quad - (2\pi)^{-1} (\kappa_1-2j-1)  V_{\sigma, 2je_1+e_2+(\kappa_1-2j-2)e_3+e_{23}} (s_1,s_2,s_3) \},
\end{align*}
and (\ref{eqn:lem_W'_case2}), Theorem \ref{thm:EF2} (ii) imply that 
$$
   Z(s_1,s_2,W,\Phi) = Z_1+Z_{2,1}+Z_{2,2}+Z_3+Z_{4,1}+Z_{4,2}, 
$$
where
\begin{align*}
 Z_1 & = \sum_{0 \le j \le \frac{\kappa_1-1}{2}} \binom{\frac{\kappa_1-1}{2}}{j}   A(0, 1, 0, 1, 0) B(1,0,0,0,1; 1,1), 
\\
 Z_{2,1} &= \sum_{0 \le j \le \frac{\kappa_1-1}{2}} \binom{\frac{\kappa_1-1}{2}}{j} 
       A(0,0,0,1,0) {B}(0,1,1,1,0; 0,2), 
\\
 Z_{2,2} &= \sum_{0 \le j \le \frac{\kappa_1-1}{2}} \binom{\frac{\kappa_1-1}{2}}{j}   
   A(0,0,0,1,0) {B} (1,1,0,0,1; 1,1), 
\\
 Z_3& = \sum_{0 \le j \le \frac{\kappa_1-1}{2}} \binom{\frac{\kappa_1-1}{2}}{j} A(0,0,0,1,1)
     {B}(0,0,1,1,0; 0,2), 
\\ 
 Z_{4,1} & = \frac{-\kappa_1+1}{2\pi}  \sum_{0 \le j \le \frac{\kappa_1-3}{2}} \binom{\frac{\kappa_1-3}{2}}{j} 
    A(0, 0,-2,1,0) {B}(0,0,1,1,0; -2,2)  
\\
 Z_{4,2} & = \sum_{0 \le j \le \frac{\kappa_1-1}{2}} \binom{\frac{\kappa_1-1}{2}}{j}  
    \frac{-\kappa_1+2j+1}{2\pi} 
    A(0,0,-2,1,0) {B}(1,0,0,0,1;-1,1).
\end{align*}
Here we used $ (-\kappa_1+2j+1) \binom{\frac{\kappa_1-1}{2}}{j} 
= (-\kappa_1+1) \binom{\frac{\kappa_1-3}{2}}{j} $ in $ Z_{4,1} $.
We use Lemma \ref{lem:BFlem1} (i) to find 
\begin{align*}
 Z_{2,1} & = 
 \frac{\GR(2s_2+\gamma_1)}{\GR(2s_2+\gamma_1+2)} 
  \cdot C(0,0,1,1,1) { D}(0,1,1,0; 1), 
\\
  Z_3 & = \frac{\GR(2s_2+\gamma_1)}{\GR(2s_2+\gamma_1+2)} 
  \cdot C(0,1,1,1,1) {D}(0,0,1,0; 1), 
\\
  Z_{4,1} & = \frac{-\kappa_1+1}{2\pi} \cdot
   \frac{\GR(2s_2+\gamma_1)}{\GR(2s_2+\gamma_1+2)} \cdot
   C(0,0,1,-1,1) {D}(0,0,1,0;  -1).
\end{align*}
In view of  
\begin{align*}
&  C(0,0,1,1,1) D(0,1,1,0; 1) + C(0,1,1,1,1) D(0,0,1,0;  1) \\
& \phantom{=} + (2\pi)^{-1}(-\kappa_1+1) \, C(0,0,1,-1,1)D(0,0,1,0;  -1) \\
& = C(0,0,1,1,1) D(0,0,1,0; -1)  + (2\pi)^{-1}(-\kappa_1+1) \, C(0,0,1,-1,1)D(0,0,1,0;  -1) \\
& = (2\pi)^{-1}( 2s_2+2\nu_1+\nu_2+\nu_3) \, C(0,0,1,-1,1)  D(0,0,1,0;  -1), 
\end{align*}
we know that 
\begin{align*} 
 Z_{2,1} + Z_3 +Z_{4,1}
& = C(0,0,1,-1,1) {D}(0,0,1,0; -1).
\end{align*}
Let us consider the sum $ Z_1 + Z_{2,2}+  Z_{4,2} $.
Since 
\begin{align*}
&  A(0,1,0,1,0) B(1,0,0,0,1; 1,1)
 +A(0,0,0,1,0) B(1,1,0,0,1; 1,1) 
\\
& + (2\pi)^{-1}(-\kappa_1+2j+1)\, A(0,0,-2,1,0) B(1,0,0,0,1; -1,1)
\\
& = A(0,0,0,1,0) B(1,0,0,0,1; -1,1) + (2\pi)^{-1} (-\kappa_1+2j+1)\, A(0,0,-2,1,0) B(1,0,0,0,1; -1,1)
\\
& = (2\pi)^{-1} (s_2+2\nu_1-1) \,A(0,0,-2,1,0) B(1,0,0,0,1; -1,1), 
\end{align*}
Lemma \ref{lem:BFlem1} (i) implies that
\begin{align*}
 Z_1+ Z_{2,2}+Z_{4,2} 
& = (2\pi)^{-1}(s_2+2\nu_1-1) \,C(0,0,1,-1,-1) {D}(1,0,0,1; 0) 
\\
& = C(0,0,1,-1,1) {D}(0,-1,1,2; -1).
\end{align*}
Here we substituted $ q \to q+1 $.
In view of 
\begin{align*}
   D(0,0,1,0; -1) + D(0,-1,1,2; -1)
& = (2\pi)^{-1} (s_2+\nu_1+\nu_3+\tfrac{\kappa_1-1}{2}) D(0,-1,1,0; -1),
\end{align*}
Lemma \ref{lem:BFlem1} (ii) leads us that 
\begin{align*}
 Z(s_1,s_2,W,\Phi)
& = (2\pi)^{-1} (s_2+\nu_1+\nu_3+\tfrac{\kappa_1-1}{2}) 
   \, C(0,0,1,-1,1) { D}(0,-1,1,0; -1)
\\
& = (2\pi)^{-1} (s_2+\nu_1+\nu_3+\tfrac{\kappa_1-1}{2})  
     \, L_1(0,1,0) L_2(0,-1,1,1) 
\\
& = L_1(0,1,0) L_2(0,0,1,1).
\end{align*} 

\medskip


Let us compute the contragredient zeta integral $ Z(s_1,s_2, \widetilde{W}, \widehat{\Phi}) $.
We first treat the case 2-(a). 
Using 
Proposition \ref{prop:contra_radial} and $ \widehat{\Phi} = \Phi $, we have 
\begin{align*}
 Z(s_1,s_2, \widetilde{W}, \widehat{\Phi}) 
& = \GR(2s_2-\gamma_1) \cdot (-\sqrt{-1})^{\kappa_1} 
     \sum_{0 \le j  \le \frac{\kappa_1}{2}} \binom{ \frac{\kappa_1}{2} }{j} 
    (-1)^{\kappa_1-2j} V_{\tilde{\sigma}, 2je_1+(\kappa_1-2j)e_3}(s_1,s_2,s_1+s_2). 
\end{align*}  
Thus in the same way as in the evaluation of $ Z(s_1,s_2, W, \Phi) $, we know that 
\begin{align*}
 Z(s_1,s_2, \widetilde{W}, \widehat{\Phi}) 
= (\sqrt{-1})^{\kappa_1}   L(s_1, \Pi_{\tilde{\sigma}}) L(s_2, \Pi_{\tilde{\sigma}}, \wedge^2) 
\end{align*}
as desired. 
The cases 2-(b) and (e) can be similarly done.

Let us consider the case 2-(c). 
Using Lemma \ref{lem:contra}, Proposition \ref{prop:contra_radial} and $ \widehat{\Phi} = \Phi $, we have
\begin{align*}
 Z(s_1,s_2, \widetilde{W}, \widehat{\Phi}) 
& = \GR(2s_2-\gamma_1) \cdot (-\sqrt{-1})^{\kappa_1}  
    \sum_{0 \le j \le \frac{\kappa_1-2}{2}} \binom{\frac{\kappa_1-2}{2}}{j} 
\\
& \quad  \times \{\sqrt{-1} \cdot \sqrt{-1}^{-1} (-1)^{\kappa_1-2j-1} 
    V_{\tilde{\sigma},  2je_1+(\kappa_1-2j-1)e_3 + e_4}(s_1+1,s_2,s_3)   \\
& \qquad -\sqrt{-1} \cdot \sqrt{-1}^{-1} (-1)^{\kappa_1-2j-2}
     V_{\tilde{\sigma}, (2j+1)e_1+(\kappa_1-2j-2)e_3 + e_4}(s_1, s_2+1, s_3)  \\
& \qquad + \sqrt{-1} \cdot \sqrt{-1}  (-1)^{\kappa_1-2j-2}
      V_{\tilde{\sigma},  (2j+1)e_1+e_2+ (\kappa_1-2j-2)e_3}(s_1, s_2, s_3+1)   \\
& \qquad -  (2\pi)^{-1} (\kappa_1-2j-2)  (-1)^{\kappa_1-2j-3}
      V_{\tilde{\sigma}, (2j+1)e_1+e_2+(\kappa_1-2j-3)e_3+e_4} (s_1,s_2,s_3) \}.
\end{align*}
Thus in the same way as in the evaluation of $ Z(s_1,s_2, W, \Phi) $, we know that 
\begin{align*}
 Z(s_1,s_2, \widetilde{W}, \widehat{\Phi}) 
= (\sqrt{-1})^{\kappa_1+2}  L(s_1, \Pi_{\tilde{\sigma}}) L(s_2, \Pi_{\tilde{\sigma}}, \wedge^2) 
\end{align*}
as desired.
The cases 2-(d) and (f) can be similarly done.

\subsection{Proof of Theorem \ref{thm:BF3}}

By Lemma \ref{lem:test3}, we know that
\begin{align*}
 W'(\hat{y}) 
& = 2^{-2} \sum_{0 \le i \le 3} 
  \int_0^{2\pi} \! \int_0^{2\pi}
   W \bigl( m_i \, \hat{y} \, \tilde{\iota}(\rk_{\theta_1}^{(2)}, \rk_{\theta_2}^{(2)} ) \bigr) 
   \exp(- \sqrt{-1} b\, \theta_2) \, \frac{d\theta_1}{2\pi} \frac{d\theta_2}{2\pi}
\\
& = \begin{cases}
   W( \hat{y}) & \mbox{cases 3-(a), (b)}, \\
  (-\sqrt{-1})^{\kappa_1+\delta_2-\delta_1} \varphi_{\sigma} (\delta_2 w_{2}+\delta_1 w_4) (\hat{y}) & \mbox{case 3-(c)}.
   \end{cases}
\end{align*}
Then Lemma \ref{lem:W'} implies that  
\begin{align} \label{eqn:lem_W'_case3}
 Z(s_1,s_2,W, \Phi_{(0,b)}) 
= \GR(2s_2+\gamma_1+b) V'(s_1,s_2,s_1+s_2).
\end{align}
Here we use the notation $ \gamma_1 = 2\nu_1+2\nu_2. $

\begin{lem} \label{lem:BFlem3}
For $ s_1,s_2 \in \bC $, and 
$ a_i, b_j, c_k, d_l \in \bC $
$ (1 \le i,k \le 7, 1 \le j \le 6, 1 \le l \le 4) $, we set 
\begin{align*}
& A(a_1,a_2,a_3,a_4,a_5,a_6; a_7) 
\\
&= \GR(2s_2+\gamma_1+a_1) \GC(s_1+\nu_1+\tfrac{\kappa_1-1}{2}+a_2) 
   \GR(s_2+2\nu_1+\kappa_1-\kappa_2-2j+ a_3)  \GR(s_2+2\nu_2+2j + a_4) 
 \\
& \quad \times 
 \frac{
  \GC(s_1+s_2+\nu_1+2\nu_2+\tfrac{\kappa_1-1}{2}+a_5)
  \GC(s_2+\nu_1+\nu_2+\tfrac{\kappa_1+\kappa_2}{2}-1+a_6)}
  { \GC(s_2+\nu_1+\nu_2+ \tfrac{\kappa_1-\kappa_2}{2}-1+a_7) },
\end{align*}
\begin{align*}
& B_0(b_1,b_2,b_3,b_4; b_5,b_6)\\
& =\frac{ \GR(s_1-q+2j+b_1) \GC(s_2-q+\nu_1+\tfrac{\kappa_1-1}{2}-\kappa_2+b_2) 
   \GR(s_1+s_2-q+2\nu_1+b_3) \GC(q+\nu_2+\tfrac{\kappa_2-1}{2}+b_4) }
     { \GR(s_1+s_2-q+2\nu_1+\kappa_1-\kappa_2+b_5)   \GR(s_1+2s_2-q+2\nu_1+2\nu_2+2j+b_6)},
\end{align*}
\begin{align*}
& C(c_1,c_2,c_3,c_4,c_5,c_6; c_7) \\
& = \GC(s_1+\nu_1+\tfrac{\kappa_1-1}{2}+c_1)  \GC(s_1+s_2+\nu_1+2\nu_2+\tfrac{\kappa_1-1}{2}+c_2)
    \GR(2s_2+2\nu_1+2\nu_2+\kappa_1-\kappa_2+c_3 )
\\
& \quad \times  
    \GR(s_2+2\nu_1+c_4 )\GR(s_2+2\nu_2+c_5 )
 \cdot \frac{ \GC(s_2+\nu_1+\nu_2+\tfrac{\kappa_1+\kappa_2}{2}-1+c_6)}{ \GC(s_2+\nu_1+\nu_2+ \tfrac{\kappa_1-\kappa_2}{2}-1+c_7) },
\end{align*}
\begin{align*}
 D_0(d_1,d_2,d_3; d_4)
& =
   \frac{ \GR(s_1-q+d_1) \GC(s_2-q+\nu_1+\tfrac{\kappa_1-1}{2}-\kappa_2 +d_2) \GC(q+\nu_2+\tfrac{\kappa_2-1}{2}+d_3)}{ 
   \GR(s_1+2s_2-q+2\nu_1+2\nu_2+\kappa_1-\kappa_2+d_4)}
\end{align*}
and 
\begin{align*}
 {B}(b_1,b_2,b_3,b_4;b_5,b_6) 
& = \frac{1}{4\pi \sqrt{-1}} \int_q B_0(b_1,b_2,b_3,b_4; b_5,b_6) \,dq,
\\
{D}(d_1,d_2,d_3; d_4) 
& = \frac{1}{4\pi \sqrt{-1}} \int_q D_0(d_1,d_2,d_3; d_4)\,dq.
\end{align*}

\medskip 
\noindent
(i) Let $ \delta \in \bZ_{\ge 0} $ such that $ \kappa_1-\kappa_2-\delta \in 2 \bZ_{\ge 0} $. 
If $ a_3+a_4+b_1+\delta = b_6 $, $ a_3+b_1+\delta = b_3 $ and $a_3+b_1=b_5 $, then we have 
\begin{align*}
& \sum_{0 \le j  \le \tfrac{\kappa_1-\kappa_2-\delta}{2} }\binom{ \tfrac{\kappa_1-\kappa_2-\delta}{2}}{j} 
   A(a_1,a_2,a_3,a_4,a_5,a_6; a_7) B(b_1,b_2,b_3,b_4; b_5,b_6)
\\
& = \frac{ \GR(2s_2+\gamma_1+a_1)}{ \GR(2s_2+\gamma_1+a_3+a_4+\delta) } 
 \cdot C(a_2,a_5,a_3+a_4, a_3+\delta, a_4,a_6;  a_7) 
  D(b_1,b_2,b_4; a_3+a_4+b_1). 
\end{align*}

\noindent
(ii) 
If $ d_4 = d_1+2d_2+2d_3  $, $ c_2 = d_1+d_2+2d_3 $, $ c_3= 2(d_2+d_3) $ and $ c_7 =d_2+d_3 $, 
then we have 
\begin{align*}
 C(c_1,c_2,c_3,c_4,c_5,c_6,c_7) {D}(d_1,d_2,d_3; d_4)
&= L_1(c_1, d_1+d_3) L_2( d_2+d_3, c_6, c_4,c_5).
\end{align*}
where
\begin{align*}
 L_1(\alpha_1,\alpha_2) & = \GC(s_1+\nu_1+\tfrac{\kappa_1-1}{2}+\alpha_1) \GC(s_1+\nu_2+\tfrac{\kappa_1-1}{2}+\alpha_2), 
\\
 L_2(\beta_1,\beta_2,\beta_3,\beta_4)
& = \GC(s_2+\nu_1+\nu_2+\tfrac{\kappa_1-\kappa_2}{2}+\beta_1)
    \GC(s_2+\nu_1+\nu_2+\tfrac{\kappa_1+\kappa_2-2}{2}+\beta_2) 
\\
& \quad \times \GR(s_2+2\nu_1+\beta_3) \GR(s_2+2\nu_2+\beta_4).
\end{align*} 
\end{lem}

\begin{proof} 
As is Lemma \ref{lem:BFlem2}, our claim follows from Lemmas \ref{lem:Barnes2nd} and \ref{lem:Saal}.
\end{proof}

\medskip

\noindent 
\underline{\bf Case 3-(a):} \ 
We know
\begin{align*} 
V'(s_1,s_2,s_3) 
& =\sum_{0 \le j  \le \frac{\kappa_1-\kappa_2}{2}} 
   \binom{\frac{\kappa_1-\kappa_2}{2}}{j} V_{\sigma, 2je_1+(\kappa_1-\kappa_2-2j)e_3+\kappa_2 e_{24}}(s_1, s_2,s_3).
\end{align*}
Then (\ref{eqn:lem_W'_case3}), Theorem \ref{thm:EF3} and Lemma \ref{lem:BFlem3} imply that  
\begin{align*}
   Z(s_1,s_2, W, \Phi)
& =  \sum_{0 \le j  \le \frac{\kappa_1-\kappa_2}{2}} \binom{ \frac{\kappa_1-\kappa_2}{2} }{j} A(0, 0, 0, 0, 0, 0; 0)  B(0, 0, 0, 0; 0, 0)
\\
& = C(0,0,0,0,0,0; 0)  D(0,0,0; 0) \\
& = L_1(0,0) L_2(0,0,0,0).
\end{align*}

\medskip 

\noindent 
\underline{\bf Case 3-(b):} \ 
As is the case 2-(f), we know
\begin{align*}
 W'(\hat{y}) 
& = (-\sqrt{-1})^{\kappa_1} (4\pi)^{-1} \{ R(2E_{1,2}) \varphi( w_{12} )(\hat{y}) 
 - \varphi(\tau_{(\kappa_1,\kappa_2,0)}(E_{1,2}^{\gk})  w_{12}  ) (y) \\
& \quad- R(2E_{2,3}) \varphi(w_{23} )(\hat{y}) + \varphi(\tau_{(\kappa_1,\kappa_2,0)}(E_{2,3}^{\gk})w_{23} ) (\hat{y})  \\
& \quad + R(2E_{3,4}) \varphi(w_{34} )(\hat{y}) - \varphi(\tau_{(\kappa_1,\kappa_2,0)}(E_{3,4}^{\gk})w_{34} ) (\hat{y})  \\
& \quad + R(2E_{1,4}) \varphi(w_{14} )(\hat{y}) - \varphi(\tau_{(\kappa_1,\kappa_2,0)}(E_{1,4}^{\gk})w_{14} ) (\hat{y})  \}.
\end{align*} 
Since
\begin{align*}
& -\tau_{(\kappa_1,\kappa_2,0)}(E_{1,2}^{\gk})  w_{12}  + \tau_{(\kappa_1,\kappa_2,0)}(E_{2,3}^{\gk})w_{23} 
 - \tau_{(\kappa_1,\kappa_2,0)}(E_{3,4}^{\gk}) w_{34}  - \tau_{(\kappa_1,\kappa_2,0)}(E_{1,4}^{\gk})w_{14}
\\
& = \sum_{0 \le j  \le \frac{\kappa_1-\kappa_2}{2}} \binom{\frac{\kappa_1-\kappa_2}{2}}{j}
  \{ 2(\kappa_1-\kappa_2-2j) u_{2j e_1+e_2+(\kappa_1-\kappa_2-2j-1)e_3+(\kappa_2-1) e_{24}+e_{23}} 
\\
& \qquad -2(\kappa_2-1) u_{2j e_1+(\kappa_1-\kappa_2-2j)e_3+(\kappa_2-2) e_{24}+e_{23}+e_{34}} \},
\end{align*}
we get 
\begin{align*}
 V'(s_1,s_2,s_3) 
& = \sum_{0 \le j  \le \frac{\kappa_1-\kappa_2}{2}}  \binom{\frac{\kappa_1-\kappa_2}{2}}{j}
\{  V_{\sigma, 2j e_1+(\kappa_1-\kappa_2-2j)e_3+(\kappa_2-1) e_{24} +e_{12}} (s_1+1,s_2,s_3) 
\\
& \quad + V_{\sigma,  2j e_1+(\kappa_1-\kappa_2-2j)e_3+(\kappa_2-1) e_{24} +e_{23}} (s_1,s_2+1,s_3)
\\
&  \quad  + V_{\sigma,  2j e_1+(\kappa_1-\kappa_2-2j)e_3+(\kappa_2-1) e_{24} +e_{34}}(s_1,s_2,s_3+1)
\\
&  \quad -  (2\pi)^{-1} (\kappa_1-\kappa_2-2j) 
    V_{\sigma,  2j e_1+e_2+(\kappa_1-\kappa_2-2j-1)e_3+(\kappa_2-1) e_{24} +e_{23}} (s_1,s_2,s_3)
\\
&  \quad  - (2\pi)^{-1} (\kappa_2-1)  
    V_{\sigma,  2j e_1+(\kappa_1-\kappa_2-2j)e_3+(\kappa_2-2) e_{24}  +e_{23}+e_{34}} ( s_1,s_2,s_3) \}.
\end{align*}
Then (\ref{eqn:lem_W'_case3}) and Theorem \ref{thm:EF3} imply that
\begin{align*}
 Z(s_1,s_2,W, \Phi)
& = Z_1 + Z_{2,1} + Z_{2,2} + Z_{3} + Z_{4,1} + Z_{4,2} + Z_{5,1} + Z_{5,2},
\end{align*}
where
\begin{align*}
 Z_1 & = \sum_{0 \le j  \le \frac{\kappa_1-\kappa_2}{2}}  \binom{\frac{\kappa_1-\kappa_2}{2}}{j}
 A(0, 1, 1, 1, 0, 0; 1) {B}(1,1,0,0; 2,1), 
\\
 Z_{2,1} & = \sum_{0 \le j  \le \frac{\kappa_1-\kappa_2}{2}}  \binom{\frac{\kappa_1-\kappa_2}{2}}{j}
  A(0,0,1,1,0,1;2) {B}(0,2,1,0; 1,2), 
\\
 Z_{2,2} & = \sum_{0 \le j  \le \frac{\kappa_1-\kappa_2}{2}}   \binom{\frac{\kappa_1-\kappa_2}{2}}{j}
   A(0,0,1,1,0,1;2) {B}(1,2,0,0;2,1), 
\\
 Z_3 & =  \sum_{0 \le j  \le \frac{\kappa_1-\kappa_2}{2}}   \binom{\frac{\kappa_1-\kappa_2}{2}}{j}
     A(0,0,1,1,1,0; 1) {B} (0,1,1,0;1,2), 
\\
 Z_{4,1} & = \frac{-\kappa_1+\kappa_2}{2\pi}
    \sum_{0 \le j  \le \frac{\kappa_1-\kappa_2-2}{2}}   \binom{\frac{\kappa_1-\kappa_2-2}{2}}{j} 
     A(0,0,-1,1,0,0; 1) {B}(0,1,1,0;-1,2), 
\\
 Z_{4,2} & =  \frac{-\kappa_1+\kappa_2}{2\pi}
   \sum_{0 \le j  \le \frac{\kappa_1-\kappa_2-2}{2}}   \binom{\frac{\kappa_1-\kappa_2-2}{2}}{j}   
     A(0,0,-1,1,0,0; 1) {B}(1,1,0,0; 0,1), 
\\
 Z_{5,1} & = \frac{-\kappa_2+1}{2\pi} 
    \sum_{0 \le j  \le \frac{\kappa_1-\kappa_2}{2}}  \binom{\frac{\kappa_1-\kappa_2}{2}}{j} 
    A(0,0,1,1,0,0;2) {B} (0,2,1,0;1,2), 
\\
 Z_{5,2} & = \frac{-\kappa_2+1}{2\pi}
  \sum_{0 \le j  \le \frac{\kappa_1-\kappa_2}{2}}   \binom{\frac{\kappa_1-\kappa_2}{2}}{j} 
   A(0,0,1,1,0,0;2) {B}(1,2,0,0;2,1). 
\end{align*}
Here we used 
$ ( -\kappa_1+\kappa_2+2j) \binom{\frac{\kappa_1-\kappa_2}{2}}{j}
  = (-\kappa_1+\kappa_2) \binom{\frac{\kappa_1-\kappa_2-2}{2}}{j}  $
in $ Z_{4,1} $ and $ Z_{4,2} $. 
Then Lemma \ref{lem:BFlem3} (i) implies that 
\begin{align*}
 Z_{2,1} & = \frac{\GR(2s_2+\gamma_1)}{\GR(2s_2+\gamma_1+2)} \cdot 
      C(0,0,2,1,1,1; 2) {D}(0,2,1; 2), 
\\
  Z_3 & = \frac{\GR(2s_2+\gamma_1)}{\GR(2s_2+\gamma_1+2)} \cdot 
    C(0,1,2,1,1,0; 1)   {D}(0,1,0; 2), 
\\
 Z_{4,1} & = \frac{\GR(2s_2+\gamma_1)}{\GR(2s_2+\gamma_1+2)} \cdot 
 \frac{ -\kappa_1+\kappa_2}{2\pi } \cdot C(0,0,0,1,1,0; 1) { D}(0,1,0; 0), 
\\
 Z_{5,1} & = \frac{\GR(2s_2+\gamma_1)}{\GR(2s_2+\gamma_1+2)} \cdot \frac{-\kappa_2+1}{2\pi}  
   \cdot  C(0,0,2,1,1,0; 2) {D}(0,2,1; 2).
\end{align*}
As is the case 2-(f) we can see that 
\begin{align*}
 Z_{2,1}+ Z_3 + Z_{4,1}+  Z_{5,1}
& = C(0,0,0,1,1,0;1) { D}(0,1,0;0).
\end{align*}
For the sum $ Z_{1} + Z_{2,2} + Z_{3} + Z_{4,2} + Z_{5,2} $, in view of
\begin{align*}
& A(0,1,1,1,0,0;1) B(1,1,0,0; 2,1) + A(0,0,1,1,0,1;2) B(1,2,0,0; 2,1)
\\
& + (2\pi)^{-1} (-\kappa_1+\kappa_2+2j) A(0,0,-1,1,0,0; 1) B(1,1,0,0; 0,1)  \\
& + (2\pi)^{-1} (-\kappa_2+1) A(0,0,1,1,0,0; 2) B(0,2,1,0; 1,2) \\
& = (2\pi)^{-1} (s_2+2\nu_1-1) A(0,0,-1,1,0,0; 1) B(1,1,0,0; 0,1),
\end{align*}
we can use Lemma \ref{lem:BFlem3} (i) to get
\begin{align*}
  Z_{1}+Z_{2,2}+ Z_{4,2}+ Z_{5,2}
& = (2\pi)^{-1} (s_2+2\nu_1-1) C(0,0,0,-1,1,0; 1) {D}(1,1,0; 1) \\
& = C(0,0,0,1,1,0; 1)  {D}(0,0,1; 0).
\end{align*}
Here we substituted $ q \to q+1 $.
Therefore Lemma \ref{lem:BFlem3} (ii) implies that
\begin{align*}
 Z(s_1,s_2,W,\Phi) 
&=  C(0,0,0,1,1,0;1)  \{ {D}(0,1,0; 0) + {D}(0,0,1; 0) \}  
\\
& = C(0,0,0,1,1,0; 0) { D}(0,0,0; 0) \\
& =  L_1(0,0) L_2(0,0,1,1).
\end{align*}

\medskip 

\noindent 
\underline{\bf Case 3-(c):} \ 
We know
\begin{align*}
 V'(s_1,s_2,s_3) 
& =  \sum_{0 \le j \le \frac{\kappa_1-\kappa_2-1}{2} } \binom{\frac{\kappa_1-\kappa_2-1}{2}}{j} 
     V_{\sigma,  2je_1+\delta_2 e_2 + (\kappa_1-\kappa_2-2j-1)e_3  + \delta_1 e_4+\kappa_2 e_{24} }(s_1,s_2,s_3).
\end{align*}
Then (\ref{eqn:lem_W'_case3}) and Theorem \ref{thm:EF3} and Lemma \ref{lem:BFlem3} lead us that
\begin{align*}
Z(s_1,s_2, W, \Phi)
& = \sum_{0 \le j \le \frac{\kappa_1-\kappa_2-1}{2} }  
    \binom{ \frac{\kappa_1-\kappa_2-1}{2} }{j}  A(1, 0, -\delta_2, \delta_2, 0, 0; 0)
   {B} (0, 0, \delta_1, 0; -\delta_2, 1) 
\\
& =  C(0, 0, 0, \delta_1, \delta_2, 0; 0) {D}(0,0,0; 0) \\
& = L_1(0,0) L_2(0,0, \delta_1,\delta_2).
\end{align*}


\medskip 

The evaluation of the contragredient zeta integral $Z(s_1,s_2, \widetilde{W}, \widehat{\Phi}) $ can be 
done in the same way as in the case 2. For example, in the case 3-(c), we have 
\begin{align*}
 Z(s_1,s_2, \widetilde{W}, \widehat{\Phi})
& = \GR(2s_2-\gamma_1) \cdot (-\sqrt{-1})^{\kappa_1+\delta_2-\delta_1} (\sqrt{-1})
\\
& \quad \times \sum_{0 \le j \le \frac{\kappa_1-\kappa_2-1}{2} } \binom{\frac{\kappa_1-\kappa_2-1}{2}}{j} 
     (\sqrt{-1})^{\delta_2-\delta_1-\kappa_2} (-1)^{\kappa_2+(\kappa_1-\kappa_2-2j-1)}
\\
& \quad \times 
     V_{\tilde{\sigma},  2je_1+\delta_2 e_2 + (\kappa_1-\kappa_2-2j-1)e_3  + \delta_1 e_4+\kappa_2 e_{24} }(s_1,s_2,s_3)
\\
& = (\sqrt{-1})^{\kappa_1 + \kappa_2} \cdot (\sqrt{-1})^{2\kappa_1+1} 
   L(s_1, \Pi_{\tilde{\sigma}}) L(s_2, \Pi_{\tilde{\sigma}}, \wedge^2)
\end{align*}
as desired. 

\def\cprime{$'$}


\end{document}